\documentclass[11pt,reqno,twosdie]{amsart}

\usepackage[asymmetric,top=3.5cm,bottom=4.3cm,left=3.1cm,right=3.1cm]{geometry}
\usepackage{amsmath,amsfonts,amsthm,mathrsfs,amssymb}%,cite}
\usepackage[usenames]{color}
\usepackage{booktabs} % for much better looking tables

\usepackage{mathtools} %for \mathclap, \mathrlap, \smashoperator etc
\usepackage{graphics}
\usepackage{framed}

\usepackage{enumerate}
\usepackage{subfigure} % make it possible to include more than one captioned figure/table in a single float

\usepackage{graphicx}
\usepackage{latexsym} 
\usepackage{enumerate}
\theoremstyle{plain}
\newtheorem{theorem}{Theorem}[section]
\newtheorem*{theorem*}{Theorem}
\newtheorem{lemma}[theorem]{Lemma}

\newtheorem{proposition}[theorem]{Proposition}
\newtheorem*{proposition*}{Proposition}

\newtheorem*{conjecture*}{Conjecture}

\theoremstyle{definition}

\theoremstyle{remark}

\renewcommand{\eqref}[1]{\textnormal{(\ref{#1})}}

\numberwithin{equation}{section}

\newcommand{\rmi}{\mathrm{i}}

\newcommand{\Rcal}{\mathcal{R}}
\newcommand{\Qcal}{\mathcal{Q}}

\title[Localization and Geometrization in Plasmon Resonances]{Localization and geometrization in plasmon resonances and geometric structures of Neumann-Poincar\'e eigenfunctions}

\author{Eemeli Bl{\aa}sten}
\address{Helsinki, Finland.}
\email{eemeli@countermail.com}

\author{Hongjie Li}
\address{Department of Mathematics, Hong Kong Baptist University, Kowloon, Hong Kong SAR.}
\email{hongjie\_li@yeah.net}

\author{Hongyu Liu}
\address{Department of Mathematics, Hong Kong Baptist University, Kowloon, Hong Kong SAR.}
\email{hongyu.liuip@gmail.com, hongyuliu@hkbu.edu.hk}

\author{Yuliang Wang}
\address{Department of Mathematics, Hong Kong Baptist University, Kowloon, Hong Kong SAR.}
\email{yuliang@hkbu.edu.hk}

\begin{document}

\begin{abstract}

This paper reports some novel and intriguing discoveries about the localization and geometrization phenomenon in plasmon resonances and the intrinsic geometric structures of 
Neumann-Poincar\'e eigenfunctions. It is known that plasmon resonance generically occurs in the quasi-static regime where the size of the plasmonic inclusion is sufficiently small 
compared to the wavelength. In this paper, we show that the global smallness condition on the plasmonic inclusion can be replaced by a local high-curvature condition, and the plasmon
resonance occurs locally near the high-curvature point of the plasmonic inclusion. We provide partial theoretical explanation and link with the geometric structures of the Neumann-
Poincar\'e (NP) eigenfunctions. The spectrum of the Neumann-Poincar\'e operator has received significant attentions in the literature. We show for the first time some intrinsic geometric 
structures of the Neumann-Poincar\'e eigenfunctions near high-curvature points.

 

\end{abstract}

\maketitle

\section{Introduction}\label{sec:Intro}

There is considerable interest in the mathematical study of plasmon materials in recent years. Plasmon materials are a type of metamaterials that are artificially engineered to allow the 
presence of negative material parameters. We refer to \cite{ADM,AMRZ,AKL,KLO}, \cite{Ack13,Ack14,ARYZ,Bos10,Brl07,CKKL,Klsap,LLL,GWM1,GWM2,GWM3,GWM4,GWM6,GWM7,GWM8,GWM9,Pen1,Pen2,Ves} 
and \cite{AKKY,AKKY2,AKM1,AKM2,DLL,LiLiu2d,LiLiu3d,LLL2,KM,LLBW} and the references therein for the relevant studies in acoustics, 
electromagnetism and elasticity, respectively.

One peculiar and intriguing phenomenon associated with the plasmon materials is the so-called anomalous resonance \cite{GWM3}. Mathematically, the plasmon resonance is associated
to the infinite dimensional kernel of a certain non-elliptic partial differential operator (PDO). In fact, the presence of negative material parameters breaks the ellipticity of the underlying 
partial differential equations (PDEs) that govern the various physical phenomena. Consequently, the non-elliptic PDO may possess a nontrivial kernel, which in turn may induce various
resonance phenomena due to appropriate external excitations. In \cite{Ack13}, applying techniques from the layer potential theory, the plasmon resonance is connected to the spectrum of the
 classical Neumann-Poincar\'e (NP) operator. Indeed, the aforementioned nontrivial kernel function of the underlying non-elliptic PDO can be represented as a single-layer potential. In 
 order for plasmon resonance to occur, the density function of the above single-layer potential has to be an eigenfunction of the corresponding Neumann-Poincar\'e operator. In such a
 way, the plasmon parameters are also connected to the eigenvalues of the corresponding NP operator in a delicate way. The spectral properties of the NP operator were recently
 extensively investigated in the literature \cite{AKM1,DLL,HKL,HP,JK18,KK18,KLY,s25}. However, the corresponding studies are mainly concerned with the spectra of various NP operators
 in different geometric or physical setups. In this paper, we discover certain intrinsic geometric structures of the NP eigenfunctions. In fact, it is shown that the NP eigenfunctions as well as
 the associated single-layer potentials possess certain curvature-dependent behaviours locally near a boundary point. To our best knowledge, this is the first study in the literature on the
 intrinsic geometric properties of the NP eigenfunctions. The geometric results can be used to provide theoretical explanation of the localization and geometrization phenomenon in
 plasmon resonances, which is another novel and intriguing discovery in this paper, and also one of the major motivations for the investigation of the geometric structures of NP
 eigenfunctions. 

The localization and geometrization in wave scattering were discovered and proposed in \cite{BL}. It states that if a certain wave scattering phenomenon occurs associated with a small
object compared to the wavelength, then the similar phenomenon occurs for a ``big" object but locally near a high-curvature boundary point. Noting that the global smallness condition
means that the curvature is intrinsically high everywhere and hence the introduction of a local high-curvature condition is a natural one for the occurrence of the local scattering behaviour.
In \cite{BL}, the localization and geometrization phenomena were shown and justified in several time-harmonic scattering scenarios. In this paper, we show that the same principle actually
holds for the plasmon resonances. In fact, in many of the existing studies on plasmon resonances, the quasi-static approximation has played a critical role where the plasmonic inclusion is
of a size much smaller than the wavelength. There are also several studies that go beyond the quasi-static limit \cite{GWM3,KLO,s25,LLLW,Ngu2}. In \cite{GWM3}, double negative
materials are employed in the shell and in \cite{Ngu2}, in addition to the employment of double negative materials, a so-called double-complementary medium structure is incorporated into
the construction of the plasmonic device. In \cite{KLO}, it is actually shown that resonance does not occur for the
classical core-shell plasmonic structure without the quasistatic approximation as long as the core and shell are strictly convex. In \cite{s25,LLLW}, in order for the plasmon
resonances to occur beyond the quasi-static approximation, the corresponding plasmonic configuration has to be designed in a subtle and delicate way. Nevertheless, we show that for
a plasmonic structure that is resonant in the quasi-static regime but non-resonant out of the quasi-static regime, the resonance always occurs locally near a high-curvature boundary point of
the plasmonic inclusion. That is, the localization and geometrization phenomenon occurs for the plasmon resonances. 
This is mainly demonstrated by certain generic numerical examples. To seek a theoretical explanation, it naturally leads to the investigation of the geometric properties of the NP
eigenfunctions as well as the associated single-layer potentials near a high-curvature boundary point. 

The focus of our study to is present the novel and intriguing discoveries on the localization and geometrization in plasmon resonances as well as the intrinsic geometric structures of the NP
eigenfunctions. We present our results mainly for the two-dimensional case though the extension to the three-dimensional case is also appealing. Moreover, in addition to the theoretical
analysis, we resort to extensive numerical experiments in our study. 

The rest of the paper is organized as follows. In Sections 2 and 3, we briefly discuss the plasmon resonances in the electrostatic and quasi-static cases. Section 4 presents the localization
and geometrization phenomenon in the plasmon resonances. In Section 5, we investigate the geometric structures of the NP eigenfunctions. The paper is concluded in Section 6 with some
relevant discussions.

\section{Plasmon resonance in electrostatics and spectral system of NP operator}

Let $D$ be a bounded domain in $\mathbb{R}^2$ with a $C^2$-smooth boundary $\partial D$ and a connected complement $\mathbb{R}^2\backslash\overline{D}$. Consider a dielectric
medium configuration as follows,
\begin{equation}\label{eq:config1}
\epsilon_\delta(x)=\begin{cases}
\epsilon_c+\mathrm{i}\delta,\quad x\in D,\\
\ \ 1,\hspace*{1cm} x\in\mathbb{R}^2\backslash\overline{D},
\end{cases}
\end{equation}
where $\epsilon_c\in\mathbb{R}_-$ and $\delta\in\mathbb{R}_+$. Let $u\in H_{loc}^1(\mathbb{R}^2)$ signify the electric field associated with the medium configuration \eqref{eq:config1},
 and it satisfies the following PDE system,
\begin{equation}\label{eq:elect1}
\begin{cases}
& \nabla\cdot(\epsilon_\delta(x)\nabla u_{\delta}(x))=f(x),\quad x\in\mathbb{R}^2,\medskip\\
& \displaystyle{u(x)=\mathcal{O}\left(\frac{1}{|x|}\right)}\quad\mbox{as\ $|x|\rightarrow \infty$},
\end{cases}
\end{equation} 
where $f\in L^2(\mathbb{R}^2)$ is compactly supported in $\mathbb{R}^2\backslash\overline{D}$ and 
\[
\int_{\mathbb{R}^2}\ f(x)\ dx=0. 
\] 
Associated with the electrostatic system \eqref{eq:elect1}, the configuration $(\epsilon_\delta, f)$ is said to be {\it resonant} if there holds
\begin{equation}\label{eq:res1}
\mathbf{E}_\delta(\epsilon_\delta, f):=\frac \delta 2\int_{D} |\nabla u_\delta|^2\ dx\rightarrow \infty\quad \mbox{as}\ \ \delta\rightarrow +0. 
\end{equation}
The condition \eqref{eq:res1} indicates that if plasmon resonance occurs, then highly oscillating behaviours are exhibited by the
resonant field around the plasmon inclusion. Mathematically, the resonance is induced by the nontrivial kernel of the non-elliptic PDO (partial differential operator) 
\begin{equation}\label{eq:pdo1}
L_{\epsilon_0} u:=\nabla\cdot(\epsilon_0\nabla u),
\end{equation} 
where $\epsilon_0$ is $\epsilon_\delta$ with $\delta$ formally taken to be zero. The kernel of $L_{\epsilon_0}$ consists of nontrivial functions satisfying
\begin{equation}\label{eq:kernel1}
u\in H_{loc}^1(\mathbb{R}^2);\quad L_{\epsilon_0} u(x)=0,\quad x\in \mathbb{R}^2;\quad u(x)=\mathcal{O}\left(\frac{1}{|x|}\right)\ \mbox{as}\ |x|\rightarrow\infty. 
\end{equation}
It is noted that $\epsilon_c$ is allowed to be negative, and hence the PDO $L_{\epsilon_0}$ is a non-elliptic operator. Therefore, if the plasmon constant $\epsilon_c$ is properly chosen, 
$\mathrm{Ker}(L_{\epsilon_0})$ as defined in \eqref{eq:kernel1} can be nonempty, which in turn can induce resonance  as described in \eqref{eq:res1} for a properly chosen external 
source $f$.

The connection to the spectral system of the Neumann-Poincar\'e operator can be described as follows. By the layer-potential theory, one seeks a solution to \eqref{eq:kernel1} of the 
following form
\begin{equation}\label{eq:s1}
u(x)=S_{\partial D}[\varphi](x),\quad \varphi\in L_0^2(\partial D),
\end{equation}
where $L_0^2(\partial D)$ is the space of square integrable functions with zero average on $\partial D$, and $S_{\partial D}[\varphi]$ is the singe-layer operator defined as
\begin{equation}\label{eq:singlayer1}
S_{\partial D}[\varphi](x):=\int_{\partial D} G(x-y)\varphi(y)ds(y), \quad x\in\mathbb{R}^2,
\end{equation}
with
\begin{equation}\label{eq:fun0}
 G(x)=\frac{1}{2\pi}\ln |x|,
\end{equation}
being the fundamental solution of the Laplace operator in two dimensions. On the boundary $\partial D$, the single layer potential enjoys the following jump relationship
\begin{equation}\label{ju}
\partial_{\nu} S_{\partial D}[\varphi]|_{\pm}(x)=\left( \pm \frac{1}{2} +  K_{\partial D}^*\right)[\varphi](x), \quad x\in\partial D,
\end{equation}
where $\partial_{\nu}$ is the outward unit normal to $\partial D$ and $\pm$ indicate the limits to $\partial D$ from outside and inside of $D$, respectively. In \eqref{ju}, the operator $K_{\partial D}^*$ is defined as
\begin{equation}\label{K0}
 K_{\partial D}^*[\varphi](x)=\frac{1}{2\pi} \int_{\partial D} \frac{\langle x-y, \nu_x\rangle }{|x-y|^2} \varphi(y) ds(y), \quad x\in \partial D,
\end{equation}
which is called the Neumann-Poincar\'e (NP) operator. By matching the transmission conditions across the boundary,
\begin{equation}
 u|_-=u|_+,\quad
 \epsilon_c  \partial_{\nu}u|_-=\partial_{\nu}u|_+,
\end{equation}
and with the help of the jump formula \eqref{ju}, solving the system \eqref{eq:kernel1} is equivalent to solving the following problem 
\begin{equation}\label{eq:spc1}
 K_{\partial D}^*[\varphi](x)=\frac{\epsilon_c+1}{2(\epsilon_c-1)} \varphi(x), \quad x\in\partial D.
\end{equation}

Clearly, according to our discussion made above, for the occurrence of the plasmon resonances, there are two critical conditions to be fulfilled from a spectral perspective 
associated with the NP operator defined in \eqref{K0}. First, the plasmon constant $\epsilon_c$ should be properly chosen such that the parameter $\lambda(\epsilon_c)$ defined as
\begin{equation}\label{eq:lambda1}
\lambda(\epsilon_c) =\frac{\epsilon_c+1}{2(\epsilon_c-1)} ,
\end{equation}
belong to the spectrum of the NP operator $K^*$. Second, the single layer potential given in \eqref{eq:s1} associated with
the NP eigenfunction $\varphi$ according to \eqref{eq:spc1} should exhibit certain highly 
oscillating behaviours around the plasmonic inclusion. The second condition naturally leads to the investigation of the structures of the NP eigenfunctions. 

\section{Plasmon resonance for small inclusions: quasi-static approximation}

In this section we consider the plasmon resonance for the wave scattering in the quasi-static regime. That is, the size of the plasmonic inclusion is much smaller than the underlying wavelength. To that end, we let $\Omega$ be a bounded domain in $\mathbb{R}^2$ with a $C^2$-smooth boundary $\partial \Omega$ and a connected complement $\mathbb{R}^2\backslash\overline{\Omega}$. Set $D= s\Omega$, where $s\in\mathbb{R}_+$ signifies a scaling parameter. Introduce the following plasmonic configuration, 
\begin{equation}\label{eq:q1}
\epsilon_{\mathcal{D},\delta}(x)=\begin{cases}
\epsilon_{c}+\mathrm{i}\delta,\quad x\in \mathcal{D},\\
\ \ 1,\hspace*{1cm} x\in\mathbb{R}^2\backslash\overline{\mathcal{D}},
\end{cases}
\end{equation}
where $\mathcal{D}=D$ or $\Omega$. Associated with the medium configuration \eqref{eq:q1} in $D$, the wave scattering is governed by the following Helmholtz system 
\begin{equation}\label{ge1}
\begin{cases}
 \nabla\cdot(\epsilon_{{D},\delta}(x)\nabla u_{\delta}(x)) + k^2 u_{\delta}(x)=f(x), \quad x\in\mathbb{R}^2,\\
\displaystyle{\lim_{|x|\rightarrow \infty}|x|^{1/2}\left( \frac{x}{|x|}\cdot\nabla u_{\delta} - \rmi k u_{\delta}\right)  \rightarrow 0\  \ \mbox{as} \quad |x|\rightarrow \infty,}
 \end{cases}
\end{equation}
where $k\in\mathbb{R}_+$ signifies a wavenumber and $f(x)$ is an external source that is compactly supported in $\mathbb{R}^2\backslash \overline{{D}}$. The last limit in \eqref{ge1} is referred to as the Sommerfeld radiation condition. \eqref{ge1} describes the transverse electromagnetic wave scattering (cf. \cite{s25}). 

Similar to the electrostatic case, if \eqref{eq:res1} occurs for the wave field in \eqref{ge1}, the configuration is said to be resonant. In order to study the plasmon resonance associated with the Helmholtz system \eqref{ge1}, by a straightforward scaling argument, the PDE system \eqref{ge1} can be transformed to 
\begin{equation}\label{ge2}
\nabla\cdot(\epsilon_{\Omega,\delta}(x)\nabla v_{\delta}(x)) + s^2k^2 v_{\delta}(x)= \tilde{f}(x), \quad x\in\mathbb{R}^2,
\end{equation}
where $v_\delta(x)=u_{\delta}(x/s)$ and $\tilde f(x)=f(x/s)$. In what follows, we introduce the following PDO, 
\begin{equation}\label{eq:pdo2}
(L_{\epsilon_{\Omega,0}} + s^2k^2) u:=\nabla\cdot(\epsilon_{\Omega,0}\nabla u) + s^2k^2 u.
\end{equation} 
Similar to our discussion in the previous section, for the occurrence of the plasmon resonance, one needs to determine a nontrivial kernel of $L_{\epsilon_{\Omega,0}} + s^2k^2$ associated with a proper choice of $\epsilon_c$. To that end, we introduce
\begin{align}
S_{\partial D}^k[\varphi](x):=& \int_{\partial D} G^k(x-y)\varphi(y)ds(y), \quad x\in\mathbb{R}^2,\label{eq:slk}\\
(K_{\partial D}^k)^*[\varphi](x):=& \int_{\partial D} \partial_{\nu_x}G^k(x-y)\varphi(y)ds(y)  \quad x\in\partial D,\label{eq:dlk}
\end{align}
where  
\[
 G^k(x)=-\frac{\rmi}{4} \mathrm{H}_0^1(k|x|),
\]
with $\mathrm{H}_0^1(t)$ zeroth-order Hankel function of the first kind. $S_{\partial D}^k$ and $(K_{\partial D}^k)^*$ are, respectively, the single-layer potential and
the  NP operator with a finite frequency $k\in\mathbb{R}_+$. According to our earlier discussion in Section 2, in order to study the plasmon resonance associated with \eqref{eq:q1}, 
it suffices to investigate the nontrivial kernel of the PDO $L_{\epsilon_{\Omega,0}} + s^2k^2$. Similar to the electrostatic case, by using the layer-potential techniques, the study is 
reduced to analyzing the spectral system of the NP operator $(K_{\partial \Omega}^{sk})^*$ and the highly oscillating behaviours of the single-layer potentials 
$S_{\partial \Omega}^{sk}[\varphi]$ with $\varphi$ being the NP eigenfunctions. Imposing the quasi-static condition,
\begin{equation}\label{eq:qs1}
s\cdot k\ll 1,
\end{equation} 
one has that
\begin{equation}\label{eq:fs1}
 -\frac{\rmi}{4} \mathrm{H}_0^1(sk|x|)=\frac{1}{2\pi}\ln|x| + \tau+\sum_{n=1}^{\infty}(b_n \ln(sk|x|)+c_n)(sk|x|)^{2n},
\end{equation}
where
\begin{equation}\label{eq:parameters1}
\begin{split}
& b_n=\frac{(-1)^n}{2\pi}\frac{1}{2^{2n}(n!)^2}, \quad c_n=-b_n\left( \gamma -\ln2-\frac{\pi \rmi}{2}-\sum_{j=1}^n\frac{1}{j} \right),\\
& \qquad\qquad \tau=\frac{1}{2\pi} (\ln(sk) + \gamma -\ln 2) -\frac{\rmi}{4},
\end{split}
\end{equation}
with $\gamma$ the Euler constant. Therefore one can derive the following asymptotic expansions
\begin{equation}\label{eq:ask}
S_{\partial \Omega}^{sk} = S_{\partial \Omega} +\tau \langle \cdot,1\rangle +(sk)^2\ln (sk) \Rcal^{sk},
\end{equation}
where $\Rcal^{sk}$ is a bounded operator from $L^2(\partial D)$ to $H^1(\partial D)$ and 
\begin{equation}\label{eq:akk}
  (K_{\partial \Omega}^{sk})^*=K_{\partial \Omega}^*+(sk)^2\ln (sk) \Qcal^{sk},
\end{equation}
where the operator $\Qcal^{sk}$ is a bounded operator from $L^2(\partial \Omega)$ to itself. Hence, under the quasi-static approximation \eqref{eq:qs1}, the plasmon resonance for the Helmholtz system \eqref{ge1} again relies on the spectral properties of the NP operator $K_{\partial \Omega}^*$ and the oscillating behaviours of the corresponding single-layer potentials that are same to the electrostatic case. In fact, it is rigorously justified in \cite{AMRZ, AKL} that under \eqref{eq:qs1} and 
\begin{equation}\label{eq:qs2}
 s^2 |\ln s| \delta^{-1} \ll1,
\end{equation}
the Helmholtz system \eqref{ge1} is resonant for $\epsilon_c$ chosen from the resonant electrostatic case. Instead of discussing more theoretical details about the plasmon resonance within the quasi-static approximation, we next present several numerical examples for demonstration. 

\begin{figure}[t]
\centering
\subfigure[]{
\includegraphics[width=0.2\textwidth]{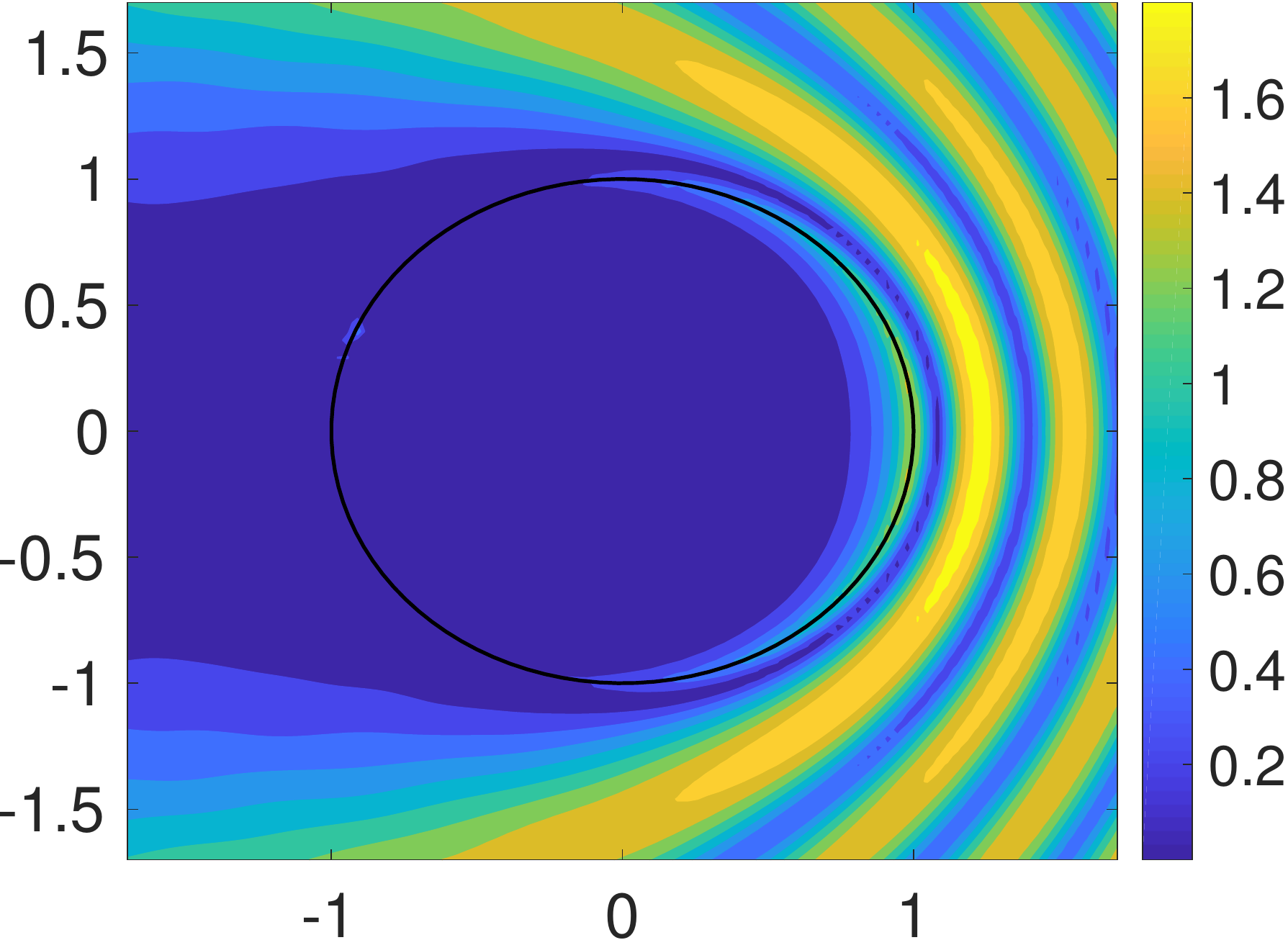}}
\subfigure[]{
\includegraphics[width=0.2\textwidth]{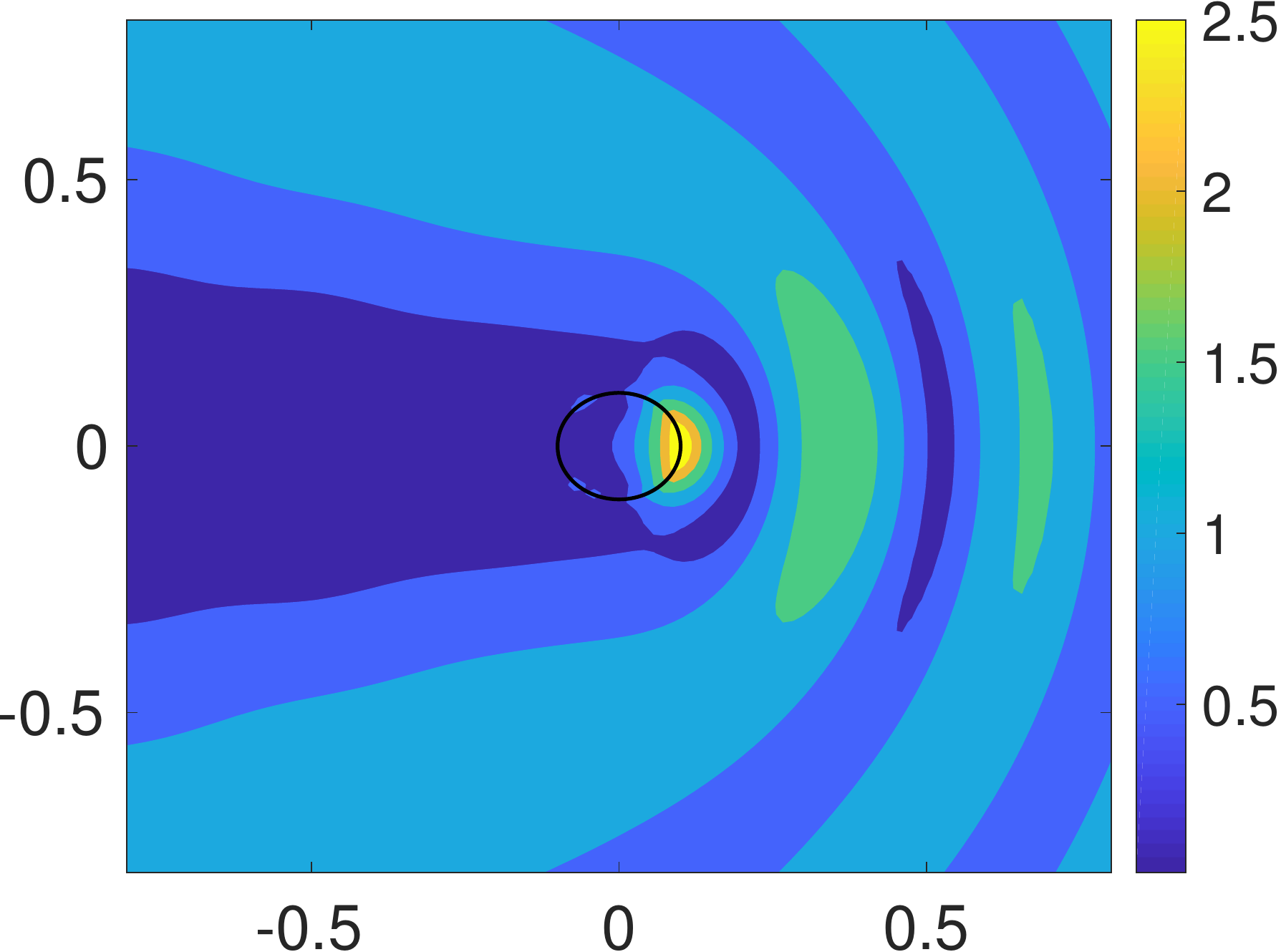}}
\subfigure[]{
\includegraphics[width=0.2\textwidth]{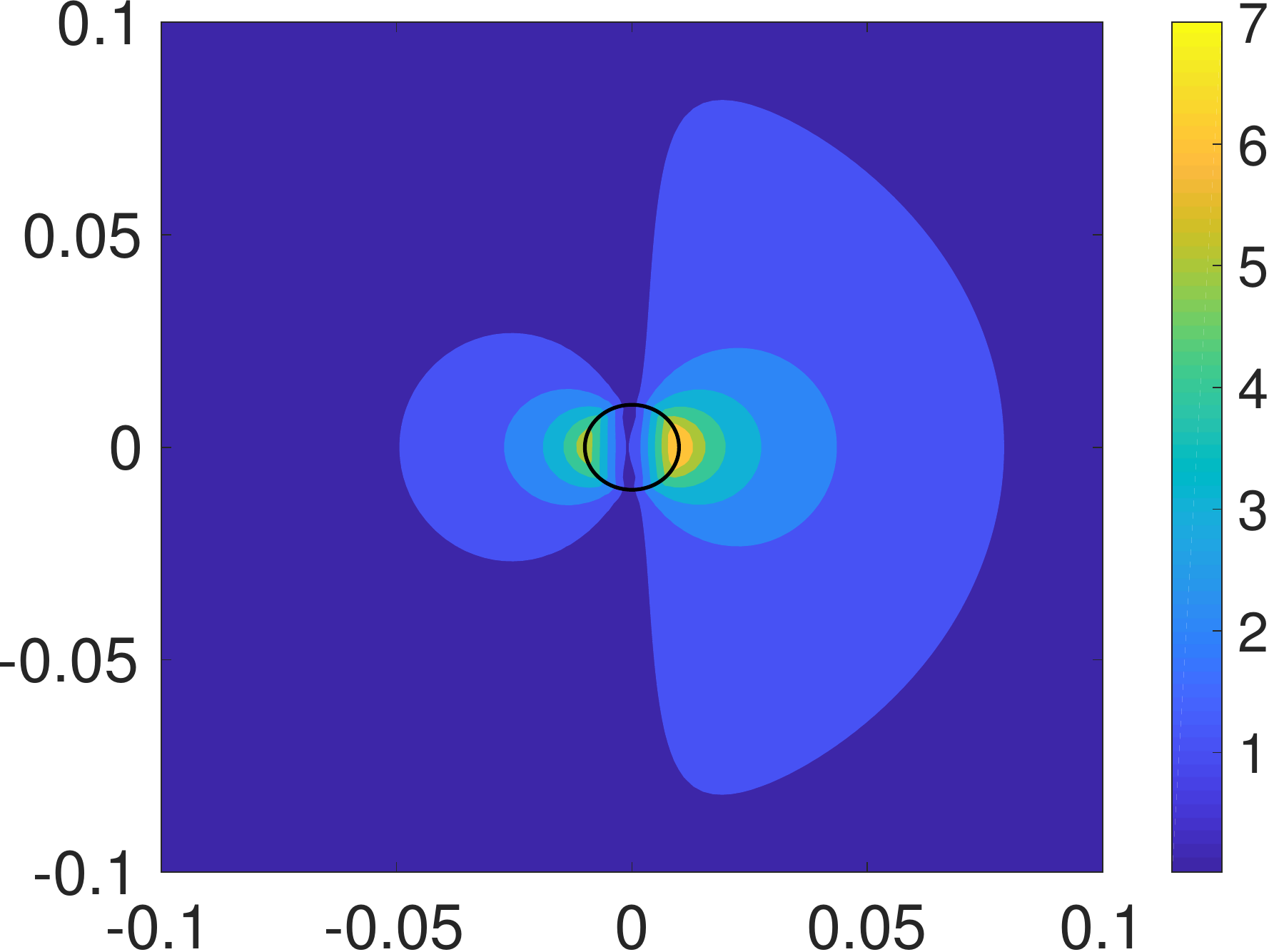}}
\subfigure[]{
\includegraphics[width=0.2\textwidth]{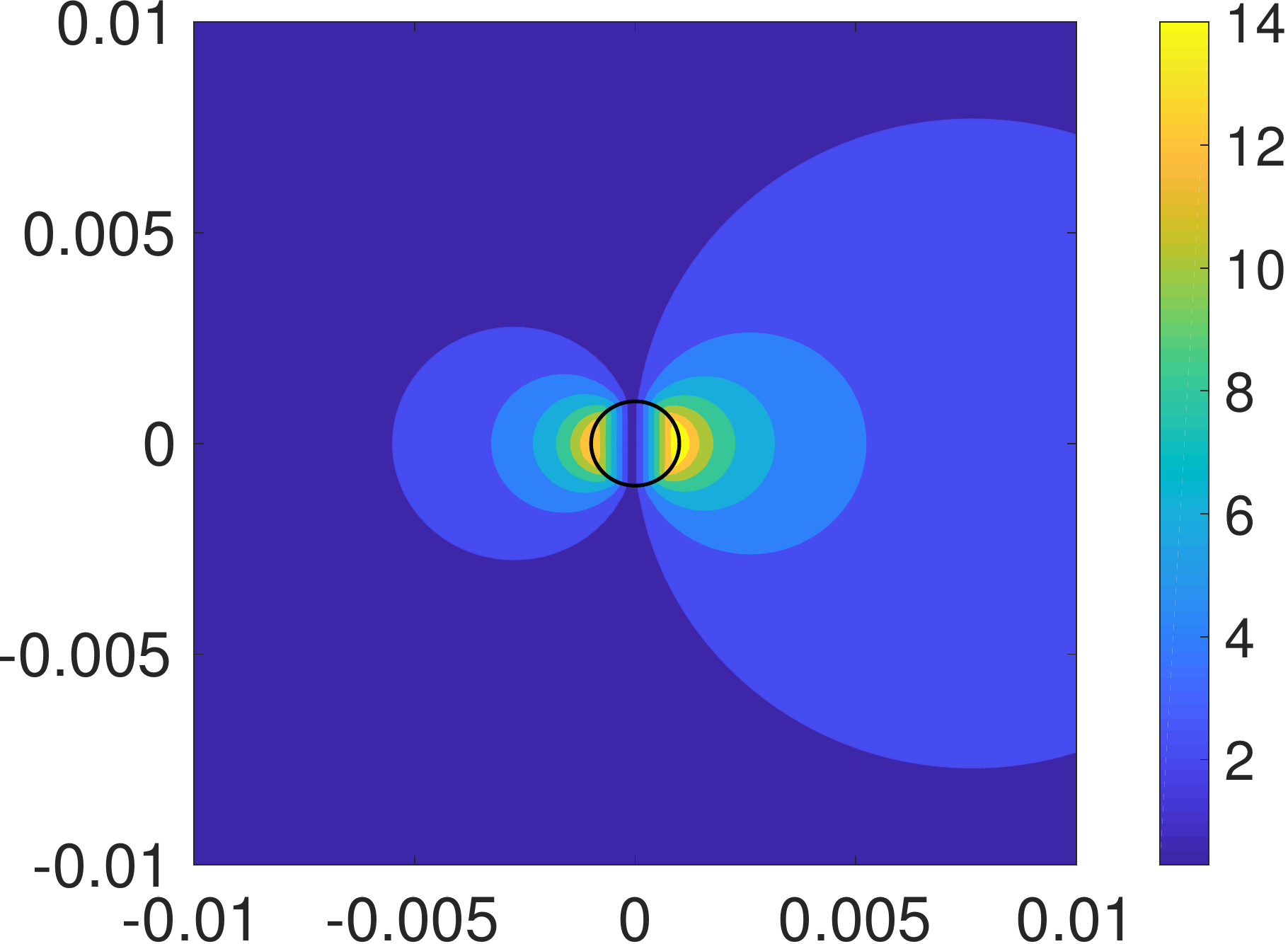}}\\
\caption{\label{figrs1} Moduli of the total wave fields of \eqref{ge1}, \eqref{eq:nn1} and \eqref{eq:nn2} with (a). $s=2$; (b).~$s=0.2$; (c). $s=0.02$; (d).~$s=0.002$.  }
\end{figure}

Consider a plasmon configuration of the form $\epsilon_{D,\delta}$ in \eqref{eq:q1} with
\begin{equation}\label{eq:nn1}
D=B_s(0),\quad \epsilon_c=-1,\quad \delta=0.001,
\end{equation}
where $B_s(0)$ is a central disc of radius $s\in\mathbb{R}_+$. The choice of $\epsilon_c=-1$ makes $\lambda(\epsilon_c)$ defined in \eqref{eq:lambda1} identically zero. 
It is noted that if $D$ is a central disk, then $0$ is actually the eigenvalue of $K_D^*$, and on the other hand, if $D$ is an arbitrary domain with a $C^2$ boundary, $K_D^*$ is 
a compact operator and $0$ is an accumulation point of its eigenvalues. Hence, with $\epsilon_c=-1$, the first condition for the occurrence of the plasmon resonance is fulfilled. 
For the corresponding Helmholtz system \eqref{ge1}, we choose
\begin{equation}\label{eq:nn2}
f=-\nabla\cdot(\epsilon_{D,\delta}\nabla u^i)-k^2 u^i,\quad u^i(x)=e^{\mathrm{i}k x\cdot d},\ k=10,\ d=(-1, 0). 
\end{equation}
That means, the wave scattering is caused by an incident plane wave which plays the role of an external source. In Fig.~ \ref{figrs1}, we plot the moduli of the total wave fields, namely $|u_\delta+u^i|$, against different parameters $s=2$, $s=0.2$, $s=0.02$ and $s=0.002$. The numerical results clearly show the critical role of the quasi-static approximation for the occurrence of the plasmon resonance. In fact, it can be seen that if the size of the plasmonic inclusion, namely $D=B_s(0)$, is not small enough compared to the wavelength, then resonance does not occur, and as $s$ becomes smaller, both conditions \eqref{eq:qs1} and \eqref{eq:qs2} are fulfilled, then resonance occurs.

\section{Localization and geometrization in plasmon resonance}\label{sect:4}

In this section, we consider the localization and geometrization for the plasmon resonance. We first present 
some numerical examples to illustrate this kind of peculiar phenomenon. Our numerical examples follow a similar setup as that 
specified in \eqref{eq:nn1} and \eqref{eq:nn2} with $s=2$. According to our study in the previous section, we know that resonance does not occur. However, we pull
out a part of the boundary of the plasmonic inclusion $D$ to form a boundary point with a relatively high curvature; see Fig.~\ref{figbc} for the geometric setup. 
Similar to the numerical experiments in Fig.~\ref{figrs1}, we numerically plot the total wave field associated to the plasmon inclusion as described above against 
the change of the curvature of the aforesaid boundary point; see Fig.~\ref{figr1}. 
\begin{figure}[t]
\centering
\includegraphics[width=0.2\textwidth]{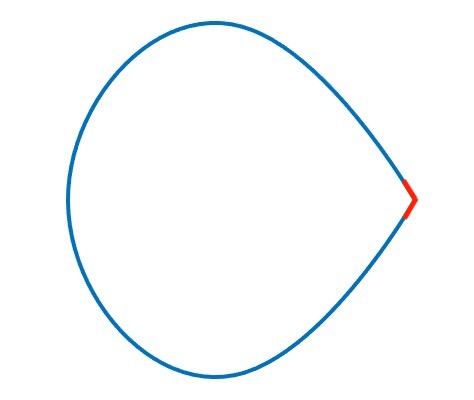}
\caption{\label{figbc} Geometry of the plasmonic inclusion $D$ with a high-curvature boundary point. }
\end{figure}
\begin{figure}[t]
\centering
\subfigure[]{
\includegraphics[width=0.2\textwidth]{i1.pdf}}
\subfigure[]{
\includegraphics[width=0.2\textwidth]{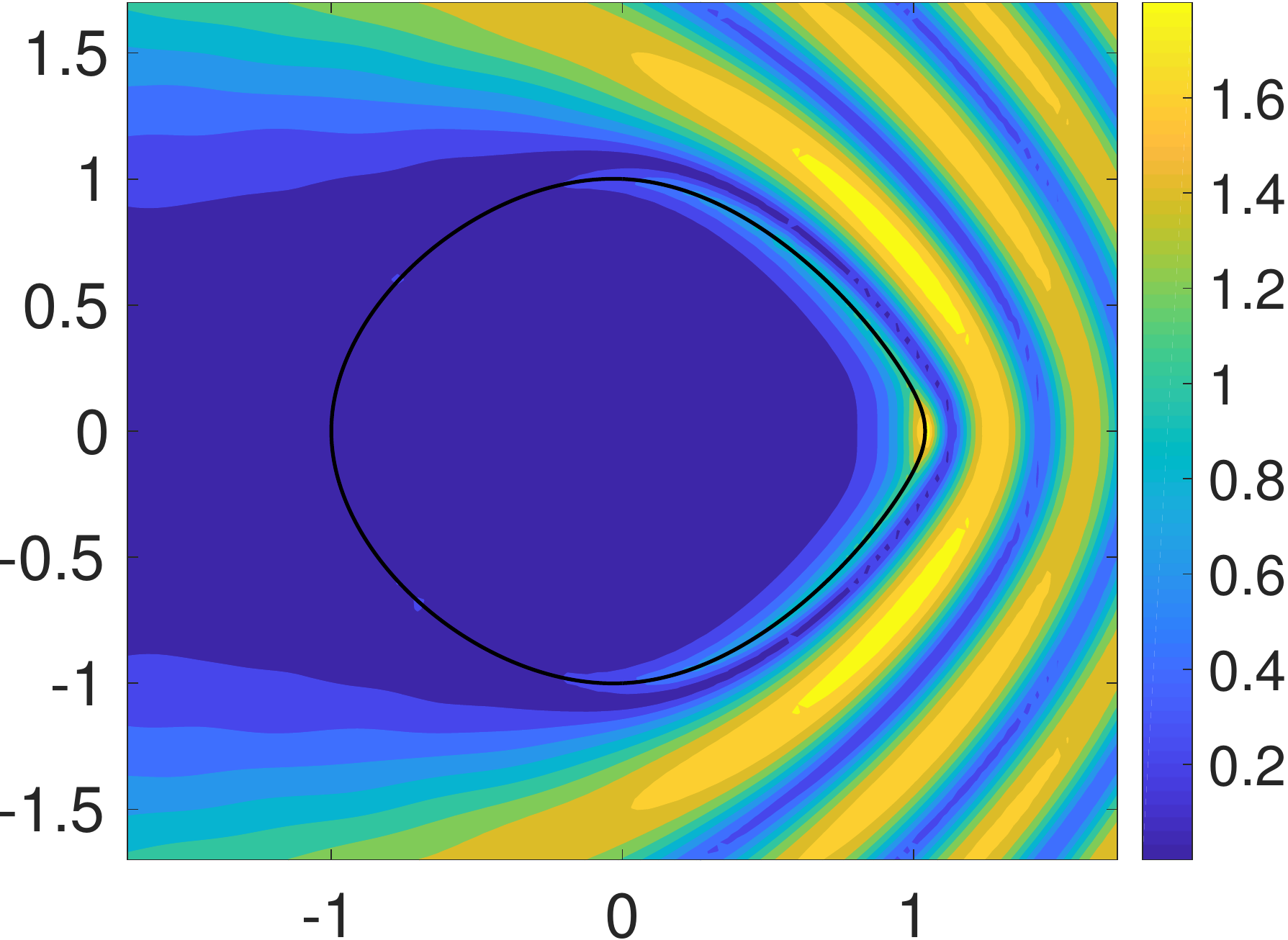}}
\subfigure[]{
\includegraphics[width=0.2\textwidth]{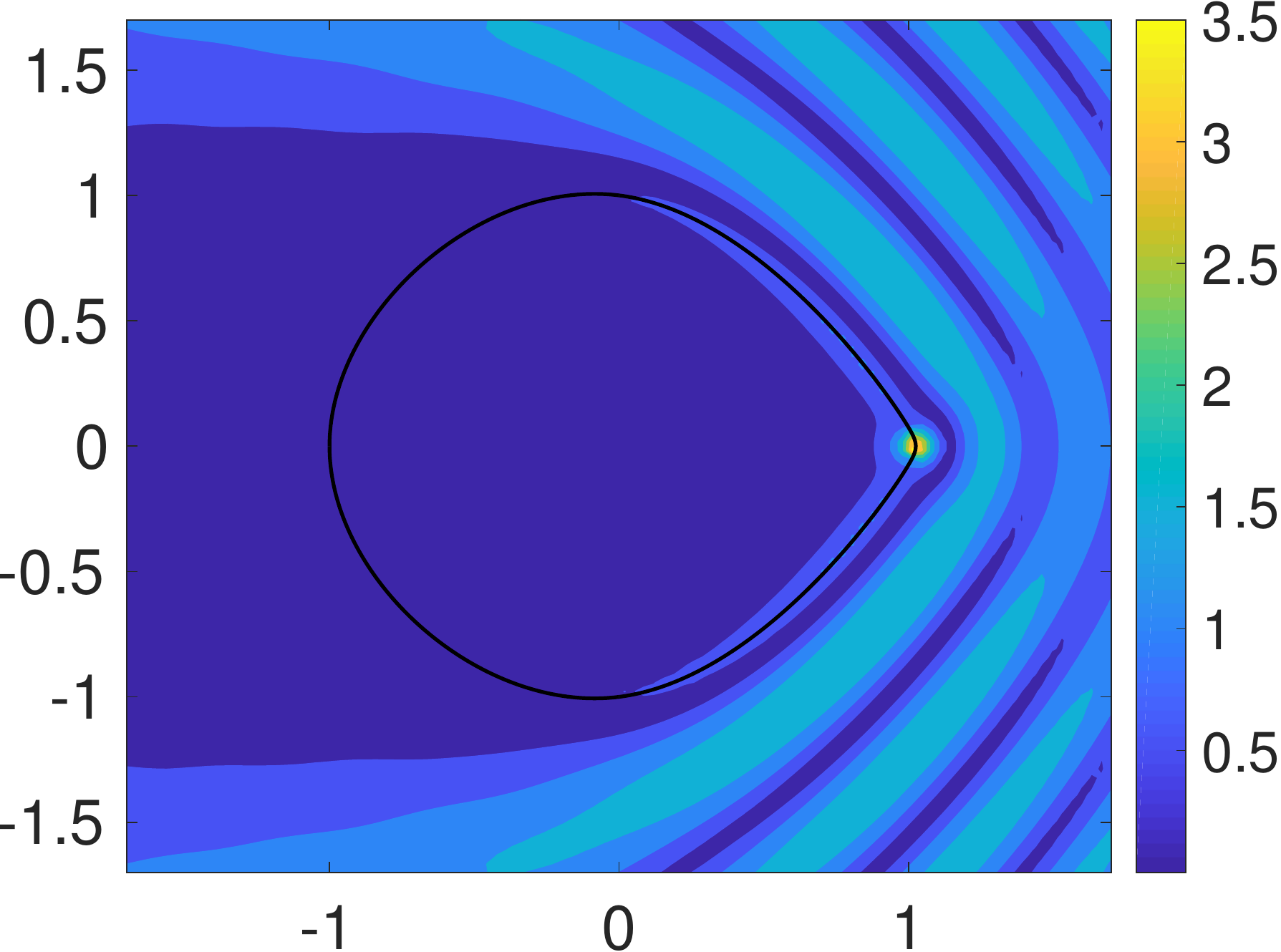}}
\subfigure[]{
\includegraphics[width=0.2\textwidth]{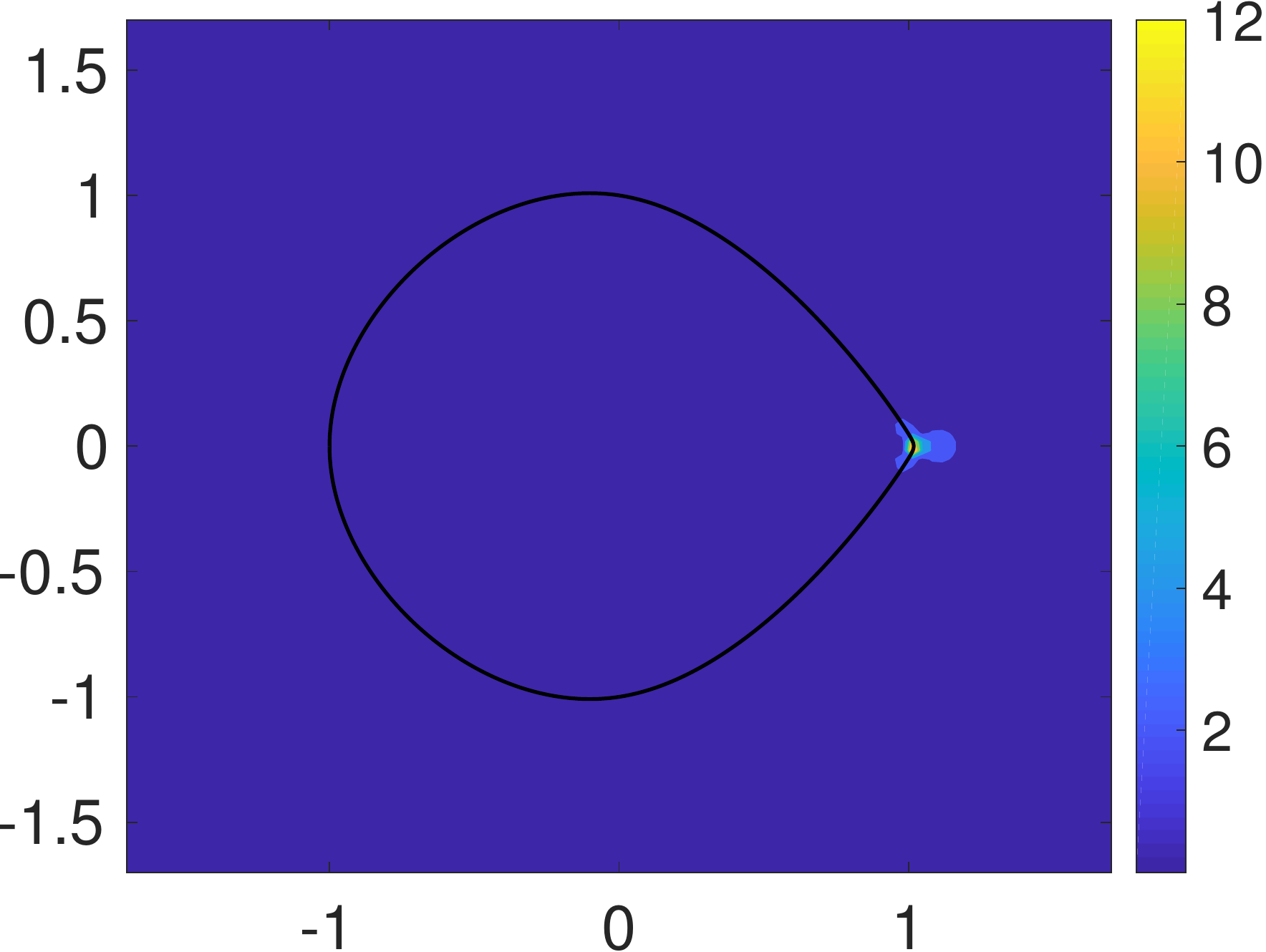}}\\
\subfigure[]{
\includegraphics[width=0.2\textwidth]{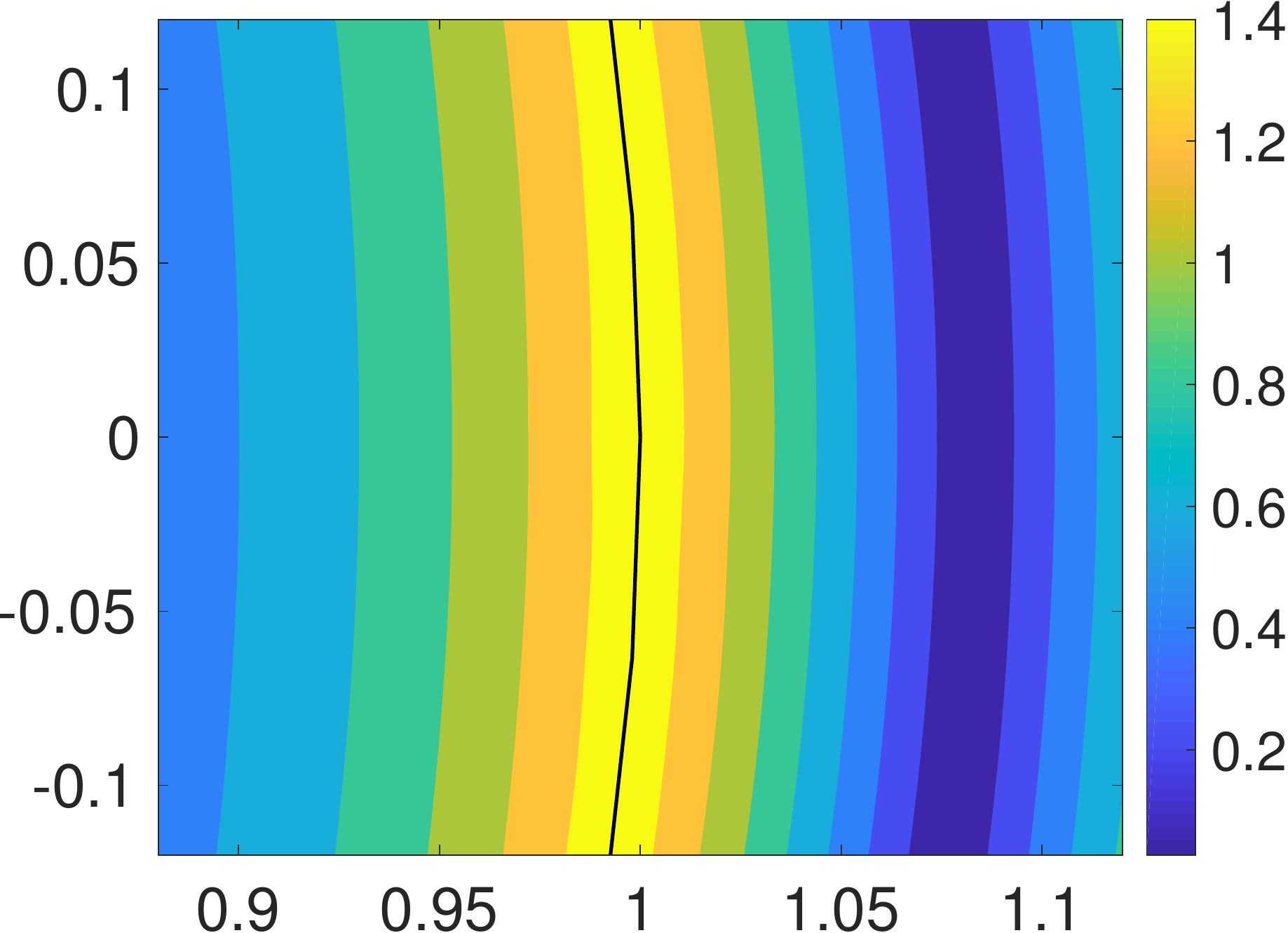}}
\subfigure[]{
\includegraphics[width=0.2\textwidth]{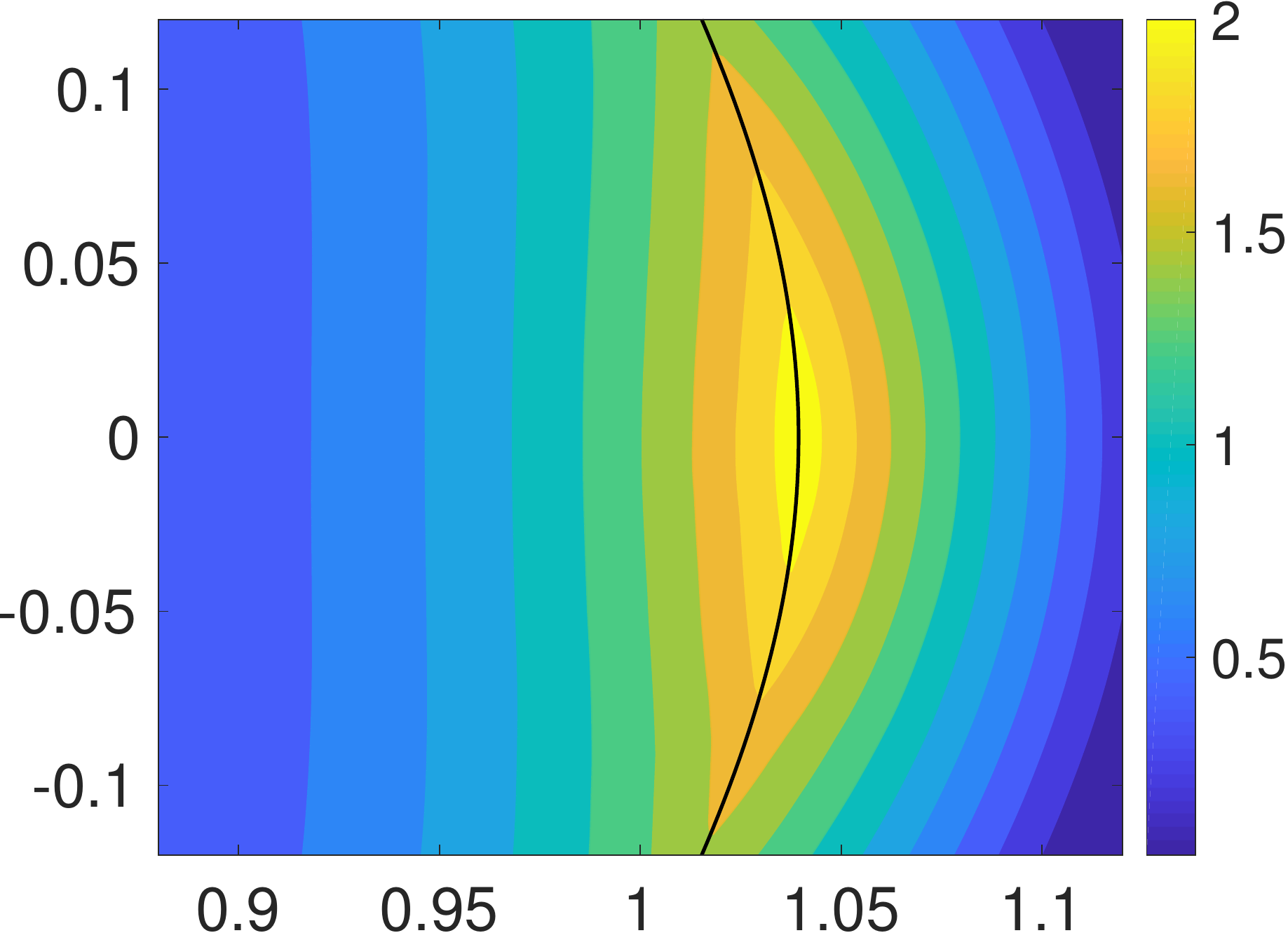}}
\subfigure[]{
\includegraphics[width=0.2\textwidth]{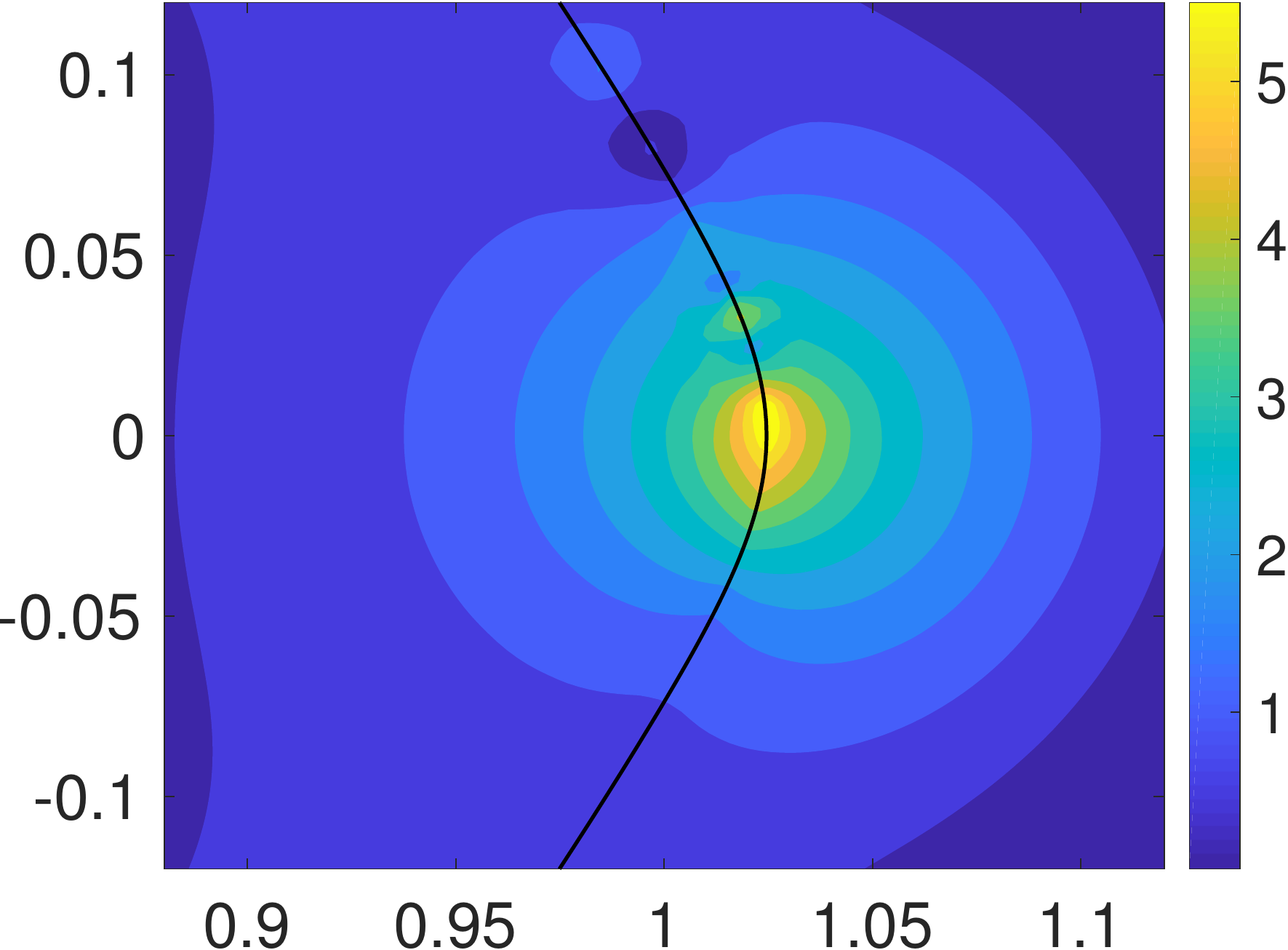}}
\subfigure[]{
\includegraphics[width=0.2\textwidth]{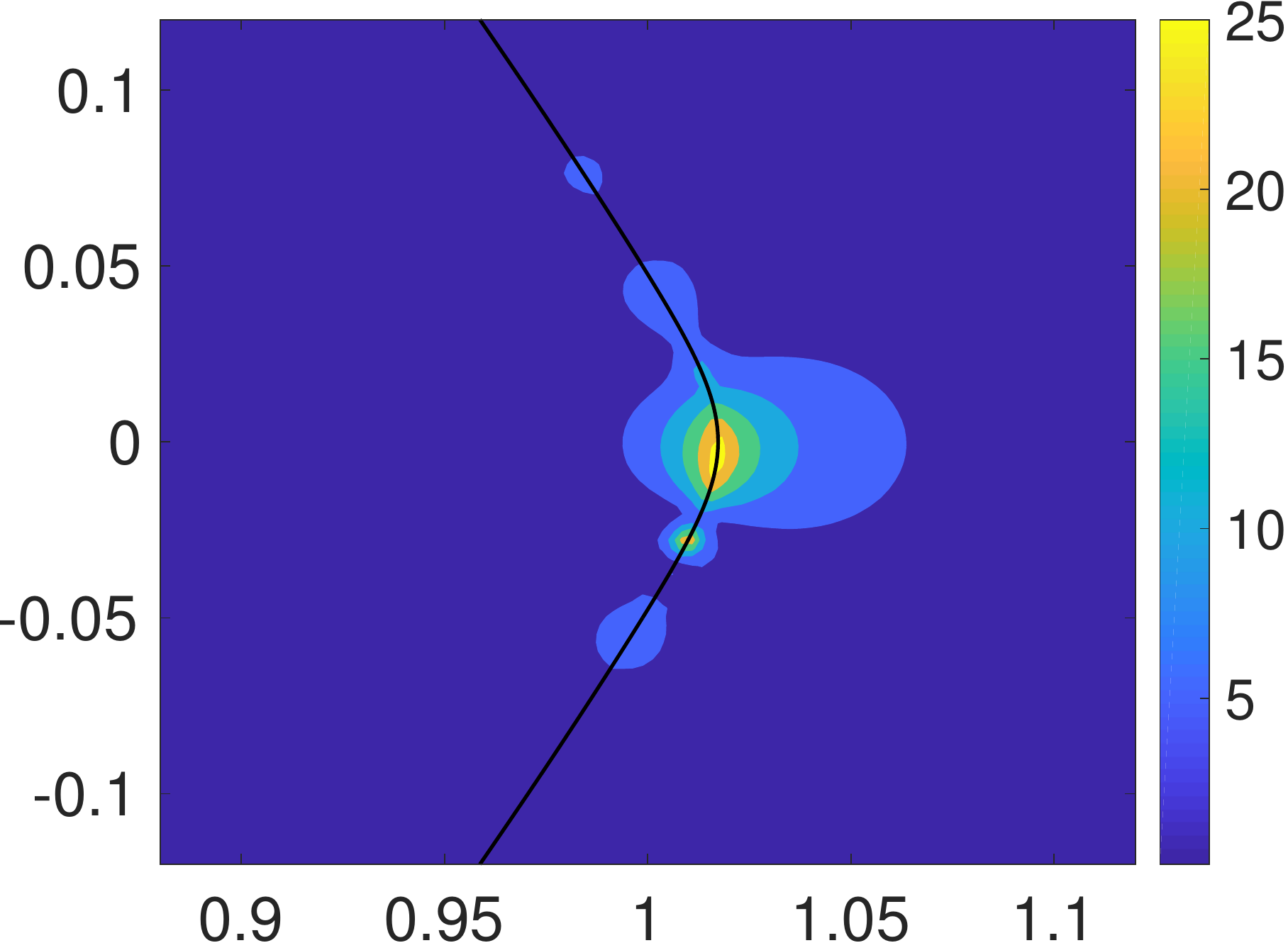}}\\
\caption{\label{figr1} Localization and geometrization phenomenon in the plasmon resonance. 
The first row plots the moduli of the total fields with increasing curvatures at the same boundary point. 
The second row plots the moduli of the total fields locally around the high-curvature point.}
\end{figure}
It can be readily seen that as the curvature of that boundary point increases to a certain degree, then resonance occurs locally 
around that high-curvature point. This is referred to as the localization and geometrization in the plasmon resonance. By localization, 
we mean that the resonance occurs only locally around a boundary point, whereas by geometrization, 
we mean that the global geometric smallness condition \eqref{eq:qs1} can be replaced by a locally high-curvature condition. 
At this point, we would like to present our novel viewpoint about the plasmon resonance. That is, on the one hand, the plasmonic parameter is unquestionably a critical 
ingredient for the occurrence of resonance, but on the other hand,  the quasi-static approximation, namely the smallness of the size of 
the plasmonic inclusion, is not the main cause for the resonance and instead, the high curvature is actually the main cause. 

Next we try to provide a theoretical explanation of the localization and geometrization phenomenon in the plasmon resonance. 
According to our earlier discussion on the plasmon resonances, respectively, in the electrostatic and quasi-static cases, one needs to 
study the quantitative properties of the eigenfunctions of the NP operator $(K_{\partial D}^k)^*$ in \eqref{eq:dlk} and the corresponding single-layer 
potential $S_{\partial D}^k[\varphi]$ in \eqref{eq:slk} locally around the high-curvature point of $\partial D$. To that end, let us consider a domain 
$D$ as plotted in Fig.~\ref{figbc} and $x_R$ be the vertex of the red part which possesses the largest curvature among all the boundary points. Set
\begin{equation}\label{eq:boundary parts}
\Gamma_1:=\partial D\cap B_\rho(x_R),\quad\Gamma_2:=\partial D\backslash\Gamma_1,
\end{equation}
where $\rho\in\mathbb{R}_+$ is sufficiently small. We have
\begin{proposition}\label{prop:1}
Let $\partial D$, $\Gamma_1$ and $\Gamma_2$ be described above. There holds
\begin{equation}\label{eq:cc1}
\big(K_{\partial D}^k\big)^*\big|_{\Gamma_1}=K_{\Gamma_1}^*+\mathcal{R}_{\Gamma_1,\rho}+\mathcal{T},
\end{equation}
where $\mathcal{T}$ is a smooth operator on $L^2(\Gamma_1)$, and $\mathcal{R}_{\Gamma_1,\rho}$ is a bounded operator on $L^2(\Gamma_1)$ satisfying 
$\|\mathcal{R}_{\Gamma_1,\rho}\|= \mathcal{O}\big((\rho k)^2\ln(\rho k)\big).$
\end{proposition}

\begin{proof}{
From \eqref{eq:boundary parts}, one has that 
\[
 \big(K_{\partial D}^k\big)^*\big|_{\Gamma_1} = \big(K_{\Gamma_1}^k\big)^*\big|_{\Gamma_1} + \big(K_{\Gamma_2}^k\big)^*\big|_{\Gamma_1},
\]
where $\big(K_{\Gamma_2}^k\big)^*\big|_{\Gamma_1}$ is smooth operator on $L^2(\Gamma_1)$. Since $\Gamma_1:=\partial D\cap B_\rho(x_R)$ 
with $\rho\in\mathbb{R}_+$ sufficiently small, from the asymptotic expression for the NP operator in \eqref{eq:akk}, 
one has by direct calculations that 
\[
  \big(K_{\Gamma_1}^k\big)^*\big|_{\Gamma_1} =K_{\Gamma_1}^*|_{\Gamma_1}+\mathcal{R}_{\Gamma_1,\rho},
\]
where $\mathcal{R}_{\Gamma_1,\rho}$ is a bounded operator on $L^2(\Gamma_1)$ satisfying 
\[
\|\mathcal{R}_{\Gamma_1,\rho}\|= \mathcal{O}\big((\rho k)^2\ln(\rho k)\big).
\]

The proof is complete. 
}
\end{proof}

Hence, by Proposition~\ref{prop:1} and our earlier discussion on the plasmon resonance in the electrostatic 
and quasi-static cases, in order to understand the localization and geometrization phenomenon illustrated in Fig.~\ref{figr1}, 
it is unobjectionable to say that one should investigate the spectral properties of the eigenfunctions of $K_{\partial D}^*$ locally 
near a high-curvature point. The rest of the paper is devoted to investigating the geometric structures of the NP eigenfunctions 
as well as the associated single-layer potentials near a high-curvature point. Finally, we mention that the geometrization with a high-curvature 
condition is a critical ingredient in our study. It is known that if $\partial D$ is $C^2$-smooth, then the corresponding NP operator is compact, 
and hence its spectrum consists only of eigenvalues. If the high-curvature point becomes a corner, then the corresponding NP operator 
possesses continuous spectra \cite{HKL,KLY,HP}, which shall make the situation more complicated. Nevertheless, the NP operator may still 
possess eigenvalues in the corner domain case, and it is worth of future investigation on the corresponding NP eigenfunctions in such a case.

\section{ Geometric structures of NP eigenfunctions}

In this section, we consider the geometric structures of the eigenfunctions of the NP operator $K_{\partial D}^*$ 
as well as the corresponding single layer potential $S_{\partial D}[\varphi]$ near a high-curvature point of $\partial D$. 

First, we present some basic results about the spectral structure of  $K_{\partial D}^*$. Throughout the rest of the paper, 
we assume that $\partial D$ is $C^2$-smooth. As discussed earlier, $K_{\partial D}^*$ is a compact operator and its spectrum consists 
of at most countably many eigenvalues that can only accumulate at $0$. We also know that (cf. \cite{Ack13})
\begin{equation}\label{eq:sk1}
\sigma(K_{\partial D}^*)\subset (-1/2, 1/2],
\end{equation}
where and also in what follows, $\sigma(K_{\partial D}^*)$ signifies the spectrum of $K_{\partial D}^*$. There holds the following property
\begin{lemma}\label{lem:c}
Suppose that $\lambda_0=1/2$ is an eigenvalue of $K_{\partial D}^*$ and $\psi_0\in L^2(\partial D)$ is an eigenfunction, i.e. $K_{\partial D}^*[\psi_0]=1/2\psi_0$. Then there holds
\[
 S_{\partial D}[\psi_0](x)=C\in\mathbb{C},\quad x\in D.
\]
\end{lemma}
\begin{proof}
Set
\[
 u(x)=S_{\partial D}[\psi_0](x), \quad x\in \mathbb{R}^2. 
\]
By Green's formula one can show that
\[
 \int_{D}|\nabla u|^2dx=\int_{\partial D}\frac{\partial u}{\partial \nu}\big|_-\overline{u}ds=\int_{\partial D}\left(K^*[\psi_0]-1/2\psi_0\right)\overline{u}ds=0.
\]
where $\nu$ signifies the exterior unit normal vector to $\partial D$. Hence, $u$ must be constant on $\partial D$. 

The proof is complete. 
\end{proof}

It is known that both $K_{\partial D}^*$ and $S_{\partial D}$ are pseudo-differential operators of order $-1$ (cf. \cite{Ned}). 
Hence, if $\psi\in L^2(\partial D)$ is an eigenfunction satisfying $K_{\partial D}^*[\psi]=\lambda \psi$ for an 
eigenvalue $\lambda\in (-1/2, 1/2]$, it can be straightforwardly verify that $\psi\in C^{0,1}(\partial D)$. 
In fact, if $\partial D$ is $C^\infty$-smooth, then $\psi\in C^\infty (\partial D)$.

Next, we investigate the geometric structures of the NP eigenfunctions. We start with the case that $\partial D$ is an ellipse whose NP eigenfunctions can be explicitly calculated. For $x=(x_1, x_2)\in\mathbb{R}^2$, we introduce the following elliptic coordinates $(\rho, \omega)$, 
\begin{equation}\label{eq:ellipse}
 x_1=R_0 \cos \omega \cosh \rho , \quad x_2=R_0\sin \omega \sinh \rho, \quad \rho>0, \; 0\leq \omega\leq 2 \pi, \ R_0\in\mathbb{R}_+
\end{equation}
An elliptic domain $D$ is defined by 
\begin{equation}\label{eq:ellipse1}
 D=\{(\rho,\omega); \rho\leq \rho_0,  \; 0\leq \omega\leq 2 \pi \},
\end{equation}
whose boundary is given by 
\begin{equation}\label{eq:ellipse2}
 \partial D=\{(\rho,\omega); \rho= \rho_0,  \; 0\leq \omega\leq 2 \pi \}.
\end{equation}
In Fig.~\ref{figeb} we give a specific example with $R_0=1$ and $\rho_0=0.05$. 
In what follows, we set 
\[
 \Xi:=R_0\sqrt{\sinh^2 \rho_0 + \sin^2\omega}. 
\]
We have

\begin{lemma}[\cite{ CKKL}]\label{lem:spe_elli}
Let $\partial D$ be an ellipse described in \eqref{eq:ellipse} and \eqref{eq:ellipse2}. There hold that 
\[
K_{\partial D}^*[\phi_{1,n}]= a_n \phi_{1,n}  \quad \mbox{and} \quad K_{\partial D}^*[\phi_{2,n}]= -a_n \phi_{2,n}, \quad n\geq 0,
\]
where
\begin{equation}\label{eq:ee1}
\phi_{1,n}=\Xi^{-1} \cos n\omega, \quad \phi_{2,n}=\Xi^{-1} \sin n\omega  \quad \mbox{and} \quad  a_n= \frac{1}{2 e^{2n\rho_0}}.
\end{equation}
Moreover, associated with the eigenfunctions in \eqref{eq:ee1}, one has for $n\geq 1$ that
\[
S_{\partial D}[\phi_{1,n}](x)=
 \begin{cases}
 -\frac{e^{n\rho} + e^{-n\rho} }{2ne^{n \rho_0}} \cos n\omega ,\quad & \rho\leq \rho_0, \\
 -\frac{e^{n\rho_0} + e^{-n\rho_0} }{2ne^{n \rho}} \cos n\omega ,\quad & \rho> \rho_0,
 \end{cases}
\]
and 
\[
S_{\partial D}[\phi_{2,n}](x)=
 \begin{cases}
 -\frac{e^{n\rho} + e^{-n\rho} }{2ne^{n \rho_0}} \sin n\omega ,\quad & \rho\leq \rho_0, \\
 -\frac{e^{n\rho_0} + e^{-n\rho_0} }{2ne^{n \rho}} \sin n\omega ,\quad & \rho> \rho_0.
 \end{cases}
\]
\end{lemma}

With the explicit forms of the NP eigenfunctions and the associated single-layer potentials, we are in a position to investigate their geometric structures. Before that, we first note that for an ellipse defined by \eqref{eq:ellipse} and \eqref{eq:ellipse2}, the corresponding curvature at a boundary point $(\rho_0, \omega)\in\partial D$ can be directly calculated to be
\begin{equation}\label{eq:cv1}
 \kappa(\omega)=\frac{\cosh \rho_0 \sinh \rho_0}{R_0(\sinh^2 \rho_0 + \sin^2\omega)^{3/2}}.
\end{equation}
Hence, the largest curvature is attainable at the two vertices with $\omega=\pi$ and $\omega=0$ respectively on the semi-major axis, denoted as $x_\circ$ and $x_*$ in what follows; see Fig.~\ref{figeb} for an illustration. Henceforth, the points on $\partial D$ that attain the largest curvature are referred to as the high-curvature points. By \eqref{eq:cv1}, the largest curvature is given by
\begin{equation}\label{eq:maxcur}
 \kappa_{\max}:=\frac{\cosh \rho_0}{R_0 \sinh^2 \rho_0 }.
\end{equation}
It is noted that for a fixed $R_0$, the curvature $\kappa_{\max}$ increases as $\rho_0$ decreases and actually one has that $\kappa_{\max}\rightarrow \infty$ as $\rho_0\rightarrow+0$. In what follows, we shall also need the conormal derivative of a function $\psi(x)$ defined over $\partial D$. Let $\partial D$ be parametrized as $x(s)$, and then the conormal derivative of $\psi(x)$ is defined as
\[
d \psi=\psi^{\prime}(x)\cdot \frac{x^{\prime}(s)}{|x^{\prime}(s)|}=\frac{d}{ds}\psi(x(s))\frac{1}{|x^{\prime}(s)|}.
\]

\begin{figure}[t]
\includegraphics[width=8cm] {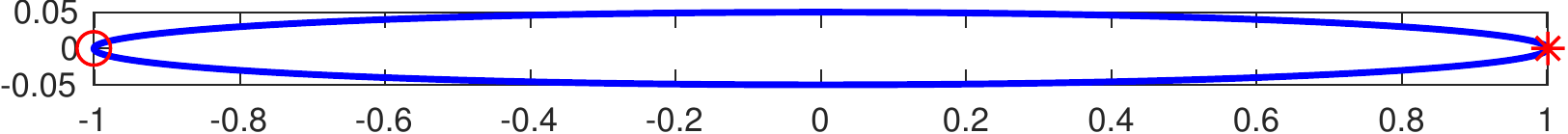}
\caption{\label{figeb} An ellipse defined by \eqref{eq:ellipse} and \eqref{eq:ellipse2} with $R_0=1$ and $\rho_0=0.05$. The left and right vertices $x_\circ$ and $x_*$, respectively, on the semi-major axis are the high-curvature points.   }
\end{figure}

\begin{proposition}\label{prop:main2}
Let $\partial D$ be an ellipse described in \eqref{eq:ellipse} and \eqref{eq:ellipse2}, and let $\phi_{1,n}$ and $\phi_{2,n}$ be the NP eigenfunctions derived in Lemma~\ref{lem:spe_elli} for $K_{\partial D}^*$ with $n\geq 1$. Then one has
\begin{enumerate}
\item $\phi_{1,n}(x)$ achieves its maximum absolute value on $\partial D$ at $x_\circ$ and $x_*$,
\begin{equation}\label{eq:gs1}
|\phi_{1,n}(x_\circ)|=|\phi_{1,n}(x_*)|=\tau_{\max},\quad \tau_{\max}:=\frac{1}{R \sinh\rho_0}, 
\end{equation}
and there holds the following asymptotic relationship as $\kappa_{\max}\rightarrow +\infty$, or equivalently $\rho_0\rightarrow +0$,
\begin{equation}\label{eq:gs2}
\tau_{\max}\sim \alpha\kappa_{\max}^{p}, \quad p=1/2,\ \ \alpha\in\mathbb{R}_+. 
\end{equation}

\item $d\phi_{2,n}$ achieves its maximum absolutely value on $\partial D$ at $x_\circ$ and $x_*$,
\begin{equation}\label{eq:gs3}
|d\phi_{2,n}(x_\circ)|=|d\phi_{2,n}(x_*)|=\tau_{\max}',\quad \tau_{\max}':=\frac{n}{R_0 \sinh^2\rho_0},
\end{equation}
and moreover there holds
\begin{equation}\label{eq:gs4}
\tau_{\max}'\rightarrow\infty \quad\mbox{as}\ \ \kappa_{\max}\rightarrow\infty. 
\end{equation}

\item $S_{\partial D}[\phi_{1,n}]$ and $d S_{\partial D}[\phi_{2,n}]$, respectively, achieve their maximum absolute values at $x_\circ$ and $x_*$. 
\end{enumerate}

\end{proposition}

\begin{proof}
{ With the explicit forms of solutions in Lemma \ref{lem:spe_elli}, the proposition can be verified by straightforward though a bit tedious calculations. }
\end{proof}

In Proposition~\ref{prop:main2}, we did not consider the case with $n=0$ due to Lemma~\ref{lem:c}. 
Clearly, the properties in Proposition~\ref{prop:main2} can be used to explain the localization and geometrization 
phenomenon discovered in Section~\ref{sect:4}, at least for the elliptic geometry case. In fact, we perform the numerical experiment in Fig.~\ref{figr1} again, 
but with the plasmonic inclusion $\partial D$ in Fig.~\ref{figbc} replaced by an ellipse in Fig.~\ref{figeb}. The loss parameter $\delta$ is set to be $0.0001$. The total 
wave field is plotted in Fig.~\ref{el_reson}. Clearly, strong resonant behaviours are observed locally around the two high-curvature points $x_\circ$ and $x_*$. 
It is remarked that the incident plane wave propagates from the left to the right and the vertex $x_\circ$ is located in the shadow region. Hence, the resonant behaviour 
around $x_*$ is stronger than that around $x_\circ$. 

\begin{figure}[t]
\centering
\subfigure[]{
\includegraphics[width=0.2\textwidth]{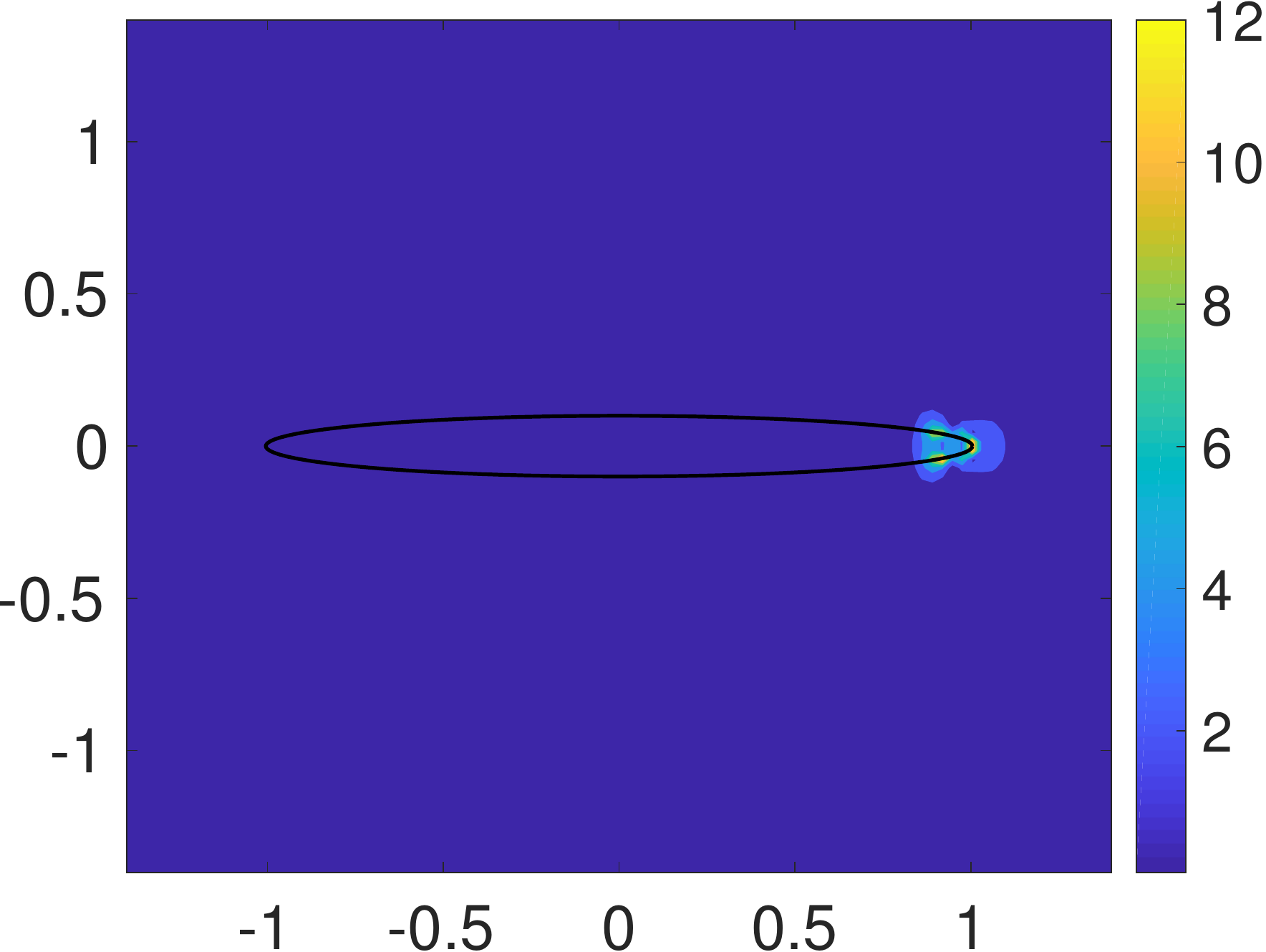}}
\subfigure[]{
\includegraphics[width=0.2\textwidth]{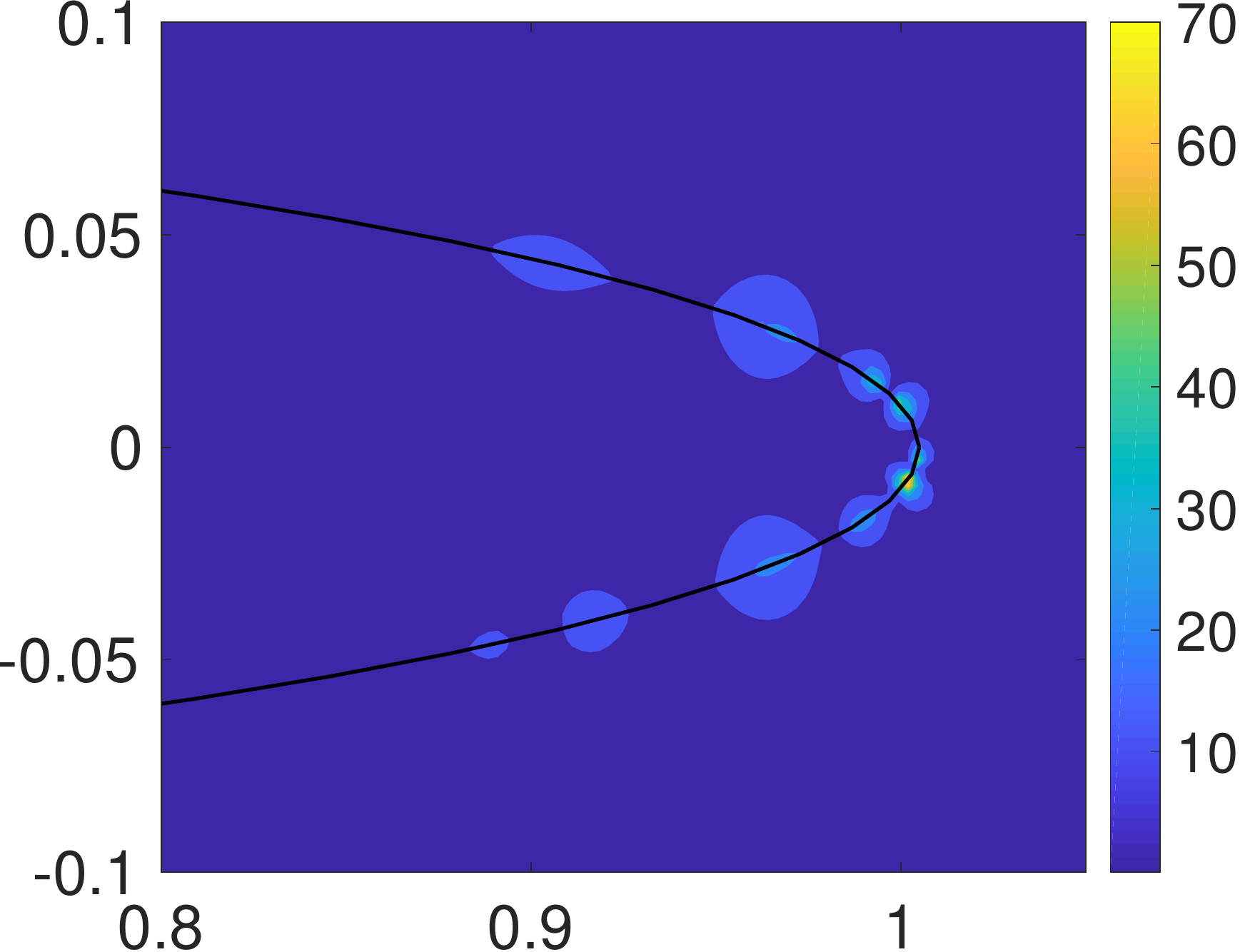}}
\subfigure[]{
\includegraphics[width=0.2\textwidth]{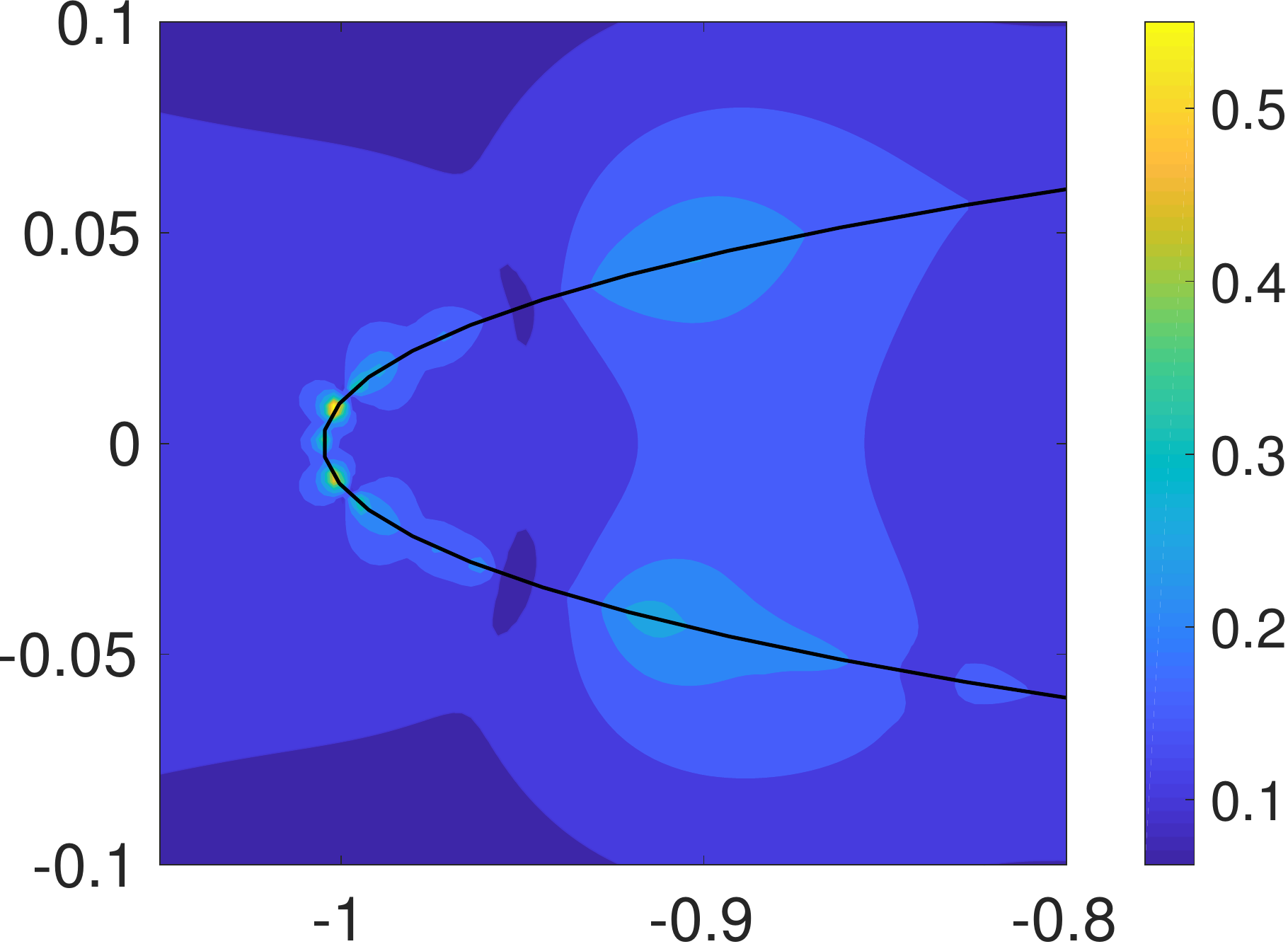}}\\
\caption{\label{el_reson} Localization and geometrization phenomenon in the plasmon resonance associated with an elliptical inclusion. (a). Modulus of 
the total wave field around the whole plamsonic inclusion; (b), (c). Modulus of the total wave field around $x_*$ and $x_o$, respectively.  }
\end{figure}

We believe those beautiful geometric structures in Proposition~\ref{prop:main2} for the NP eigenfunctions and the associated single-layer 
potentials hold for more general geometries. However, dealing with the general geometries, it is unpractical to derive the explicit 
forms of the NP eigenfunctions and the associated single-layer potentials. We have conducted extensive numerical experiments within general geometries and indeed 
the NP eigenfunctions exhibit certain intrinsic geometric structures near a boundary point with a high curvature. Before presenting our discoveries, we first introduce the notion
of a symmetric domain. Consider a star-shaped domain $D$ whose boundary $\partial D$ is parametrized as follows,
\begin{equation}\label{eq:boundcon}
 \partial D=r(\theta)\hat x(\theta),\ \ \hat{x}(\theta)=(\cos\theta, \sin \theta), \quad \theta\in[0,2\pi),
\end{equation}
where $r\geq 0$ is the radial function. If there exists $n\in\mathbb{N}$ such that 
\begin{equation}\label{eq:period}
 r(\theta)=r(\theta+2\pi/n),
\end{equation}
the the domain $D$ is said to be $n$-symmetric. Clearly, the domain in Fig.~\ref{figbc} is 1-symmetric and the domain in Fig.~\ref{figeb} is 2-symmetric. 
An $n$-symmetric domain possesses $n$ high-curvature points. 

The major numerical discoveries can be summarized as follows: 

\begin{enumerate}
\item Suppose $D$ is convex with $\partial D$ satisfying \eqref{eq:boundcon}, and $\lambda\in\sigma(K_{\partial D}^*)$. 
If $\lambda$ is positive and simple, then the absolute values of both its eigenfunction and the corresponding single-layer potential blow up at the 
high-curvature point(s) on $\partial D$ as the corresponding curvature goest to infinity; whereas if $\lambda$ is positive and multiple, 
then there exists at least one of the eigenfunctions such that the absolute values of both the eigenfunction and the associated single-layer potential blow up 
at the high-curvature point(s) on $\partial D$ as the corresponding curvature goest to infinity. If $\lambda$ is negative, then similar 
conclusions hold at the high-curvature point(s), but for the conormal derivatives of the eigenfunction and the associated single-layer potential. 

\item Suppose $D$ is concave at the high-curvature point(s) with $\partial D$ satisfying \eqref{eq:boundcon}, 
and $\lambda\in\sigma(K_{\partial D}^*)$. If $\lambda$ is negative and simple, then the absolute values of both its eigenfunction and 
the associated single-layer potential blow up at the high-curvature point(s) on $\partial D$ as the corresponding 
curvature goest to infinity; whereas if $\lambda$ is positive and multiple, then there exists at one of the eigenfunctions 
such that the absolute values of both the eigenfunction and the associated single-layer potential blow up at the high-curvature point(s) on $\partial D$ 
as the corresponding curvature goest to infinity. If $\lambda$ is positive, then similar conclusions hold at the high-curvature point(s), but for the 
conormal derivatives of the eigenfunction and the associated single-layer potential. 

\item If $D$ is non-symmetric, then the NP eigenfunction or its conormal derivative as well as the corresponding single-layer potential may still possess the 
blow-up behaviour at a high-curvature point, but the situation is more complicated, 
and there is no definite conclusion about it. 

\end{enumerate} 

%As mentioned earlier, we have conducted extensive numerical experiments to support the above assertions. In what follows, we only present a few 
%representative ones and we also refer to the arXiv preprint version of this paper \cite{ELLWarXiv} for more numerical examples. 

\subsection{Numerical method}

We first introduce the numerical method used to calculate the spectral system of the NP operator defined in \eqref{K0} and the corresponding single layer potential.  
Assume that the boundary of $D$, i.e. $\partial D$, is parameterized by $r(t)$, $t\in(0,2\pi)$. Then the NP operator can be expressed as follows
\begin{equation}\label{opk1}
 K^*[\varphi](x)=\frac{1}{2\pi} \int_{0}^{2\pi} \frac{\langle r(s)-r(t), \nu_s\rangle }{|r(s)-r(t)|^2} \varphi(r(t))|r^{\prime}(t)| dt.
\end{equation}
To numerically calculate this integral, we first discretize the integral line into $n$ panels and on each panel, we utilize the 16-point Gauss-Legendre quadrature formula. 
We point out that from the expression in \eqref{opk1}, there is the singularity when $t=s$, namely $y=x$. However, noting that $r(t)\in C^2(\partial D)$, the singularity can be removed
by using the following identity, 
\[
  \lim_{t\rightarrow s} \frac{\langle r(s)-r(t), \nu_s\rangle }{|r(s)-r(t)|^2}=-\frac{\langle r^{\prime\prime}(s), \nu_s\rangle} {2|r^{\prime}(s)|^2}, 
\]
where $\nu_s$ signifies the exterior unit normal vector to $\partial D$ at $r(s)$.

\subsection{A convex 1-symmetric domain}

Let us first consider a domain $D$ with one high-curvature point, denoted as $x_*$, as shown in Fig.~ \ref{fig1}, and the largest curvature is $500$.
\begin{figure}[t]
\includegraphics[width=3cm] {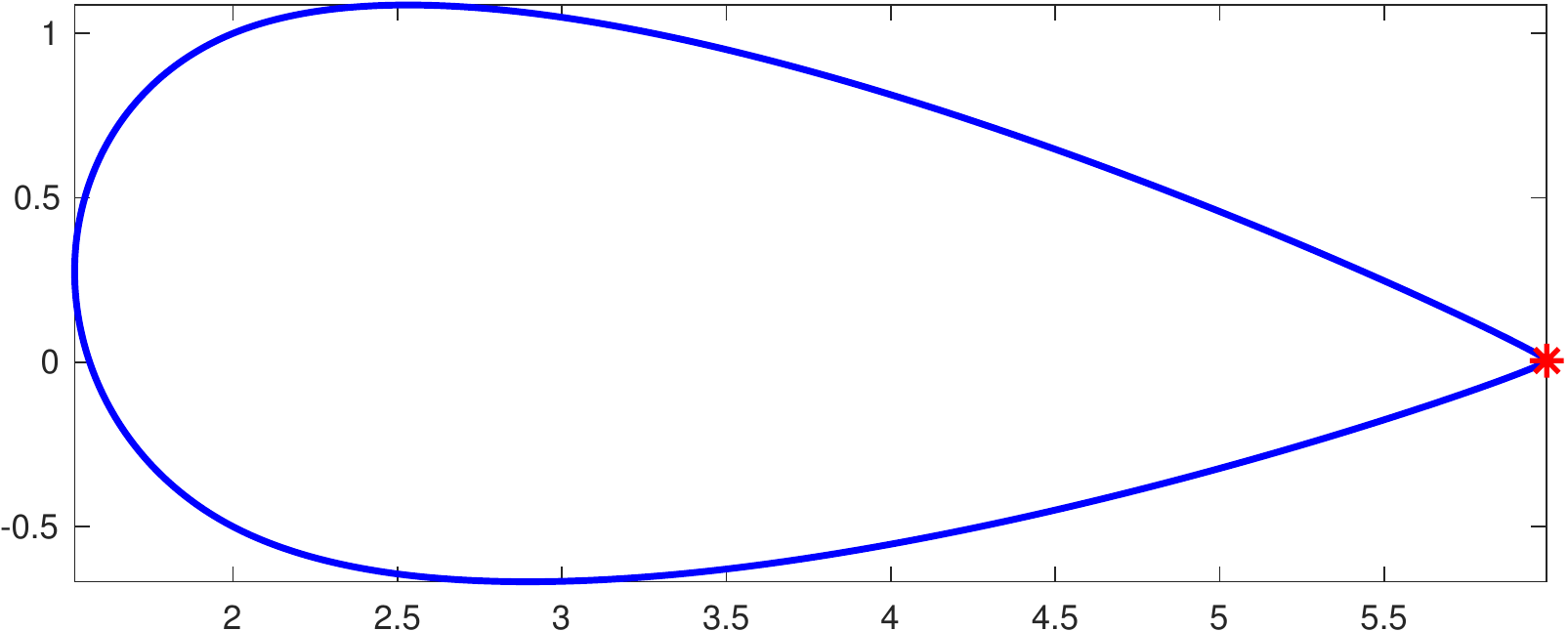}
\caption{\label{fig1} A convex 1-symmetric domain.}
\end{figure}
The first seven largest NP eigenvalues (in terms of the absolute value) are numerically found to be
\begin{equation}\label{eq:egg1}
\begin{split}
& \lambda_0=0.5,\ \  \lambda_1=0.2575,\ \  \lambda_2=-0.2575,\ \ \lambda_3=0.1365,\\
&\qquad\ \lambda_4=-0.1365,\ \  \lambda_5=0.0685,\ \ \lambda_6=-0.0685.
 \end{split}
\end{equation}
It is remarked that all of the eigenvalues are simple. 

Fig.~\ref{fig2} plots the eigenfunctions as well as the associated single-layer potentials, respectively, for the positive eigenvalues  
$\lambda_1=0.2575$ and $\lambda_3=0.1365$. The numerical results clearly support our assertion about the 
NP eigenfunctions associated to simple positive eigenvalues. 

%Fig.~\ref{fig2} plots the eigenfunctions with respect to the arc length, the corresponding 
%single layer potentials and the single layer potentials around the high-curvature point $x_*$ shown in Fig.~ \ref{fig1} 
%for the positive eigenvalues $\lambda_1=0.2575$ and $\lambda_3=0.1365$. In the Fig.~\ref{fig2}, 
%$a,d$ show that the eigenfunctions blow up at the high-curvature point $x_*$ for the positive eigenvalue, 
%which accords with the conclusion for the ellipse case stated in Proposition \ref{prop:main2}, $(1)$. 
%$b,c,e,f$ show that the associated single layer potentials blow up at the high-curvature point $x_*$ for the positive eigenvalue, 
%which accords with the conclusion for the ellipse case stated in Proposition \ref{prop:main2}, $(3)$.

\begin{figure}[t]
\centering
\subfigure[]{
\includegraphics[width=0.2\textwidth]{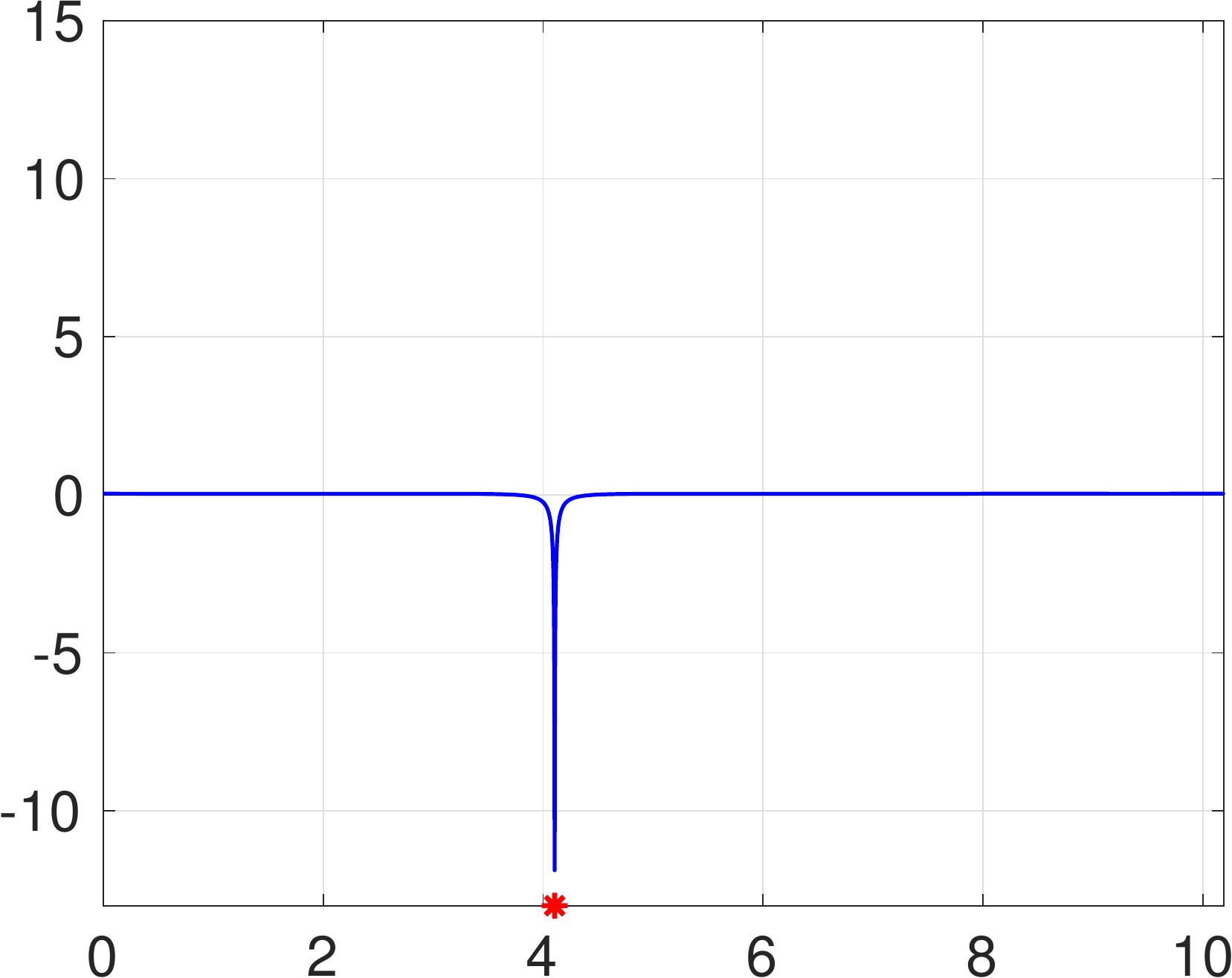}}
\subfigure[]{
\includegraphics[width=0.2\textwidth]{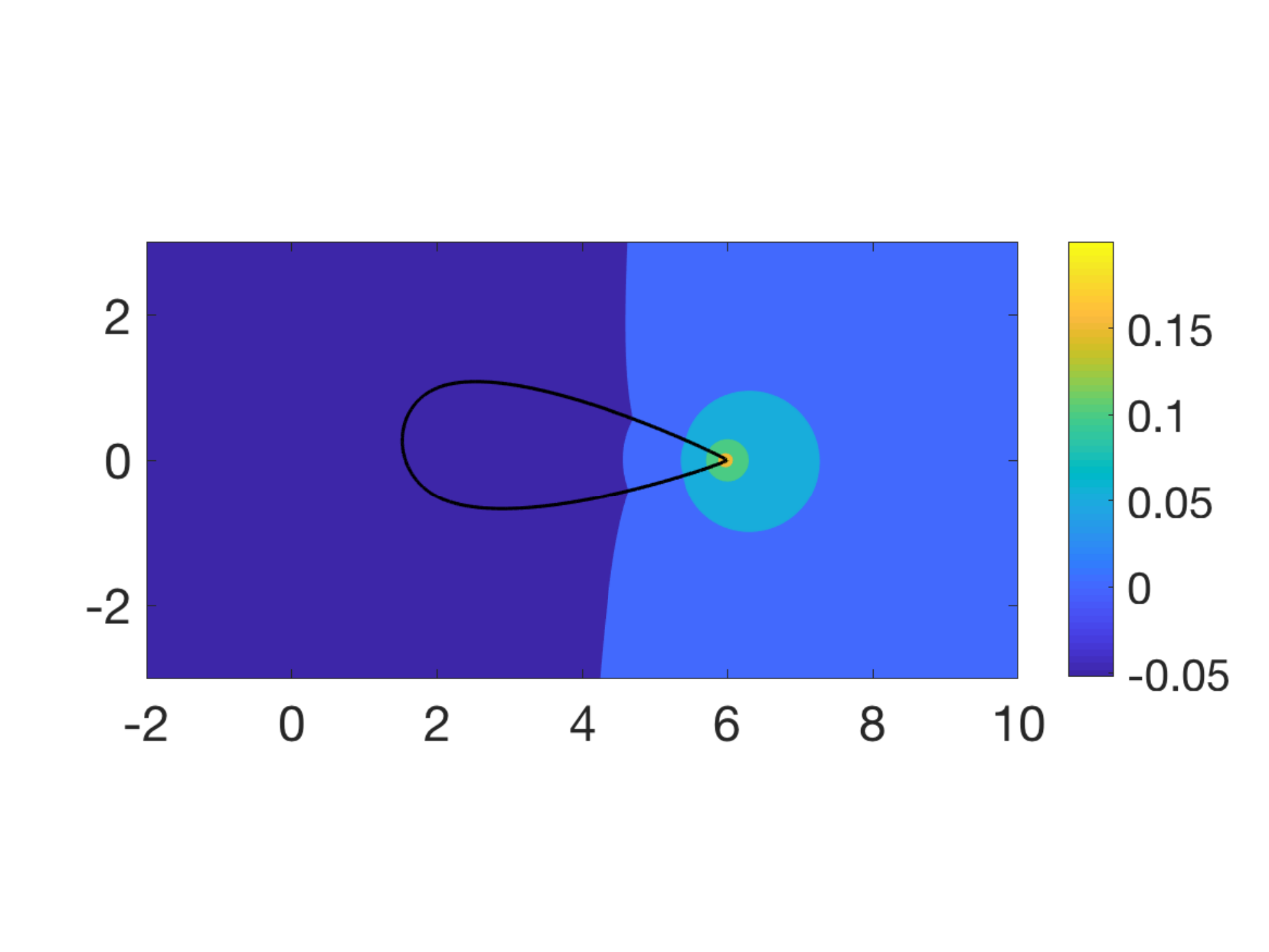}}
\subfigure[]{
\includegraphics[width=0.2\textwidth]{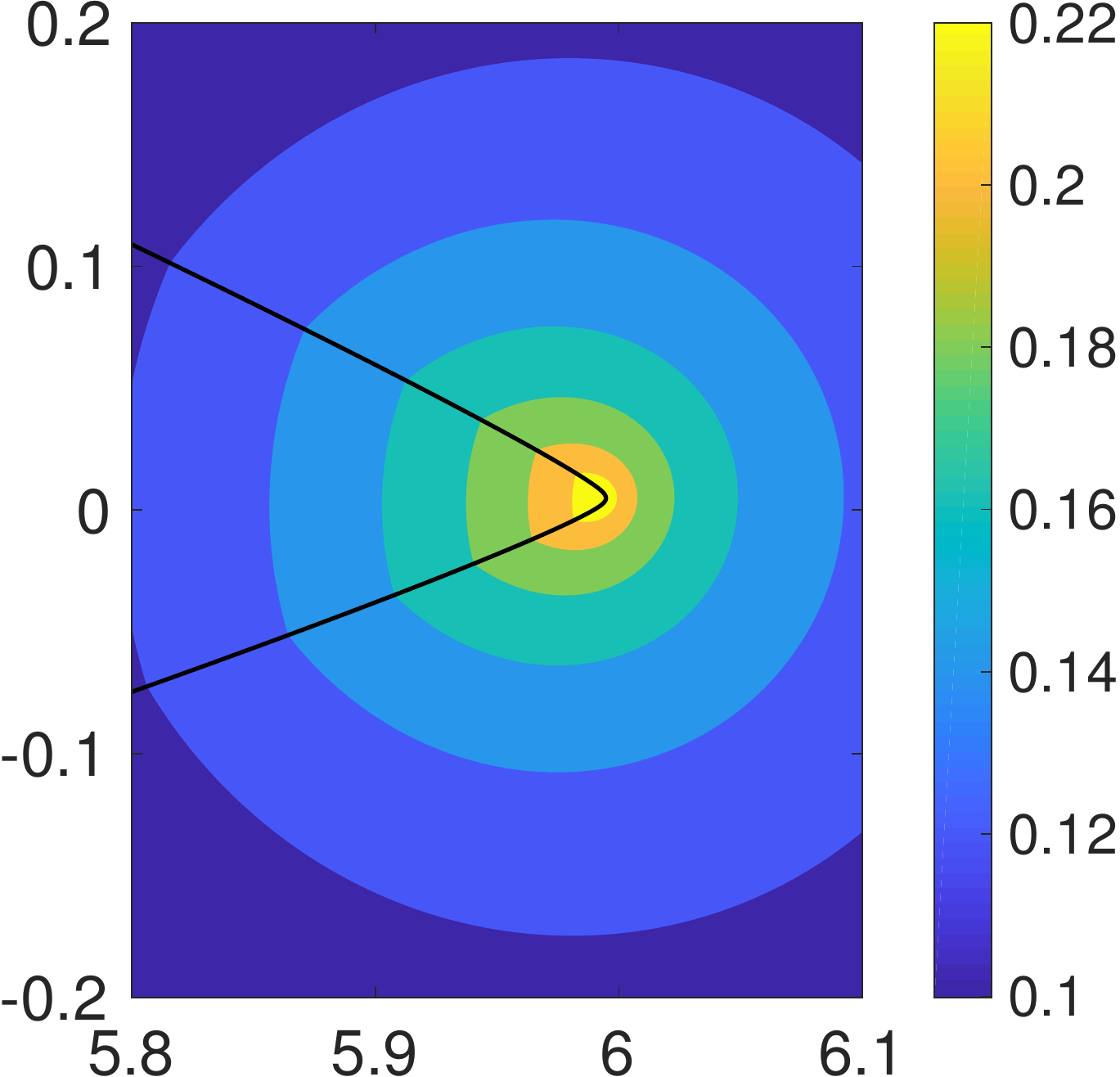}}\\
\subfigure[]{
\includegraphics[width=0.2\textwidth]{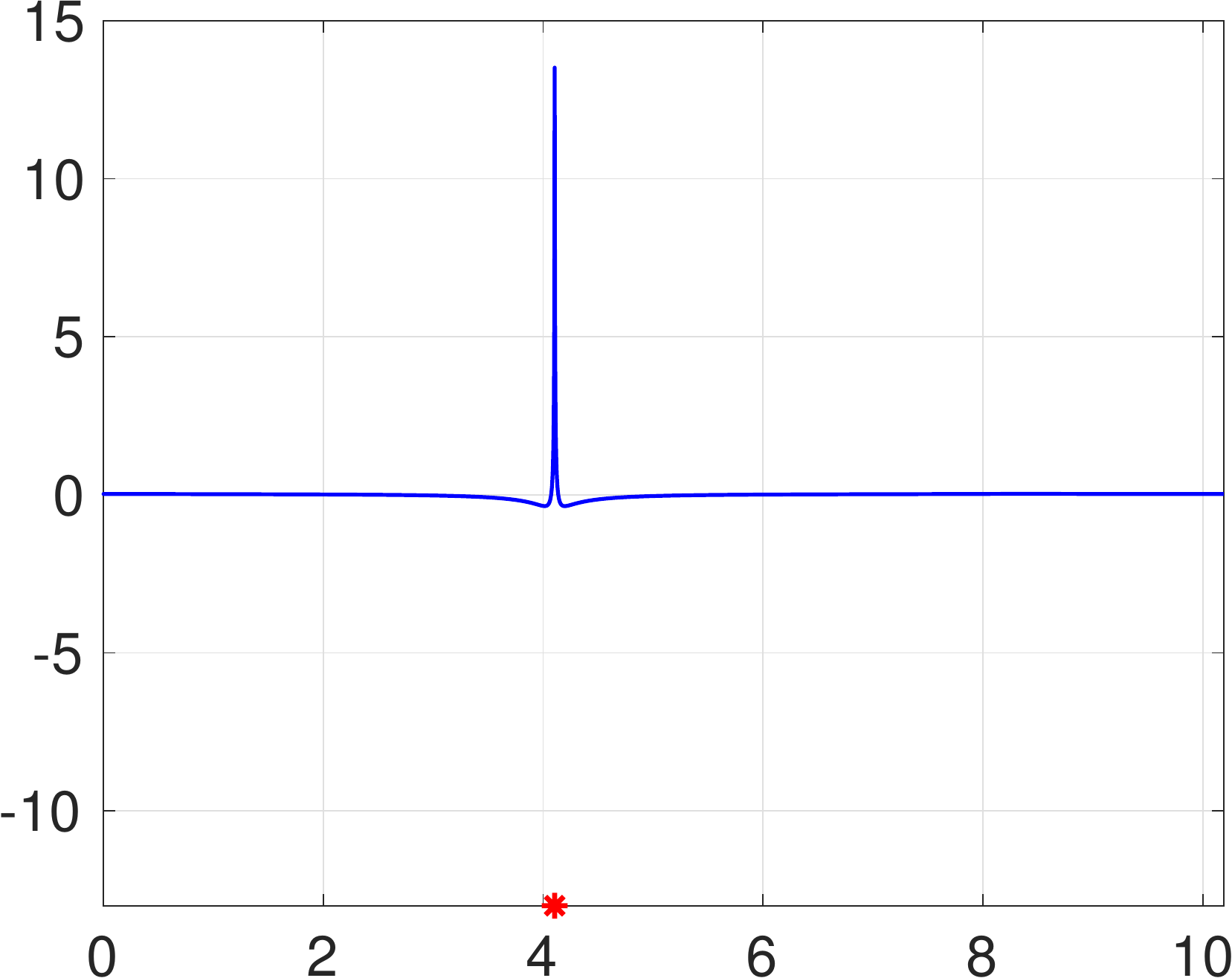}}
\subfigure[]{
\includegraphics[width=0.2\textwidth]{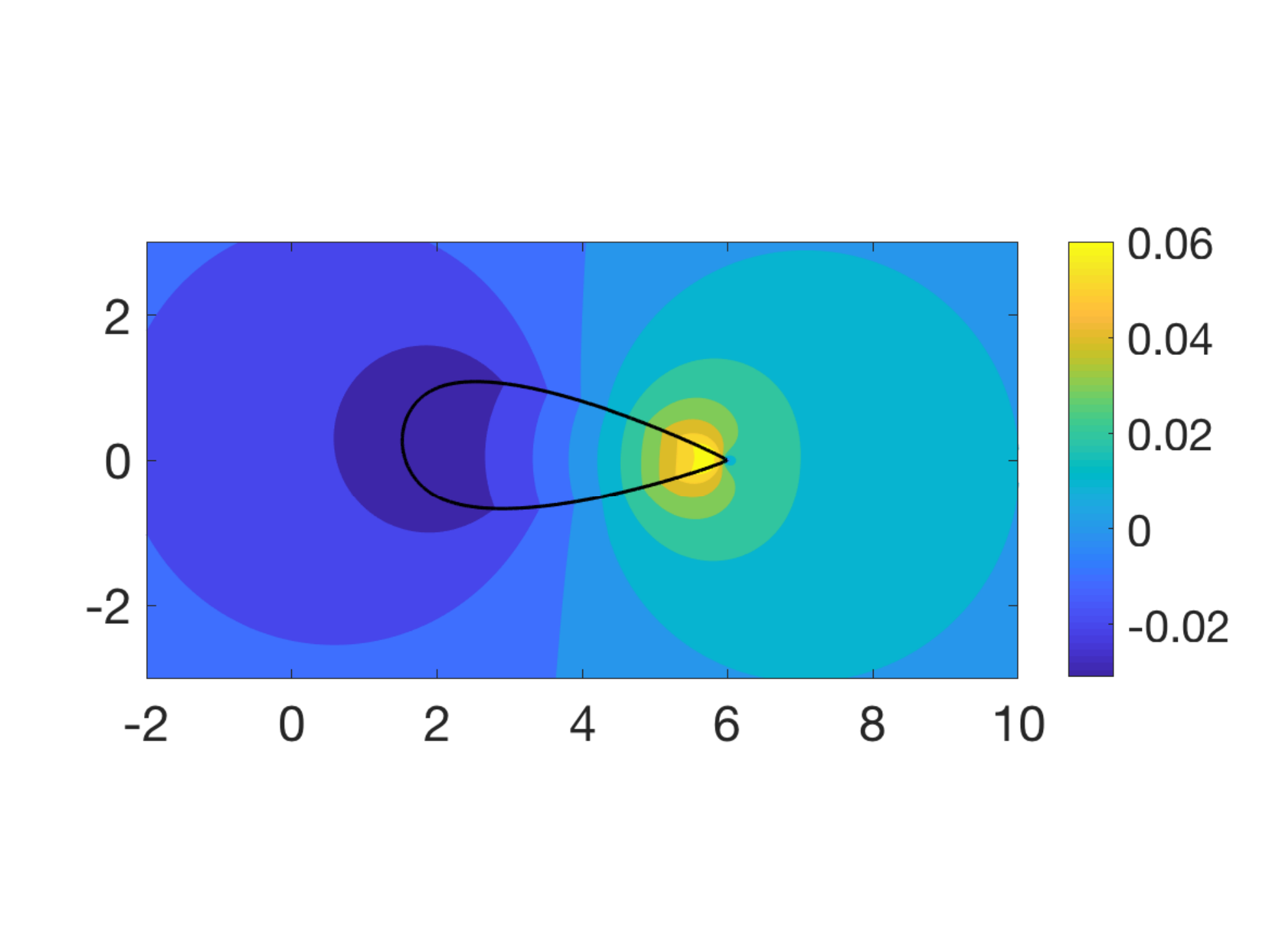}}
\subfigure[]{
\includegraphics[width=0.2\textwidth]{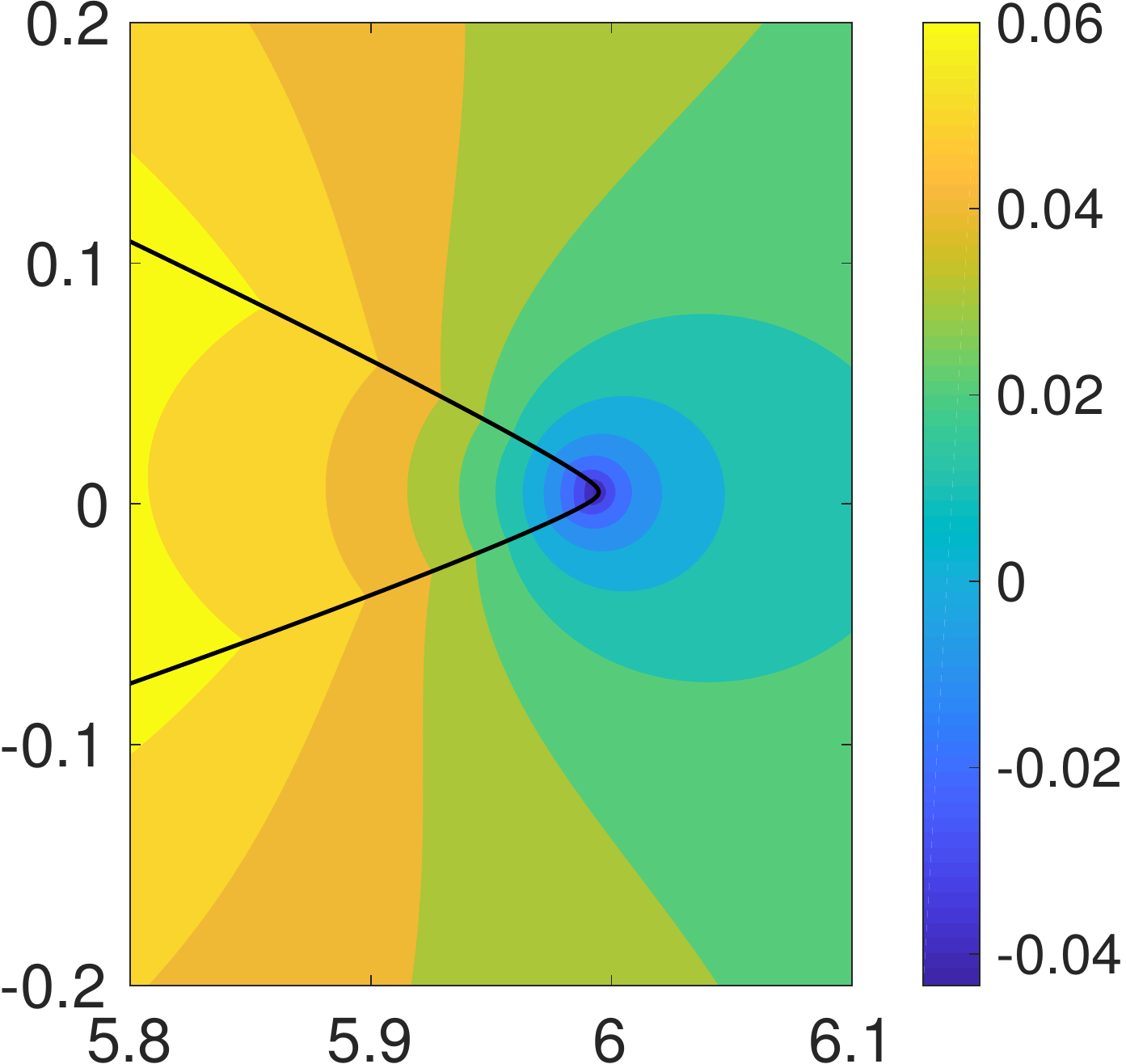}}\\
\caption{\label{fig2} (a). Plotting of the eigenfunction for $\lambda_1=0.2575$ with respect to the arc length; (b).
The associated single-layer potential for $\lambda_1=0.2575$; (c). The single-layer potential around the high-curvature point; 
(d), (e), (f). The corresponding items for $\lambda_3=0.1365$. }
\end{figure}

Fig.~\ref{fig3} plots the eigenfunctions as well as the corresponding conormal derivatives and single-layer potentials for the negative eigenvalues 
$\lambda_2=-0.2575$ and $\lambda_4=-0.1365$, respectively. The numerical results clearly support our assertion about the 
NP eigenfunctions associated to simple negative eigenvalues.

\begin{figure}
\centering
\subfigure[]{
\includegraphics[width=0.2\textwidth]{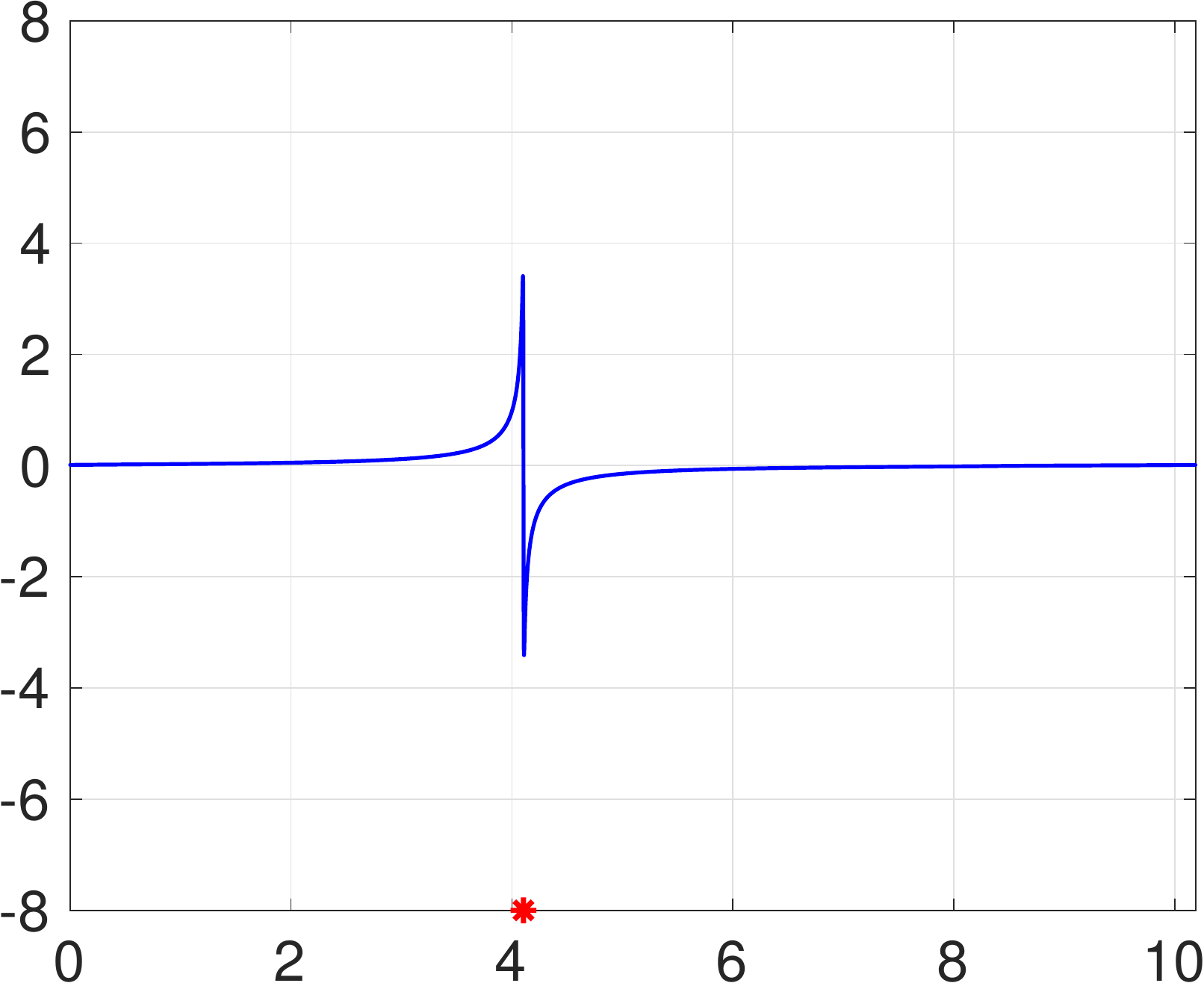}}
\subfigure[]{
\includegraphics[width=0.2\textwidth]{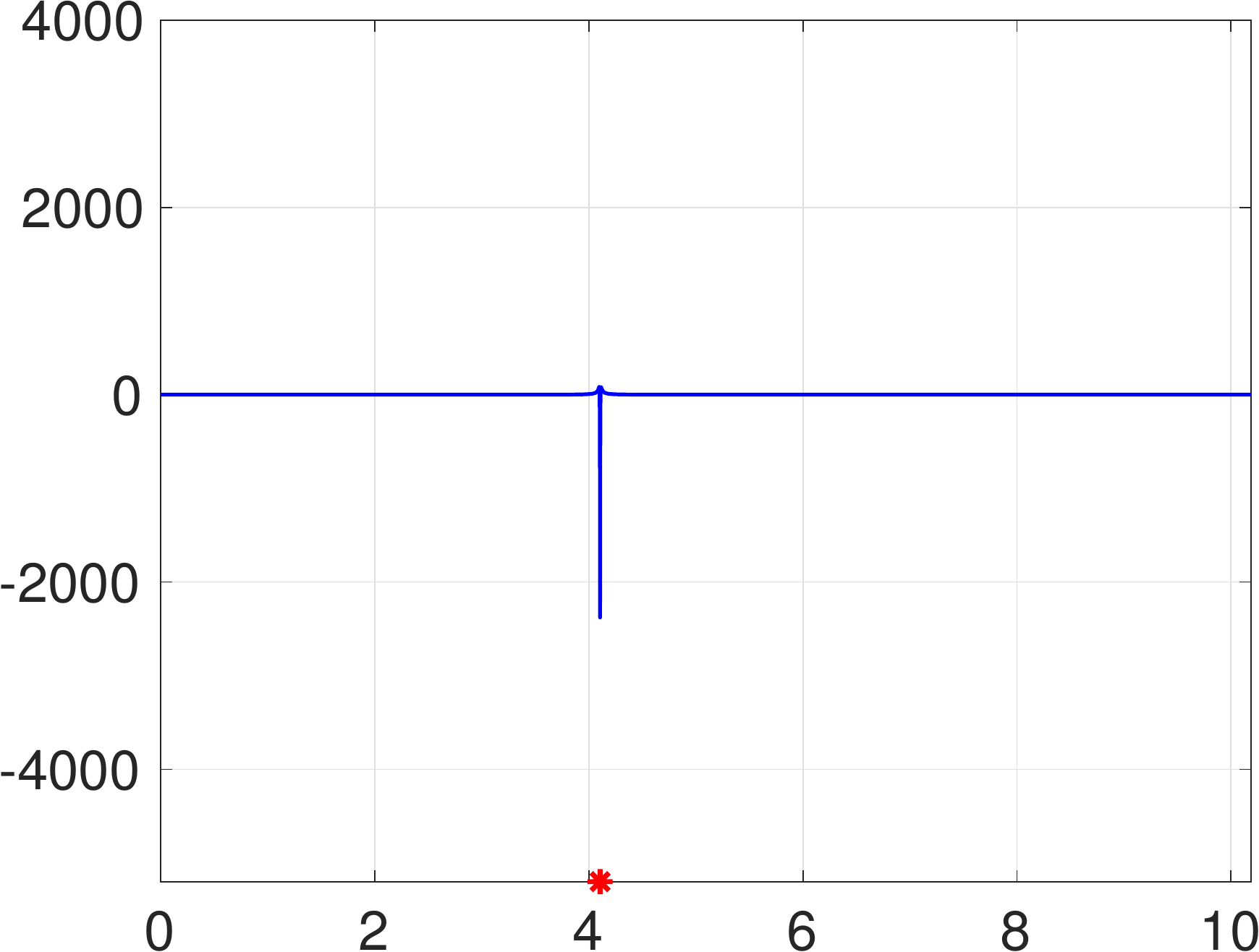}}
\subfigure[]{
\includegraphics[width=0.2\textwidth]{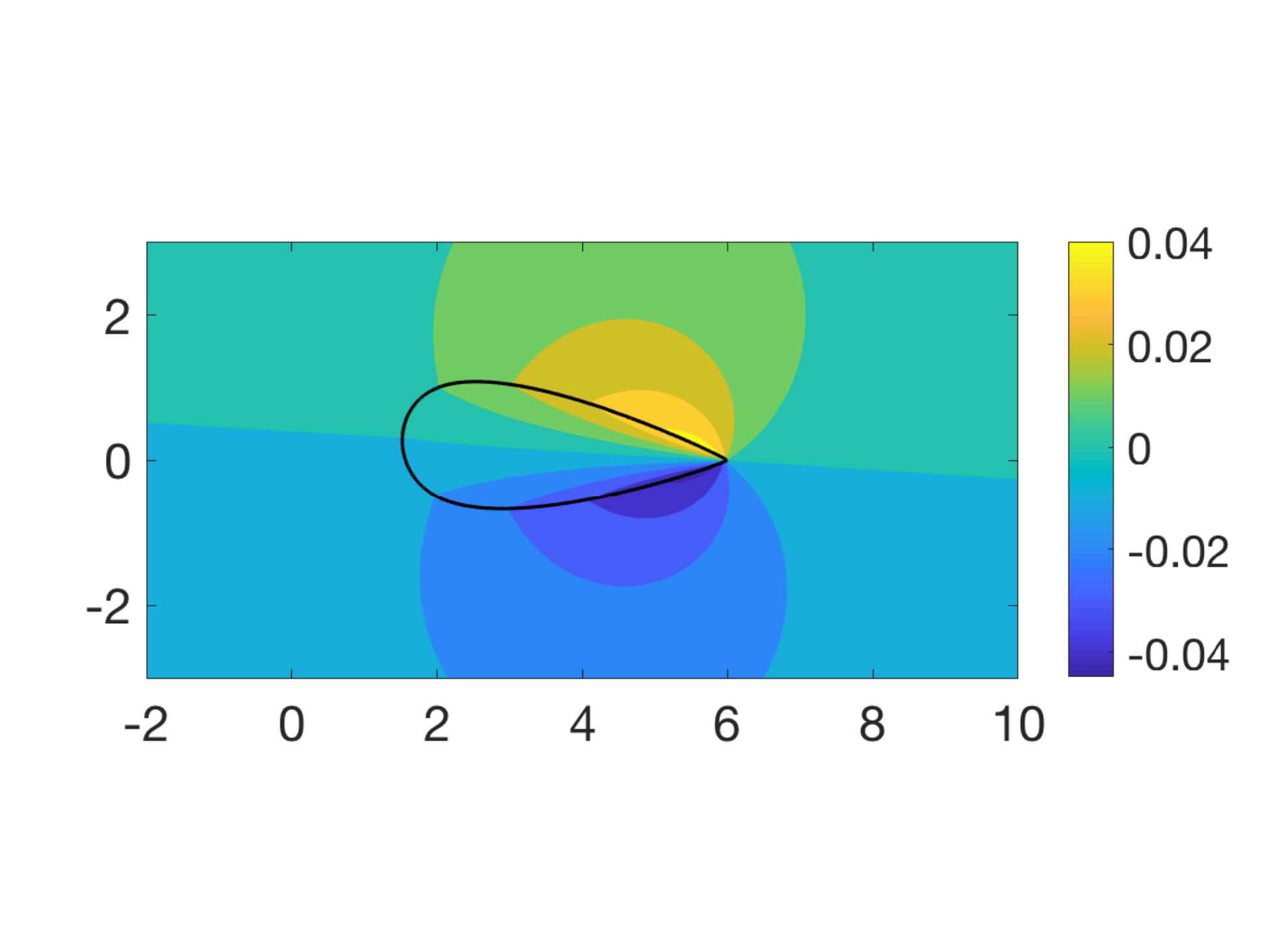}}
\subfigure[]{
\includegraphics[width=0.2\textwidth]{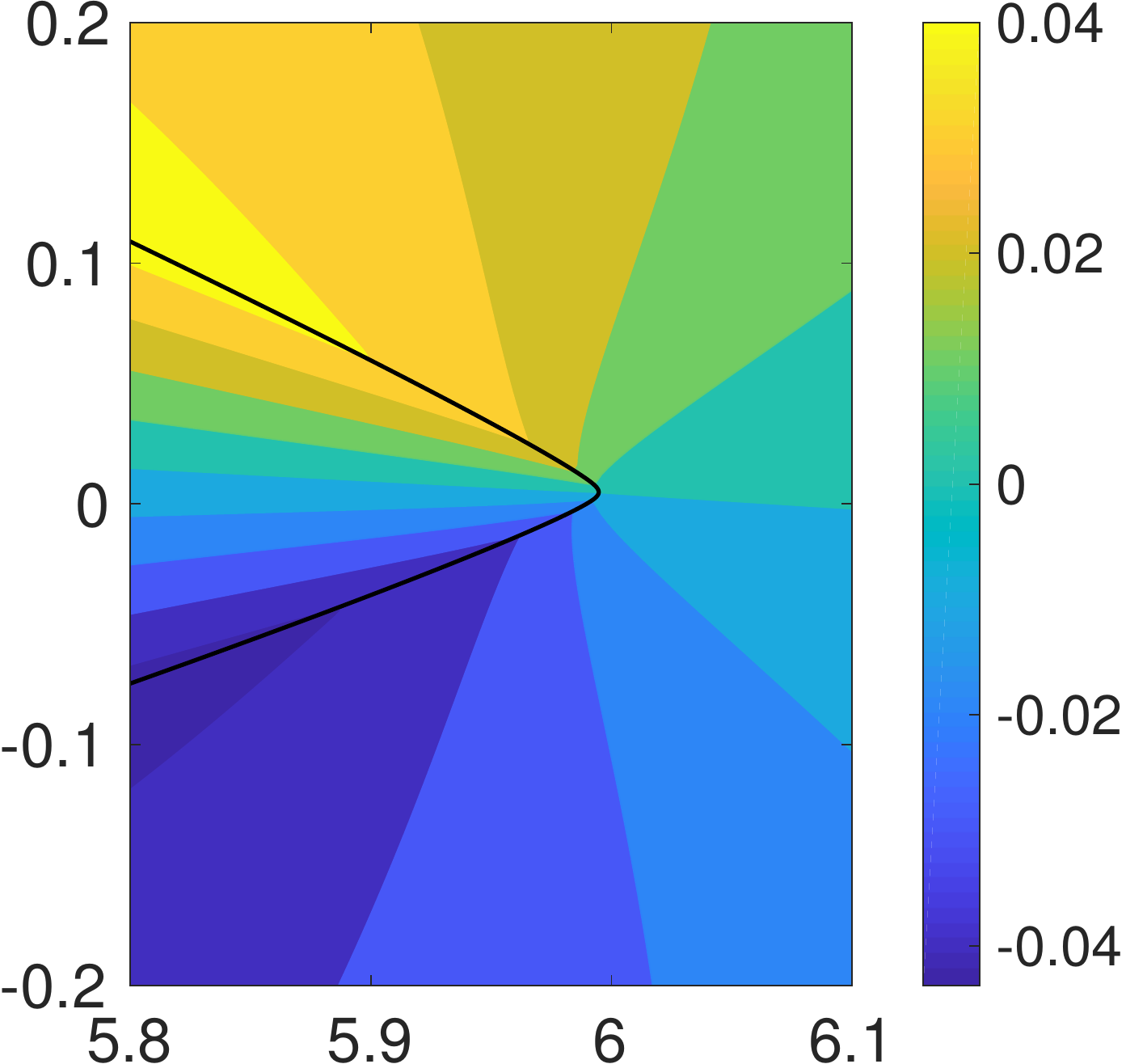}}\\
\subfigure[]{
\includegraphics[width=0.2\textwidth]{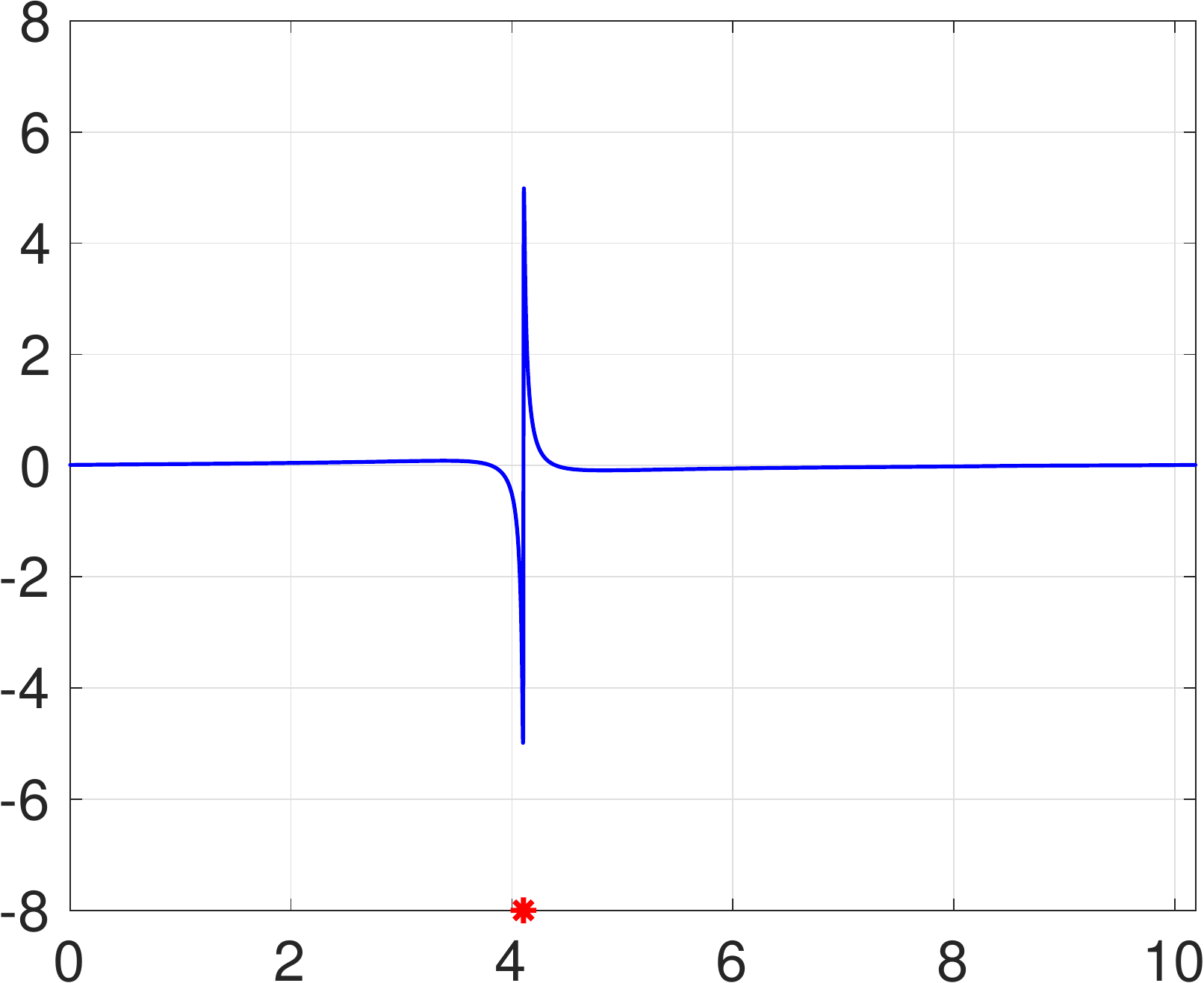}}
\subfigure[]{
\includegraphics[width=0.2\textwidth]{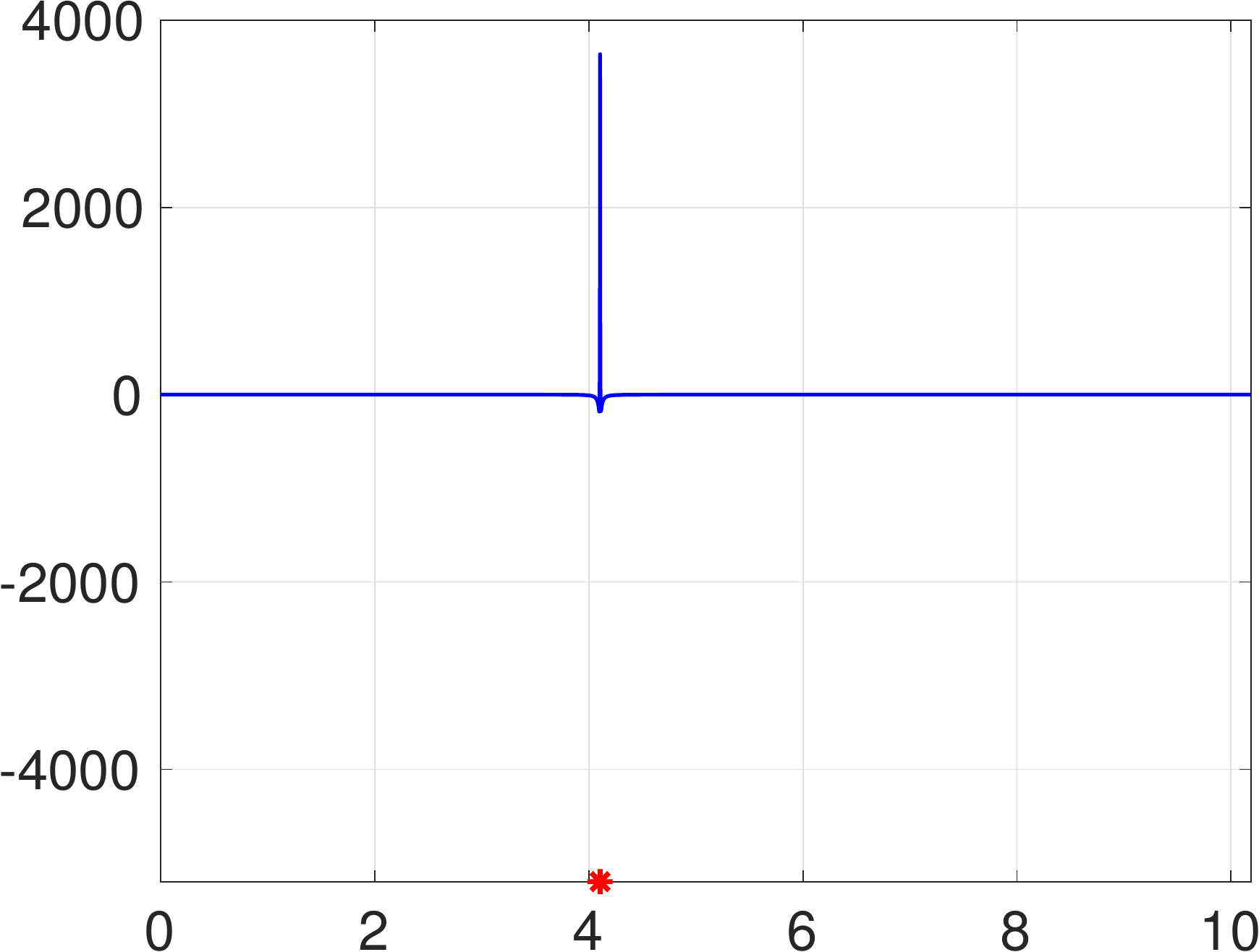}}
\subfigure[]{
\includegraphics[width=0.2\textwidth]{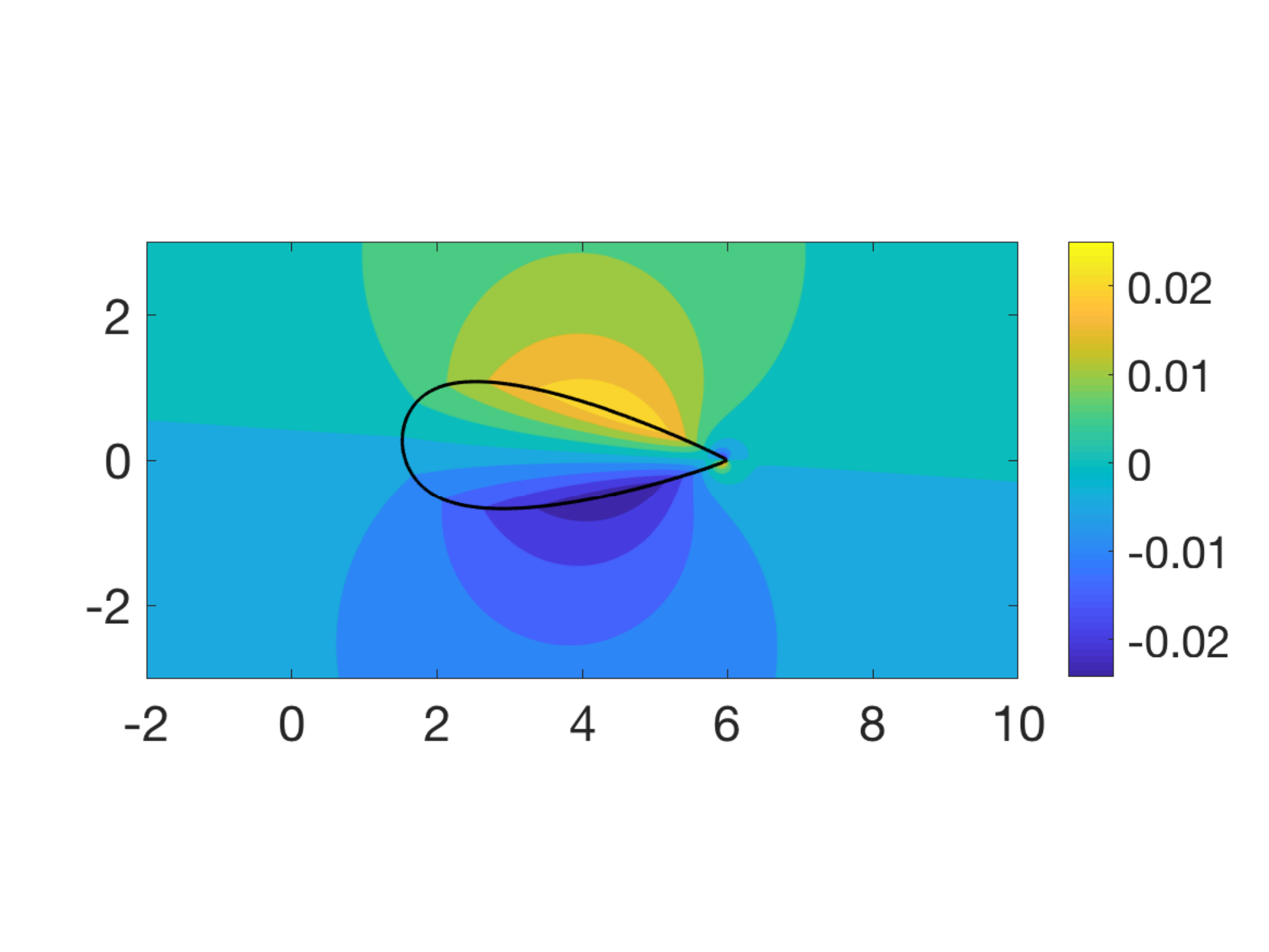}}
\subfigure[]{
\includegraphics[width=0.2\textwidth]{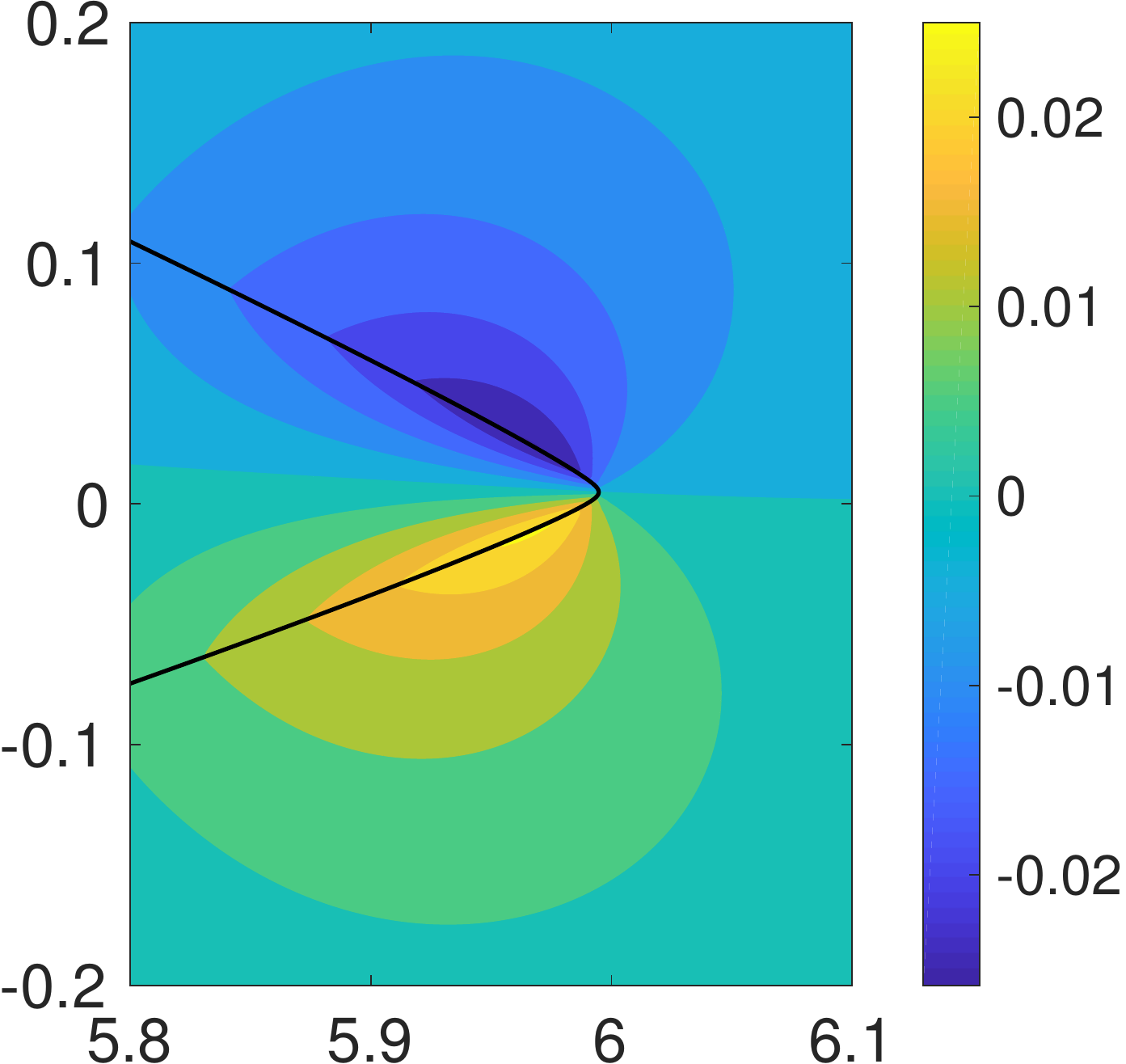}}\\

\caption{\label{fig3} (a), (b). Plotting of the eigenfunction and its conormal derivative for $\lambda_2=-0.2575$; 
(c), (d). The associated single-layer potential for $\lambda_2=-0.2575$; (e), (f), (g), (h). 
The corresponding items for the negative eigenvalues $\lambda_4=-0.1365$.}
\end{figure}

Fig.~\ref{fig4} plots the eigenfunctions with respect to arc length for the eigenvalues 
$\lambda_1=0.2575$ and $\lambda_2=-0.2575$ with different maximum curvature $500$, $1000$ and $1500$.

\begin{figure}
\centering
\subfigure[]{
\includegraphics[width=0.2\textwidth]{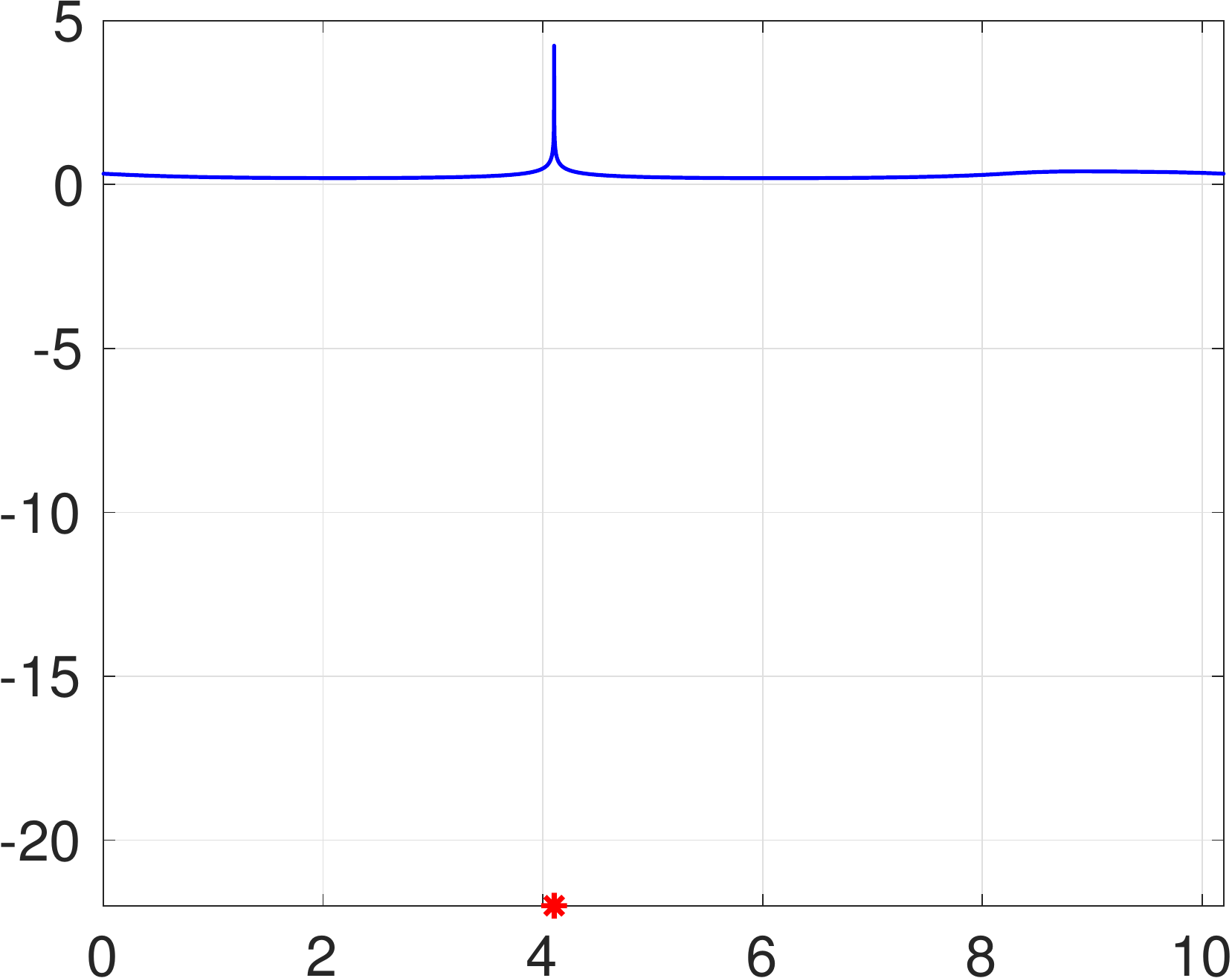}}
\subfigure[]{
\includegraphics[width=0.2\textwidth]{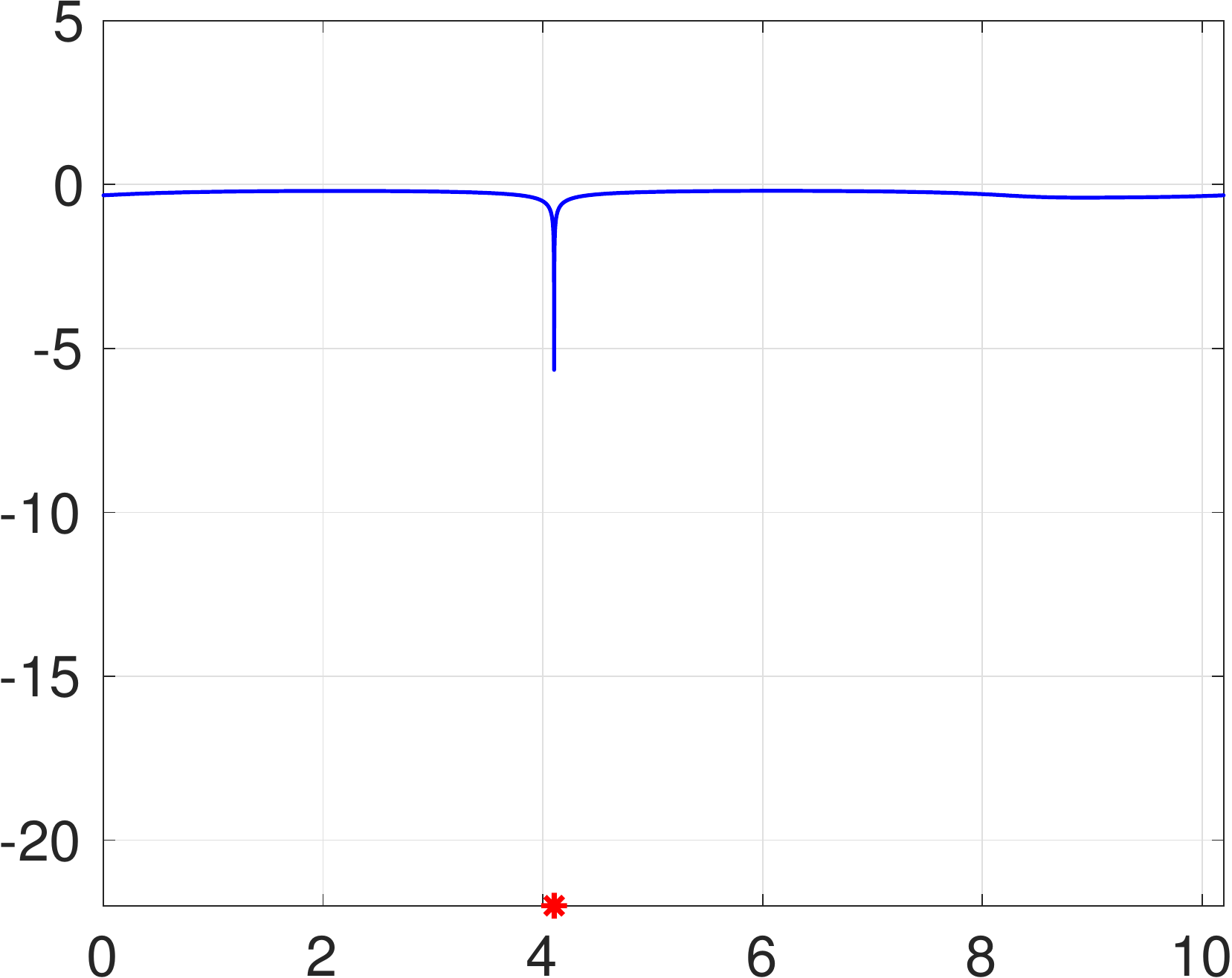}}
\subfigure[]{
\includegraphics[width=0.2\textwidth]{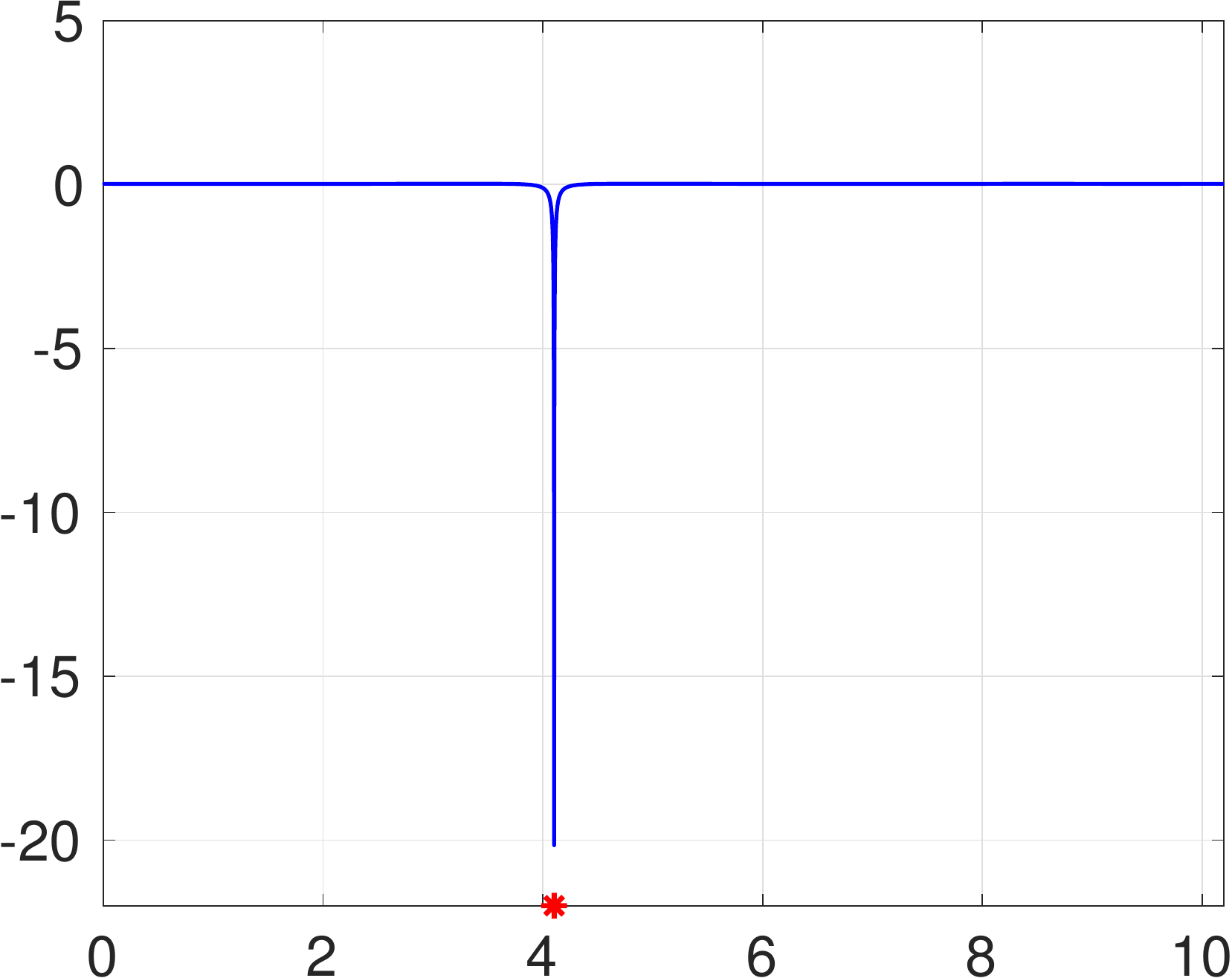}}\\
\subfigure[]{
\includegraphics[width=0.2\textwidth]{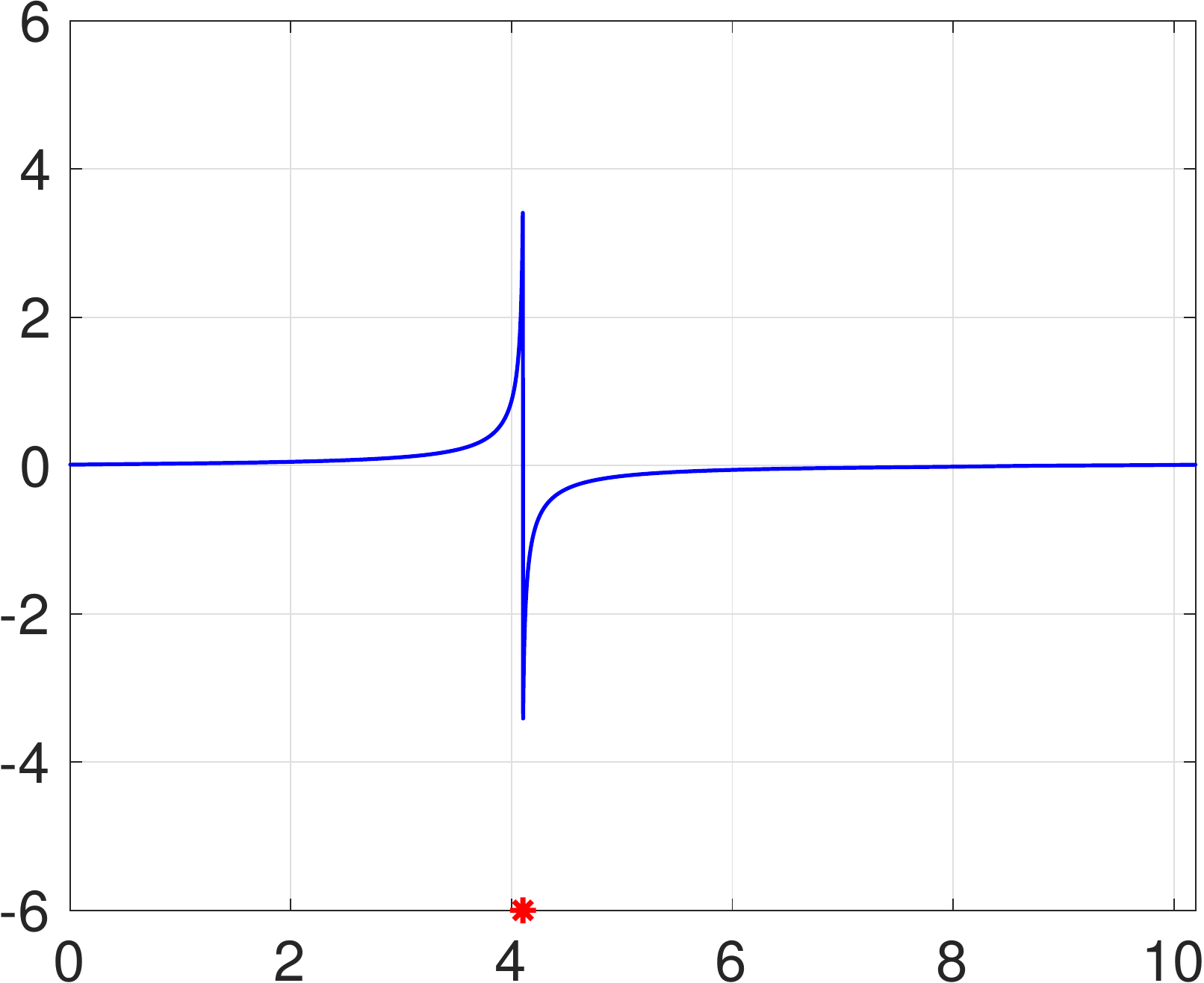}}
\subfigure[]{
\includegraphics[width=0.2\textwidth]{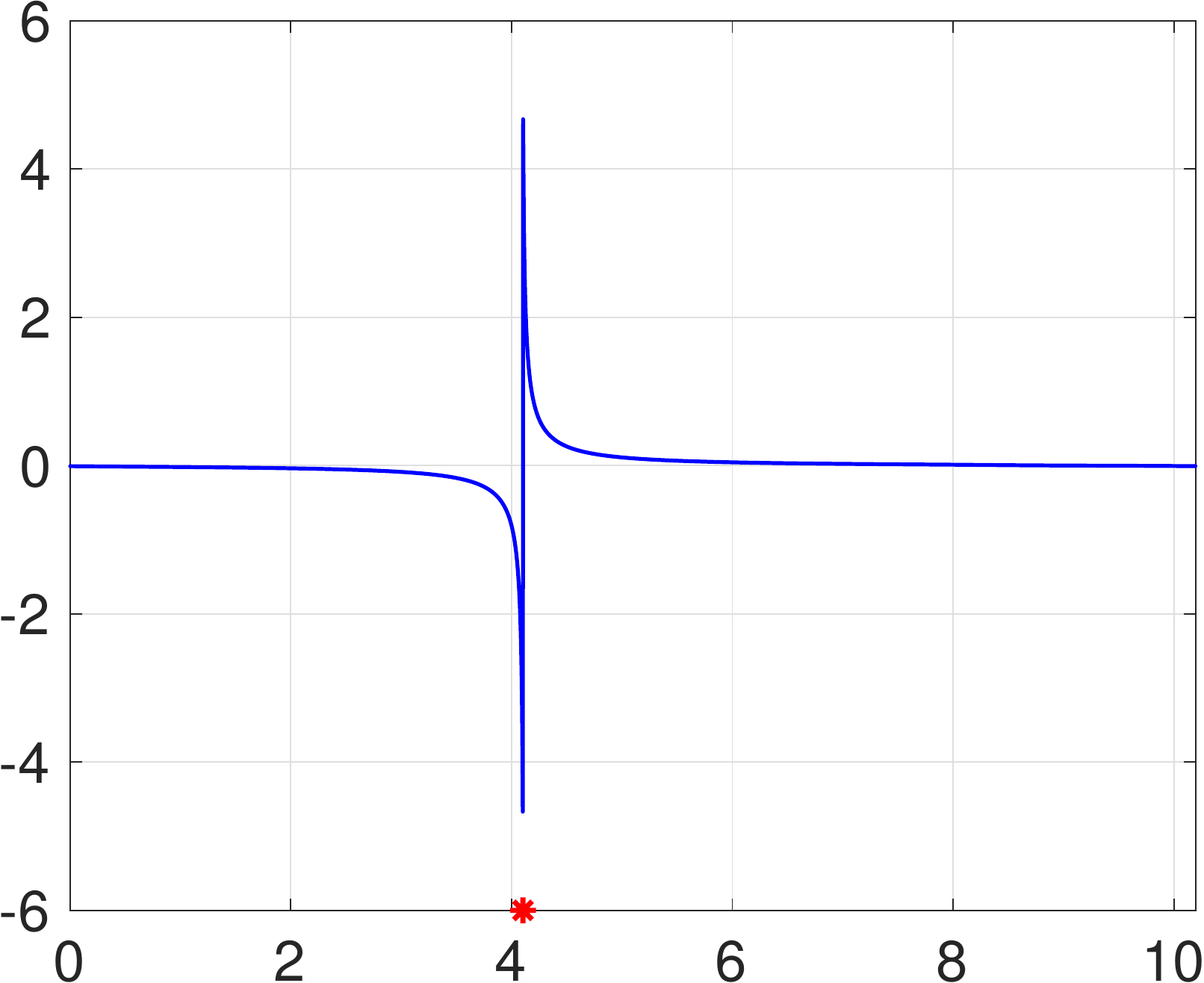}}
\subfigure[]{
\includegraphics[width=0.2\textwidth]{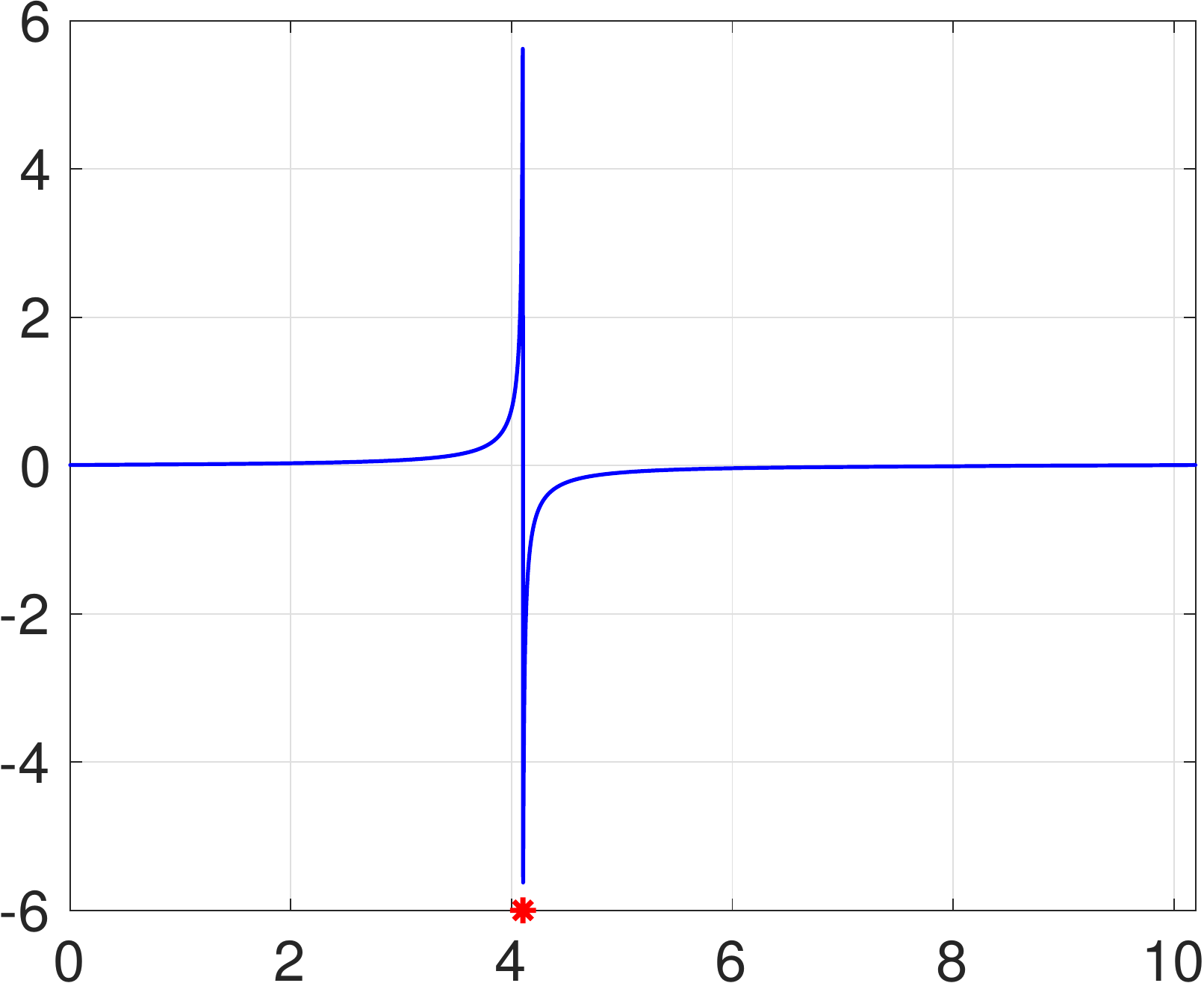}}\\
\caption{\label{fig4} (a), (b), (c).  Plotting the eigenfunctions for the positive eigenvalues $\lambda_1=0.2575$ with different maximum curvature $500$, $1000$ and $1500$. 
(d), (e), (f). The corresponding items for the negative eigenvalue $\lambda_2=-0.2575$.}
\end{figure}

We also numerically investigate the blow-up rate of the eigenfunction or its conormal derivative at a high-curvature point and plot the logarithm of the absolute value of the eigenfunctions at the high-curvature point for the positive eigenvalues $\lambda_1$, $\lambda_3$ and $\lambda_5$, and the logarithm of the absolute value of the derivative of the eigenfunctions at the high-curvature point for the negative eigenvalues $\lambda_2$, $\lambda_4$ and $\lambda_6$ with respect to different curvature in Fig.~\ref{fig5}. We find that they always 
follow the following rule
\begin{equation}\label{eq:growthrate}
 \psi_{\max} \sim a \kappa_{\max}^p\quad\mbox{as}\ \ \kappa_{\max}\rightarrow+\infty, 
\end{equation}
where $\alpha, p\in\mathbb{R}_+$ and $\psi_{\max}$ signifies the absolute value at the high-curvature point for the eigenfunction if the corresponding eigenvalue 
is positive, and for the conormal derivative of the eigenfunction if the eigenvalue is negative. In fact, by increasing the curvature at the point $x_*$and using
a standard regression, we can find the blow-up rates for different eigenvalues in \eqref{eq:egg1}. The parameters from the
regression are listed in Table~\ref{tab1}. 

%
%Fig.~\ref{fig4} shows that both the absolute value of the eigenfunction for the positive eigenvalue and the absolute value of the derivative of the eigenfunction for the negative eigenvalue at high-curvature point $x_*$ increase as the the curvature $\kappa_{x_*}$ increases, with $\kappa_{x_*}$ denoting the curvature at $x_*$. Therefore we plot the logarithm of the absolute value of the eigenfunctions at the high-curvature point for the positive eigenvalues $\lambda_1$, $\lambda_3$ and $\lambda_5$, and the logarithm of the absolute value of the derivative of the eigenfunctions at the high-curvature point for the negative eigenvalues $\lambda_2$, $\lambda_4$ and $\lambda_6$ with respect to different curvature in Fig.~\ref{fig5}. Denote by $\psi_{\max}$ the  maximum of the absolute value of the eigenfunction for the positive eigenvalue, or the maximum of absolute value of the derivative of the eigenfunction for the negative eigenvalue and from Fig.~\ref{fig5} one can conclude that by regression
%\[
% \psi_{\max} \sim a \kappa_{\max}^p.
%\]
%The coefficients of the regression are shown in Table.~\ref{tab1}.
\begin{figure}
\includegraphics[width=0.2\textwidth] {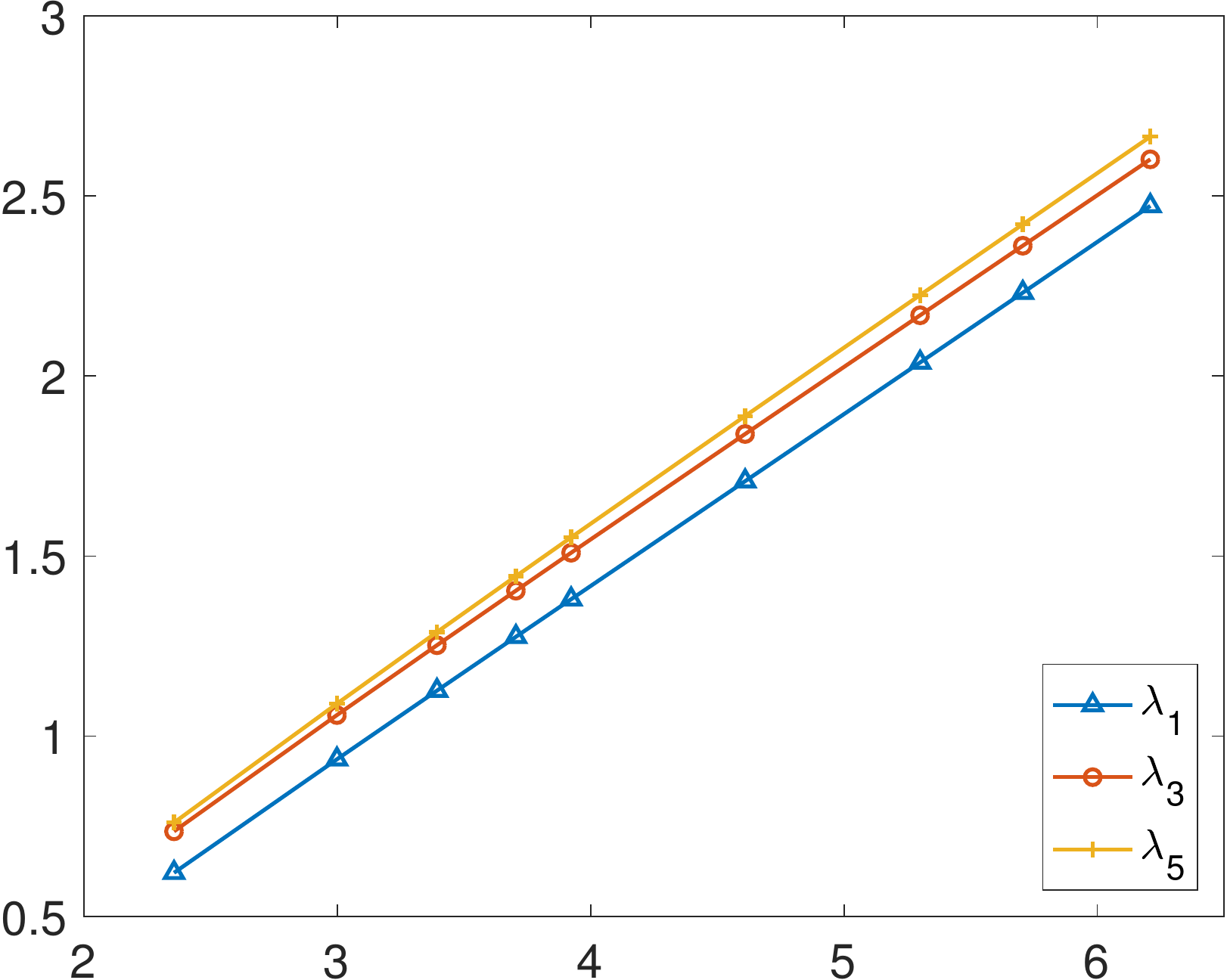}
\hspace{0.8cm}
\includegraphics[width=0.2\textwidth] {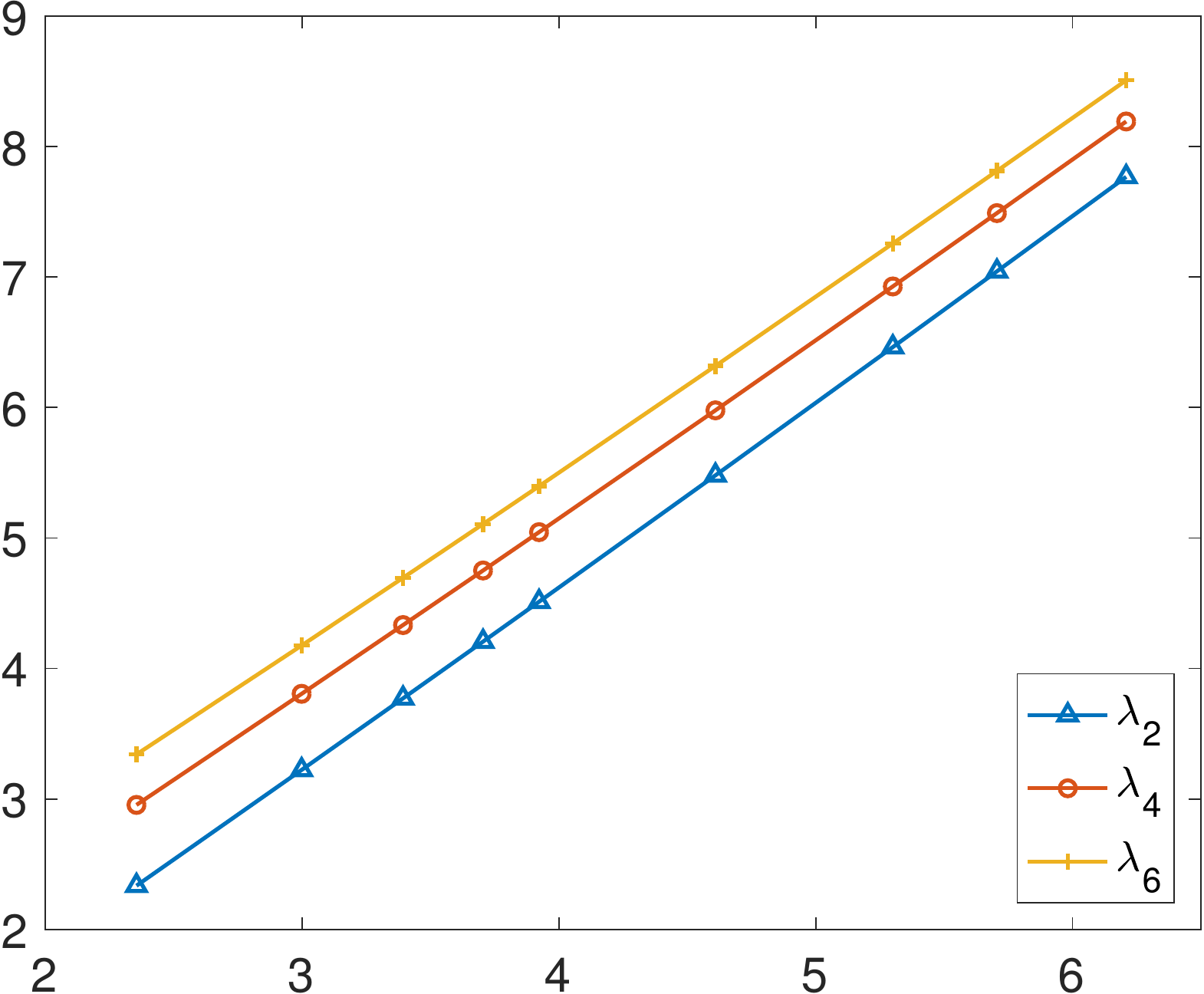}
\caption{\label{fig5} The left one plots logarithm of the eigenfunction at the high-curvature point for the positive eigenvalues $\lambda_1$, $\lambda_{3}$ and $\lambda_{5}$, and the right one plots logarithm of the derivative of the eigenfunction at the high-curvature point for the simple negative eigenvalues $\lambda_2$, $\lambda_{4}$ and $\lambda_{6}$ with respect to different curvature.}
\end{figure}
\begin{table}[t]
  \centering
  \subtable[]{
    \centering
    \begin{tabular}{cccc}
      \toprule
      & $\lambda_1$ & $\lambda_3$ & $\lambda_5$ \\
      \midrule
      $p$ & 0.4793 & 0.4824 & 0.4925 \\[5pt]
      $\ln(\alpha)$ & -0.5022 & -0.3886 & -0.3867 \\
      \bottomrule
    \end{tabular}}
  %  \caption{}
  \hspace{1.5cm}
\subtable[]{
    \centering
    \begin{tabular}{cccc}
      \toprule
      & $\lambda_2$ & $\lambda_4$ & $\lambda_6$ \\
      \midrule
      $p$ & 1.4108 & 1.3602 & 1.3423 \\[5pt]
      $\ln(\alpha)$ & -1.0102 & -0.2763 & -0.1483 \\
      \bottomrule
    \end{tabular}
    }
 %\end{subtable}
  \caption{The parameters of the form \eqref{eq:growthrate} from the regression associated with the 
  eigenvalues in \eqref{eq:egg1}: (a) $\lambda_j, j=1, 3, 5$; (b) $\lambda_j, j=2,4,6$.}
  \label{tab1}
\end{table}

%Here we remark that the eigenvalue $\lambda_5$ and $\lambda_6$ are simple eigenvalues and 
%\[
% \lambda_5=0.0685 \quad \mbox{and} \quad \lambda_6=-0.0685.
%\]

\subsection{A convex 3-symmetric domain}

In this subsection, we consider a convex 3-symmetric domain as shown in Fig.~\ref{fig11}, which possesses three high-curvature points that are 
denoted by $x_{*}$, $x_{\triangle}$ and $x_{o}$ as shown in Fig.~ \ref{fig11}. The largest curvature is 
\begin{equation}\label{eq:lc1}
 \kappa_{x_{*}}=\kappa_{x_{\triangle}}=\kappa_{x_{o}}=500.
\end{equation}
\begin{figure}[t]
\includegraphics[width=3cm] {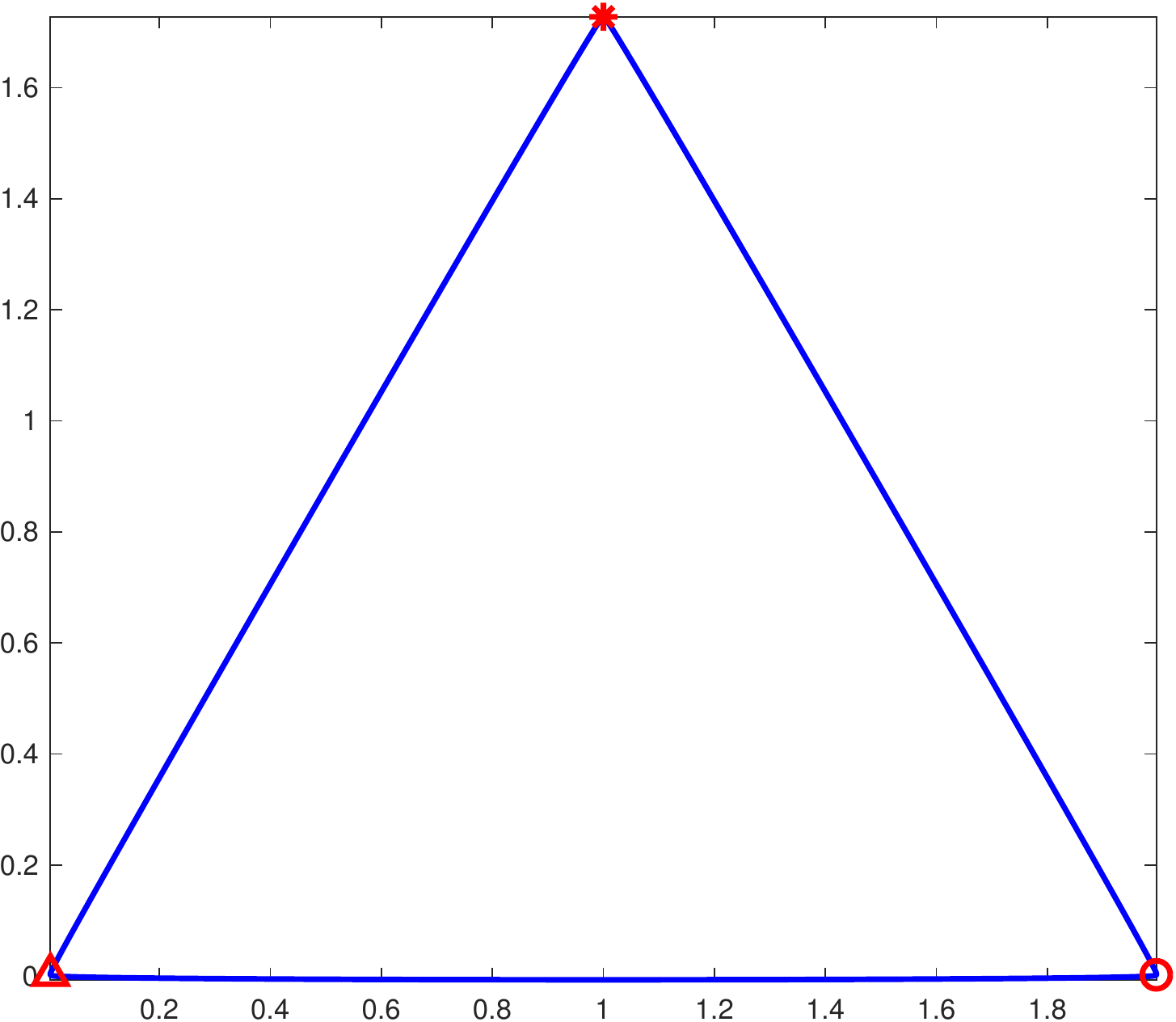}
\caption{\label{fig11} A convex 3-symmetric domain.}
\end{figure}
The first seven largest eigenvalues (in terms of the absolute value) are numerically found to be
\begin{equation}\label{eq:egg3_1}
\begin{split}
\lambda_0&=0.5, \  \lambda_1= \lambda_2=0.2850,\  \lambda_3= \lambda_4=-0.2850, \ \lambda_5=0.2583,\ \lambda_6=-0.2583, \\
   & \  \lambda_7= \lambda_3=0.1906,\  \lambda_9= \lambda_10=-0.1906, \ \lambda_{11}=0.1568,\ \lambda_{12}=-0.1568, \\
   & \  \lambda_{13}= \lambda_{14}=0.1111,\  \lambda_{15}= \lambda_{16}=-0.1111, \ \lambda_{17}=0.0875,\ \lambda_{18}=-0.0875.
\end{split}
\end{equation}
Compared to the study in the previous subsection, there are multiple NP eigenvalues occurring for the 3-symmetric domain. Hence, we can verify our assertion
about the NP eigenfunction associated to a multiple NP eigenvalue. In the following, we first show the case for the simple eigenvalue and then the case for the multiple eigenvalue.

Fig.~\ref{fig12} plots the eigenfunctions as well as the associated single-layer potentials, respectively, for the positive eigenvalues  
$\lambda_5=0.2583$. The numerical results clearly support our assertion about the 
NP eigenfunctions associated to simple positive eigenvalues. 

\begin{figure}
\centering
\subfigure[]{
\includegraphics[width=0.2\textwidth]{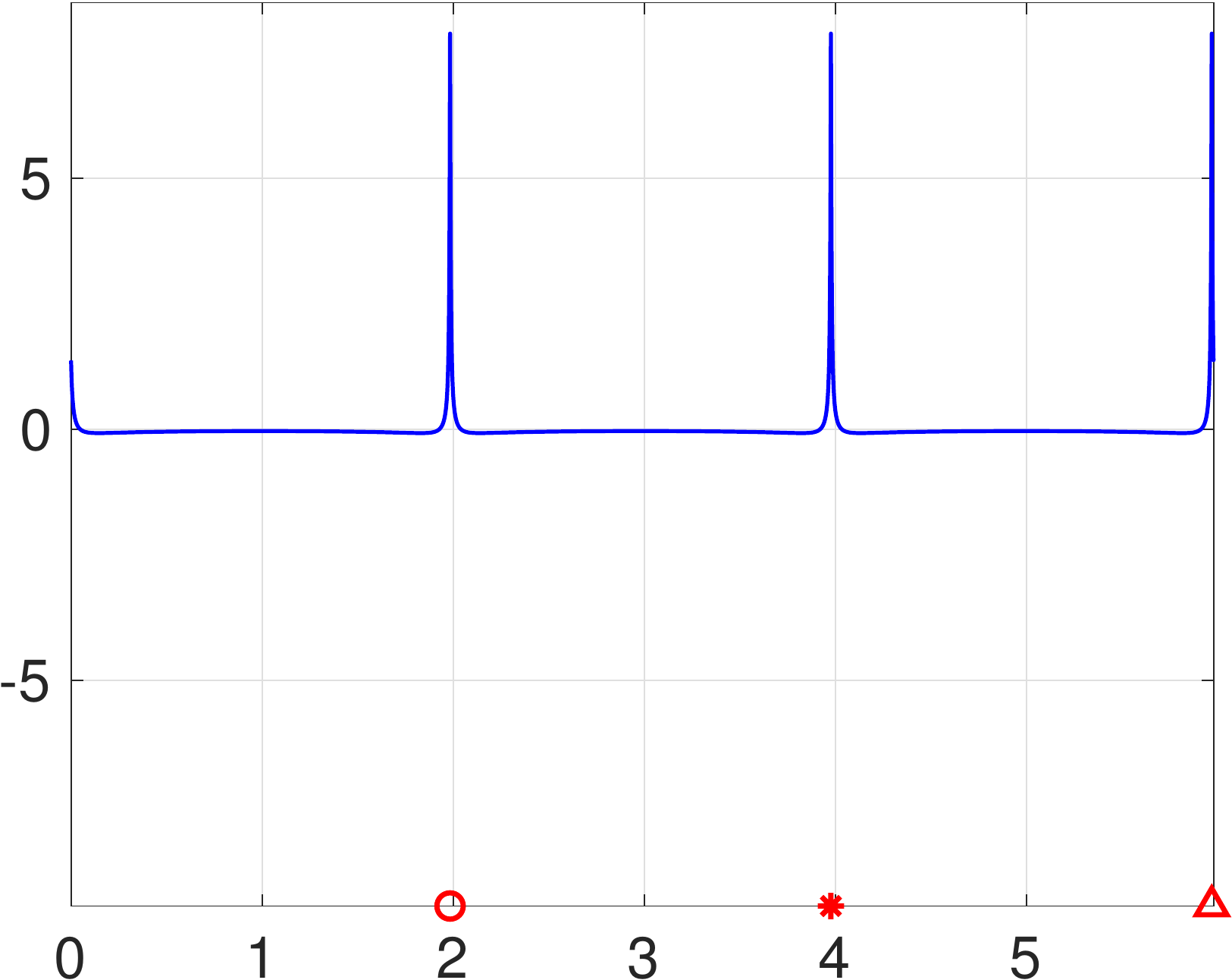}}
\subfigure[]{
\includegraphics[width=0.2\textwidth]{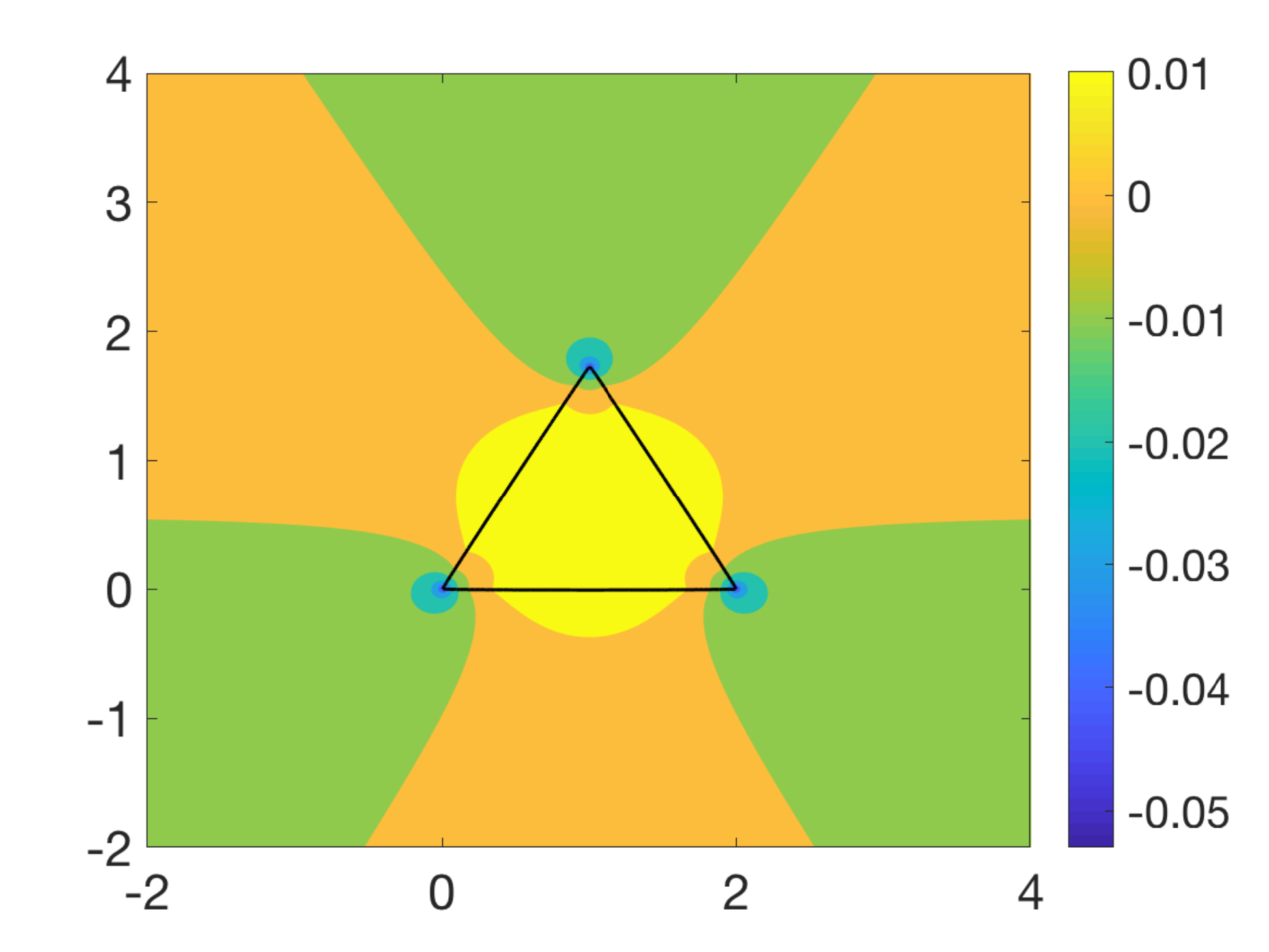}}
\subfigure[]{
\includegraphics[width=0.2\textwidth]{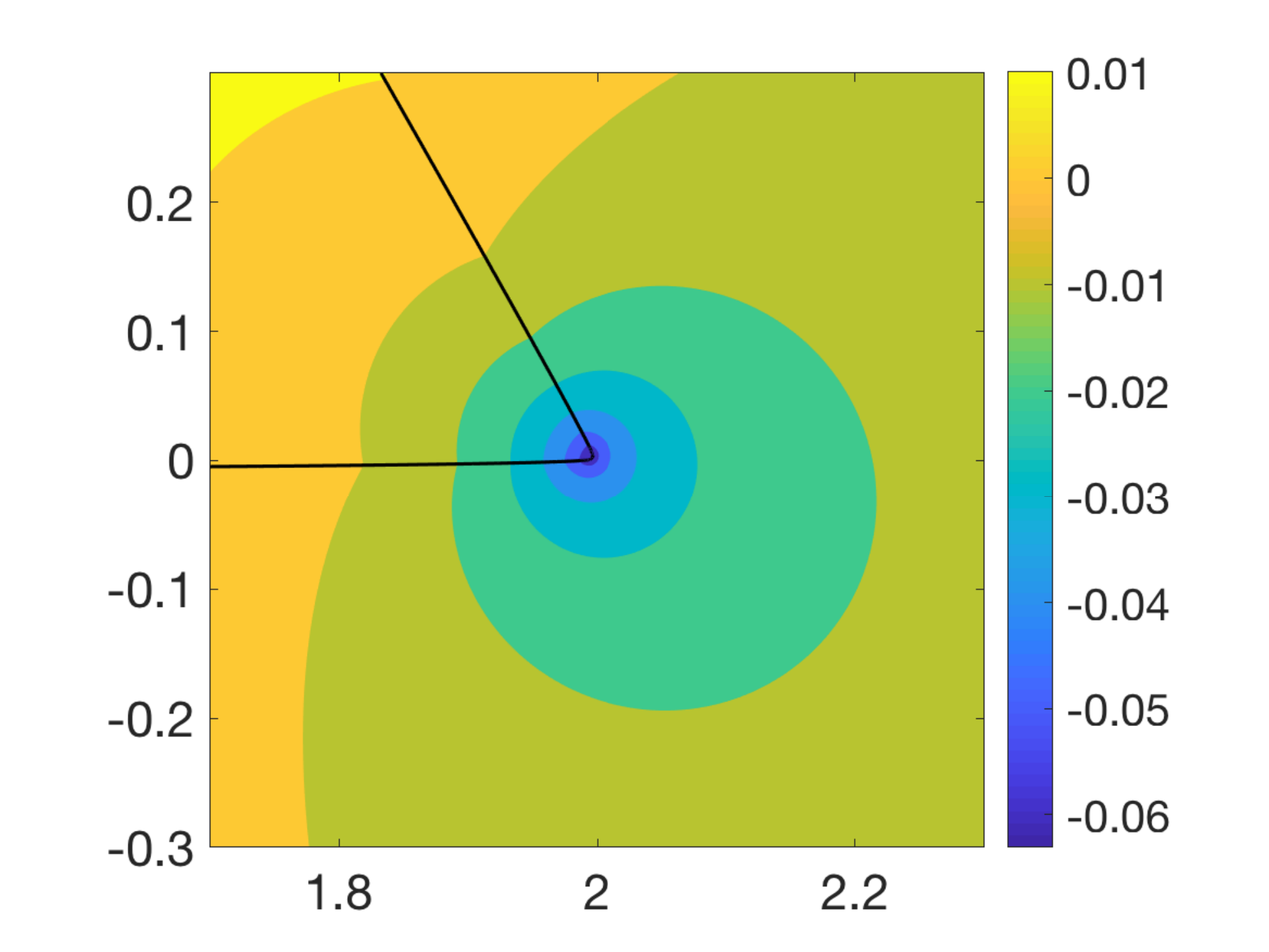}}
\caption{\label{fig12} (a). Plotting of the eigenfunction for $\lambda_5=0.2583$ with respect to the arc length; (b).
The associated single-layer potential for $\lambda_5=0.2583$; (c). The single-layer potential around the high-curvature point;  }
\end{figure}

Fig.~\ref{fig13} plots the eigenfunctions as well as the corresponding conormal derivatives and single-layer potentials for the negative eigenvalues 
$\lambda_6=-0.2583$. The numerical results clearly support our assertion about the NP eigenfunctions associated to simple negative eigenvalues.

\begin{figure}
\centering
\subfigure[]{
\includegraphics[width=0.2\textwidth]{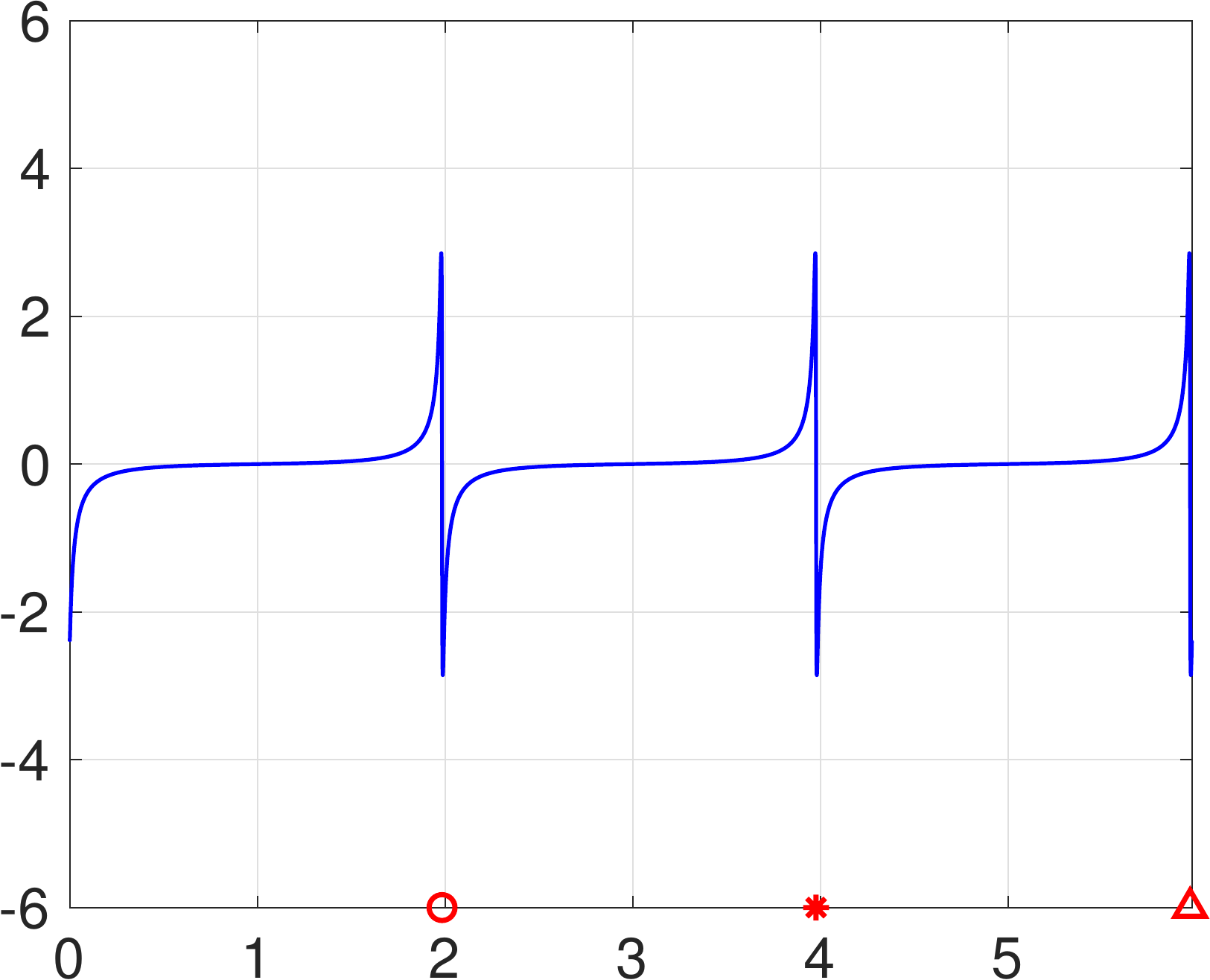}}
\subfigure[]{
\includegraphics[width=0.2\textwidth]{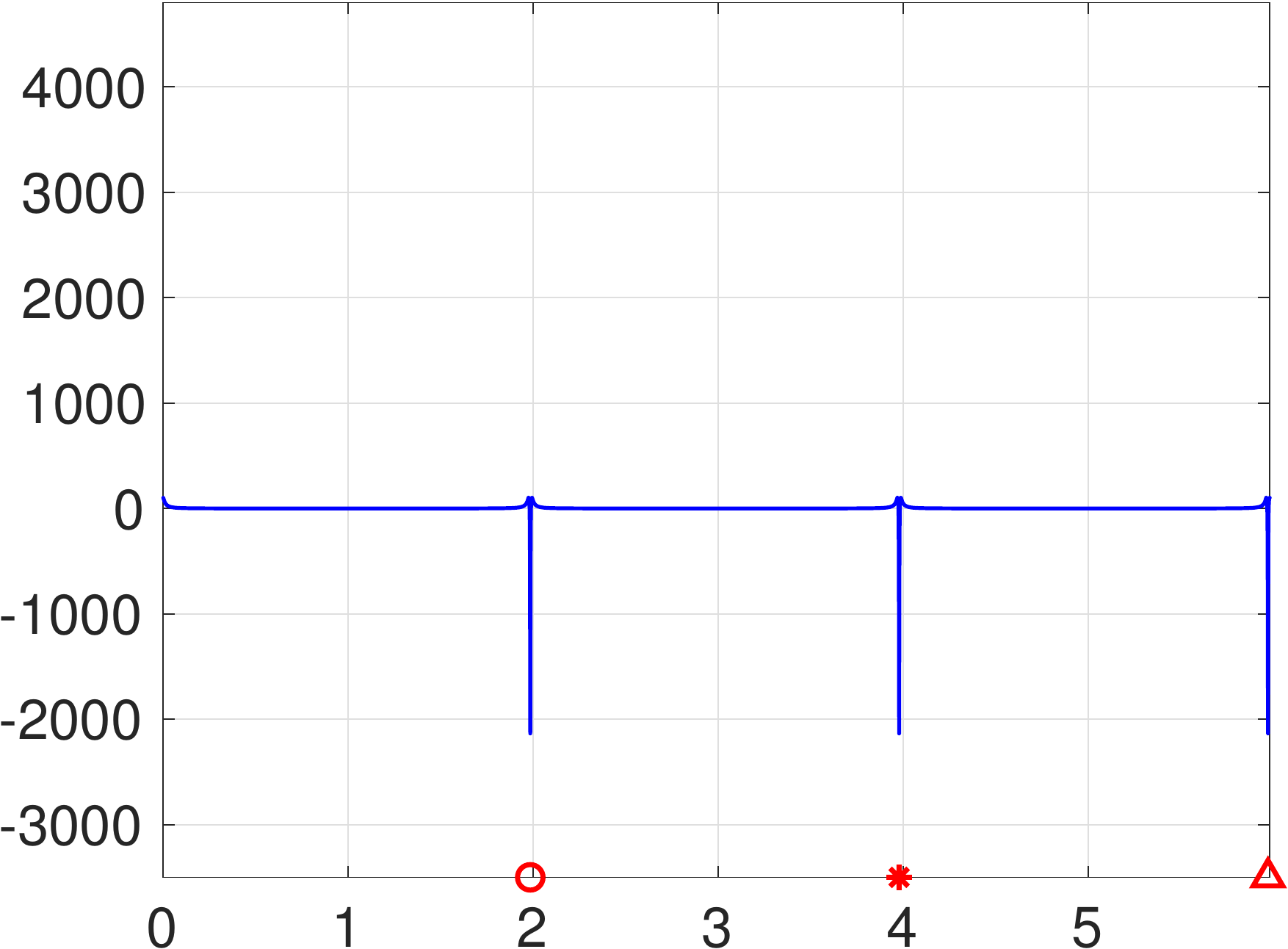}}
\subfigure[]{
\includegraphics[width=0.2\textwidth]{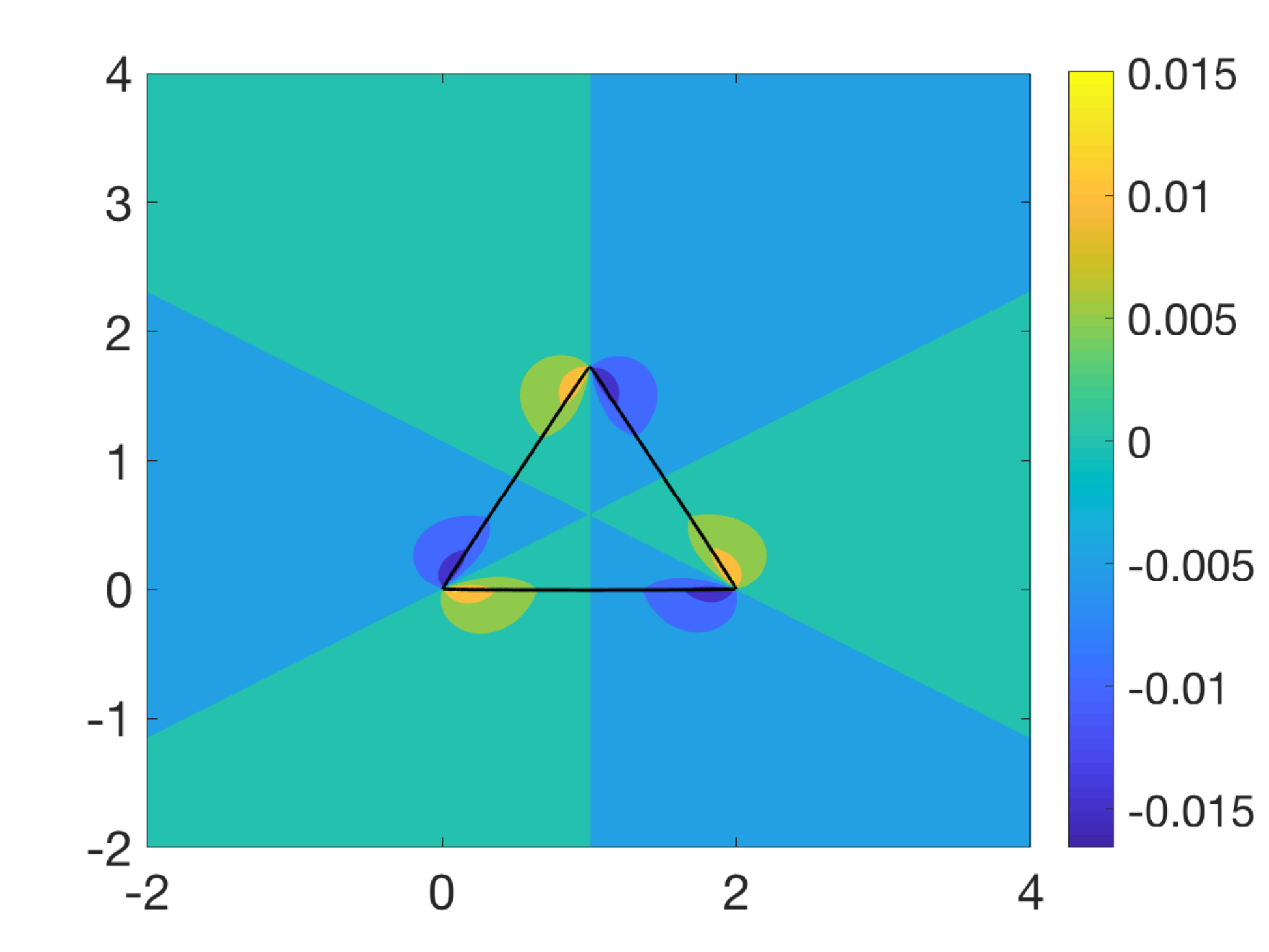}}
\subfigure[]{
\includegraphics[width=0.2\textwidth]{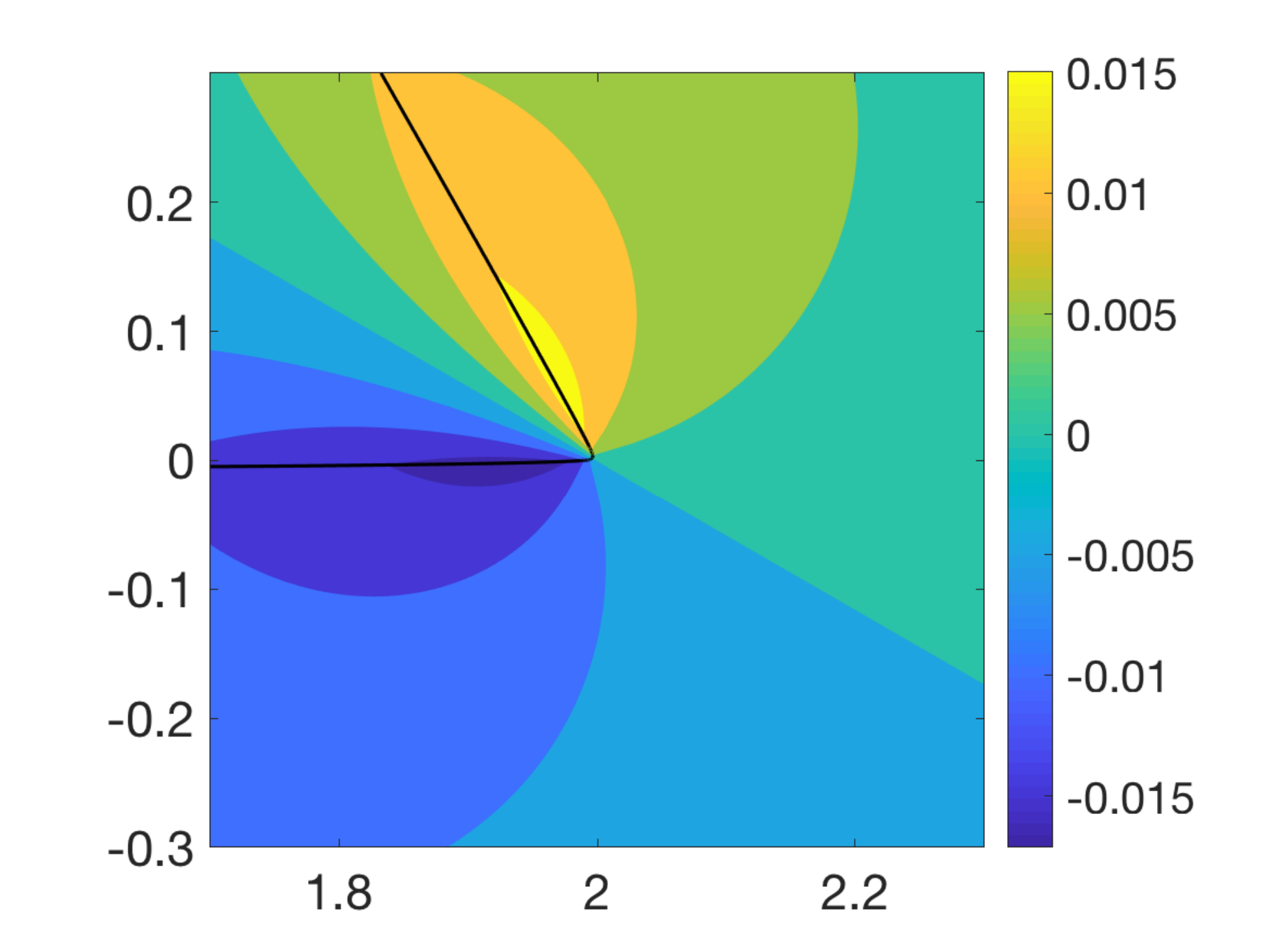}}
\caption{\label{fig13}  (a), (b). Plotting of the eigenfunction and its conormal derivative for $\lambda_6=-0.2583$; 
(c), (d). The associated single-layer potential for $\lambda_6=-0.2583$. }
\end{figure}

Fig.~\ref{fig14} plots the eigenfunctions with respect to arc length for the eigenvalues $\lambda_5=0.2583$ and $\lambda_6=-0.2583$ with different maximum curvature $500$, $1000$ and $1500$.

\begin{figure}
\centering
\subfigure[]{
\includegraphics[width=0.2\textwidth]{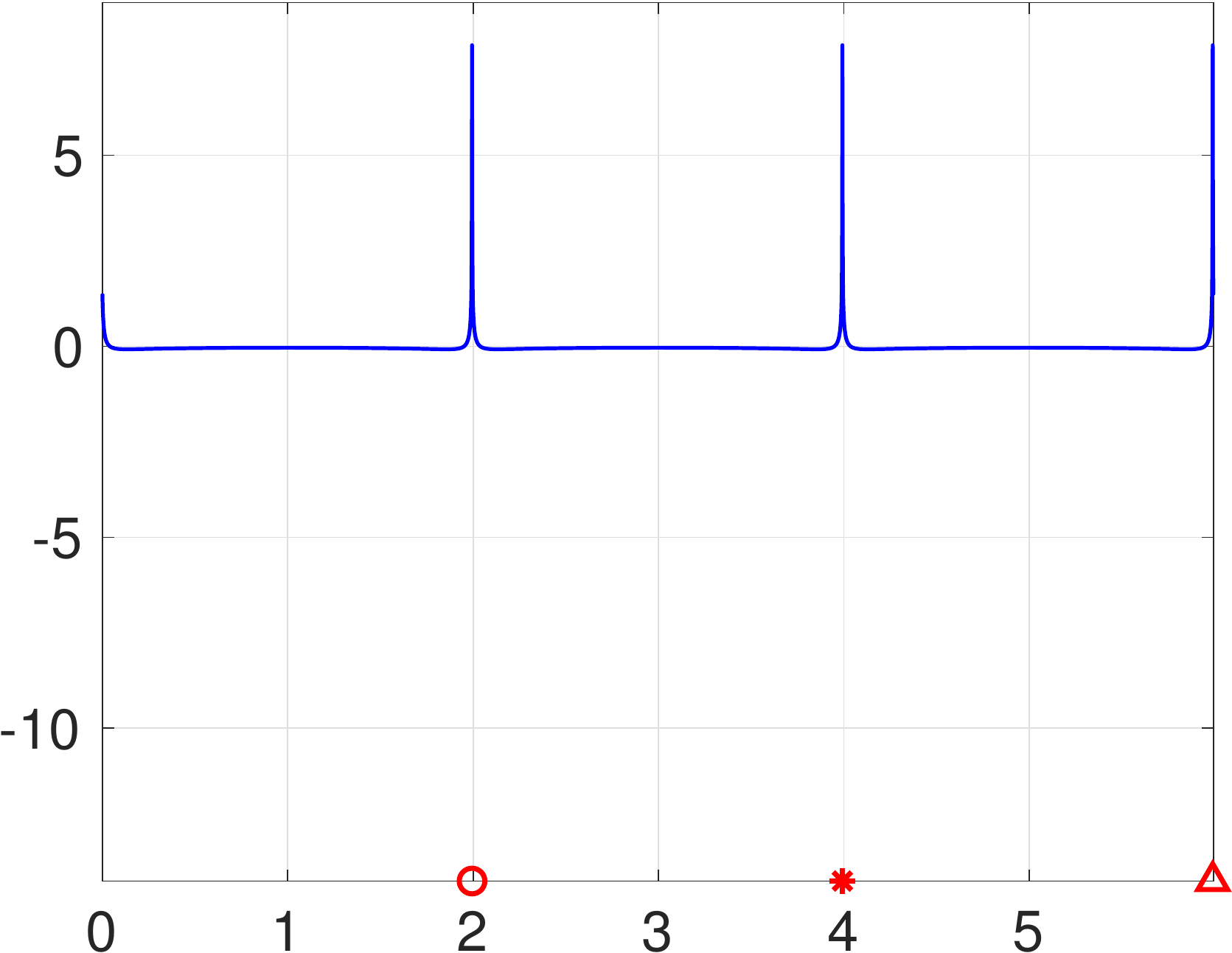}}
\subfigure[]{
\includegraphics[width=0.2\textwidth]{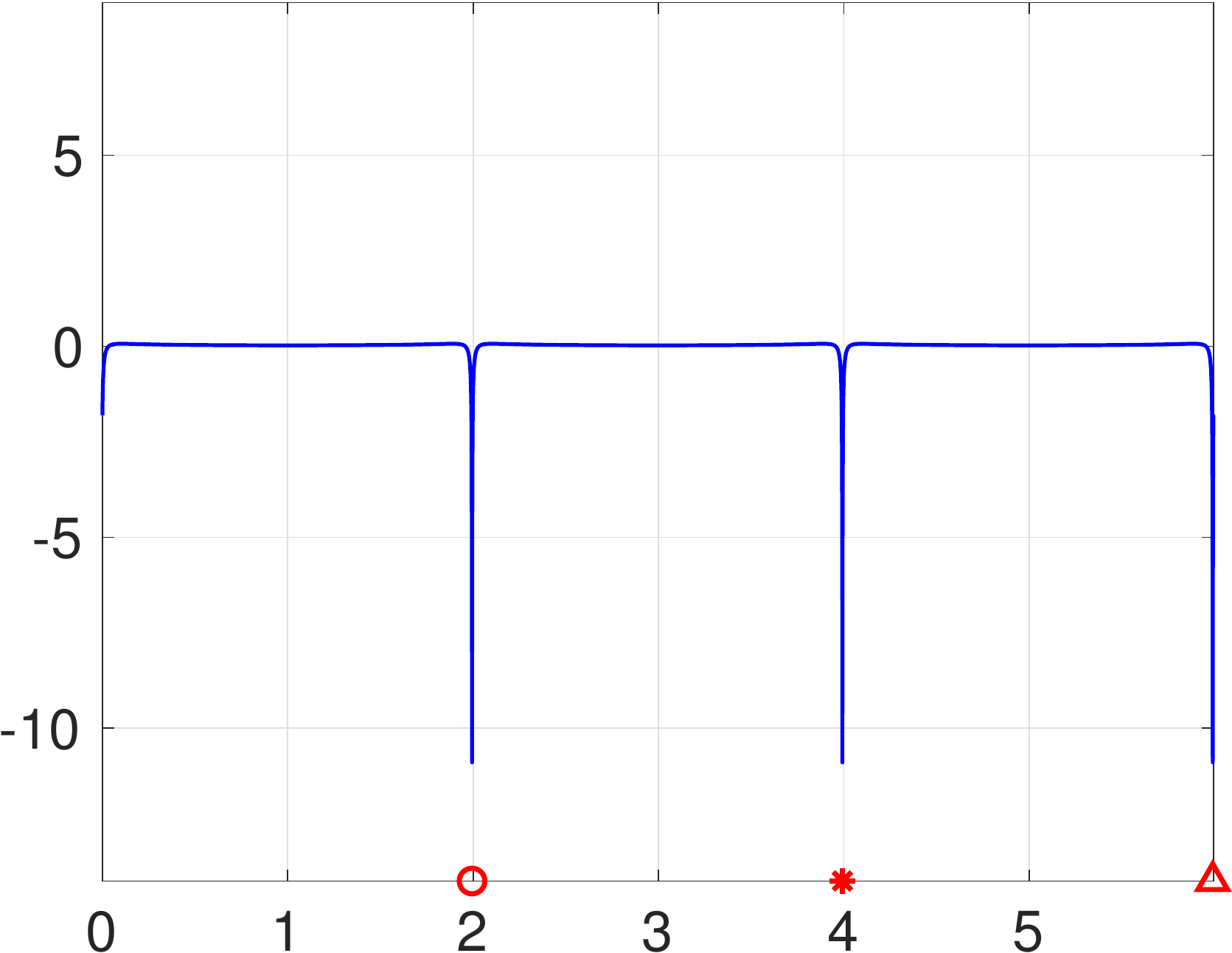}}
\subfigure[]{
\includegraphics[width=0.2\textwidth]{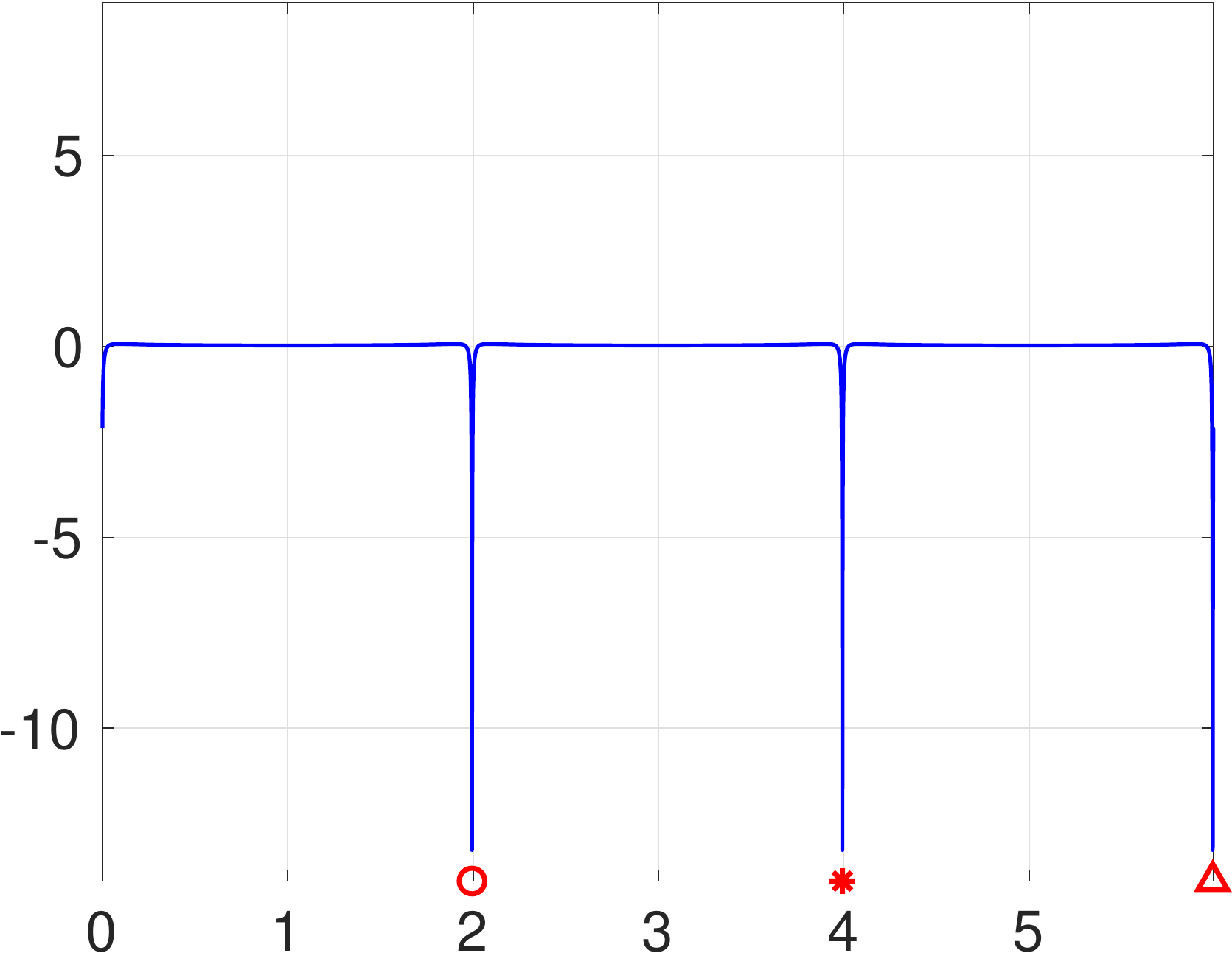}}\\
\subfigure[]{
\includegraphics[width=0.2\textwidth]{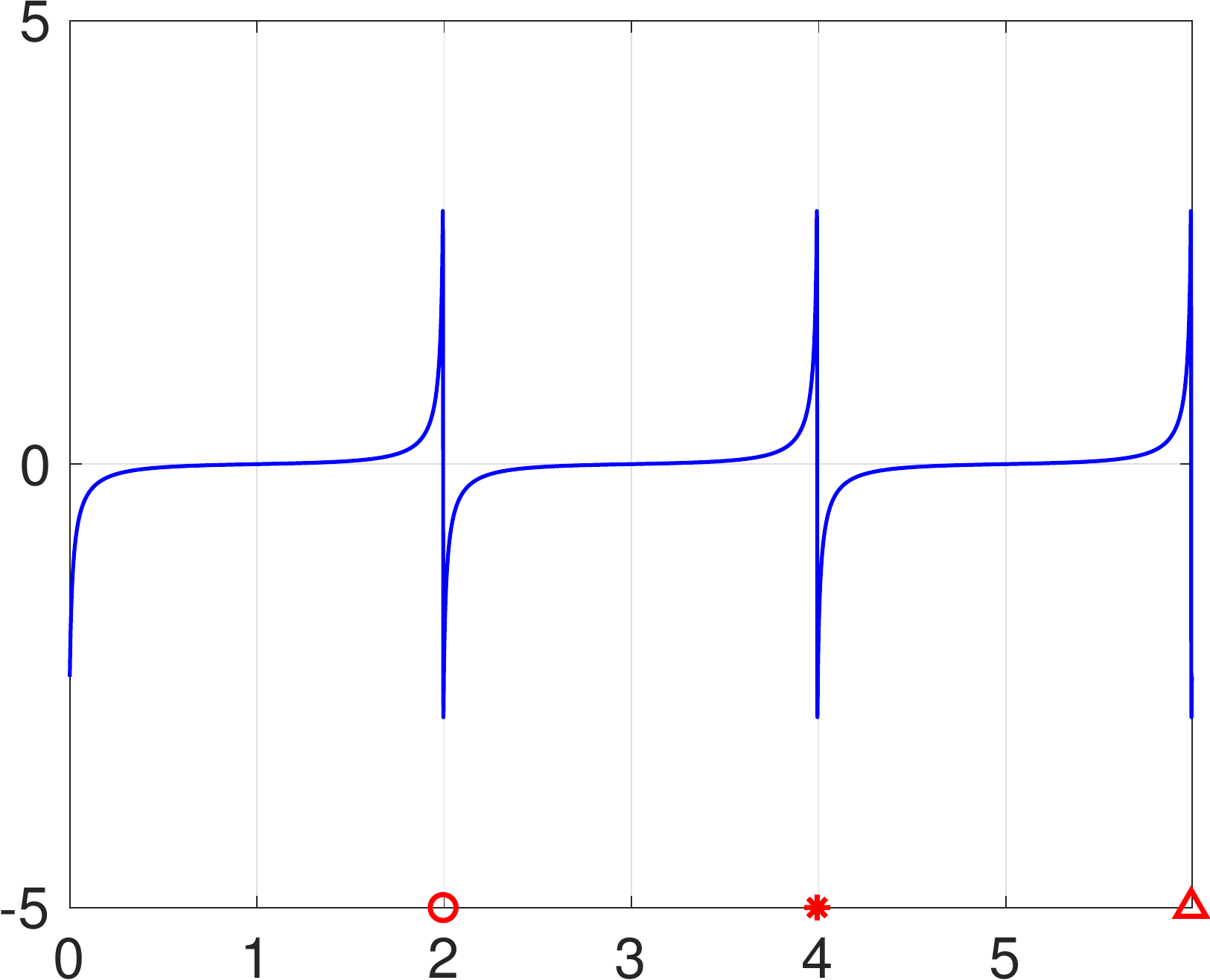}}
\subfigure[]{
\includegraphics[width=0.2\textwidth]{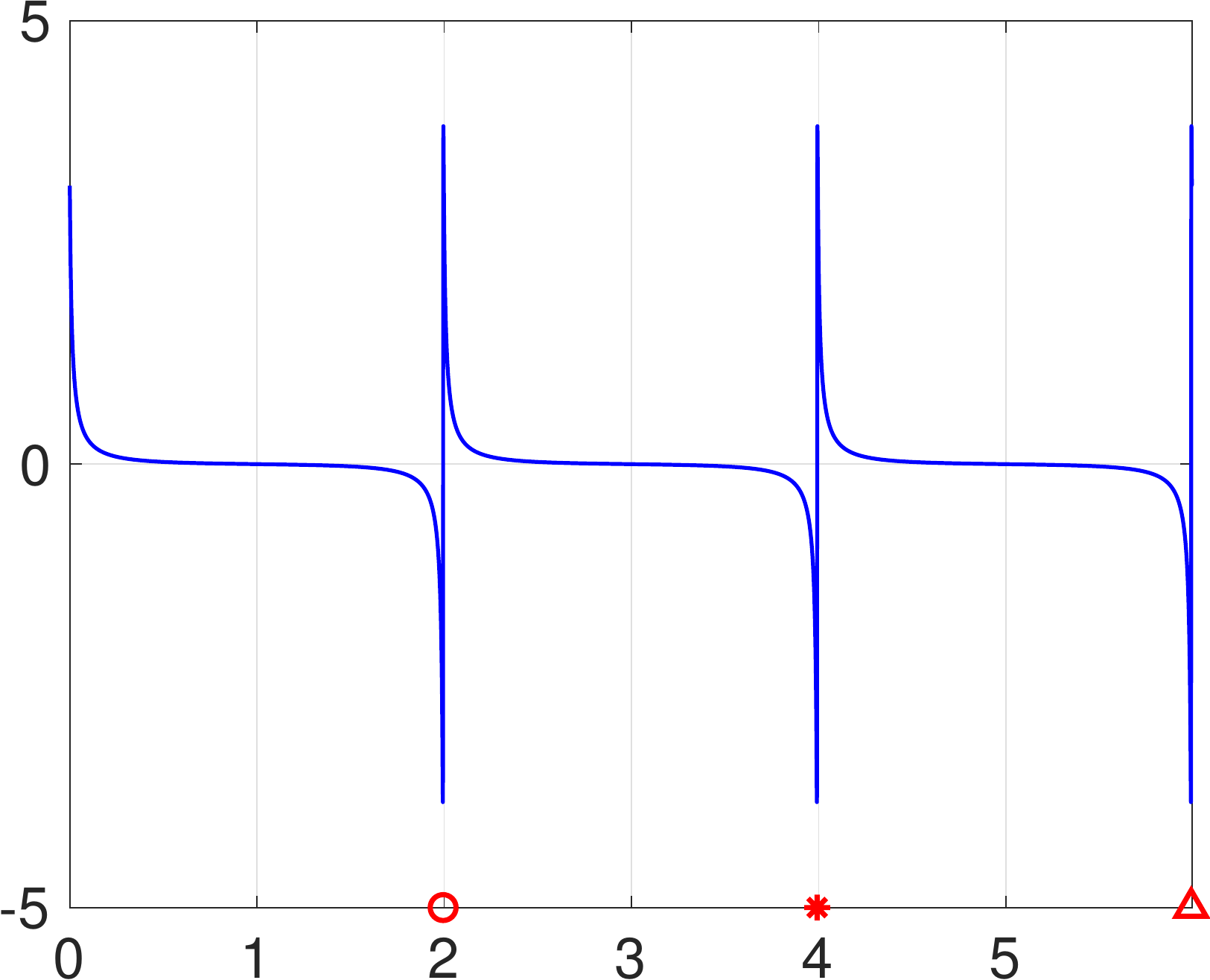}}
\subfigure[]{
\includegraphics[width=0.2\textwidth]{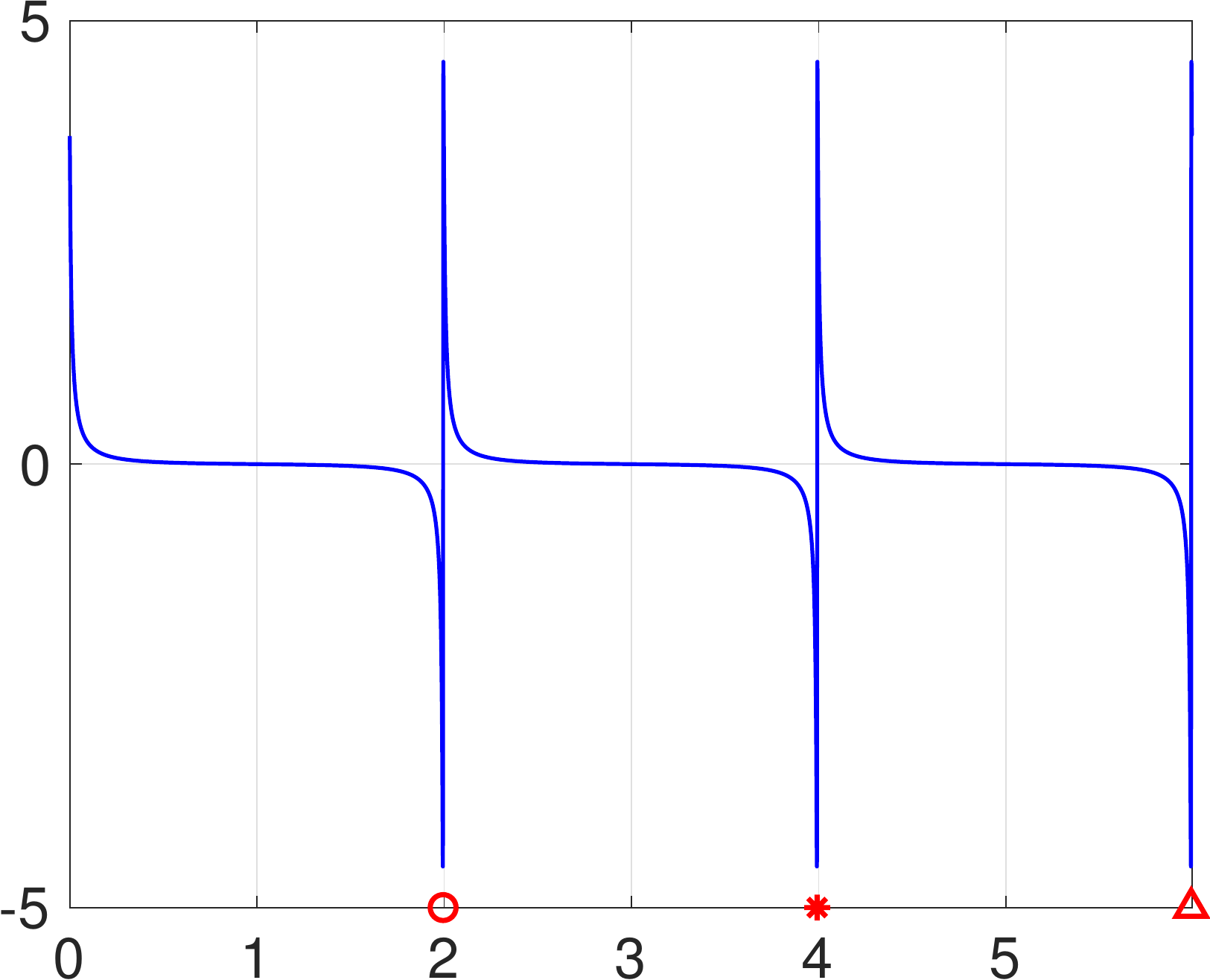}}\\
\caption{\label{fig14} (a), (b), (c).  Plotting the eigenfunctions for the positive eigenvalues $\lambda_5=0.2583$ with different maximum curvature $500$, $1000$ and $1500$. 
(d), (e), (f). The corresponding items for the negative eigenvalue $\lambda_6=-0.2583$.}
\end{figure}

We next investigate the blow-up rate of the NP eigenfunction or its conormal derivative with respect to the curvature. Therefore we plot the logarithm of the absolute value of the eigenfunctions at the high-curvature point for the positive eigenvalues $\lambda_5$, $\lambda_{11}$ and $\lambda_{17}$, and the logarithm of the absolute value of the derivative of the eigenfunctions at the high-curvature point for the negative eigenvalues $\lambda_6$, $\lambda_{12}$ and $\lambda_{18}$ with respect to different curvature in Fig.~\ref{fig15}. It turns out that blow-up rate also follows the 
rule in \eqref{eq:growthrate}. By regression, we numerically determine the corresponding parameters for those different eigenvalues 
in \eqref{eq:egg3_1}, and they are listed in Table~\ref{tab3}. 

\begin{figure}
\includegraphics[width=0.2\textwidth] {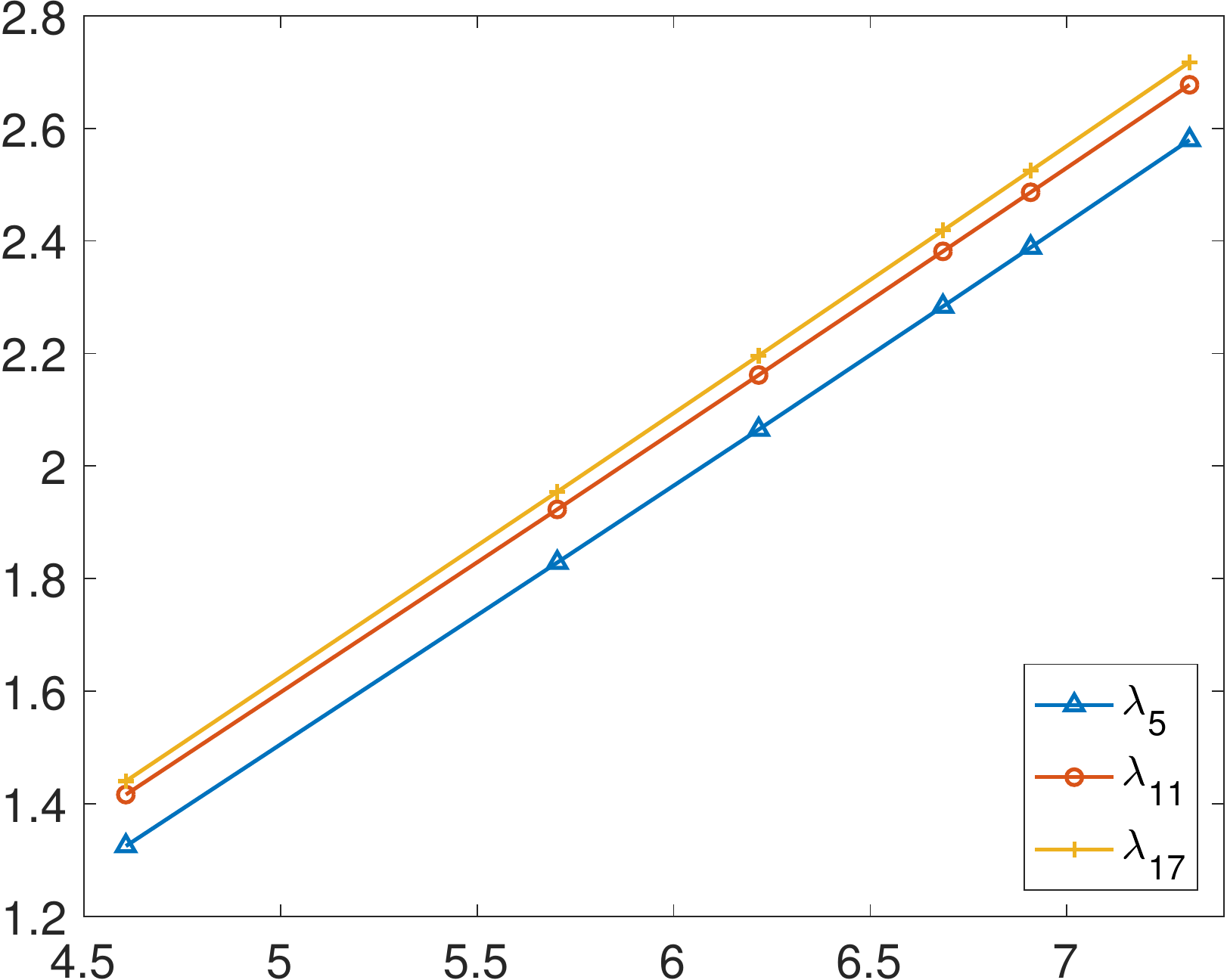}
\hspace{0.8cm}
\includegraphics[width=0.2\textwidth] {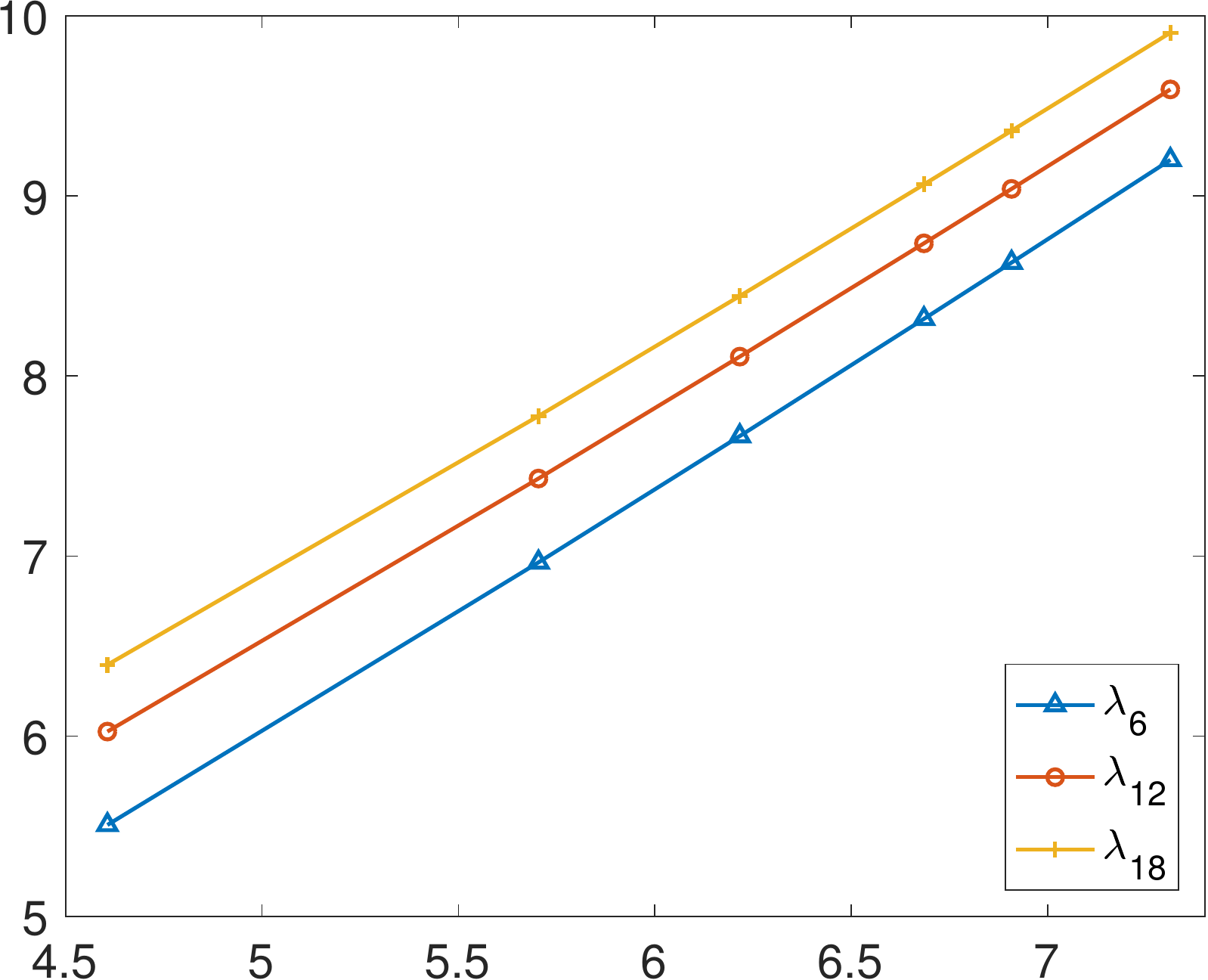}
\caption{\label{fig15} The left one plots logarithm of the eigenfunction at the high-curvature point for the simple positive eigenvalues $\lambda_5$, $\lambda_{11}$ and $\lambda_{17}$, and the right one plots logarithm of the derivative of the eigenfunction at the high-curvature point for the simple negative eigenvalues $\lambda_6$, $\lambda_{12}$ and $\lambda_{18}$ with respect to different curvature.}
\end{figure}
\begin{table}[t]
  \centering
  \subtable[]{
    \centering
    \begin{tabular}{cccc}
      \toprule
      & $\lambda_5$ & $\lambda_{11}$ & $\lambda_{17}$ \\
      \midrule
      $p$ & 0.4632 & 0.4659 & 0.720 \\[5pt]
      $\ln(a)$ & -0.8117 & -0.7322 & -0.7360 \\
      \bottomrule
    \end{tabular}}
  %  \caption{}
  \hspace{1.5cm}
\subtable[]{
    \centering
    \begin{tabular}{cccc}
      \toprule
      & $\lambda_6$ & $\lambda_{12}$ & $\lambda_{18}$ \\
      \midrule
      $p$ & 1.3641 & 1.3170 & 1.2963 \\[5pt]
      $\ln(a)$ & -0.7958 & -0.0611 & 0.4046 \\
      \bottomrule
    \end{tabular}
    }
 %\end{subtable}
  \caption{The coefficients of the regression; (a) $\lambda_j, j=5, 11, 17$; (b) $\lambda_j, j=6, 12, 18$.}
  \label{tab3}
\end{table}
%
%Here we give two remarks. The first one is that when changing the curvature at the point $x_{o}$, the curvature at the other two points $x_{*}$ and $x_{\triangle}$ is also changing and they are alway the same, i.e.
%\[
% \kappa_{x_{o}}= \kappa_{x_{*}}= \kappa_{x_{\triangle}},
%\]
%since the domain is symmetric. The second one is that the eigenvalues $\lambda_i$, $i=11,12,17,18$, all are the simple eigenvalues with 
%\[
% \lambda_{11}=0.1568, \quad  \lambda_{12}=-0.1568, \quad \lambda_{17}=0.0875 \quad \mbox{and} \quad \lambda_{18}=-0.0875.
%\]
%
Next we show the corresponding properties for the multiple eigenvalues.
Fig.~\ref{figsm1} plots the two linearly independent eigenfunctions associated with the multiple eigenvalue  
$\lambda_1=\lambda_2=0.2850$, as well as the corresponding single-layer potentials. The numerical results
clearly support our assertion about the NP eigenfunctions associated with multiple positive NP eigenvalues. 

% we plot the eigenfunctions, 
%the corresponding single layer potentials and the single layer 
%potentials around the high-curvature point $x_o$ in Fig.~\ref{figsm1}. 
%The first row in Fig.~\ref{figsm1} shows that the eigenfunction 
%and the corresponding single layer potential blow up at the point $x_*$ 
%for the multiple positive eigenvalue $\lambda_1=0.2850$ and the second row 
%shows that the the eigenfunction and the corresponding single layer potential 
%blow up at the point $x_{\triangle}$ for the multiple positive eigenvalue $\lambda_2=0.2850$.

\begin{figure}[h]
\centering
\subfigure[]{
\includegraphics[width=0.2\textwidth]{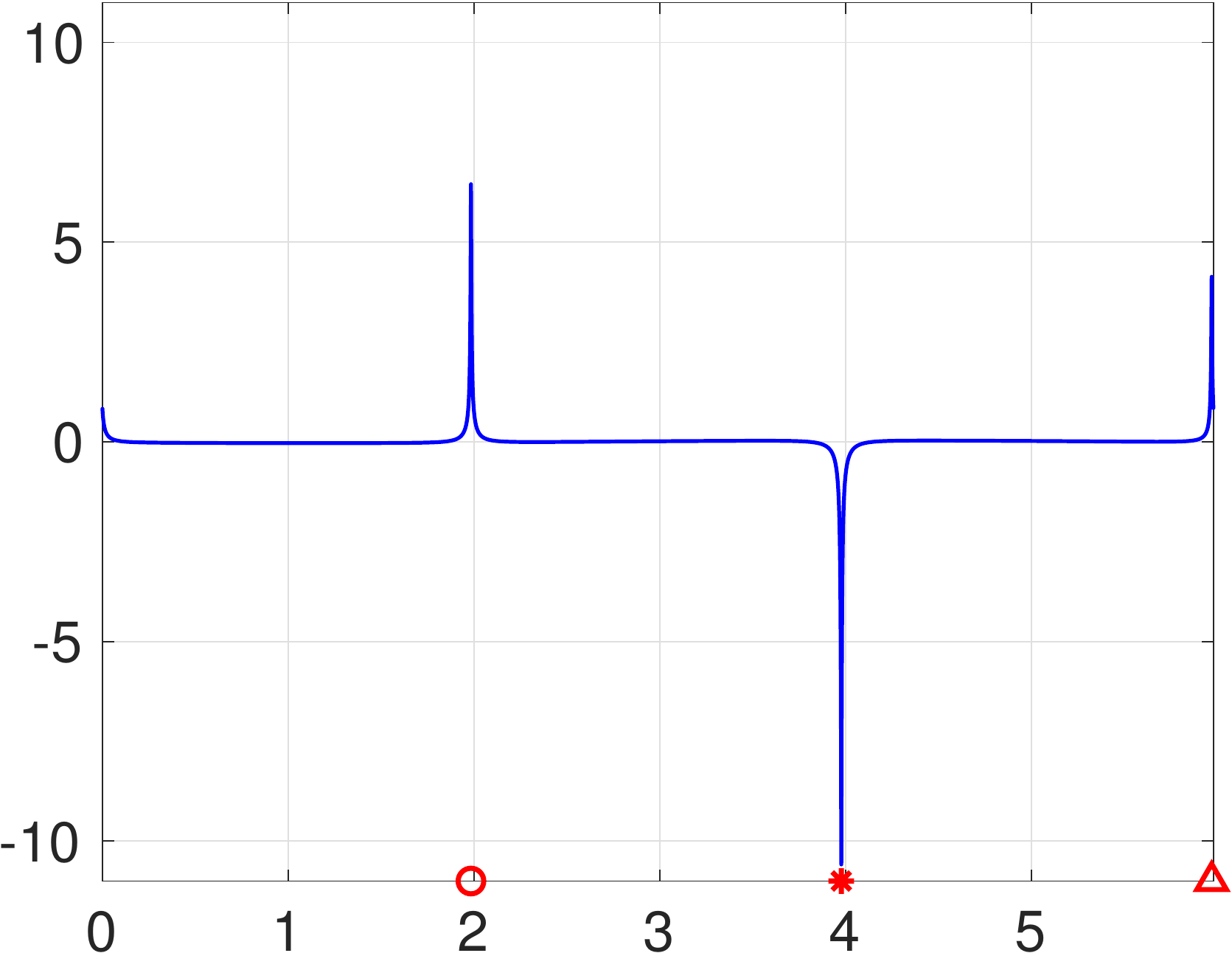}}
\subfigure[]{
\includegraphics[width=0.225\textwidth]{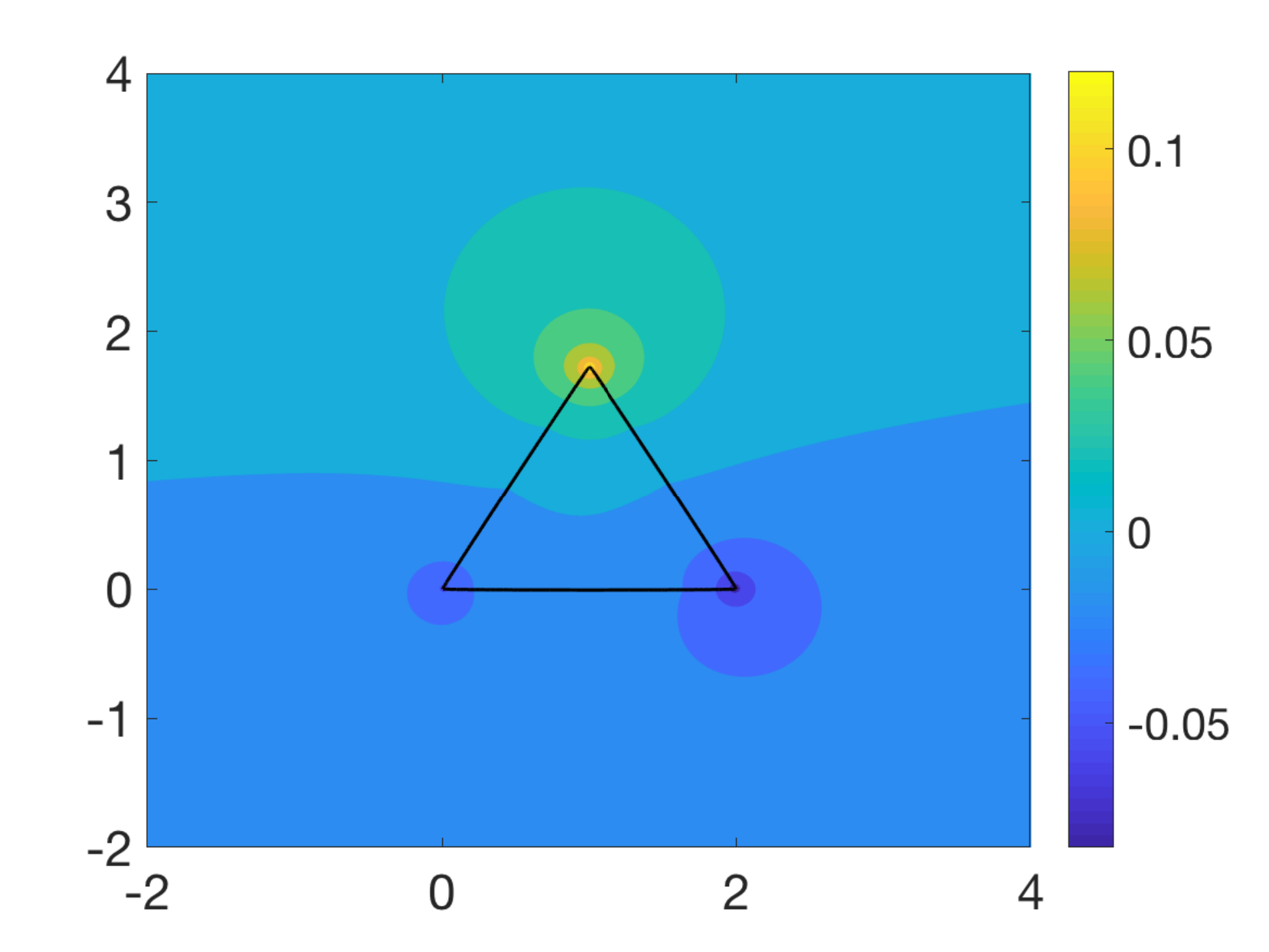}}
\subfigure[]{
\includegraphics[width=0.225\textwidth]{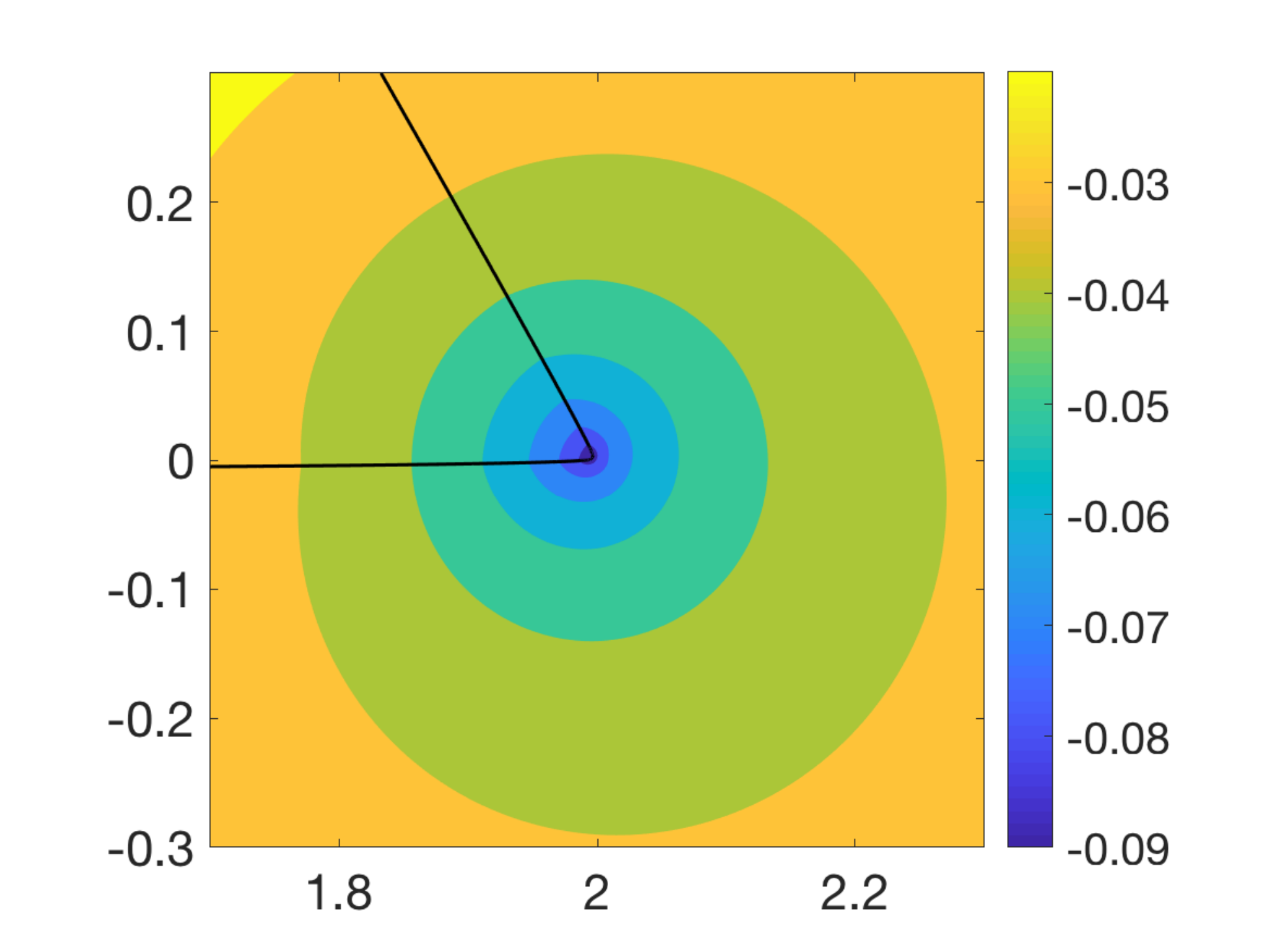}}\\
\subfigure[]{
\includegraphics[width=0.2\textwidth]{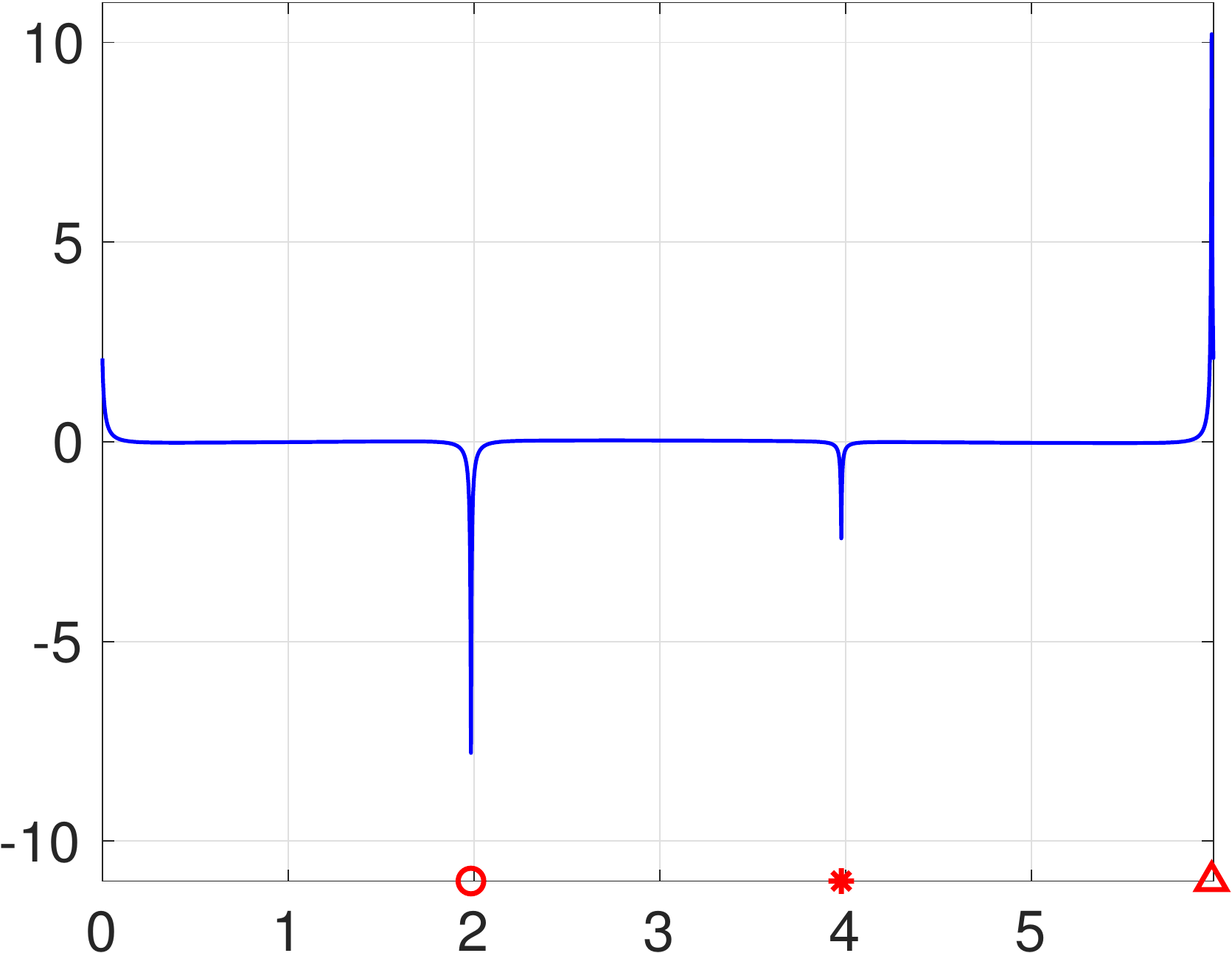}}
\subfigure[]{
\includegraphics[width=0.225\textwidth]{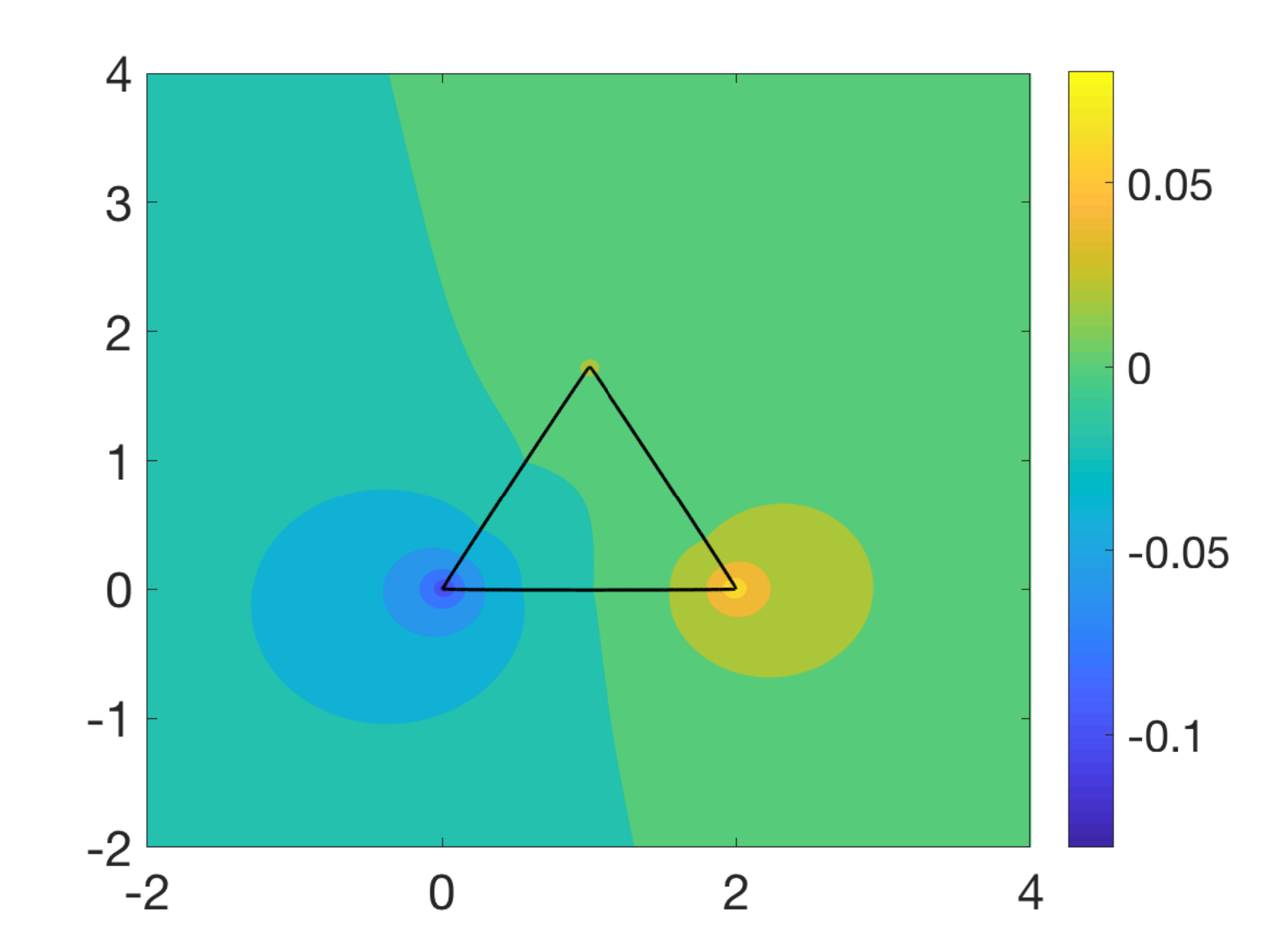}}
\subfigure[]{
\includegraphics[width=0.225\textwidth]{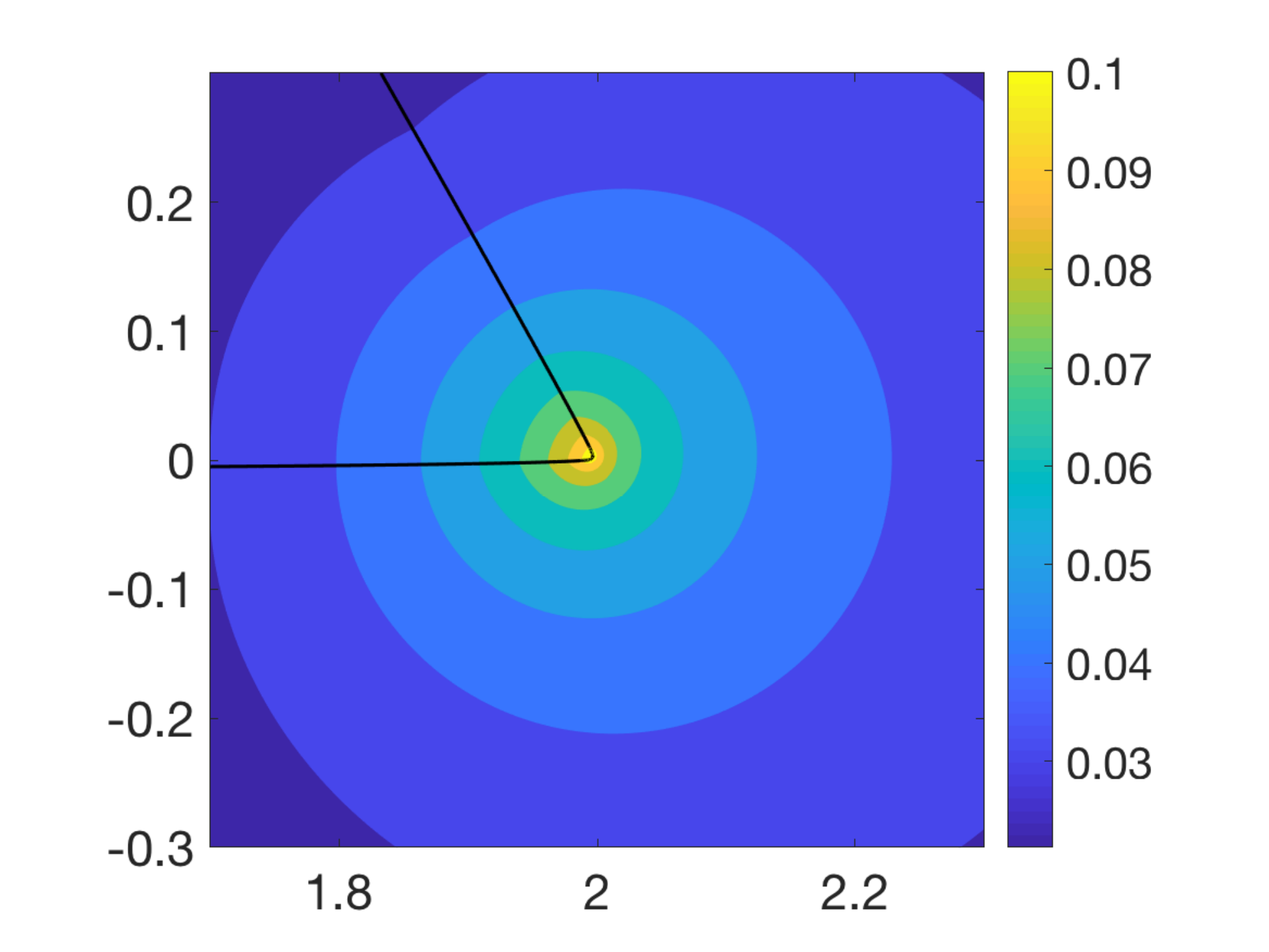}}\\
\caption{\label{figsm1} 
(a), (b), (c). The first eigenfunction associated with $\lambda_1=\lambda_2=0.2850$ as well as the corresponding single-layer potential; 
(d), (e), (f). The second eigenfunction associated with $\lambda_1=\lambda_2=0.2850$ as well as the corresponding single-layer potential. }
\end{figure}

Fig.~\ref{figsm2} plots the two linearly independent eigenfunctions associated with the multiple eigenvalue  
$\lambda_3=\lambda_4=-0.2850$, as well as the corresponding conormal derivatives and the corresponding single-layer potentials. The numerical results
clearly support our assertion about the NP eigenfunctions associated with multiple negative NP eigenvalues. 

%As for the multiple positive eigenvalues $\lambda_3=\lambda_4=-0.2850$, we plot the eigenfunctions, the conormal derivative of the eigenfunctions, the associated single layer potentials and the single layer potentials around the high-curvature point $x_o$ in Fig.~\ref{figsm2}. The first row and the second row in Fig.~\ref{figsm2} show that the conormal derivative of the eigenfunctions and  the corresponding single layer potentials blow up at the high-curvature point $x_o$ for the multiple negative eigenvalues $\lambda_3=\lambda_4=-0.2850$.
\begin{figure}
\centering
\subfigure[]{
\includegraphics[width=0.205\textwidth]{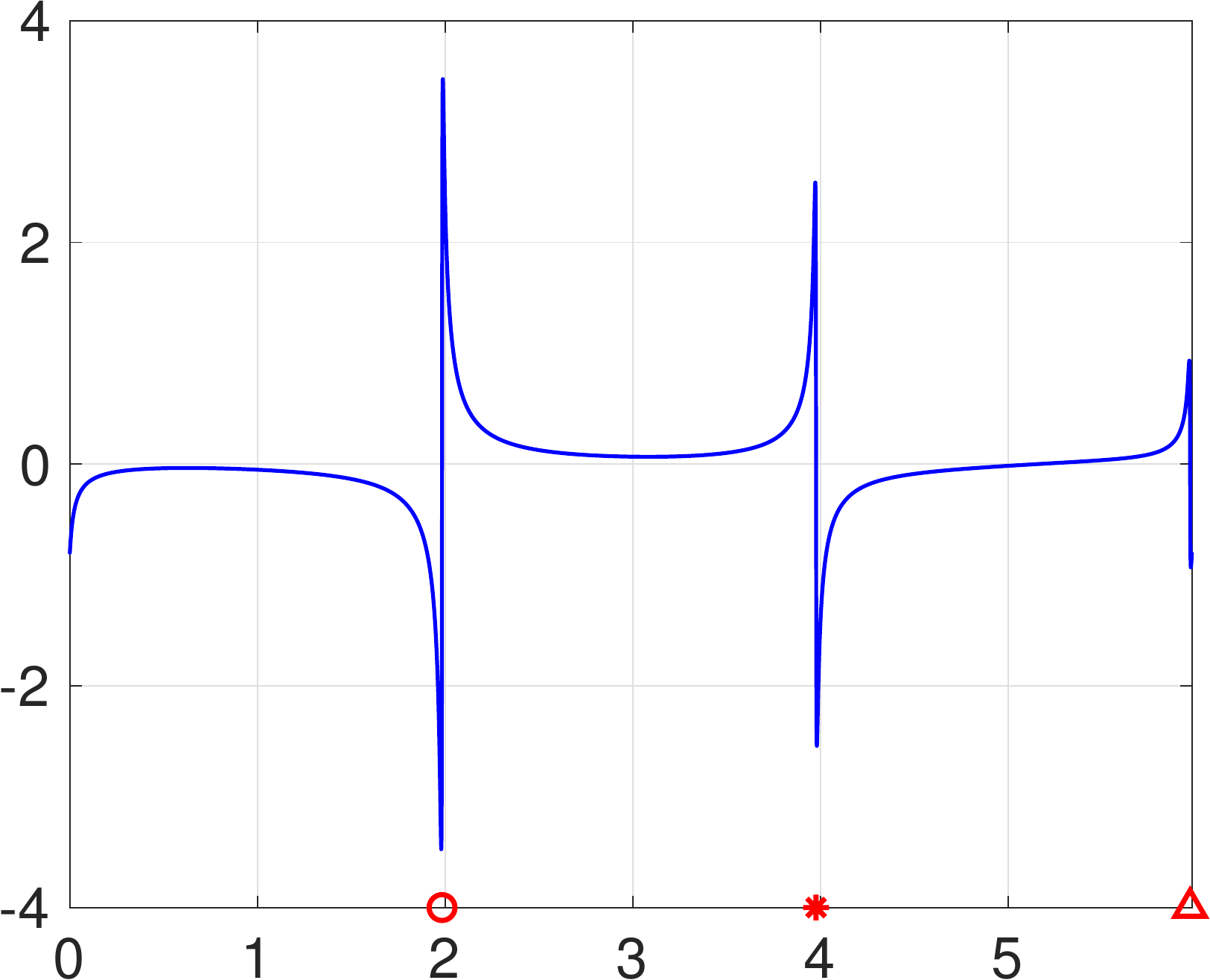}}
\subfigure[]{
\includegraphics[width=0.22\textwidth]{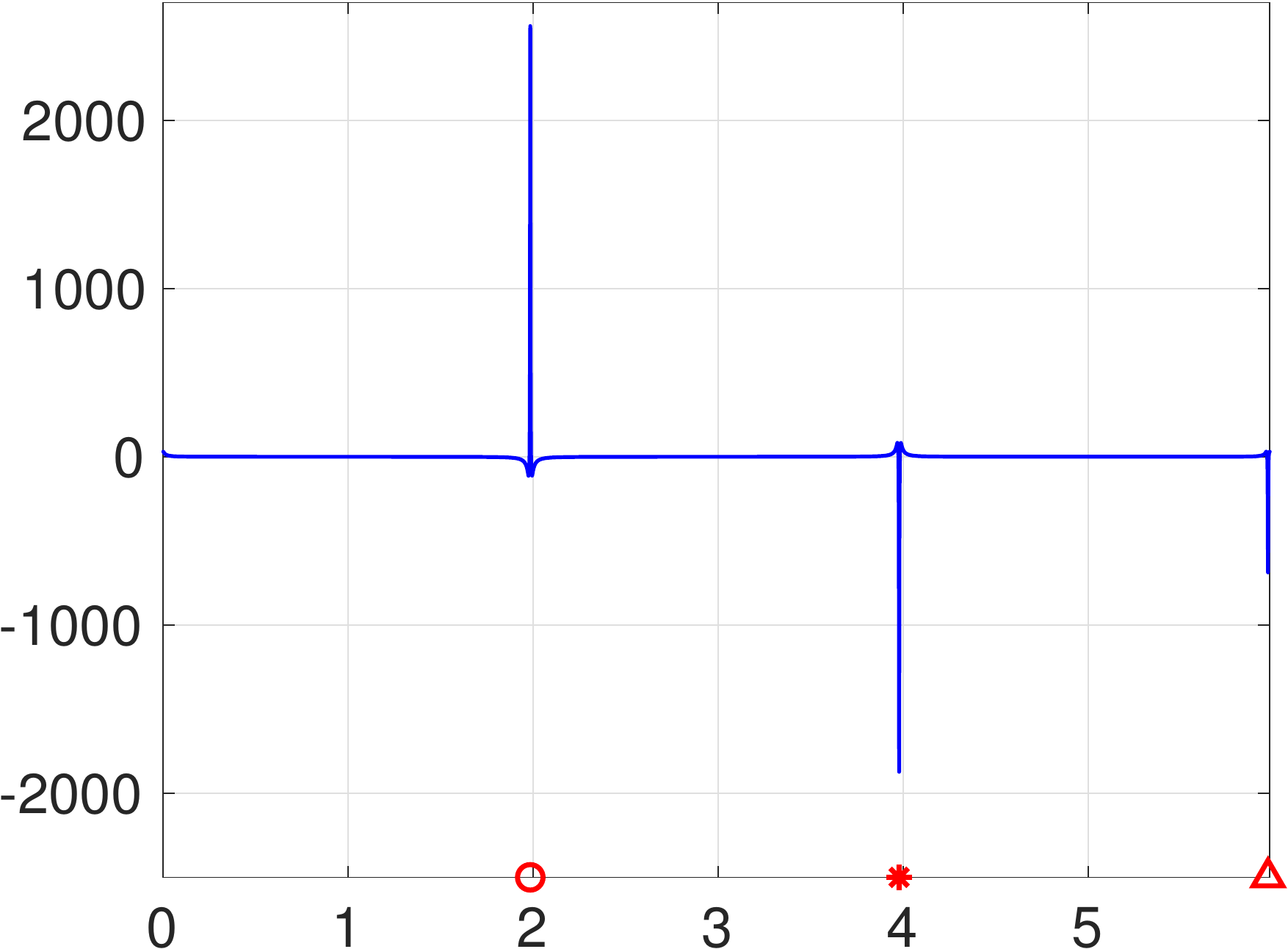}}
\subfigure[]{
\includegraphics[width=0.23\textwidth]{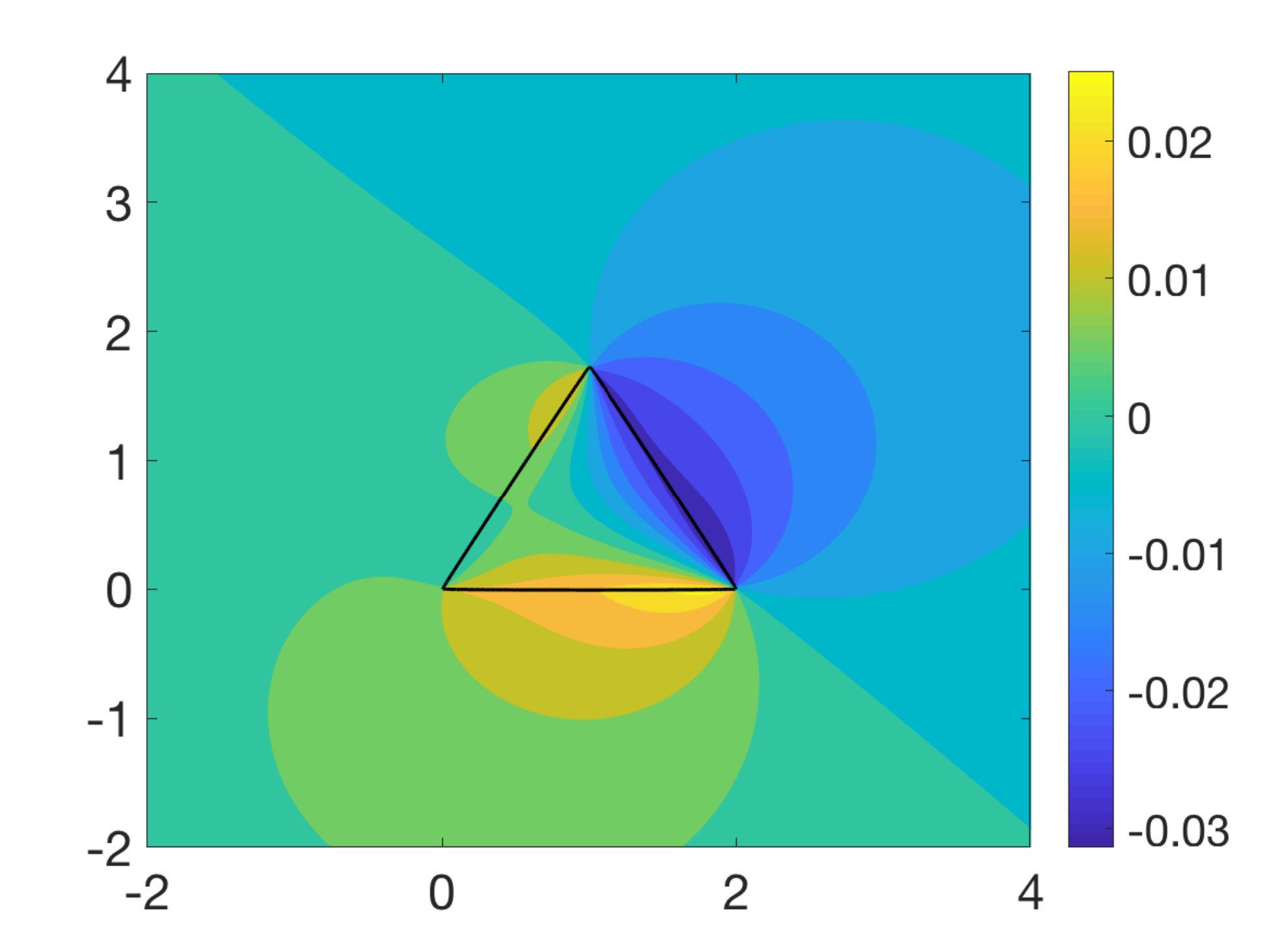}}
\subfigure[]{
\includegraphics[width=0.23\textwidth]{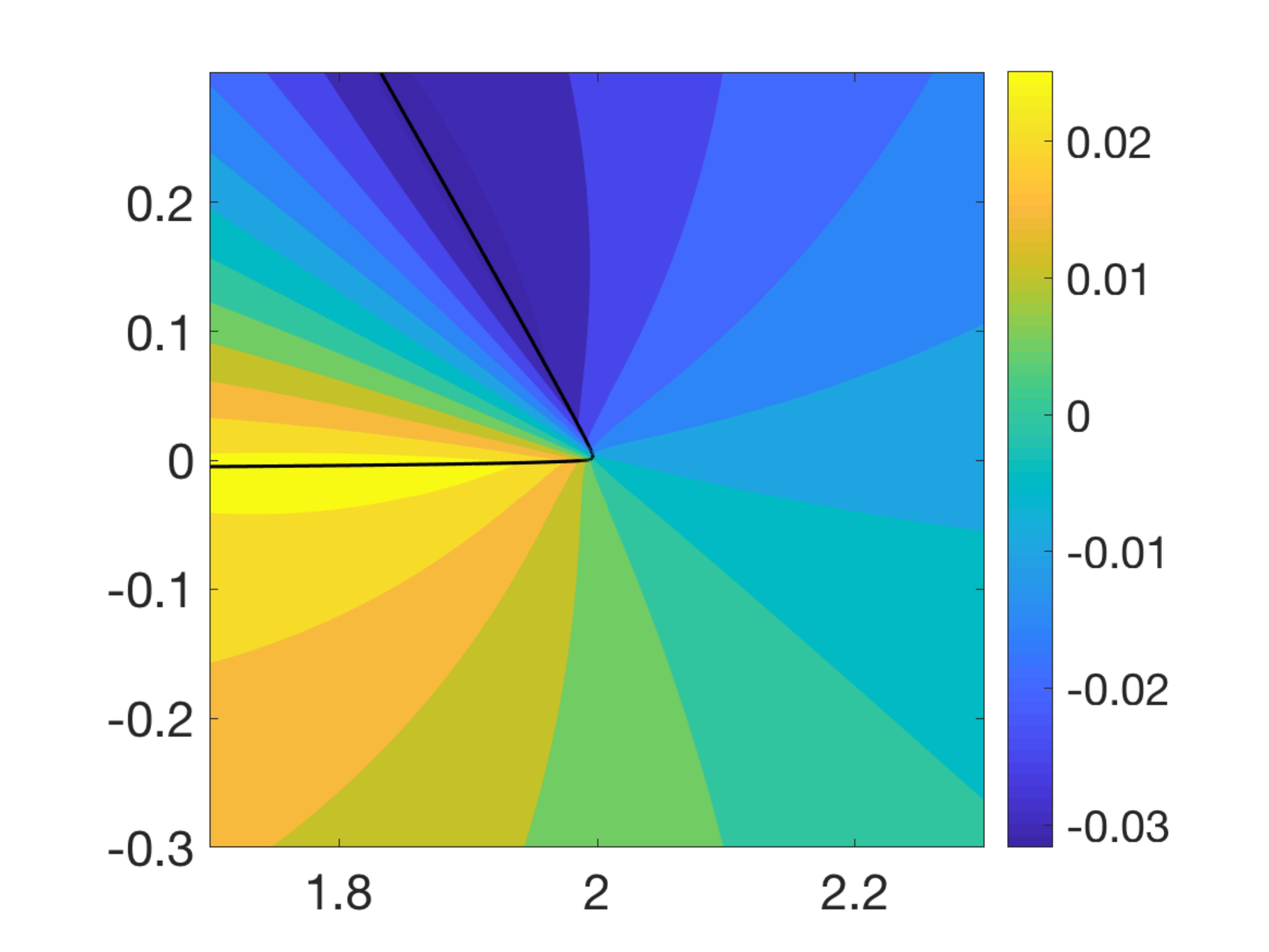}}\\
\subfigure[]{
\includegraphics[width=0.205\textwidth]{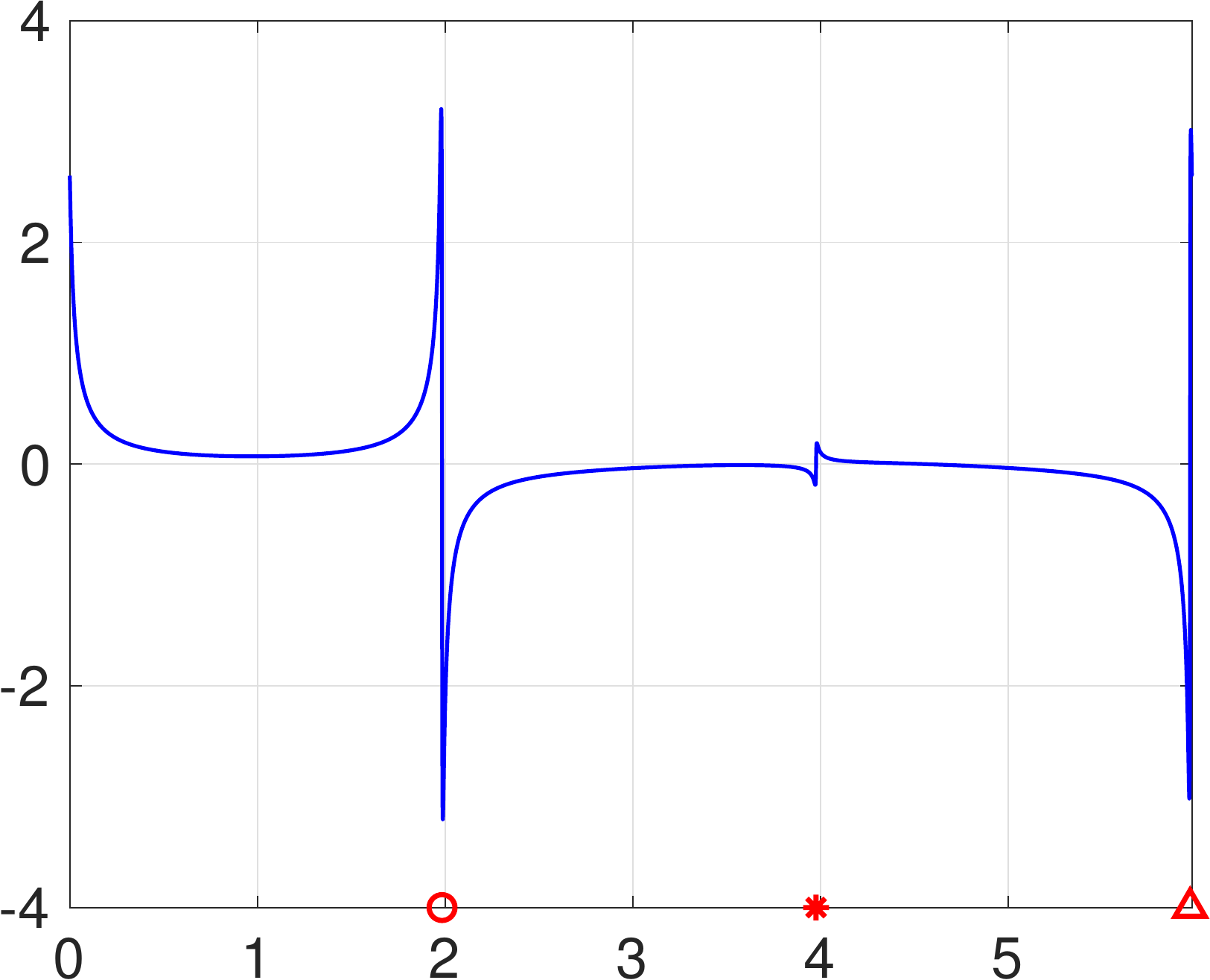}}
\subfigure[]{
\includegraphics[width=0.22\textwidth]{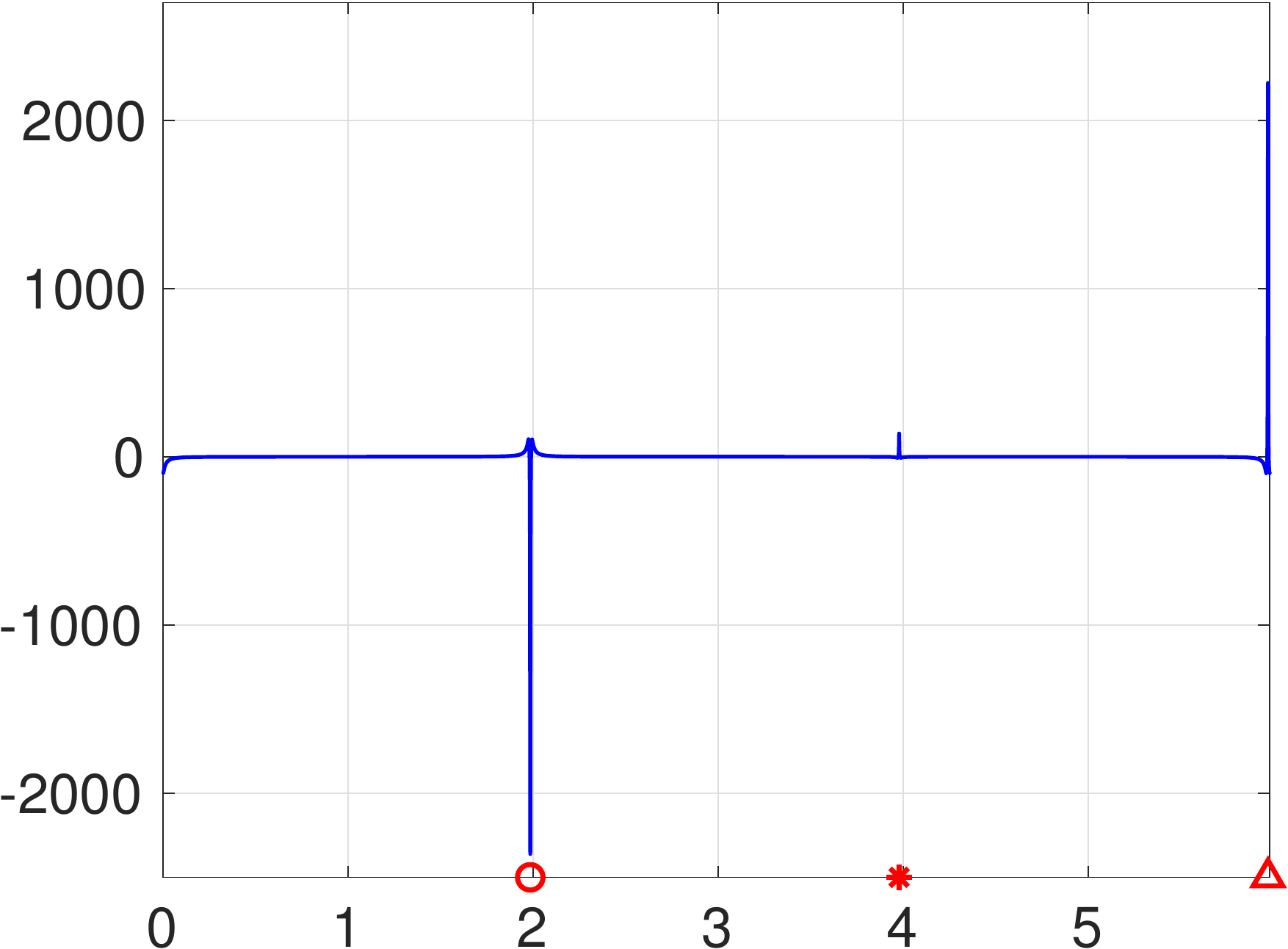}}
\subfigure[]{
\includegraphics[width=0.23\textwidth]{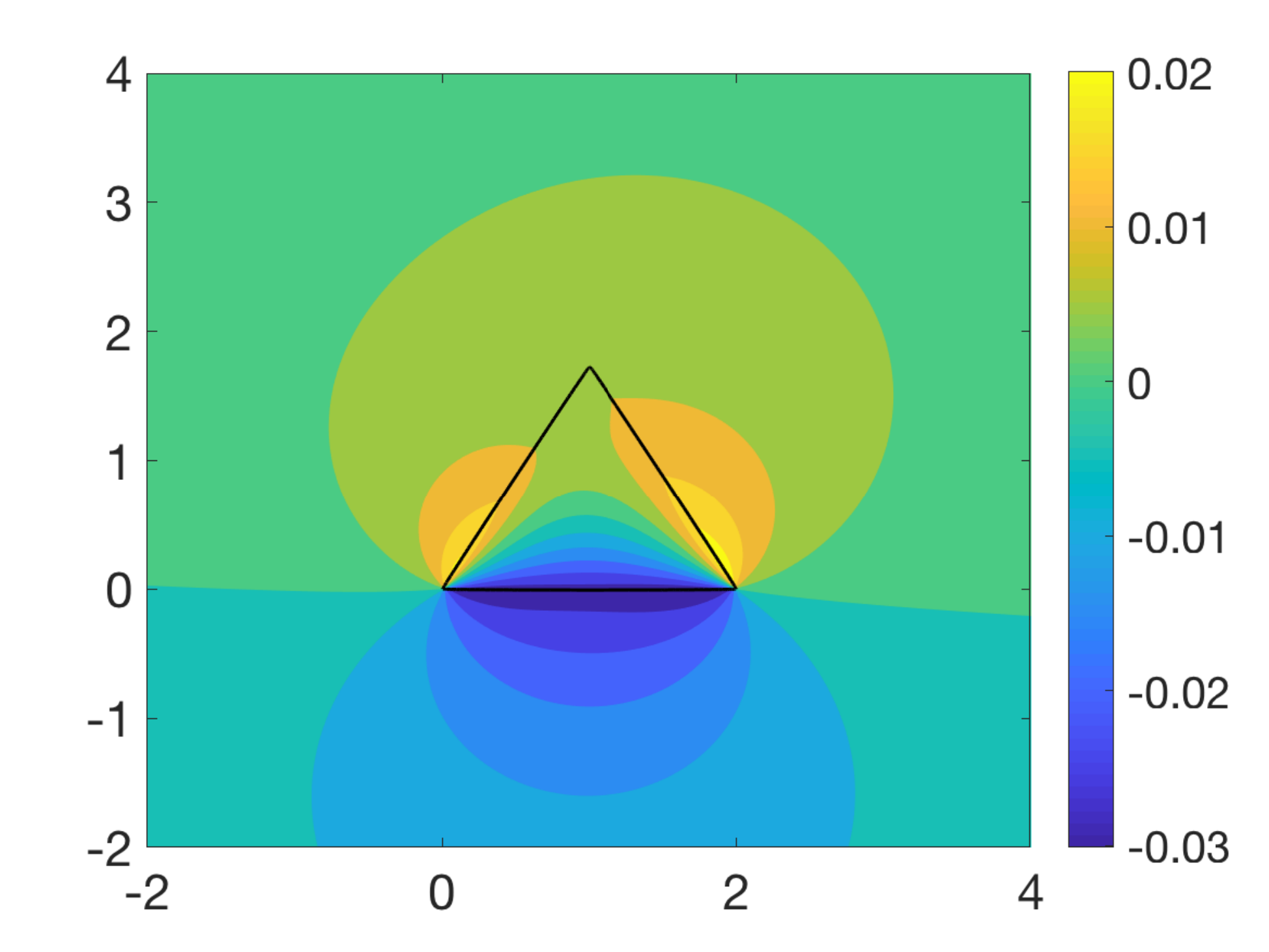}}
\subfigure[]{
\includegraphics[width=0.23\textwidth]{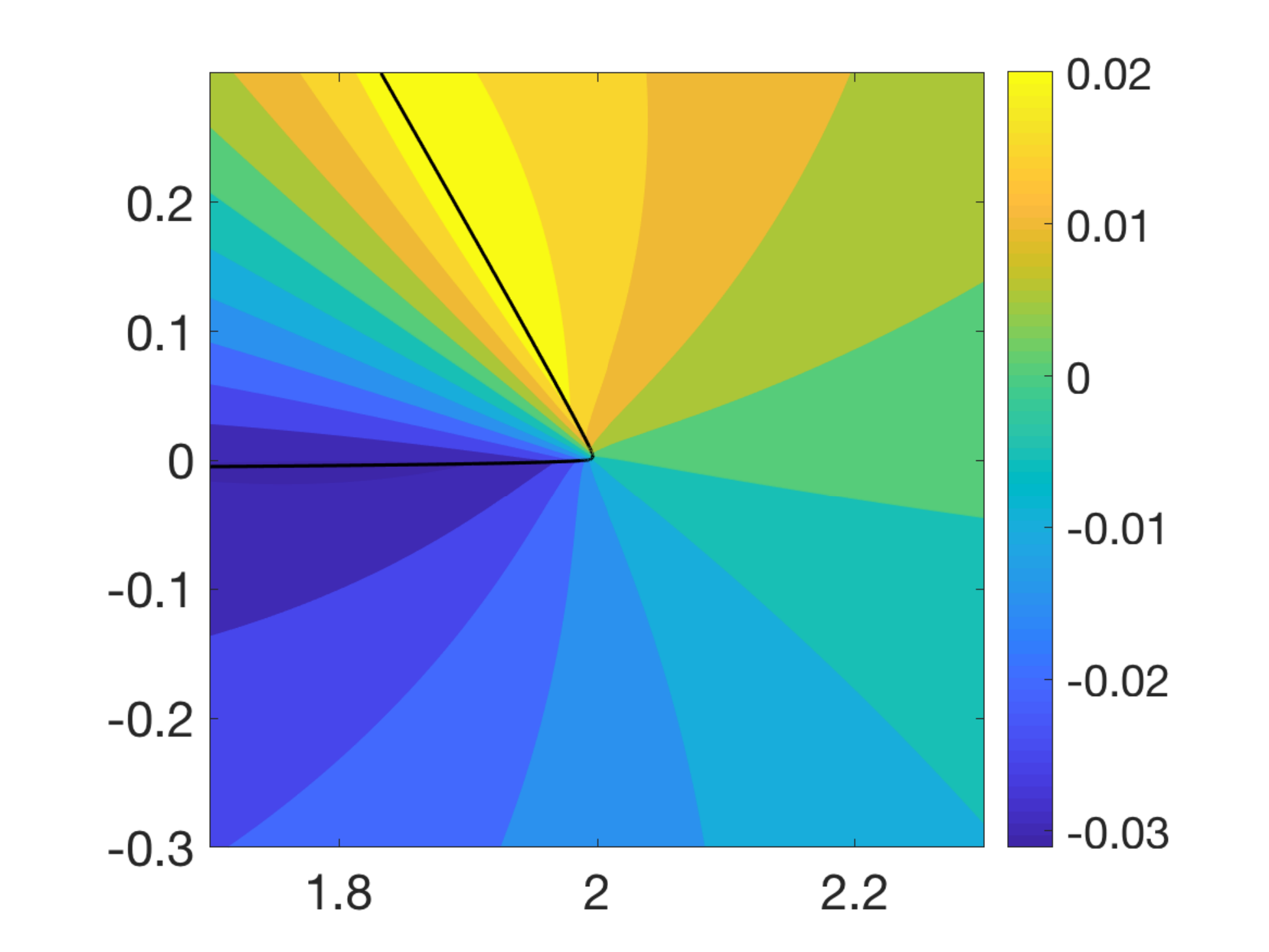}}\\
\caption{\label{figsm2}
 (a), (b), (c), (d). The first eigenfunction associated with $\lambda_3=\lambda_4=-0.28500$ as well as the corresponding conormal derivatives and the corresponding single-layer potential; 
(e), (f), (g), (h). The second eigenfunction associated with $\lambda_3=\lambda_4=-0.2850$ as well as the corresponding conormal derivatives and the corresponding single-layer potential. }
\end{figure}

%Here we mention that we only list the numerical results for the multiple eigenvalue $\lambda_j$, $j=1,2,3,4$ and the 
%properties for other multiple eigenvalues are similar. From the discussion above, one can conclude that for the 
%boundary of the domain $D$ given in Fig.~ \ref{fig11}, if the eigenvalue is multiple and positive,  then there exists 
%at least one of the eigenfunctions such that both the eigenfunction and the associated single layer potential blow up 
%at the high curvature point on $\partial D$; and if the eigenvalue is multiple and negative, then similar conclusions holds 
%as above, but for the conormal derivative of the eigenfunction and the associated single layer potential. 
%

%
%
\subsection{A convex 4-symmetric domain}

In this subsection, we consider a convex 4-symmetric domain as shown in Fig.~\ref{fig16}, which possesses four high-curvature points that are 
denoted by $x_{\times}$ $x_{o}$, $x_{*}$ and $x_{\triangle}$ as shown in Fig.~ \ref{fig16}. The largest curvature is 
\begin{equation}
 \kappa_{x_{\times}}=\kappa_{x_{o}}=\kappa_{x_{*}}=\kappa_{x_{\triangle}}=500.
\end{equation}
\begin{figure}
\includegraphics[width=3cm] {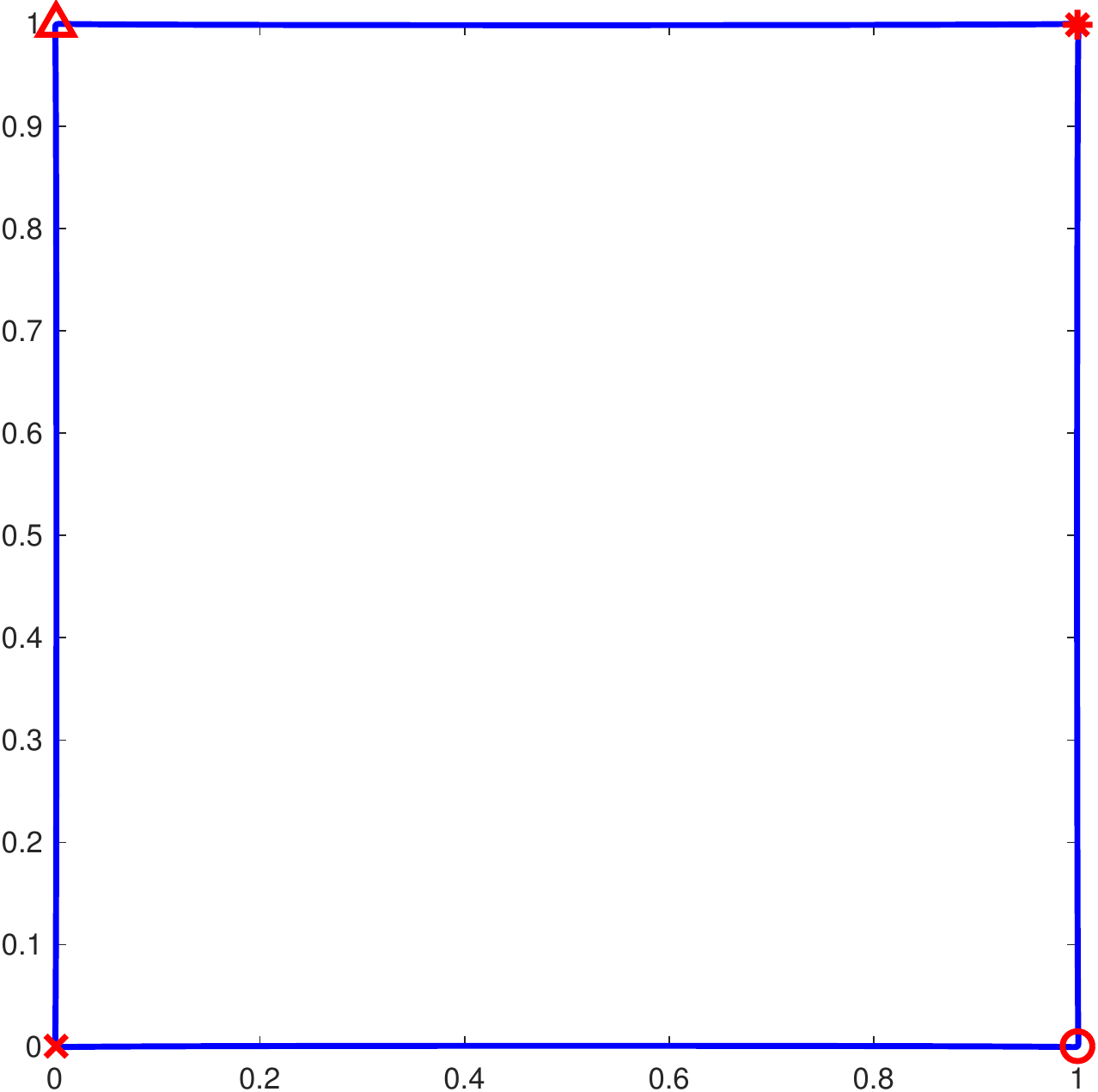}
\caption{\label{fig16} The boundary with four high-curvature points.}
\end{figure}
The first eleven largest eigenvalues (in terms of the absolute value) are numerically found to be
\begin{equation}\label{eq:egg4}
\begin{split}
\lambda_0=0.5, & \  \lambda_1=0.2183,\  \lambda_2=-0.2183, \  \lambda_3= \lambda_4=0.2113,\ \lambda_5= \lambda_6=-0.2113,\\
  &   \lambda_7=0.1934,\quad  \lambda_8=-0.1934, \quad  \lambda_9=0.1313,\quad \lambda_{10}=-0.1313.
  \end{split}
\end{equation}
 There are multiple NP eigenvalues occurring for the 4-symmetric domain. Hence, we can verify our assertion
about the NP eigenfunction associated to a multiple NP eigenvalue. In the following, we first show the case for the simple eigenvalue and then the case for the multiple eigenvalue.

Fig.~\ref{fig17} plots the eigenfunctions as well as the associated single-layer potentials, respectively, for the positive eigenvalues  
$\lambda_1=0.2183$. The numerical results clearly support our assertion about the 
NP eigenfunctions associated to simple positive eigenvalues.

\begin{figure}
\centering
\subfigure[]{
\includegraphics[width=0.2\textwidth]{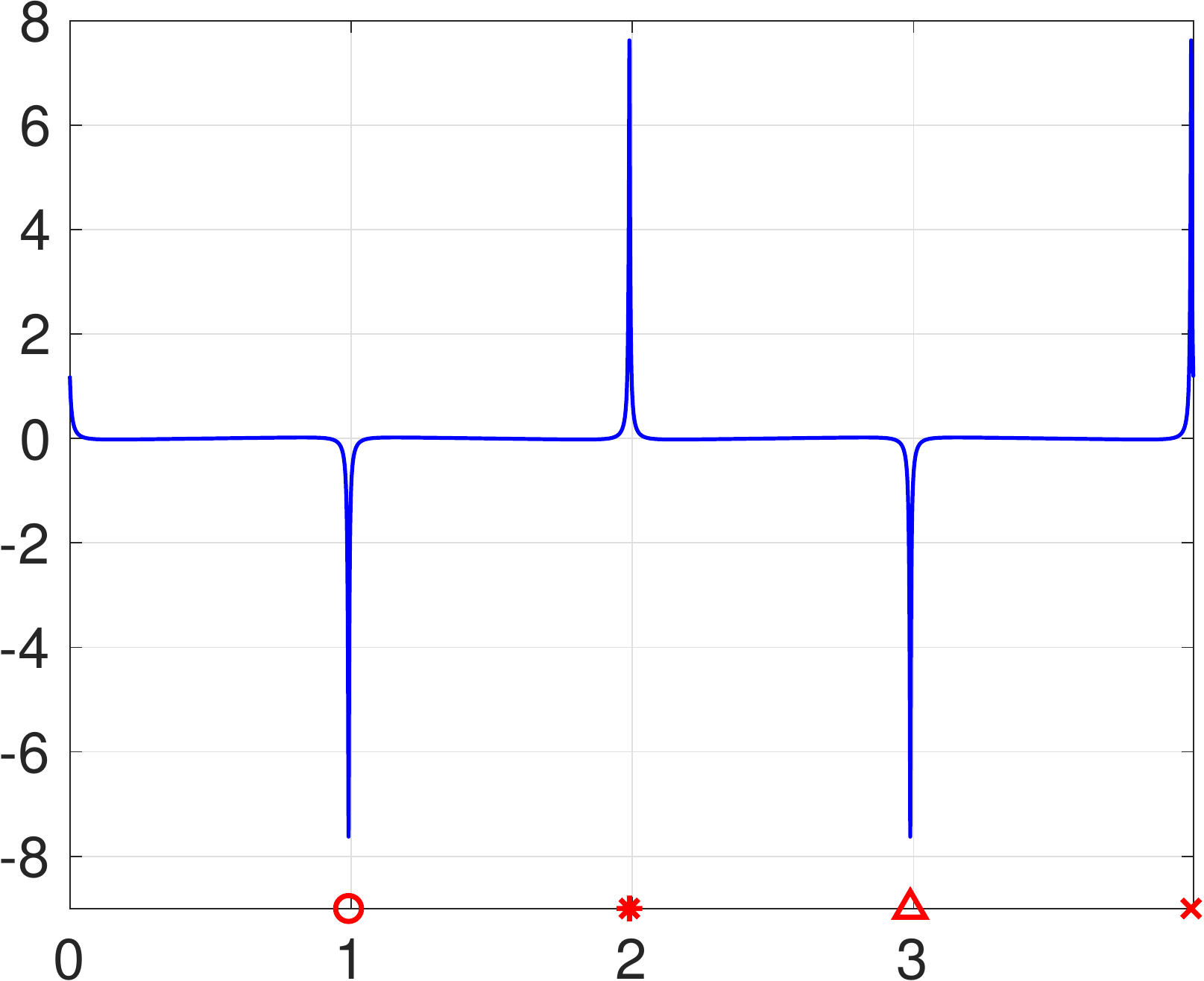}}
\subfigure[]{
\includegraphics[width=0.2\textwidth]{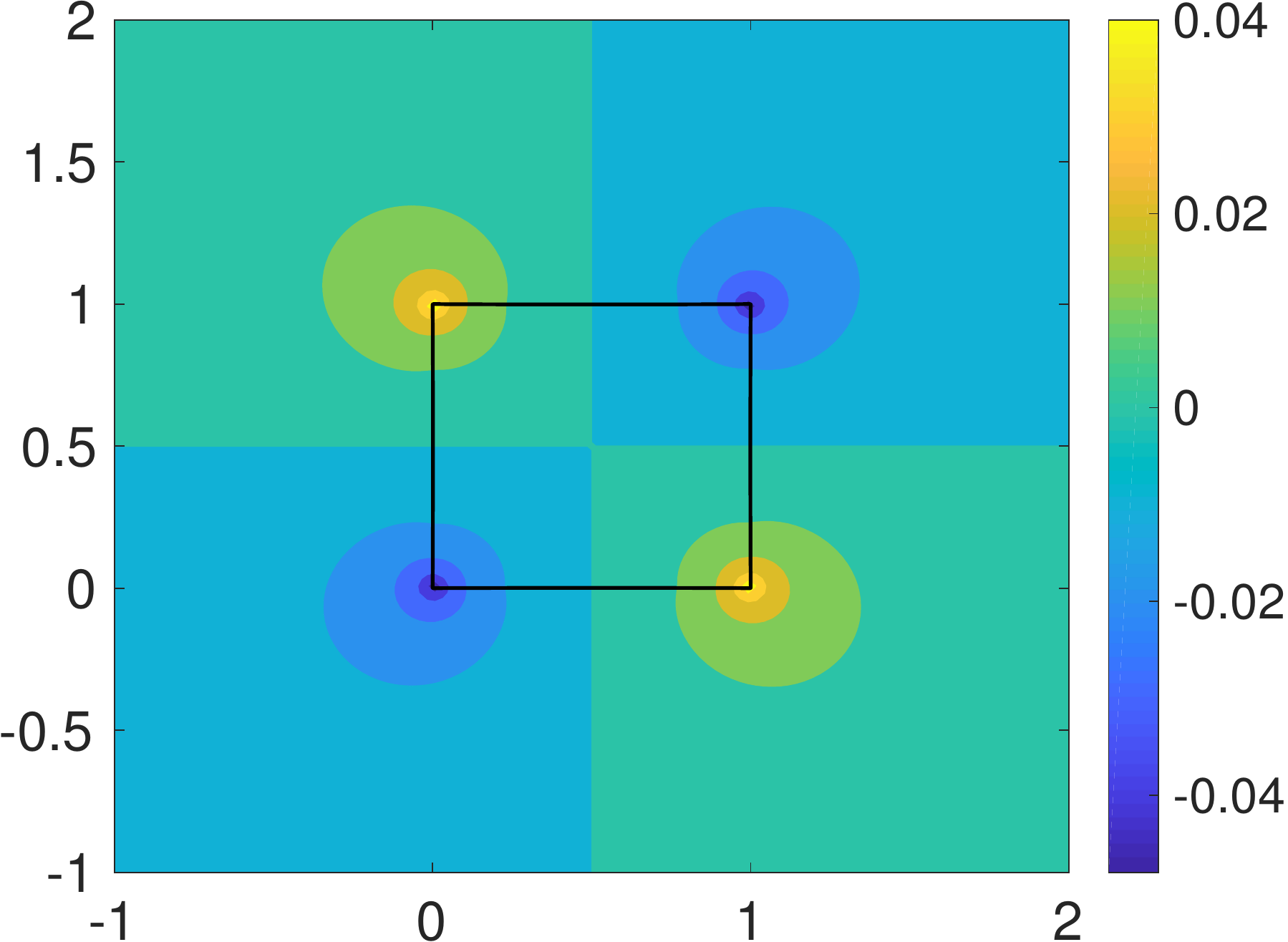}}
\subfigure[]{
\includegraphics[width=0.2\textwidth]{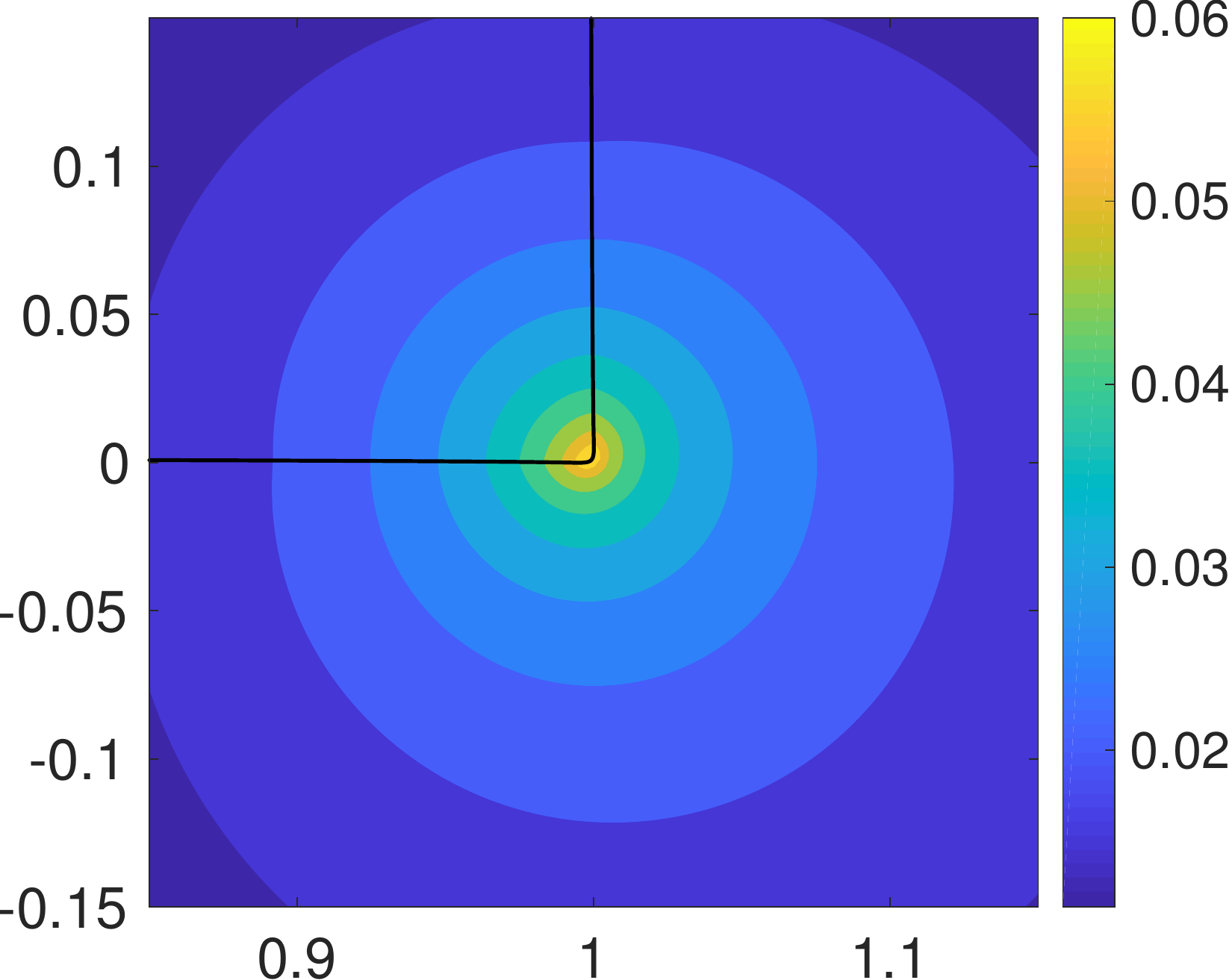}}\\
\caption{\label{fig17} (a). Plotting of the eigenfunction for $\lambda_1= 0.2183$ with respect to the arc length; (b).
The associated single-layer potential for $\lambda_1= 0.2183$; (c). The single-layer potential around the high-curvature point.}
\end{figure}

Fig.~\ref{fig18} plots the eigenfunctions as well as the corresponding conormal derivatives and single-layer potentials for the negative eigenvalues 
$\lambda_2= -0.2183$. The numerical results clearly support our assertion about the NP eigenfunctions associated to simple negative eigenvalues.

%Fig.~\ref{fig18} plots the eigenfunction with respect to the arc length, the conormal derivative of the eigenfunction with respect to the arc length, the associated single layer potential and the single layer potentials around the high-curvature point $x_o$ for the negative eigenvalue  $\lambda_2= -0.2183$, respectively. In the Fig.~\ref{fig18}, $a,b$ show that the conormal derivative of the eigenfunction blows up at the high-curvature points $x_{\times}$ $x_{o}$, $x_{*}$ and $x_{\triangle}$ for the simple negative eigenvalue, which accords with the conclusion for the ellipse case stated in Proposition \ref{prop:main2}, $(2)$. $c,d$ show that the conormal derivative of the associated single layer potential blows up at the high-curvature points $x_{\times}$ $x_{o}$, $x_{*}$ and $x_{\triangle}$ for the simple negative eigenvalue, which accords with the conclusion for the ellipse case stated in Proposition \ref{prop:main2}, $(3)$.
%
%
\begin{figure}
\centering
\subfigure[]{
\includegraphics[width=0.2\textwidth]{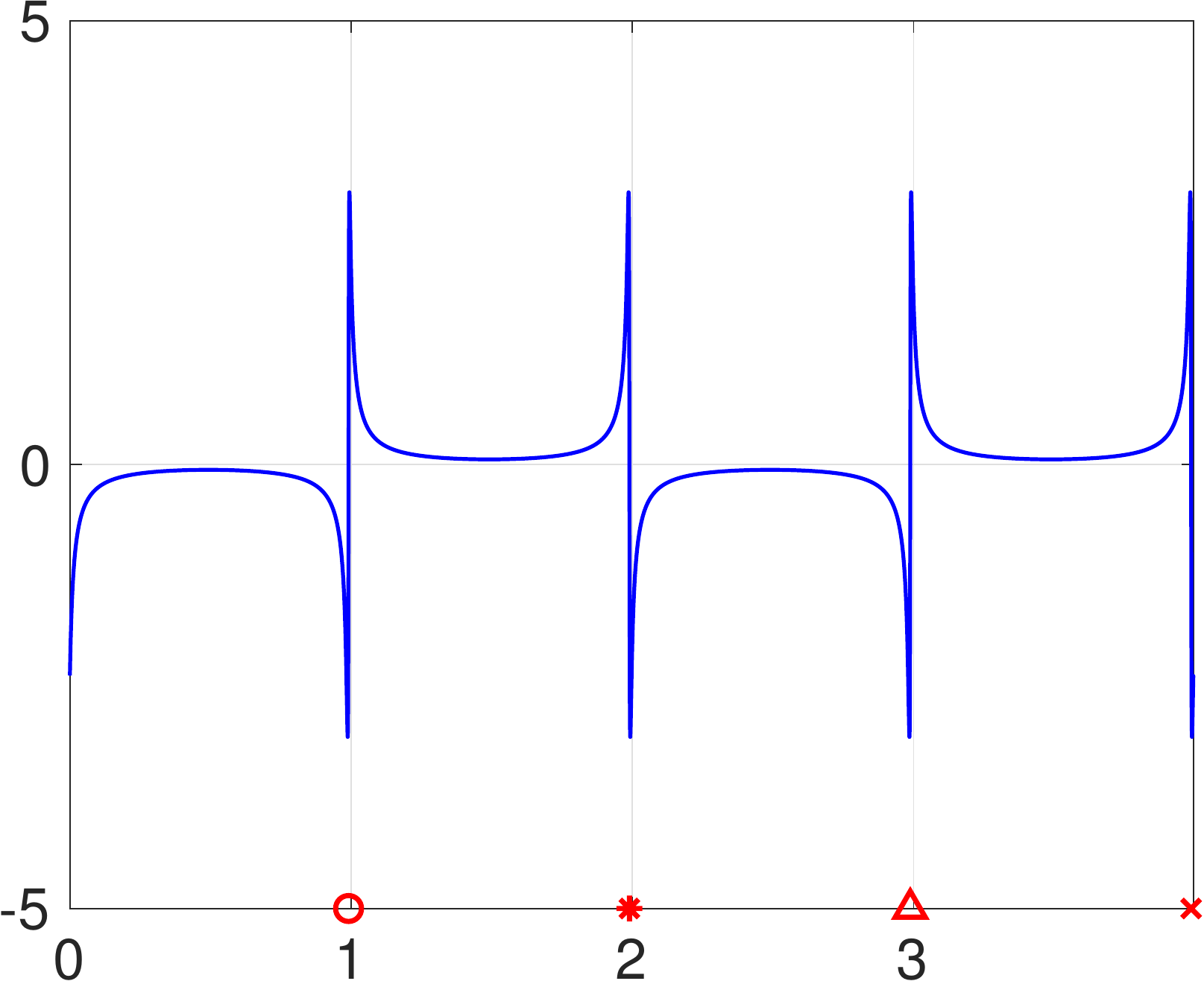}}
\subfigure[]{
\includegraphics[width=0.2\textwidth]{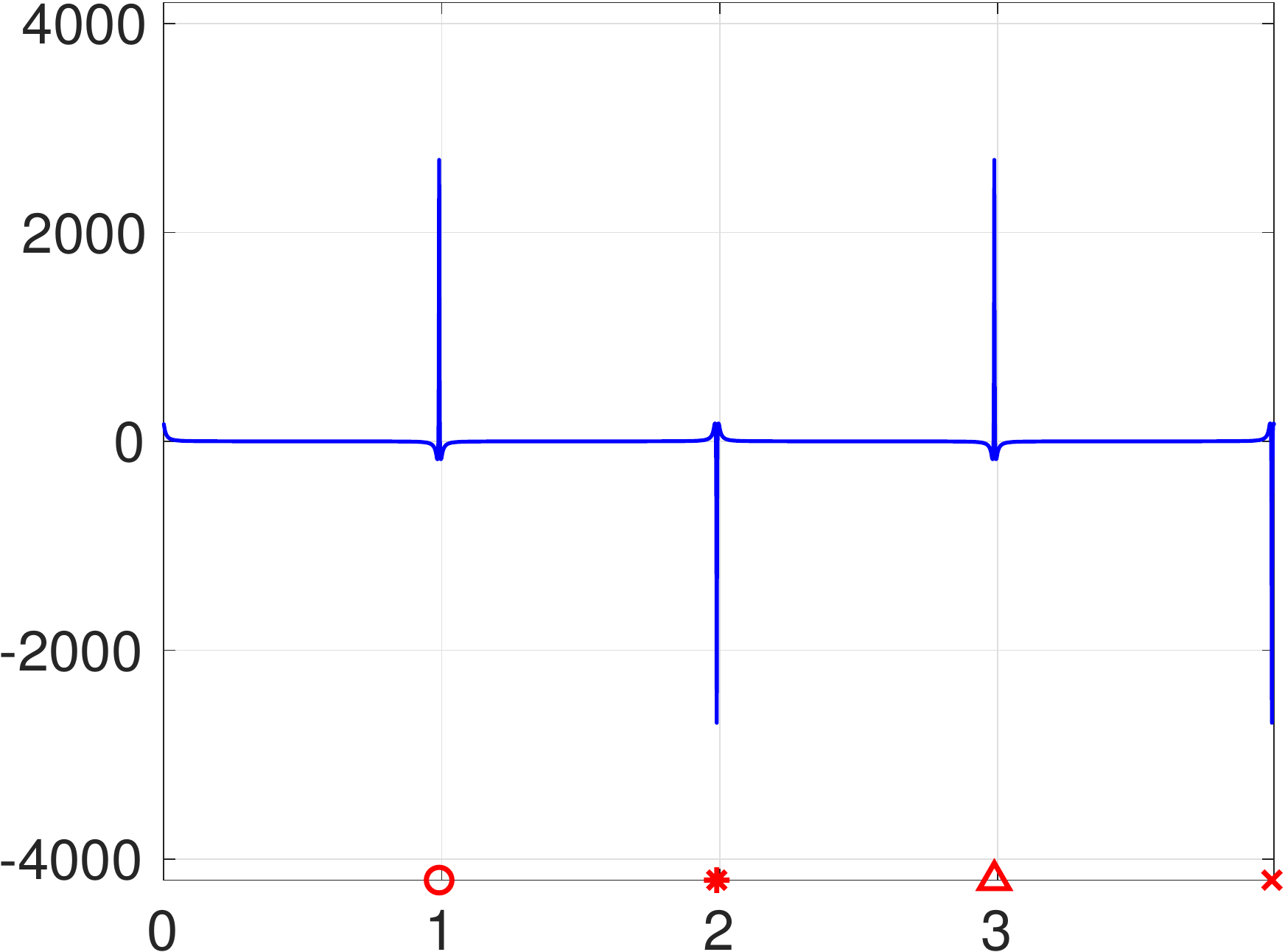}}
\subfigure[]{
\includegraphics[width=0.2\textwidth]{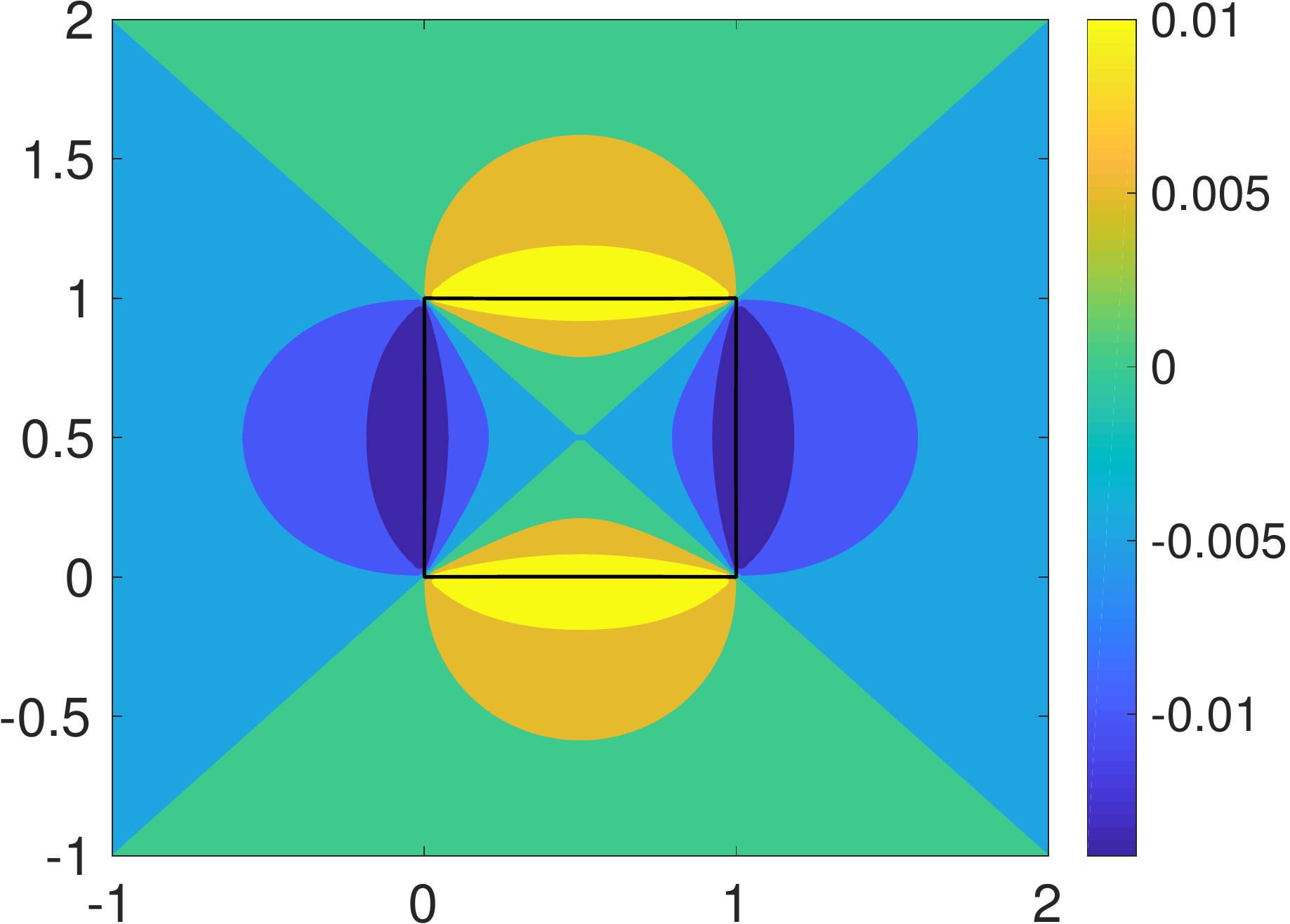}}
\subfigure[]{
\includegraphics[width=0.2\textwidth]{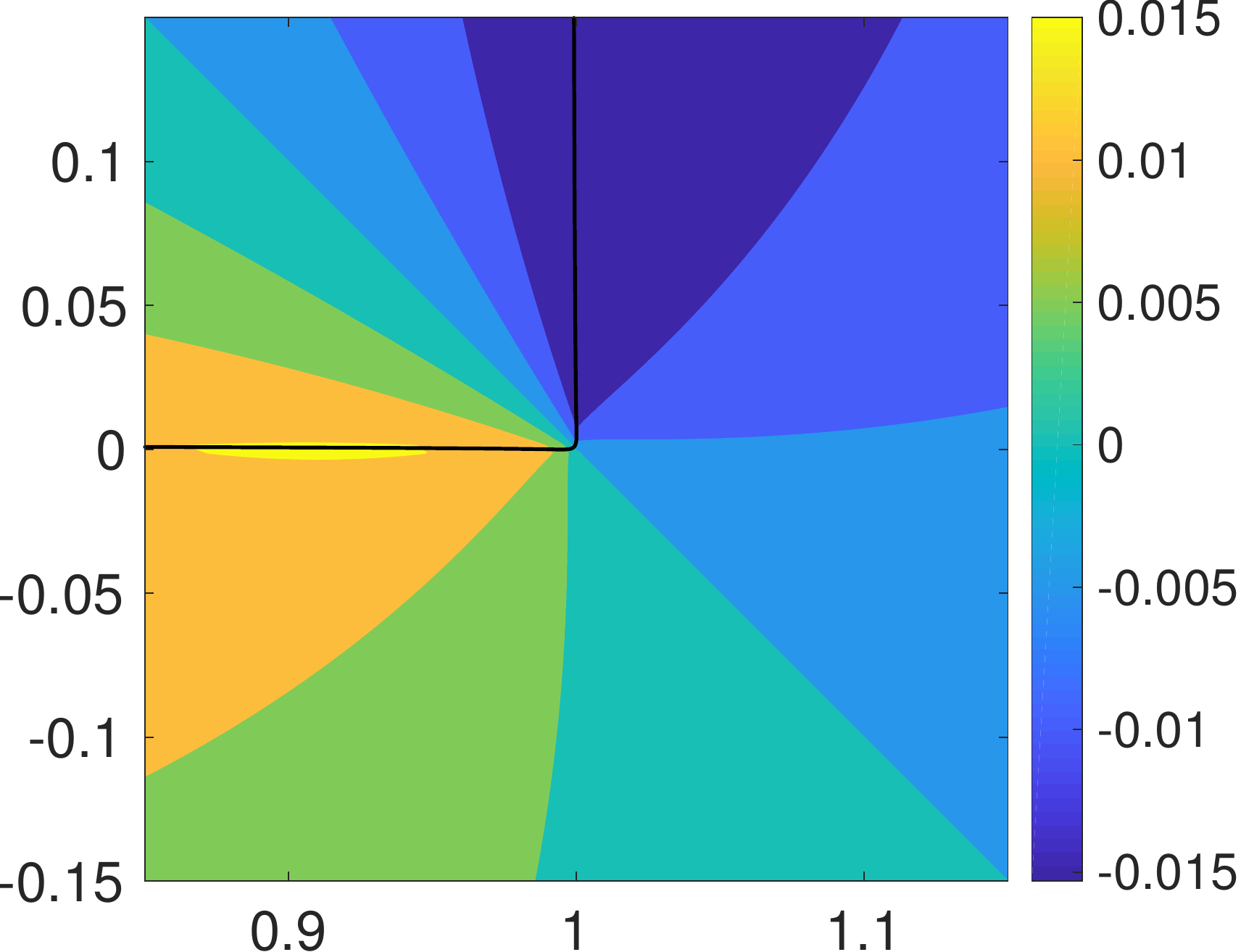}}\\
\caption{\label{fig18}  (a), (b). Plotting of the eigenfunction and its conormal derivative for $\lambda_2=-0.2183$; 
(c), (d). The associated single-layer potential for $\lambda_2=-0.2183$.}
\end{figure}

Fig.~\ref{fig19} plots the eigenfunctions with respect to arc length for the eigenvalues $\lambda_1=0.2183$ and $\lambda_2=-0.2183$ with different maximum curvature $500$, $1000$ and $1500$.

\begin{figure}
\centering
\subfigure[]{
\includegraphics[width=0.2\textwidth]{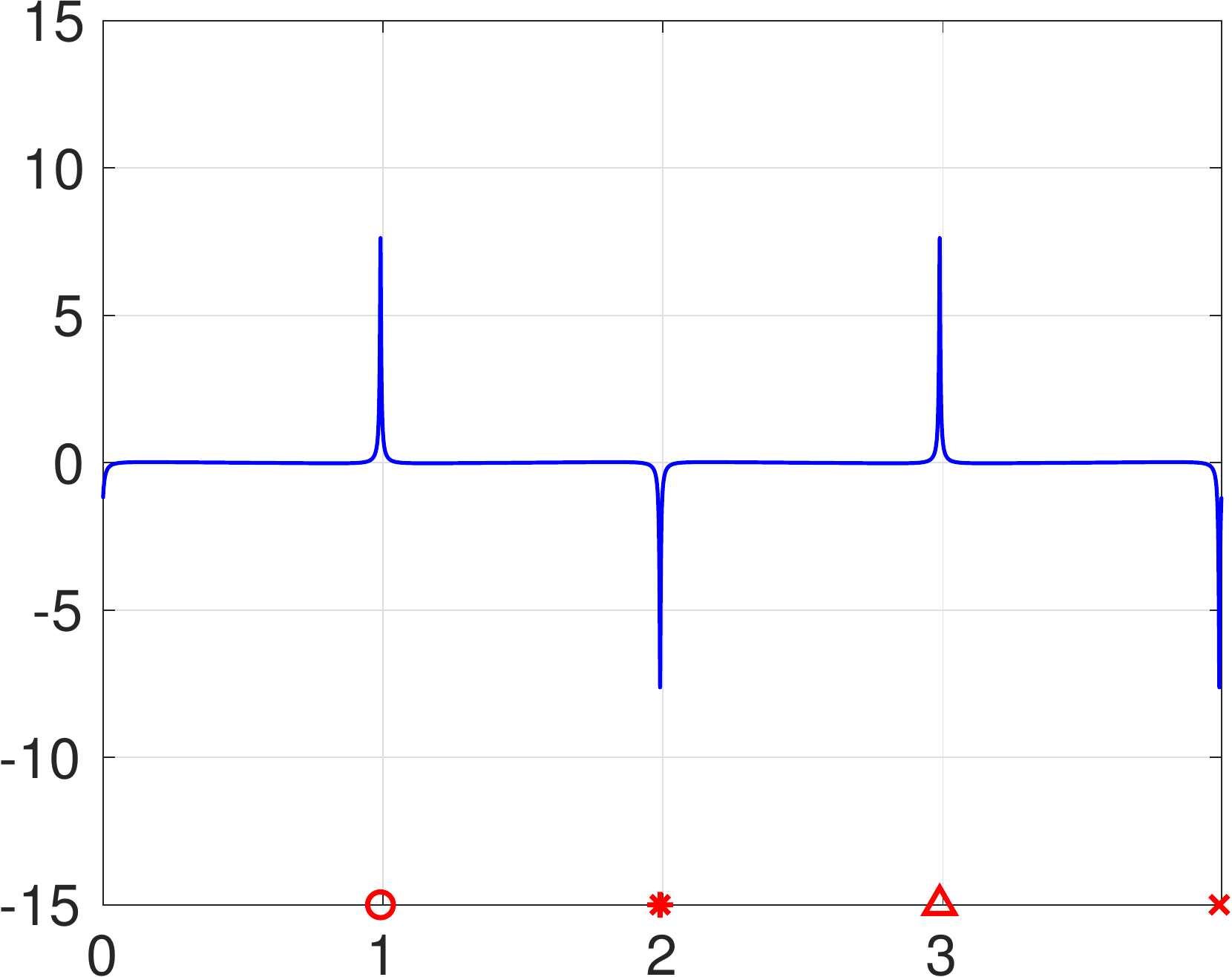}}
\subfigure[]{
\includegraphics[width=0.2\textwidth]{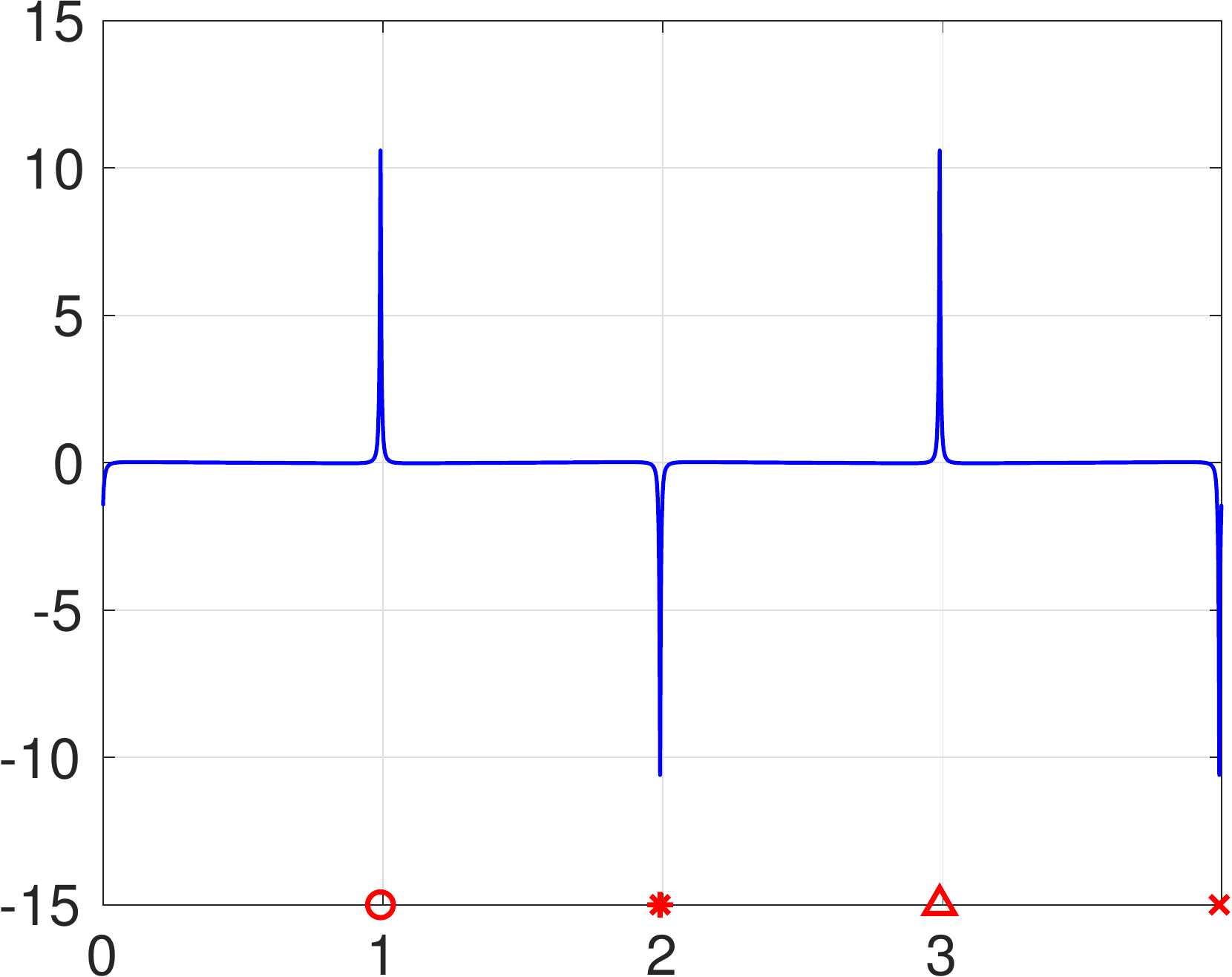}}
\subfigure[]{
\includegraphics[width=0.2\textwidth]{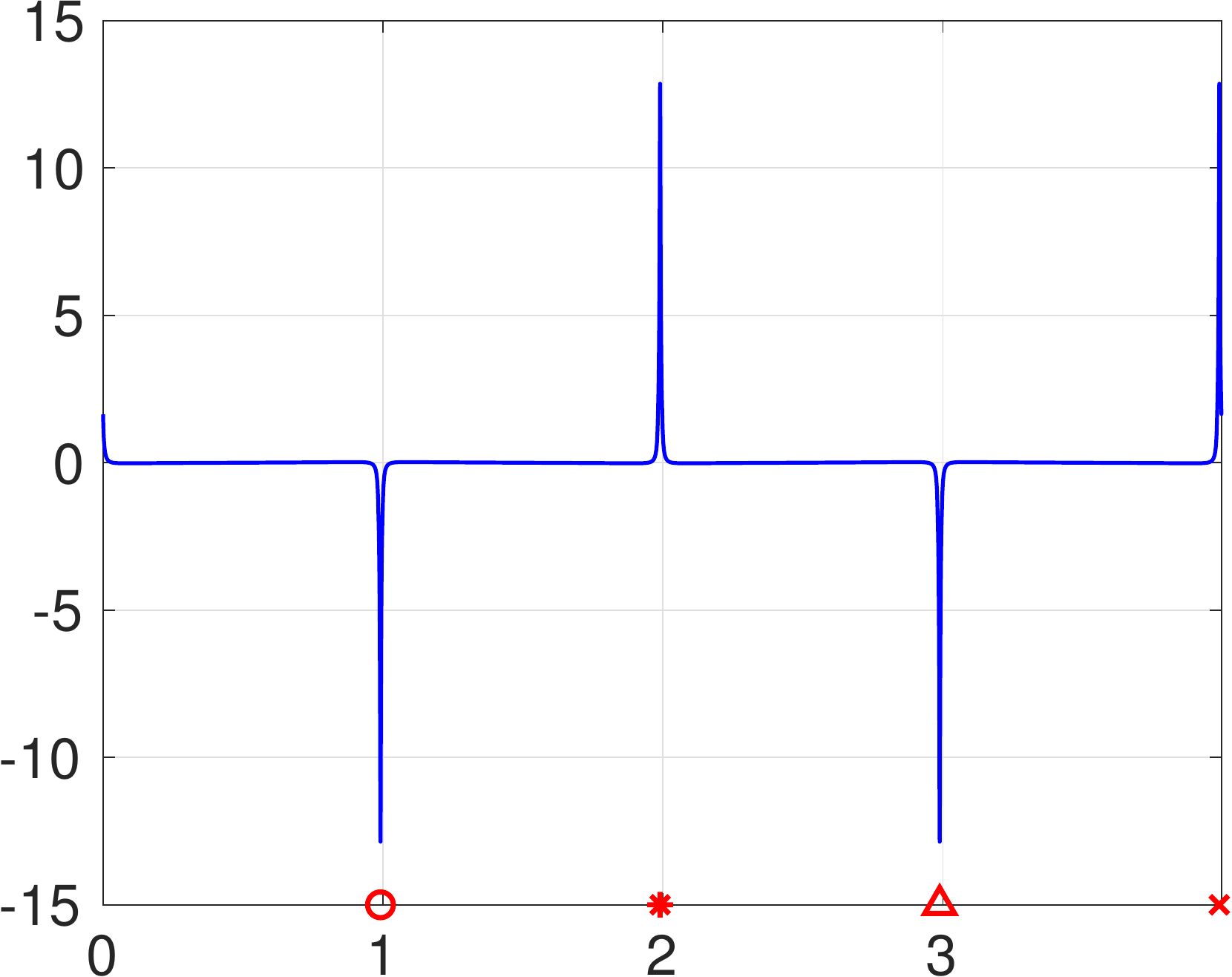}}\\
\subfigure[]{
\includegraphics[width=0.2\textwidth]{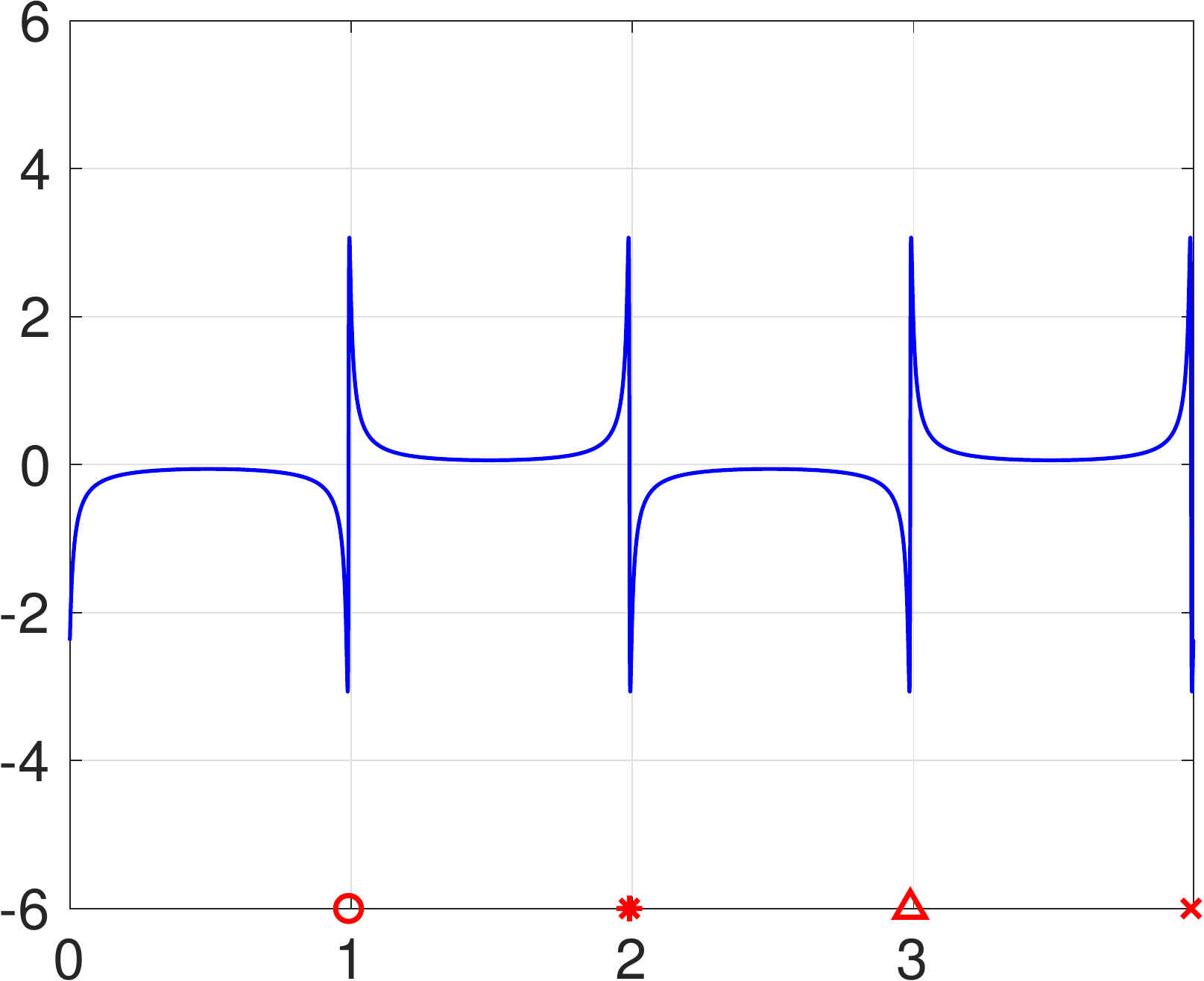}}
\subfigure[]{
\includegraphics[width=0.2\textwidth]{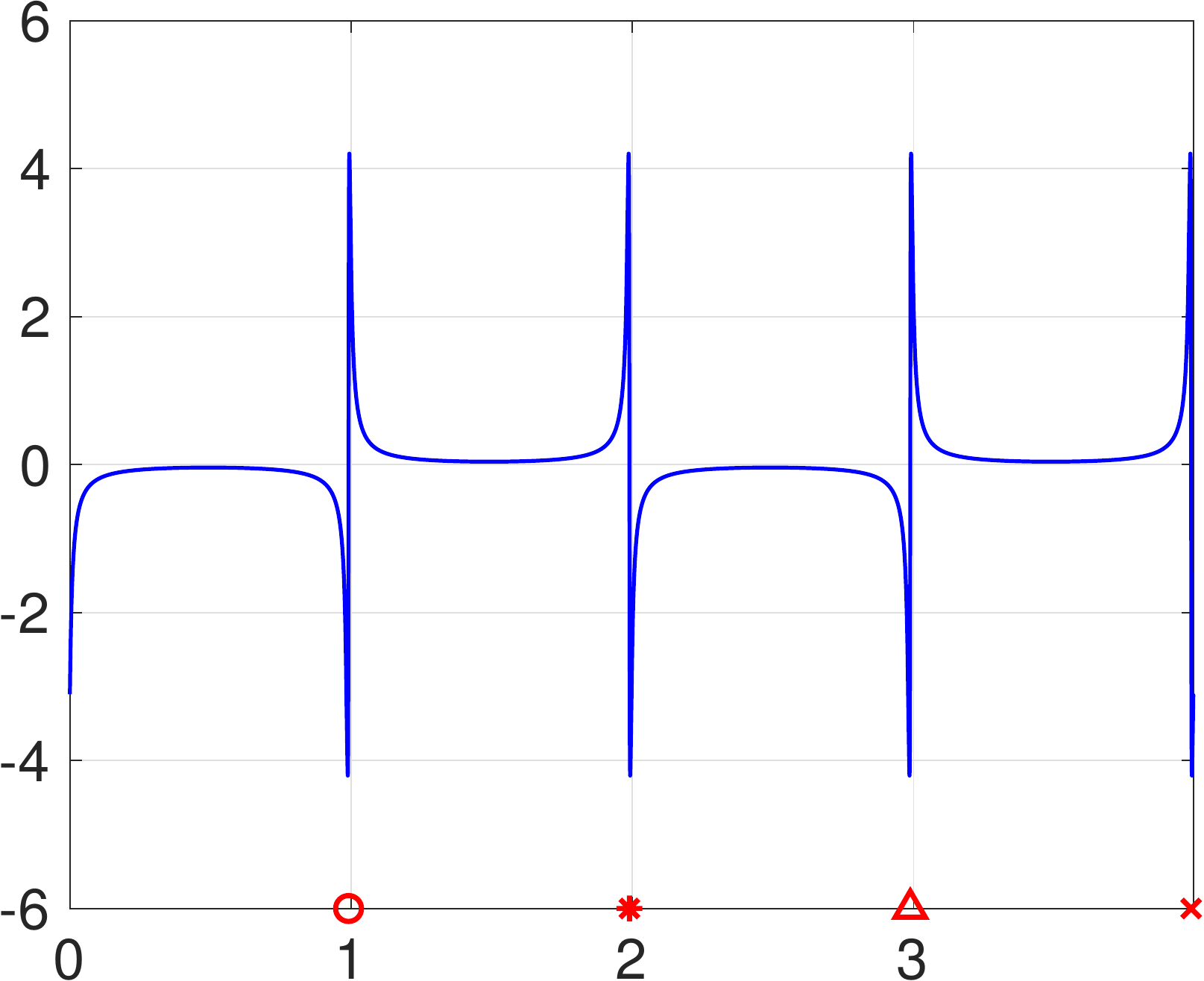}}
\subfigure[]{
\includegraphics[width=0.2\textwidth]{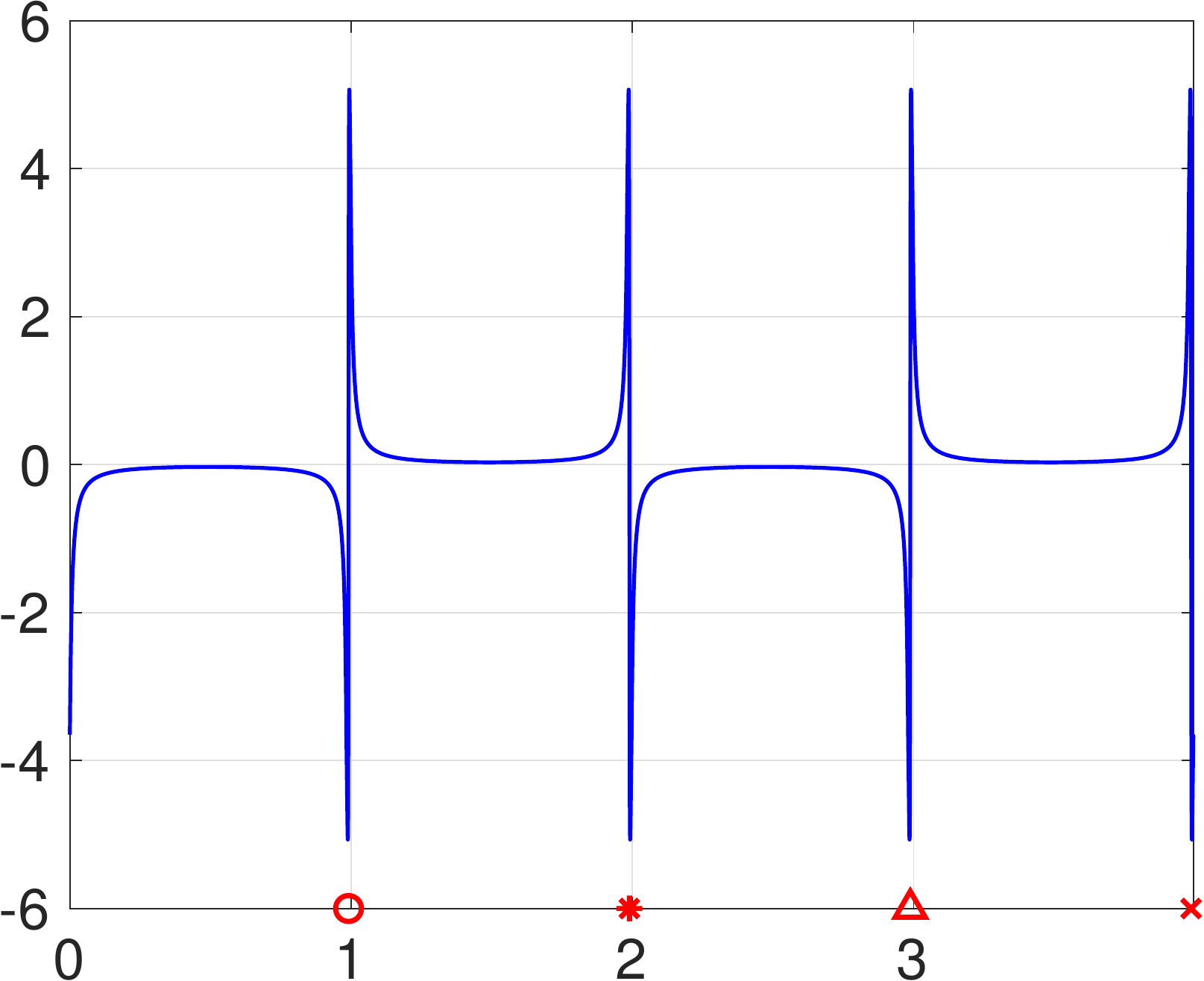}}\\
\caption{\label{fig19} (a), (b), (c).  Plotting the eigenfunctions for the positive eigenvalues $\lambda_1=0.2183$ with different maximum curvature $500$, $1000$ and $1500$. 
(d), (e), (f). The corresponding items for the negative eigenvalue $\lambda_2=-0.2183$.}
\end{figure}

We next investigate the blow-up rate of the NP eigenfunction or its conormal derivative with respect to the curvature.Therefore we plot the logarithm of the absolute value of the eigenfunctions at the high-curvature point for the positive eigenvalues $\lambda_1$, $\lambda_{7}$ and $\lambda_{9}$, and the logarithm of the absolute value of the conormal derivative of the eigenfunctions at the high-curvature point for the simple negative eigenvalues $\lambda_2$, $\lambda_{8}$ and $\lambda_{10}$ with respect to different curvature in Fig.~\ref{fig20}. It turns out that blow-up rate also follows the 
rule in \eqref{eq:growthrate}. By regression, we numerically determine the corresponding parameters for those different eigenvalues 
in \eqref{eq:egg4}, and they are listed in Table~\ref{tab4}.

%Fig.~\ref{fig19} shows that both the absolute value of the eigenfunction for the positive eigenvalue and the absolute value of the conormal derivative of the eigenfunction for the negative eigenvalue at high-curvature point $x_o$ increase as the the curvature $\kappa_{x_o}$ increases, with $\kappa_{x_o}$ denoting the curvature at $x_o$. Therefore we plot the logarithm of the absolute value of the eigenfunctions at the high-curvature point for the positive eigenvalues $\lambda_1$, $\lambda_{7}$ and $\lambda_{9}$, and the logarithm of the absolute value of the conormal derivative of the eigenfunctions at the high-curvature point for the simple negative eigenvalues $\lambda_2$, $\lambda_{8}$ and $\lambda_{10}$ with respect to different curvature in Fig.~\ref{fig19}. Denote by $\psi_{\max}$ the  maximum of the absolute value of the eigenfunction for the positive eigenvalue, or the maximum of absolute value of the conormal derivative of the eigenfunction for the negative eigenvalue and from Fig.~\ref{fig19} one can conclude that by regression
%\[
% \psi_{\max} \sim a \kappa_{\max}^p.
%\]
%The coefficients of the regression are shown in Table.~\ref{tab4}.
\begin{figure}
\includegraphics[width=0.2\textwidth] {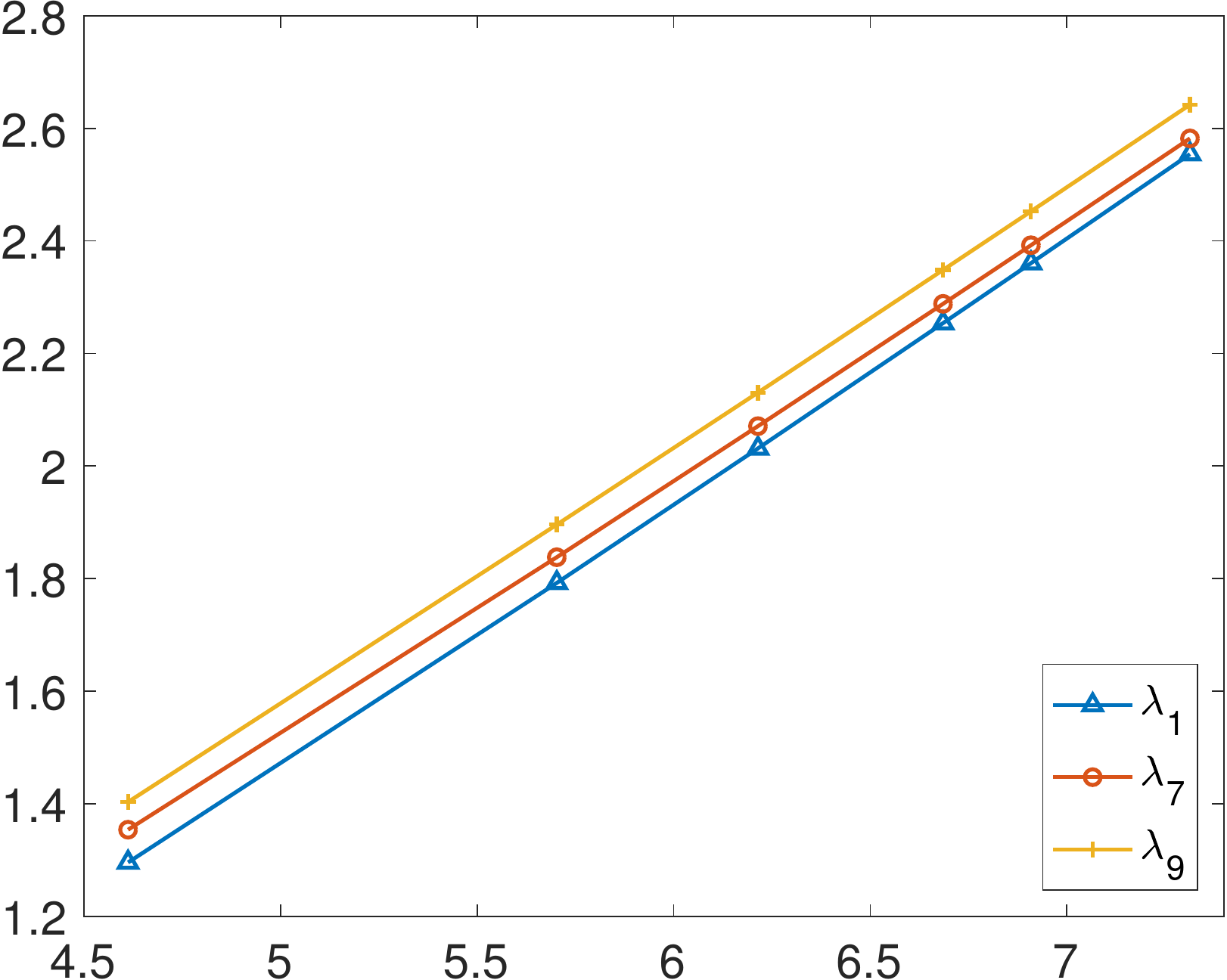}
\hspace{0.8cm}
\includegraphics[width=0.2\textwidth] {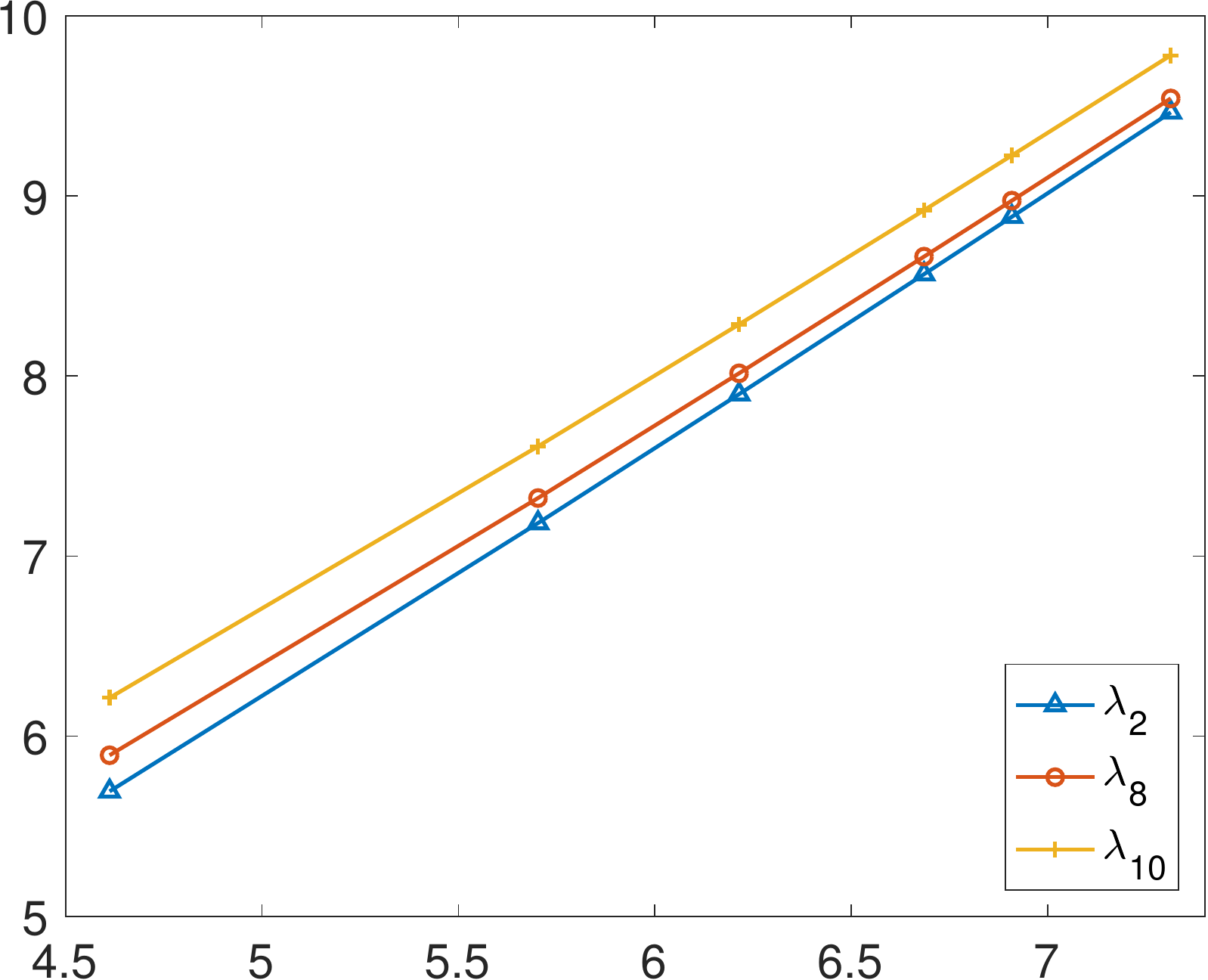}
\caption{\label{fig20} The logarithm of the eigenfunction at the high-curvature point $x_o$ for the positive eigenvalues $\lambda_1$, $\lambda_{7}$ and $\lambda_{9}$, and the negative eigenvalues $\lambda_2$, $\lambda_{8}$ and $\lambda_{10}$ with respect to different curvature.}
\end{figure}

\begin{table}[t]
  \centering
  \subtable[]{
    \centering
    \begin{tabular}{cccc}
      \toprule
      & $\lambda_1$ & $\lambda_{7}$ & $\lambda_{9}$ \\
      \midrule
      $p$ & 0.4657 & 0.4544 & 0.4584 \\[5pt]
      $\ln(a)$ & -0.8579 & -0.7477 & -0.7147 \\
      \bottomrule
    \end{tabular}}
  %  \caption{}
  \hspace{1.5cm}
\subtable[]{
    \centering
    \begin{tabular}{cccc}
      \toprule
      & $\lambda_2$ & $\lambda_{8}$ & $\lambda_{10}$ \\
      \midrule
      $p$ & 1.3951 & 1.3492 & 1.3198 \\[5pt]
      $\ln(a)$ & -0.7562 & -0.3501 & 0.1044 \\
      \bottomrule
    \end{tabular}
    }
 %\end{subtable}
  \caption{The coefficients of the regression; (a) $\lambda_j, j=1,7,9$; (b) $\lambda_j, j=2,8,10$.}
  \label{tab4}
\end{table}

\subsection{A concave 1-symmetric domain}

In this subsection, we consider a concave 1-symmetric domain as shown in Fig.~\ref{fig21}, which possesses one high-curvature points that are
denoted by $x_*$ and $x_o$. The largest curvature is
\begin{equation}
 \kappa_{\max}=\kappa_{x_*}= \kappa_{x_o}=500.
\end{equation}
\begin{figure}
\includegraphics[width=4cm] {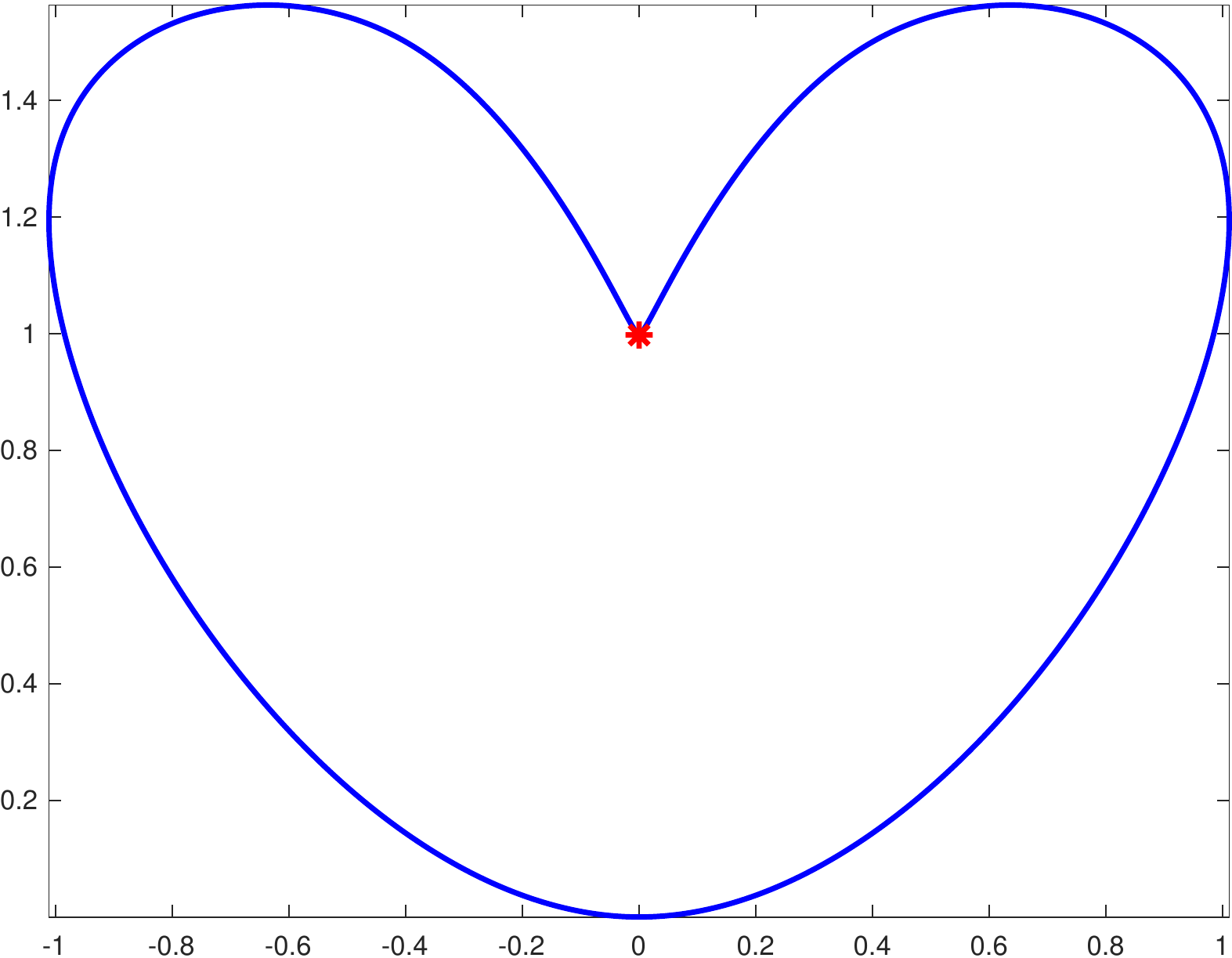}
\caption{\label{fig21} the boundary with one high-curvature point}
\end{figure}
The first five largest NP eigenvalues (in terms of the absolute value) are
\begin{equation}\label{eq:egg3_2}
\begin{split}
& \lambda_0=0.5, \quad  \lambda_1=0.3310,\quad \lambda_2=-0.3310,\quad\lambda_3=0.2142,\\
& \quad \lambda_4=-0.2142, \quad \lambda_5=0.1262 \quad \lambda_6=-0.1262.
\end{split}
\end{equation}

Fig.~\ref{fig22} plots the eigenfunctions, their conormal derivatives and the associated single-layer potentials, respectively, associated with the eigenvalues
 $\lambda_1=0.3310$ and $\lambda_3=0.2142$. The numerical results clearly support our earlier assertion about the NP eigenfunctions associated with simple
 positive eigenvalues. 
%
%Fig.~\ref{fig22} plots the eigenfunctions with respect to the arc length, the conormal derivative of the eigenfunctions with respect to the arc length, the associated single layer potentials and the single layer potentials around the high-curvature point $x_*$ for the positive eigenvalues  $\lambda_1=0.3310$ and $\lambda_3=0.2142$ respectively. In Fig.~\ref{fig22}, $a,b,e,f$ show that the conormal derivative of the eigenfunctions blow up at the high-curvature point $x_*$ for the positive eigenvalue. $c,d,g,h$ show that the conormal derivative of the associated single layer potentials blow up at the high-curvature point $x_*$ for the positive eigenvalue.
%
\begin{figure}
\centering
\subfigure[]{
\includegraphics[width=0.2\textwidth]{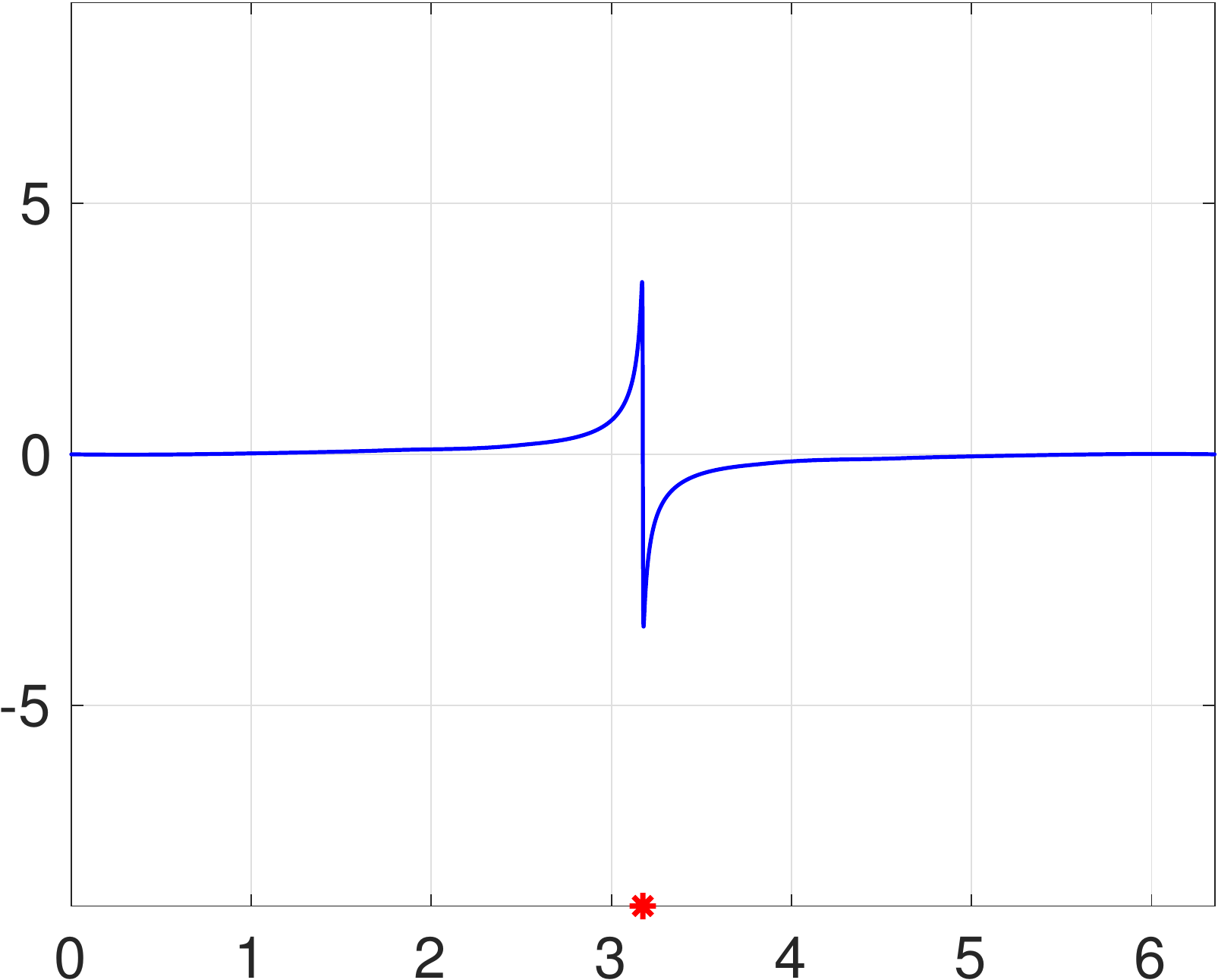}}
\subfigure[]{
\includegraphics[width=0.2\textwidth]{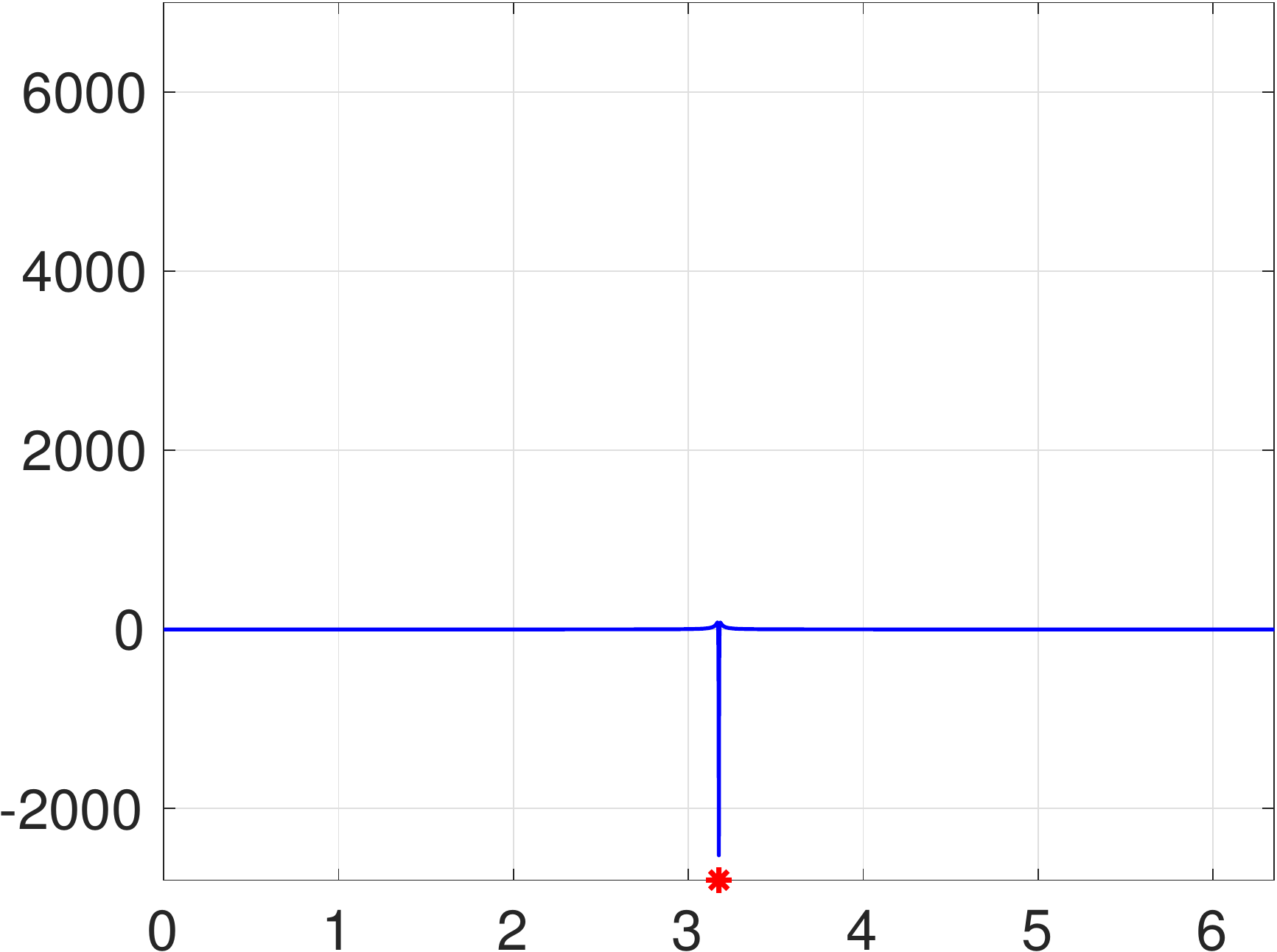}}
\subfigure[]{
\includegraphics[width=0.2\textwidth]{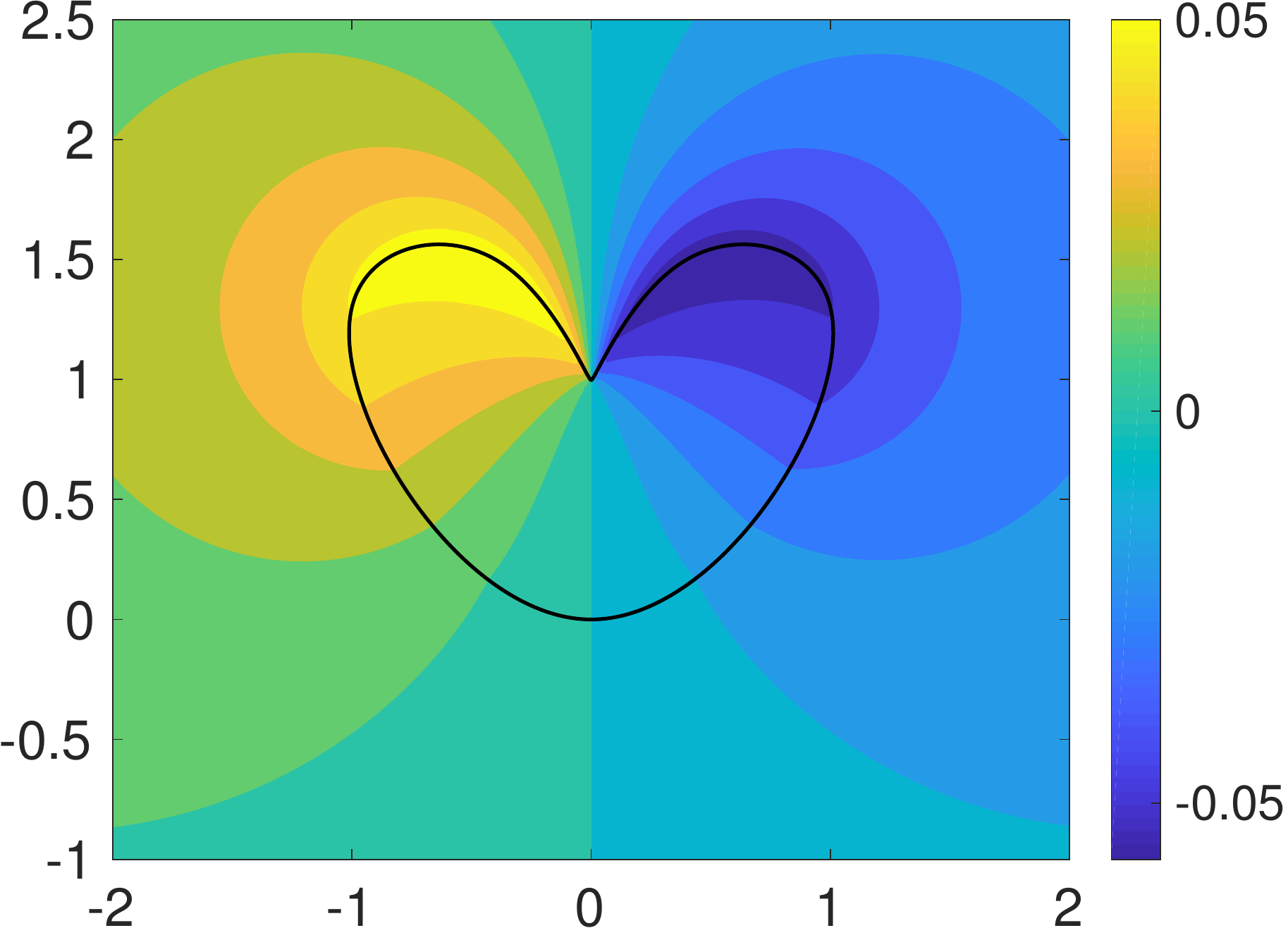}}
\subfigure[]{
\includegraphics[width=0.2\textwidth]{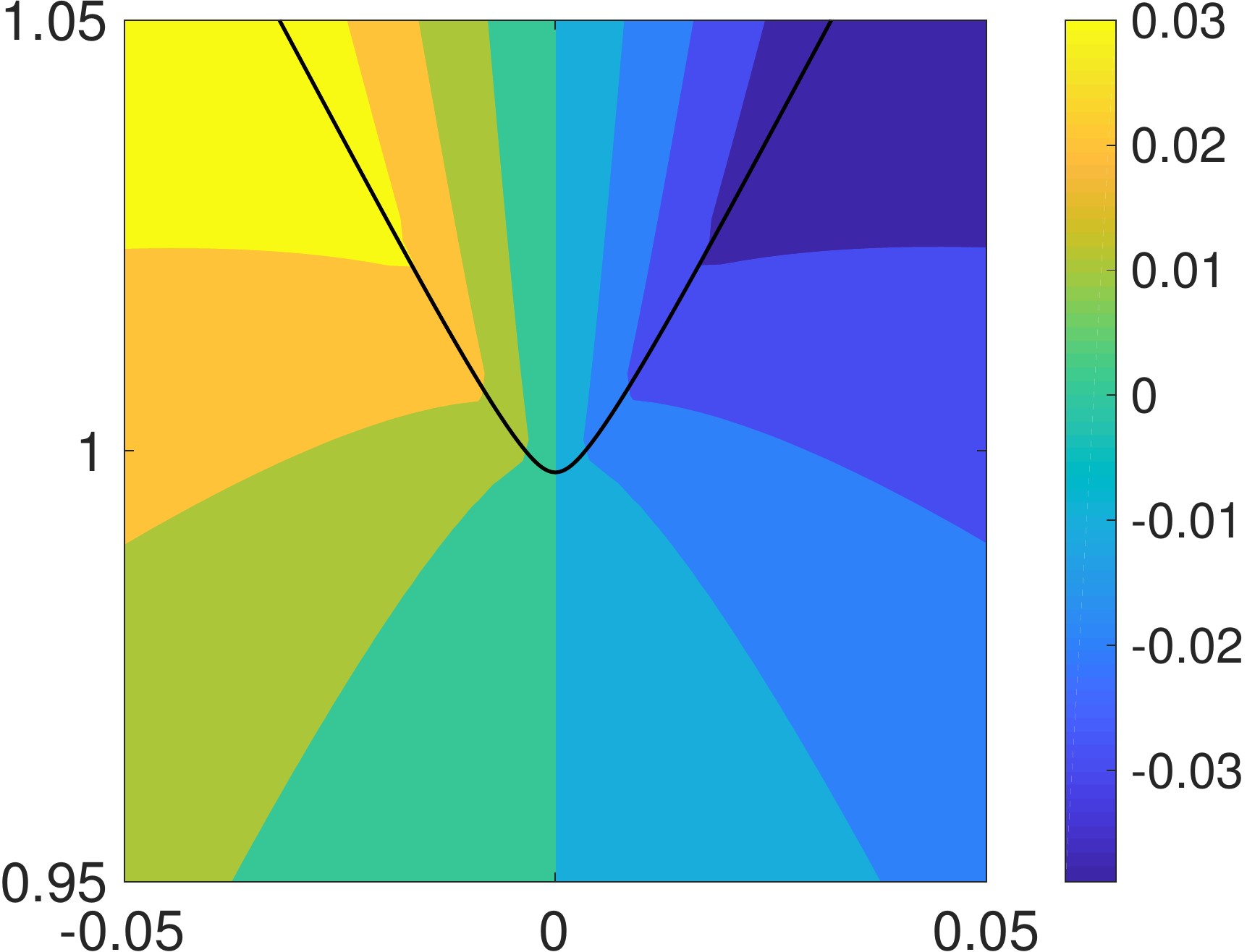}}\\
\subfigure[]{
\includegraphics[width=0.2\textwidth]{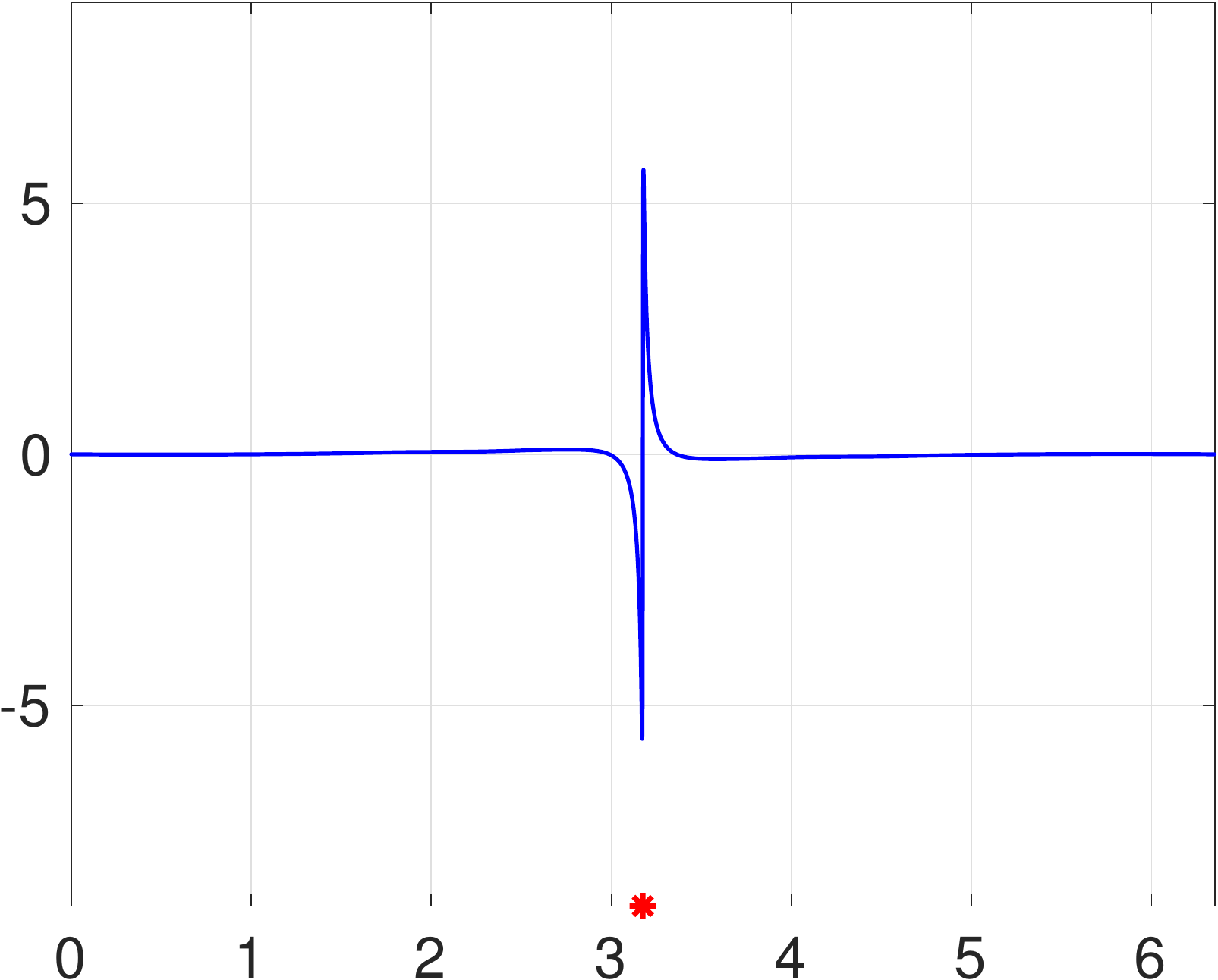}}
\subfigure[]{
\includegraphics[width=0.2\textwidth]{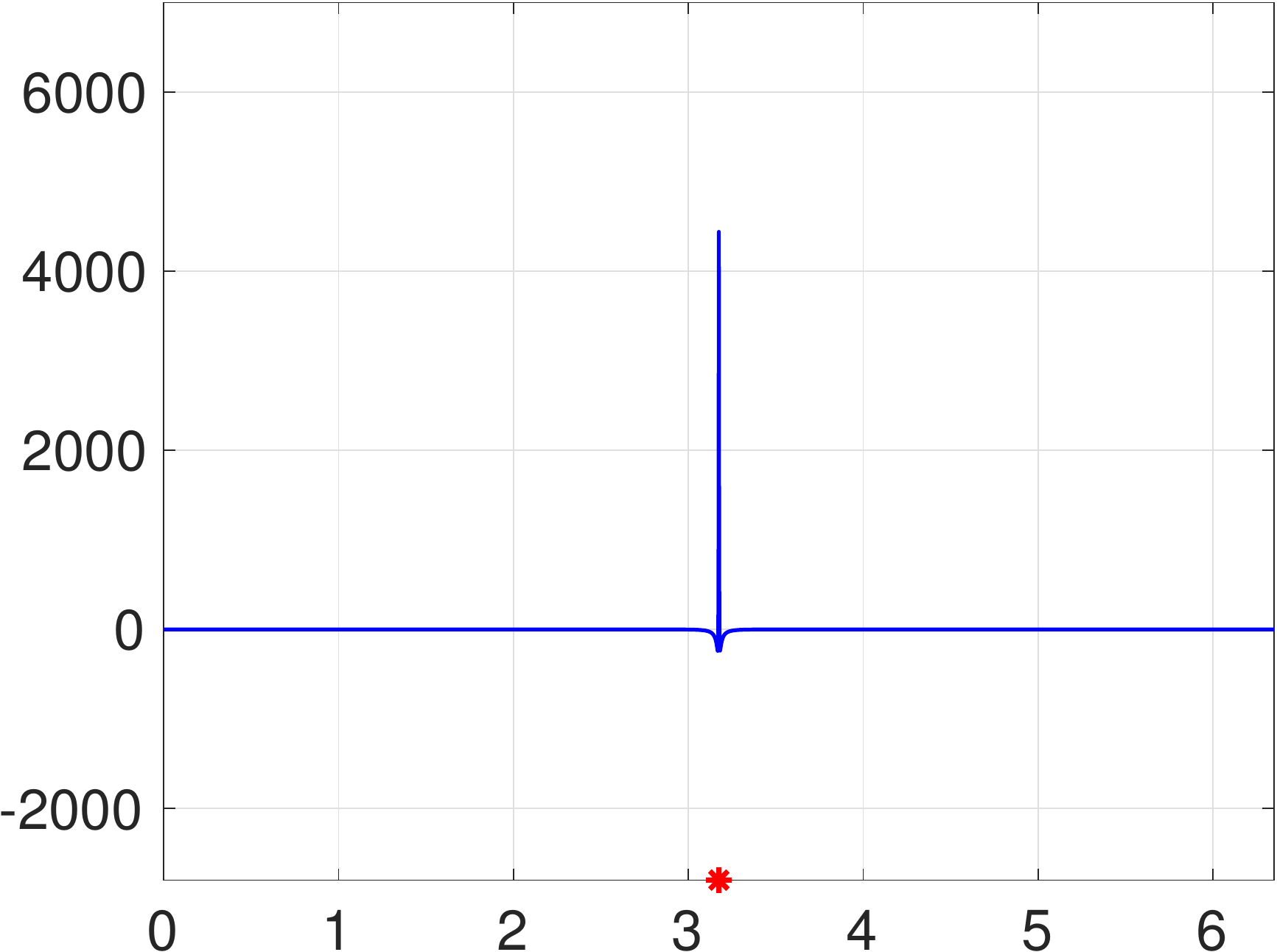}}
\subfigure[]{
\includegraphics[width=0.2\textwidth]{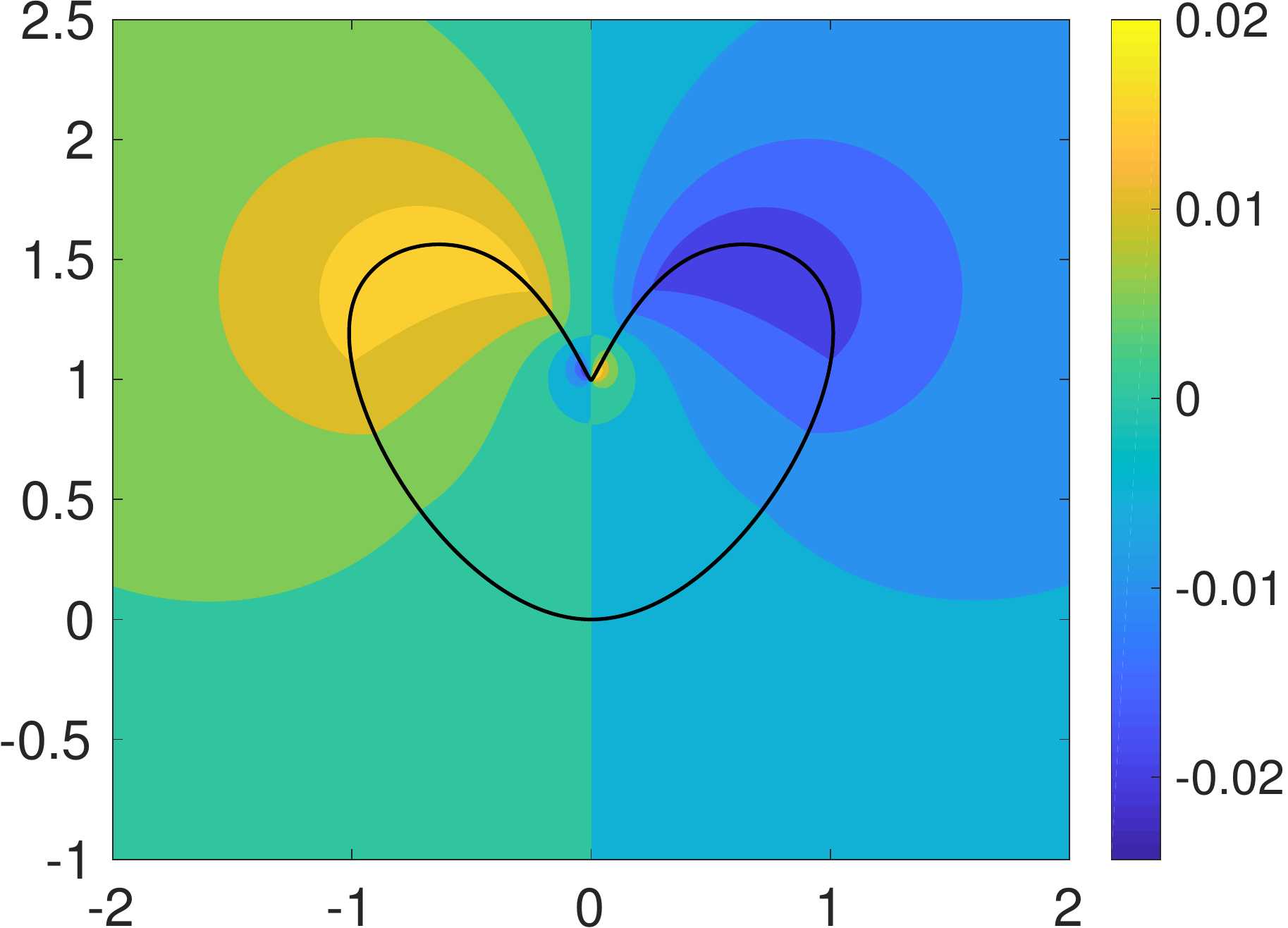}}
\subfigure[]{
\includegraphics[width=0.2\textwidth]{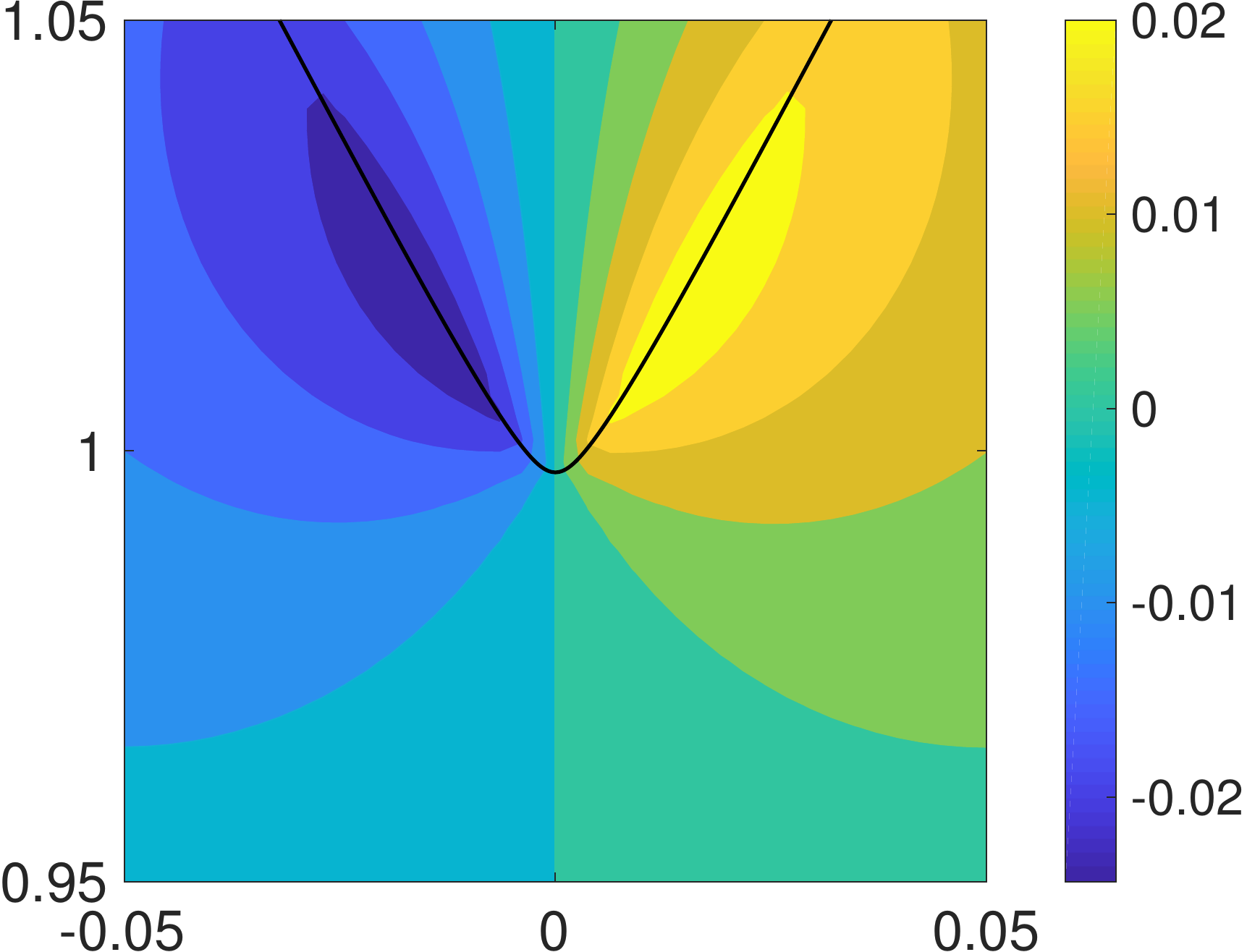}}
\caption{\label{fig22} (a), (b), (c), (d). The eigenfunction, its conormal derivative, and the corresponding
single-layer potential associated with $\lambda_1=0.3310$; (e), (f), (g), (h). The corresponding items
associated with $\lambda_3=0.2142$.}
\end{figure}

Fig.~\ref{fig23} plots the eigenfunctions, their conormal derivatives and the associated single-layer potentials, respectively, associated with the eigenvalues
 $\lambda_2=-0.3310$ and $\lambda_4=-0.2142$. The numerical results clearly support our earlier assertion about the NP eigenfunctions associated with simple
 negative eigenvalues.

%Fig.~\ref{fig23} plots the eigenfunctions with respect to the arc length, the corresponding single layer potentials and the single layer potentials around the high-curvature point $x_*$ for the negative eigenvalues $\lambda_2=-0.3310$ and $\lambda_4=-0.2142$. In Fig.~\ref{fig23}, $a,d$ show that the eigenfunctions blow up at the high-curvature point $x_*$ for the negative eigenvalue. $b,c,e,f$ show that the associated single layer potentials blow up at the high-curvature point $x_*$ for the negative eigenvalue.
%
\begin{figure}
\centering
\subfigure[]{
\includegraphics[width=0.2\textwidth]{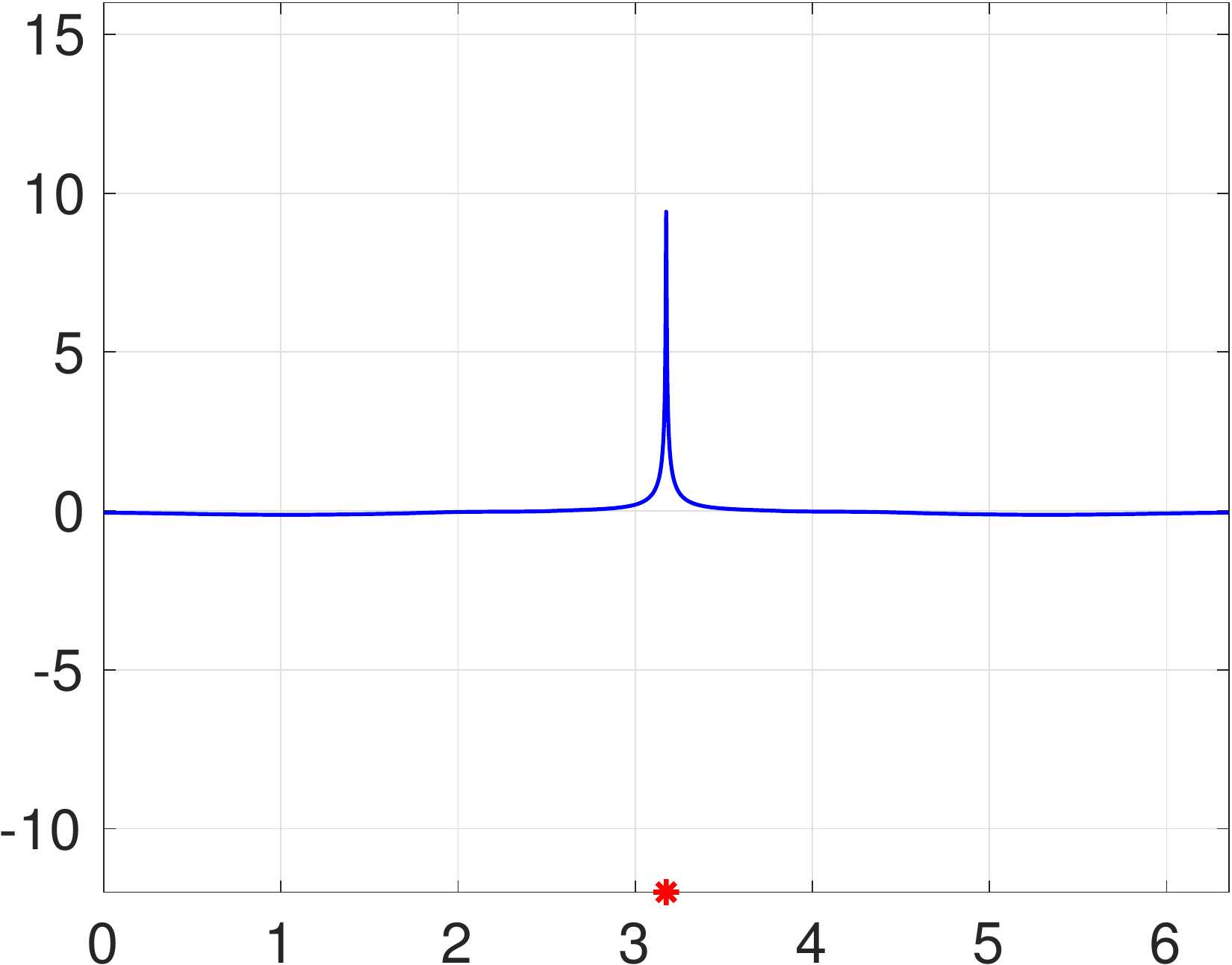}}
\subfigure[]{
\includegraphics[width=0.2\textwidth]{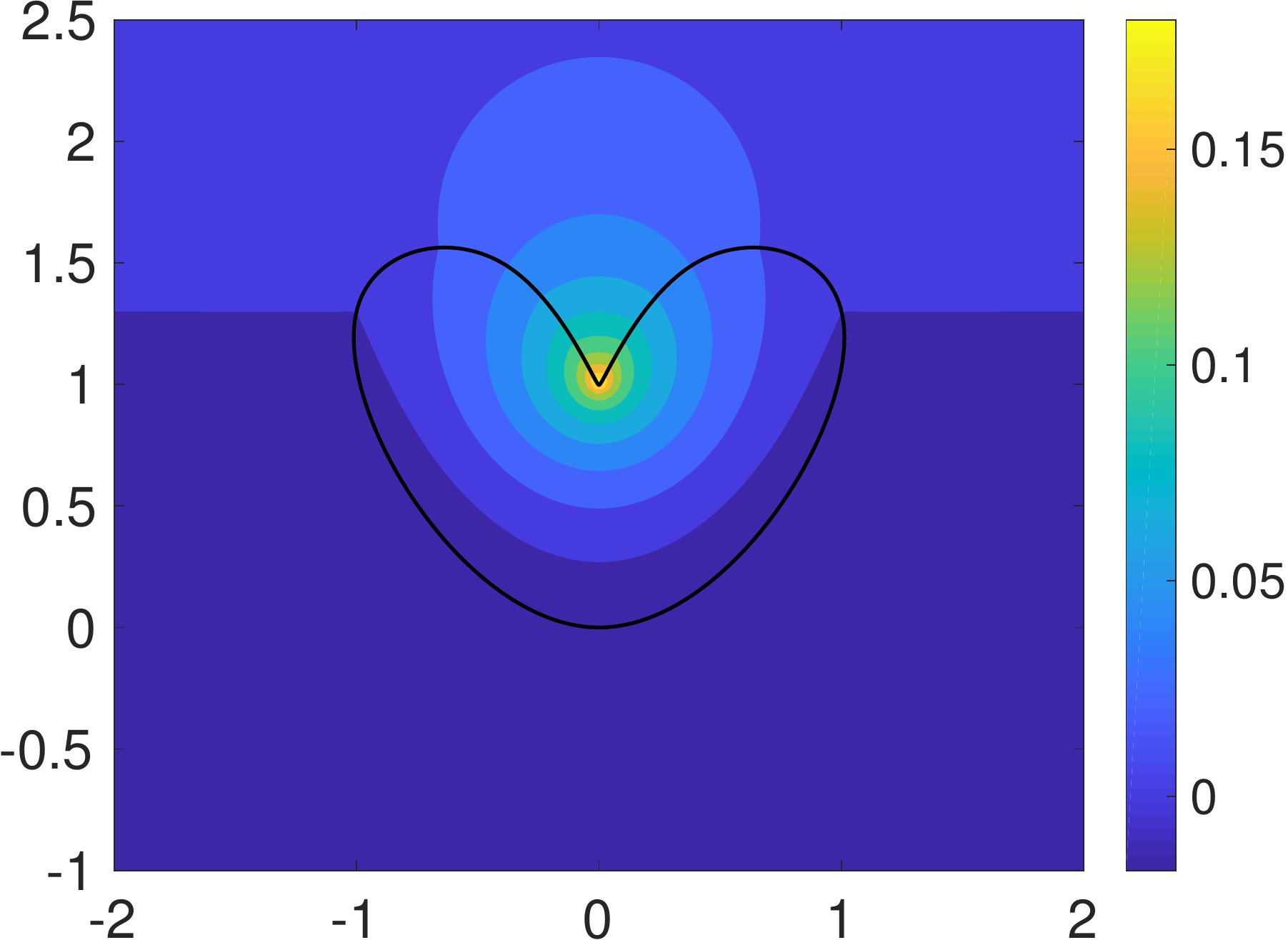}}
\subfigure[]{
\includegraphics[width=0.2\textwidth]{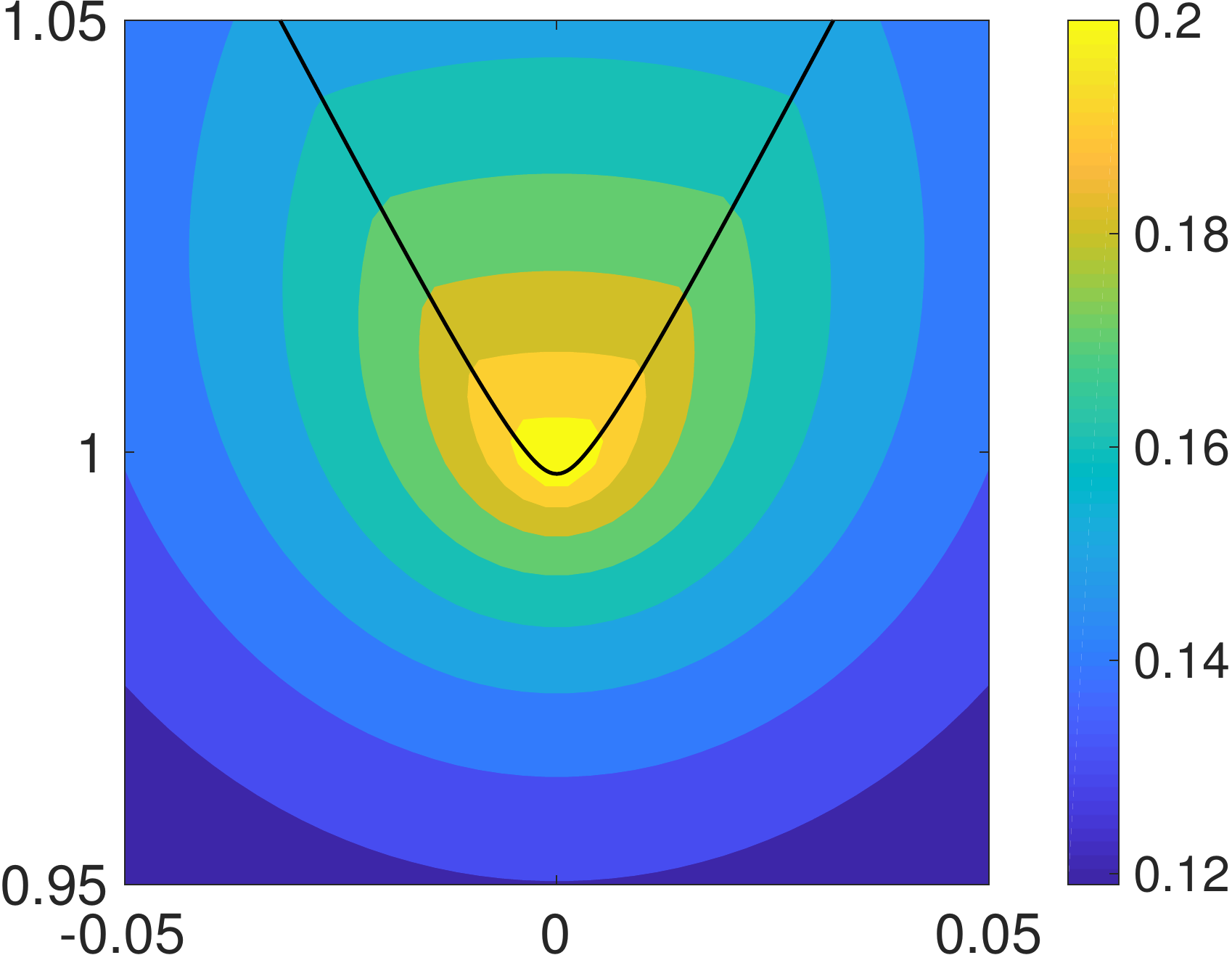}}\\
\subfigure[]{
\includegraphics[width=0.2\textwidth]{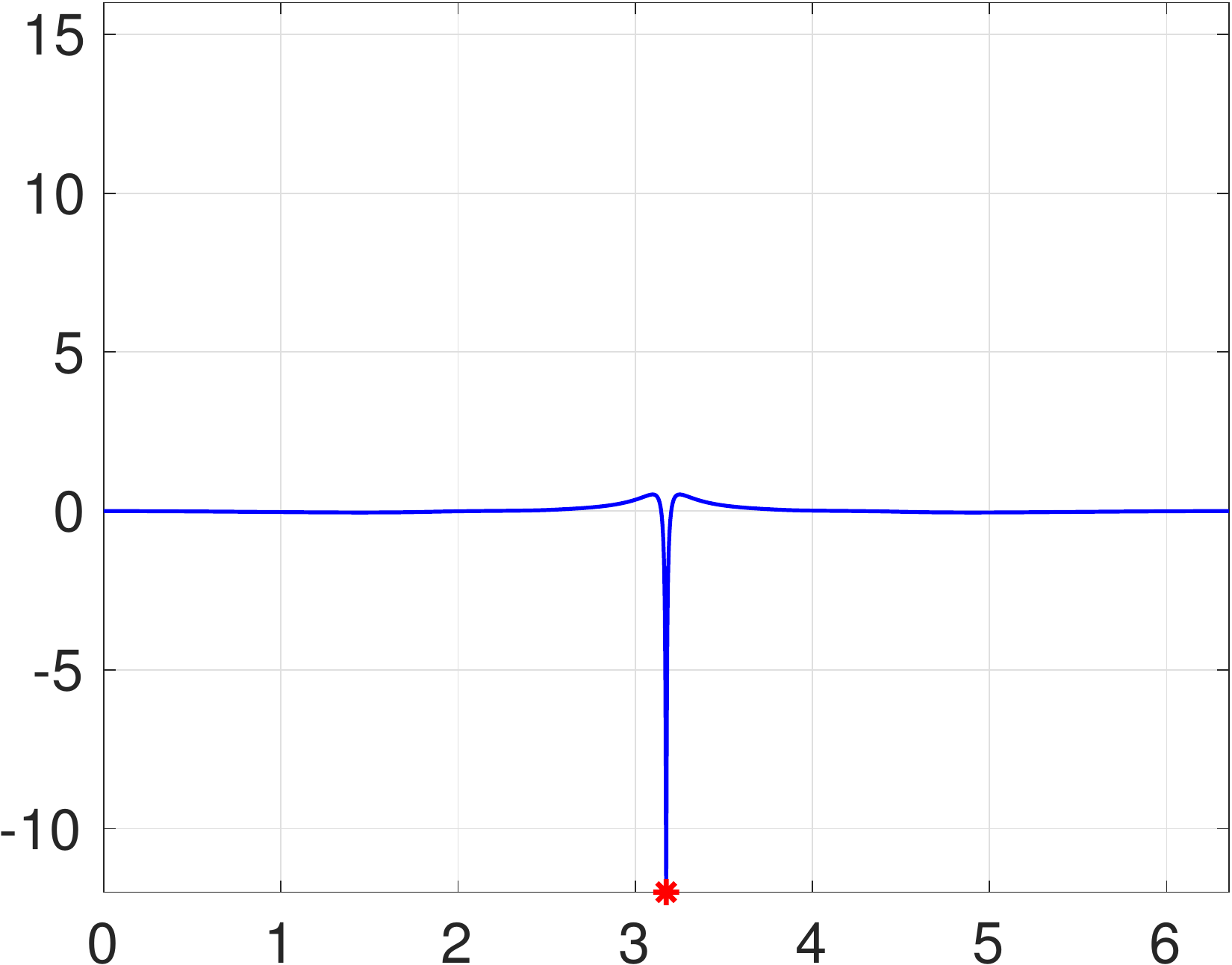}}
\subfigure[]{
\includegraphics[width=0.2\textwidth]{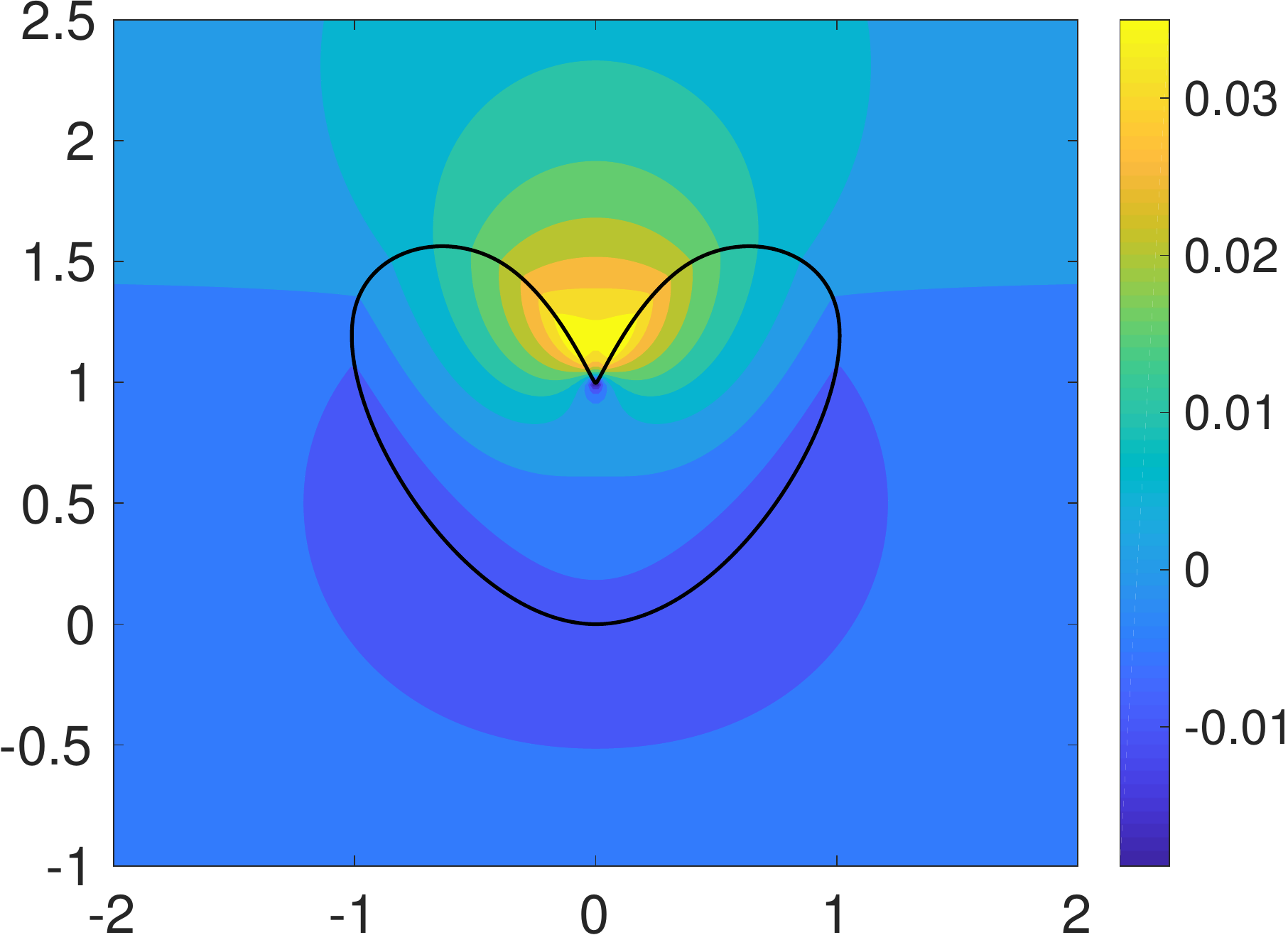}}
\subfigure[]{
\includegraphics[width=0.2\textwidth]{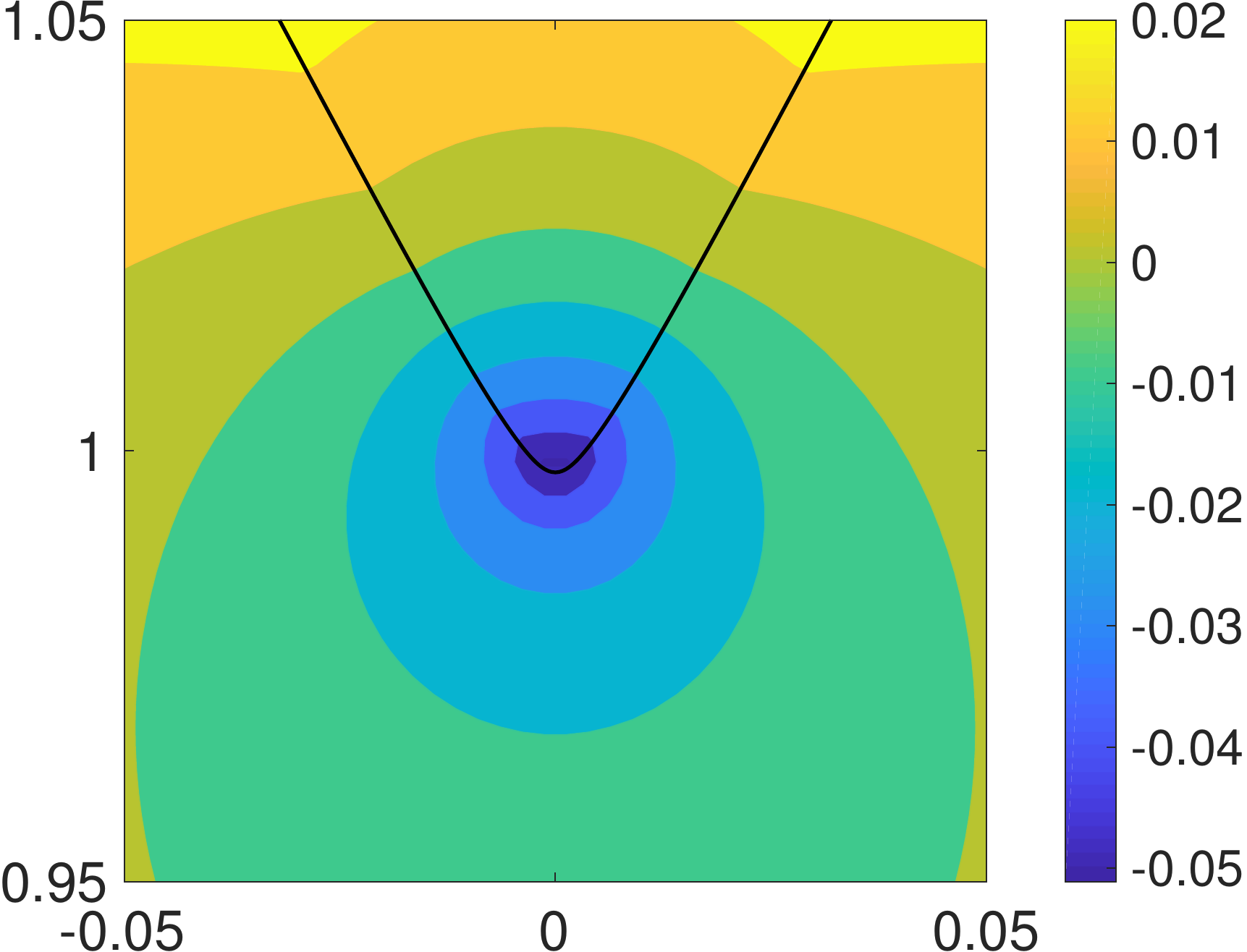}}\\
\caption{\label{fig23}  (a), (b), (c). The eigenfunction, its conormal derivative, and the corresponding
single-layer potential associated with $\lambda_2=-0.3310$; (d), (e), (f). The corresponding items
associated with $\lambda_4=-0.2142$.}
\end{figure}

Fig.~\ref{fig24} plots the eigenfunctions with respect to arc length for the eigenvalues $\lambda_1=0.3310$ and $\lambda_2=-0.3310$ with different maximum curvature $500$, $1000$ and $1500$.

\begin{figure}
\centering
\subfigure[]{
\includegraphics[width=0.2\textwidth]{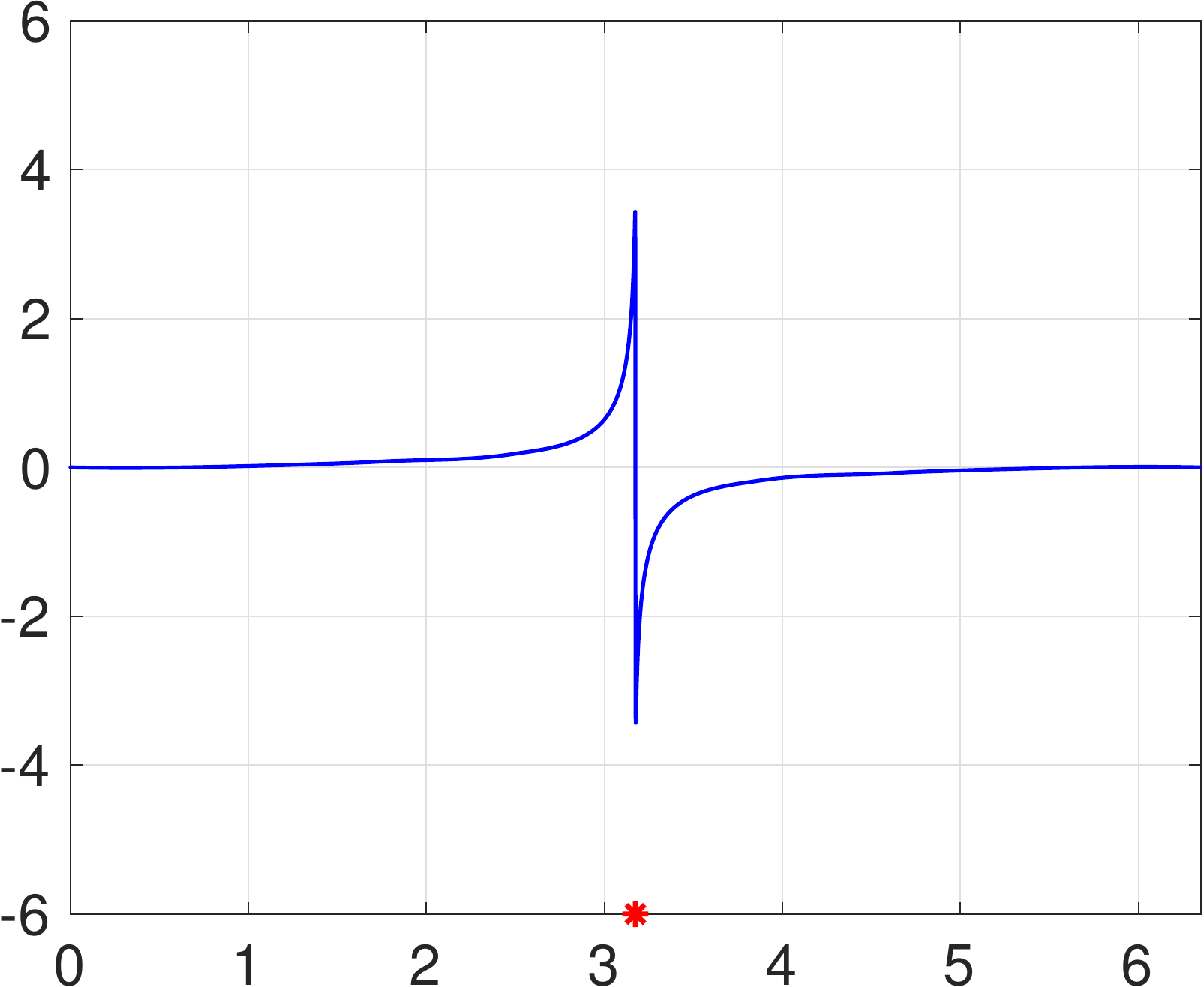}}
\subfigure[]{
\includegraphics[width=0.2\textwidth]{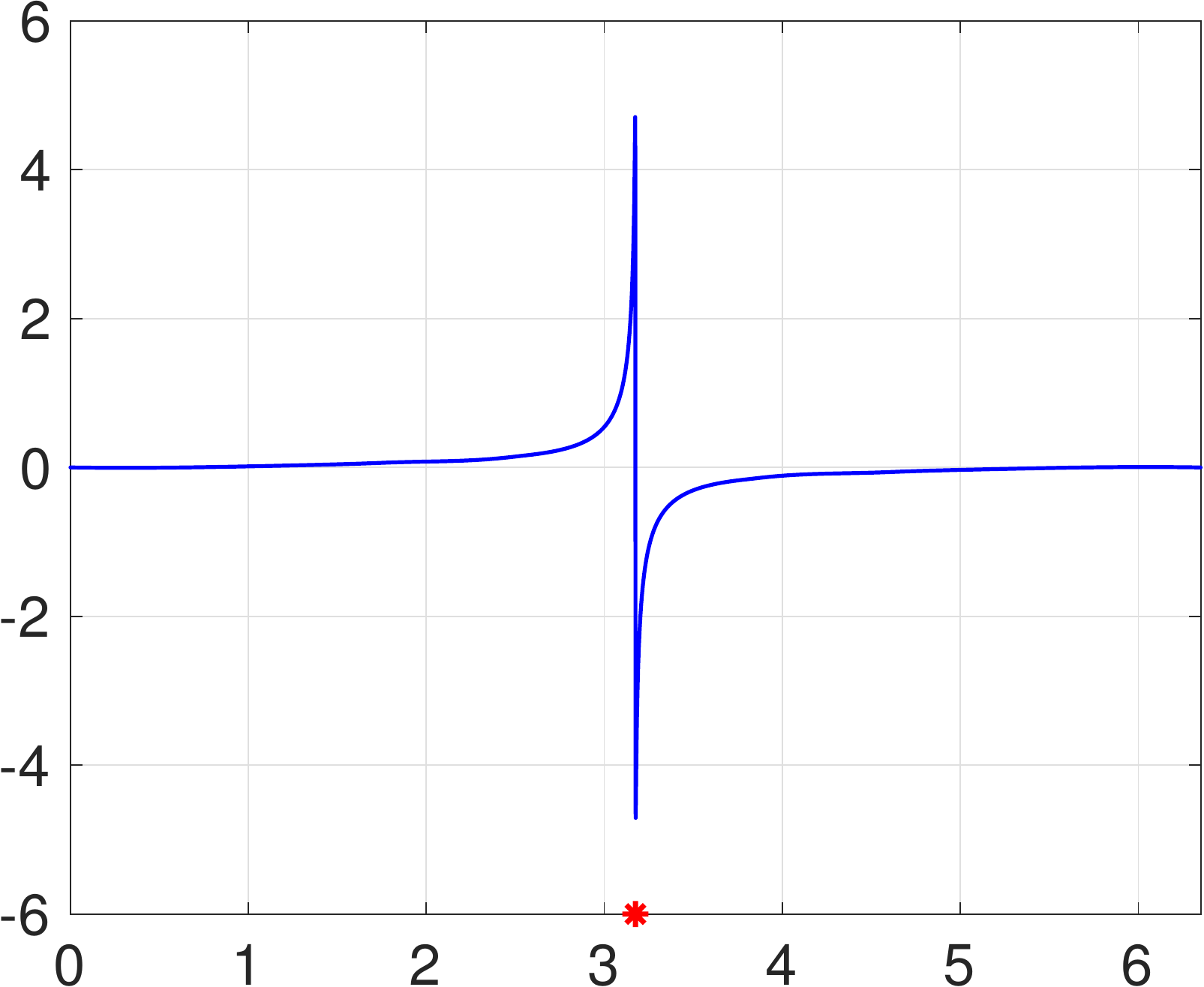}}
\subfigure[]{
\includegraphics[width=0.2\textwidth]{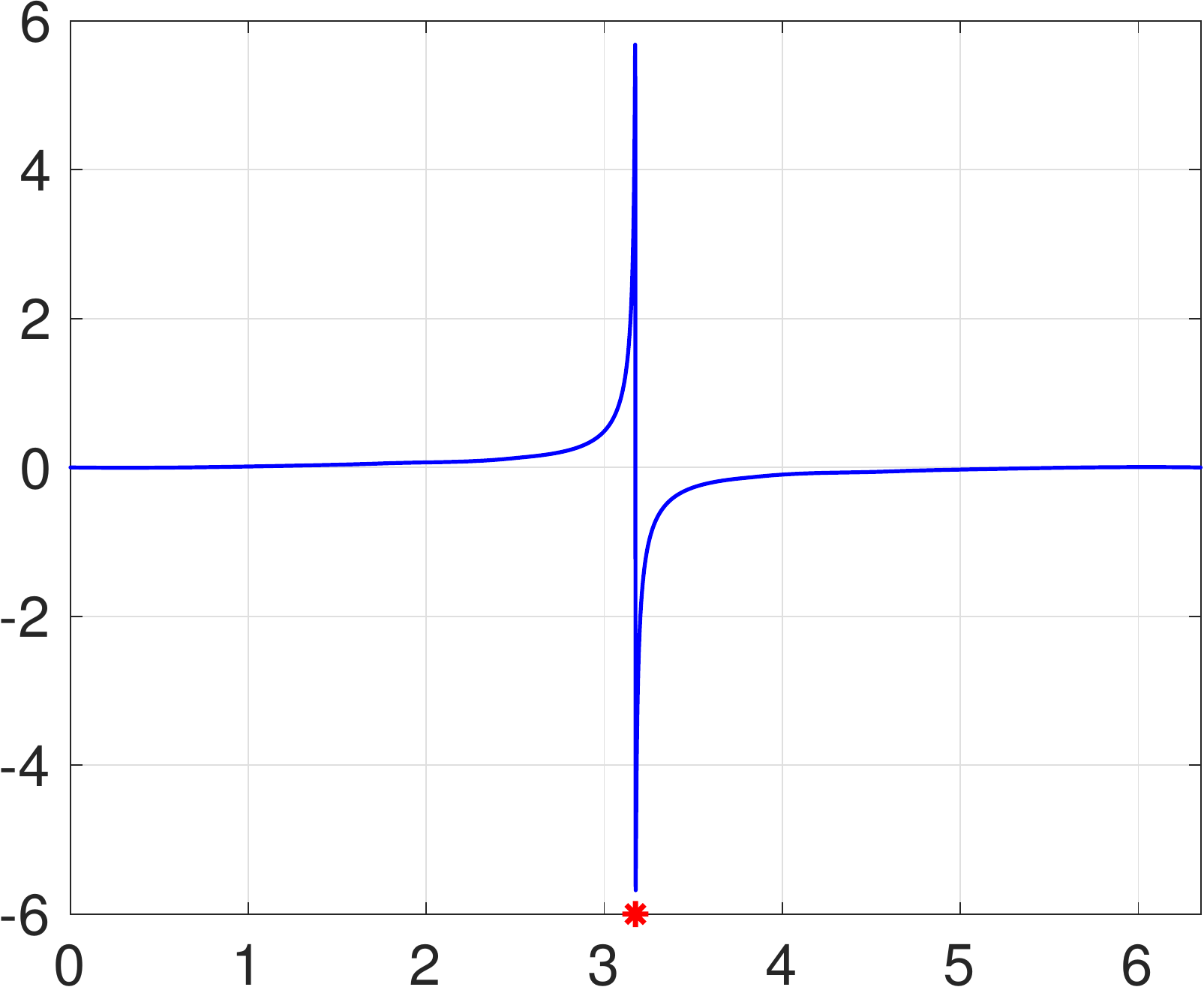}}\\
\subfigure[]{
\includegraphics[width=0.2\textwidth]{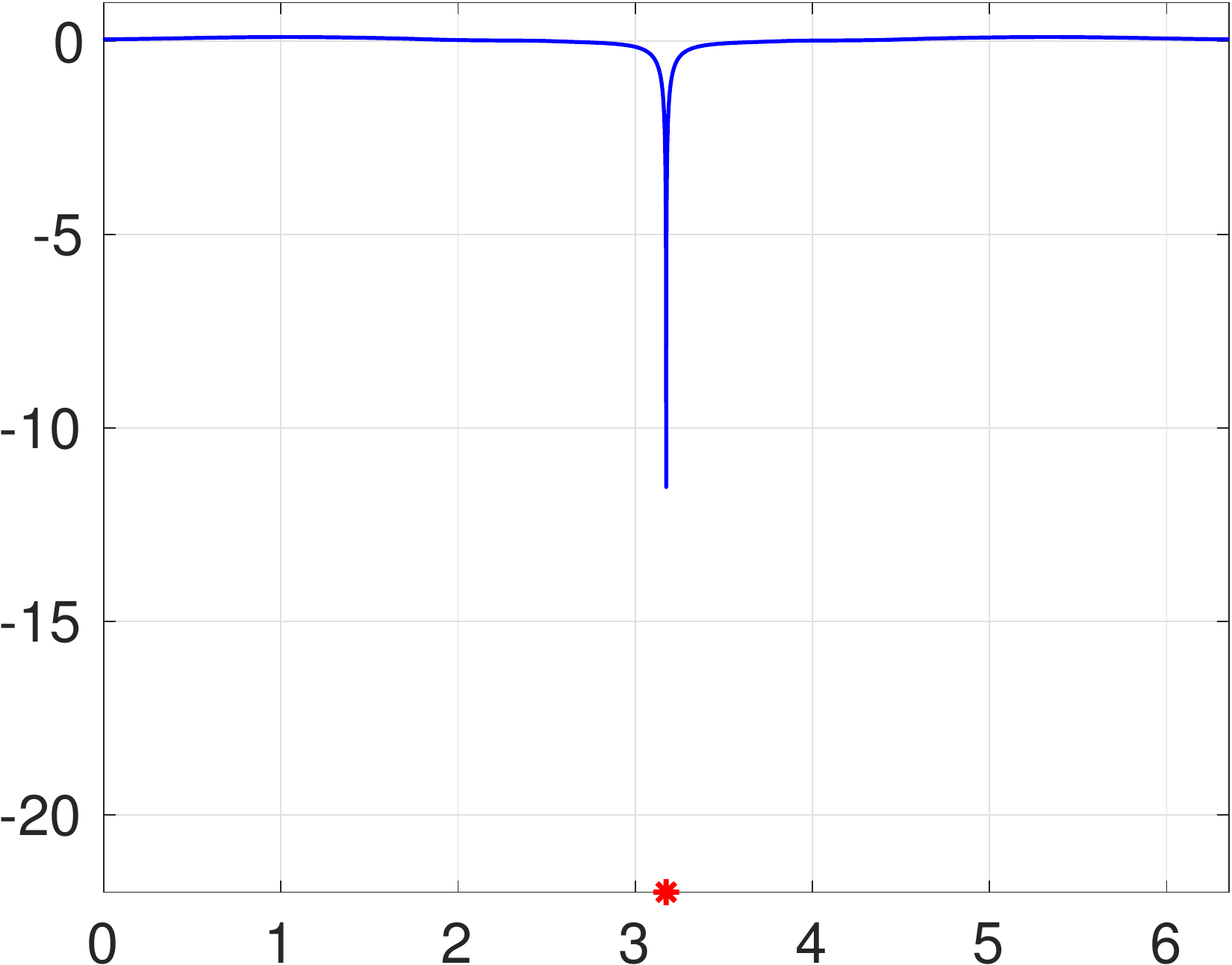}}
\subfigure[]{
\includegraphics[width=0.2\textwidth]{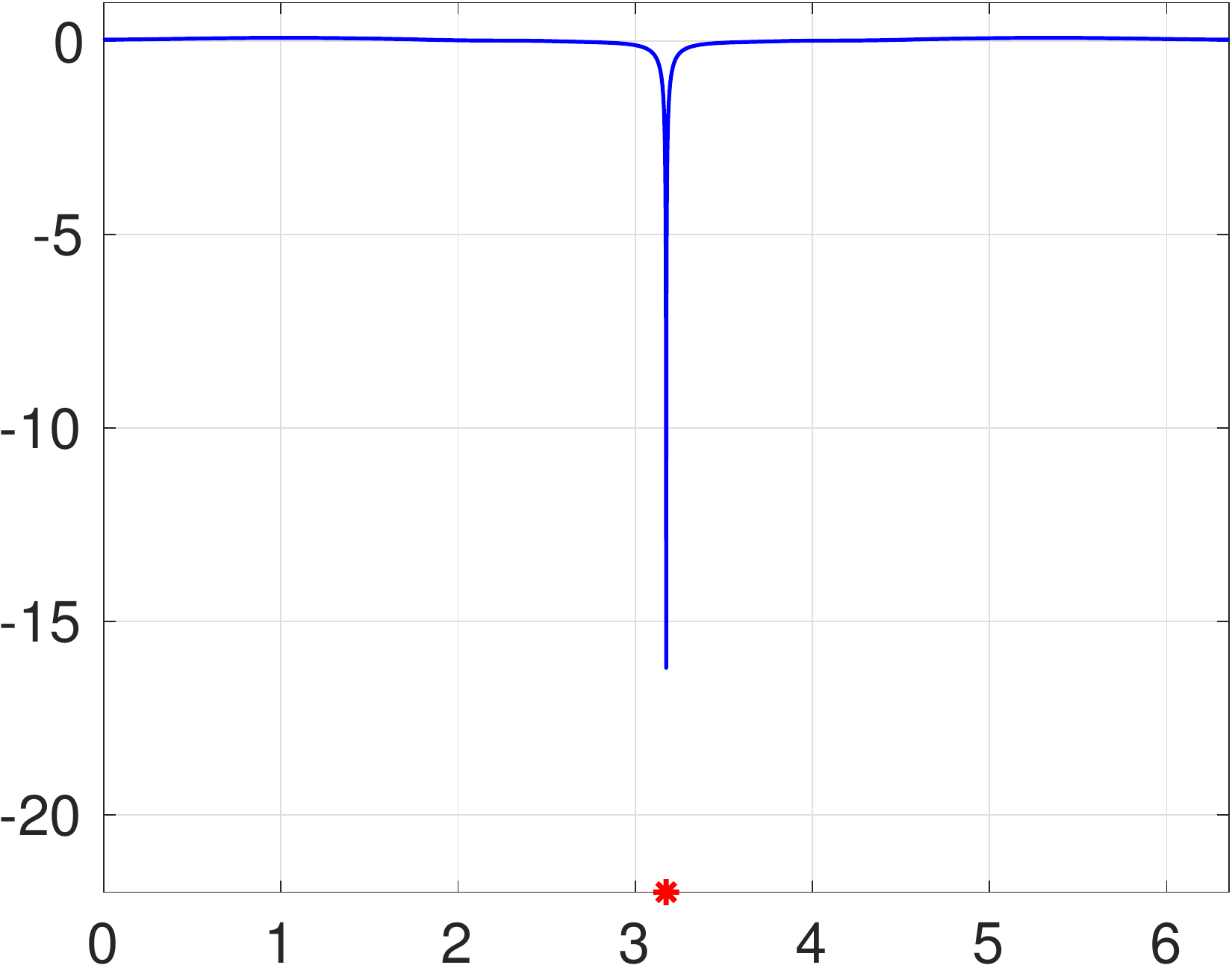}}
\subfigure[]{
\includegraphics[width=0.2\textwidth]{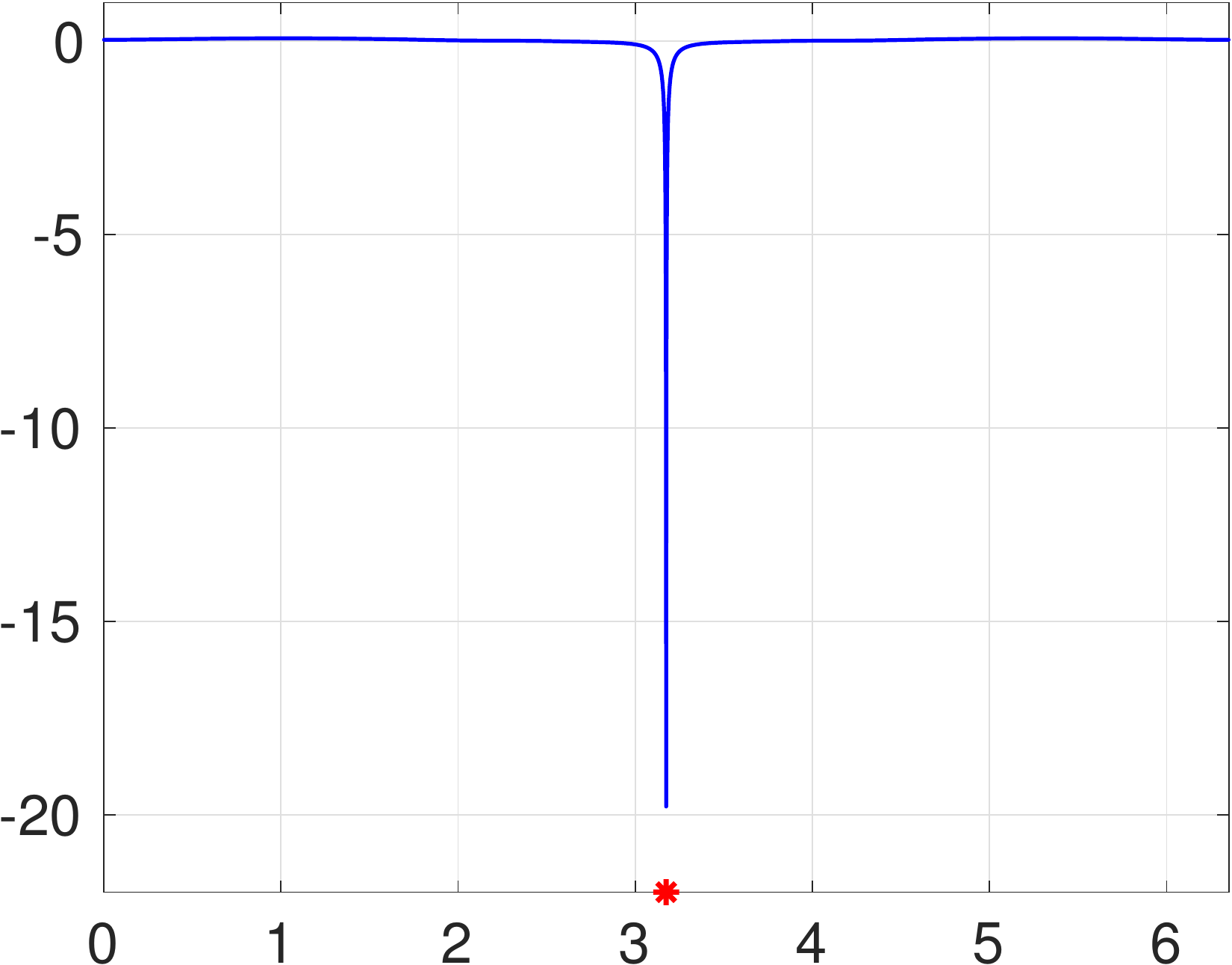}}\\
\caption{\label{fig24} (a), (b), (c).  Plotting the eigenfunctions for the positive eigenvalues $\lambda_1=0.3310$ with different maximum curvature $500$, $1000$ and $1500$. 
(d), (e), (f). The corresponding items for the negative eigenvalue $\lambda_2=-0.3310$.}
\end{figure}

%
%Fig.~\ref{fig24} shows that both the absolute value of the derivative of the eigenfunction for the positive eigenvalue and the absolute value of the eigenfunction for the negative eigenvalue at high-curvature point $x_*$ increase as the the curvature $\kappa_{x_*}$ increases, with $\kappa_{x_*}$ denoting the curvature at $x_*$. Therefore we plot the logarithm of the derivative of the absolute value of the eigenfunctions at the high-curvature point for the positive eigenvalues $\lambda_1$, $\lambda_3$ and $\lambda_5$, and the logarithm of the absolute value of the eigenfunctions at the high-curvature point for the negative eigenvalues $\lambda_2$, $\lambda_4$ and $\lambda_6$ with respect to different curvature in Fig.~\ref{fig25}. Denote by $\psi_{\max}$ the  maximum of the absolute value of the eigenfunction for the negative eigenvalue, or the maximum of absolute value of the conormal derivative of the eigenfunction for the positive eigenvalue and from Fig.~\ref{fig25} one can conclude that by regression
%\[
% \psi_{\max} \sim a \kappa_{\max}^p.
%\]
%The coefficients of the regression are shown in Table.~\ref{tab5}.
%
\begin{figure}
\includegraphics[width=0.2\textwidth] {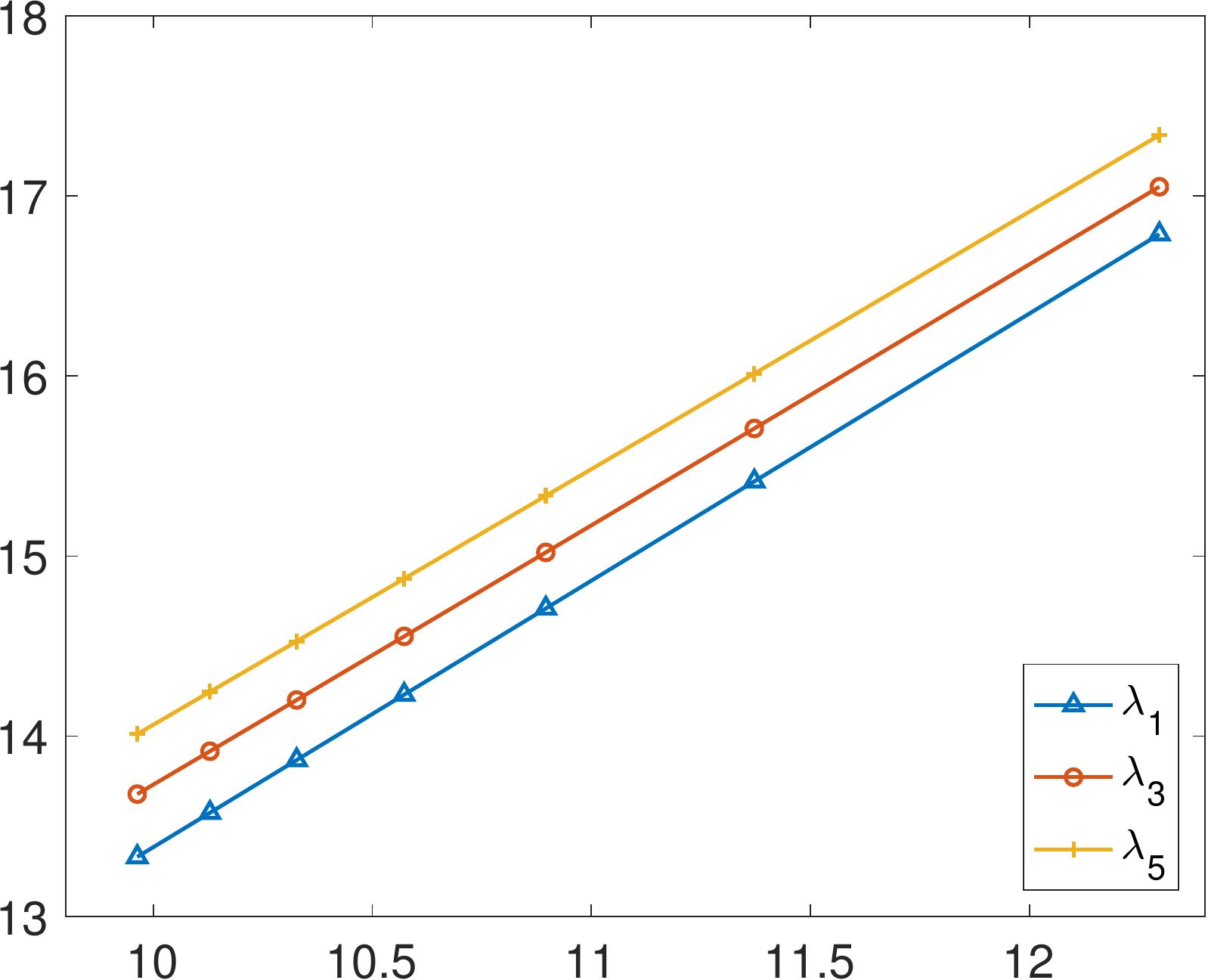}
\hspace{0.8cm}
\includegraphics[width=0.2\textwidth] {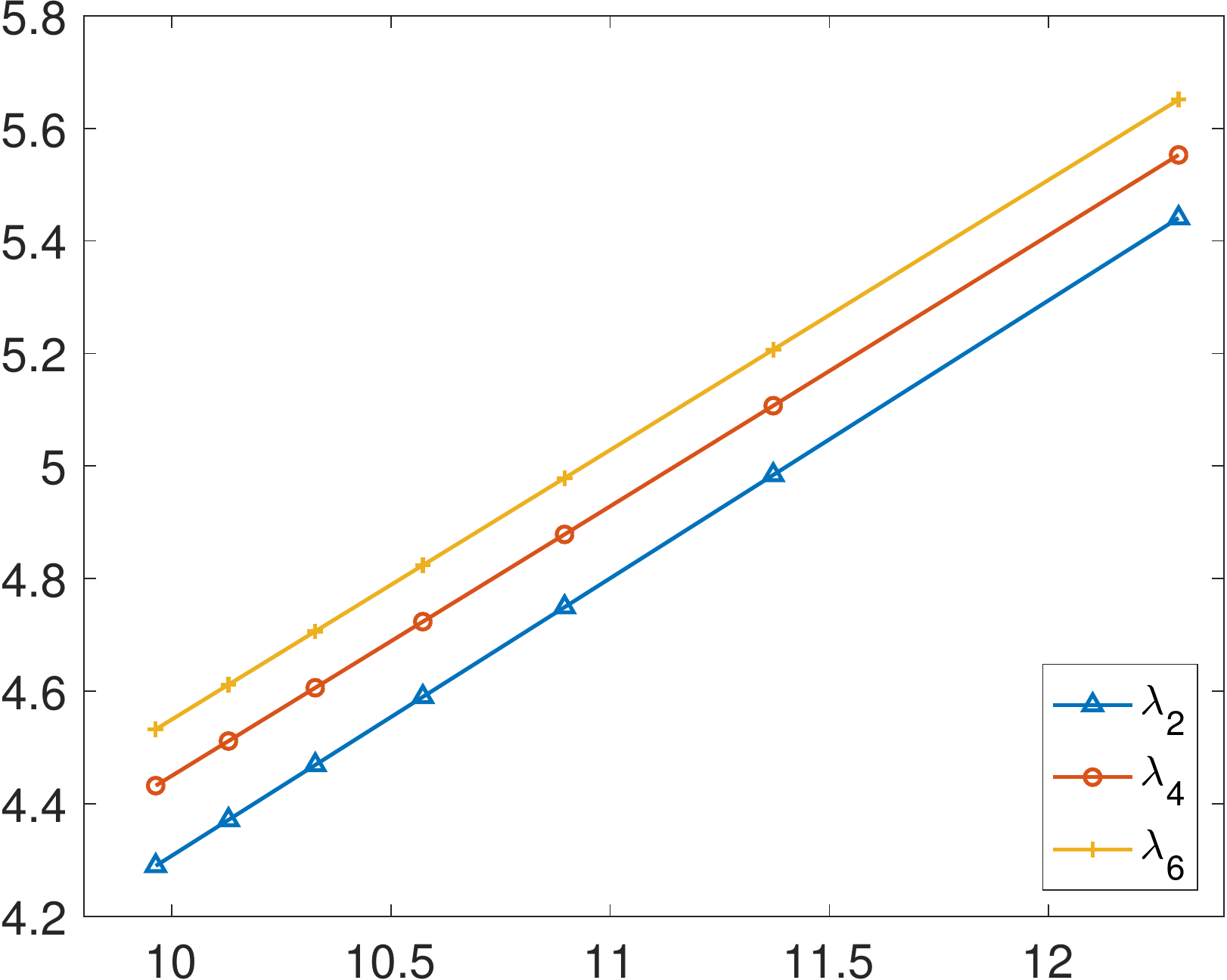}
\caption{\label{fig25} The logarithm of the absolute value of the conormal derivative of the eigenfunction at the high-curvature point for the positive eigenvalues $\lambda_1$, $\lambda_3$ and $\lambda_5$ and the logarithm of the absolute value of the eigenfunction at the high-curvature point is the largest for the negative eigenvalues $\lambda_2$, $\lambda_4$ and $\lambda_6$ with respect to different curvature.}
\end{figure}

We next investigate the blow-up rate of the NP eigenfunction or its conormal derivative with respect to the curvature. Therefore we plot the logarithm of the derivative of the absolute value of the eigenfunctions at the high-curvature point for the positive eigenvalues $\lambda_1$, $\lambda_3$ and $\lambda_5$, and the logarithm of the absolute value of the eigenfunctions at the high-curvature point for the negative eigenvalues $\lambda_2$, $\lambda_4$ and $\lambda_6$ with respect to different curvature in Fig.~\ref{fig25}. It turns out that blow-up rate also follows the 
rule in \eqref{eq:growthrate}. By regression, we numerically determine the corresponding parameters for those different eigenvalues 
in \eqref{eq:egg3_2}, and they are listed in Table~\ref{tab5}. 

\begin{table}[t]
  \centering
  \subtable[]{
    \centering
    \begin{tabular}{cccc}
      \toprule
      & $\lambda_1$ & $\lambda_3$ & $\lambda_5$ \\
      \midrule
      $p$ & 1.4814 & 1.4454 & 1.4253 \\[5pt]
      $\ln(a)$ & -1.4367 & -0.7261 & -0.1924 \\
      \bottomrule
    \end{tabular}}
  %  \caption{}
  \hspace{1.5cm}
\subtable[]{
    \centering
    \begin{tabular}{cccc}
      \toprule
      & $\lambda_2$ & $\lambda_4$ & $\lambda_6$ \\
      \midrule
      $p$ & 0.4934 & 0.4803 & 0.4795 \\[5pt]
      $\ln(a)$ & -0.6267 & -0.3545 & -0.2455 \\
      \bottomrule
    \end{tabular}
    }
 %\end{subtable}
  \caption{The coefficients of the regression; (a) $\lambda_j, j=1, 3, 5$; (b) $\lambda_j, j=2,4,6$.}
  \label{tab5}
\end{table}

%Here we remark that the eigenvalue $\lambda_5$ and $\lambda_6$ are simple eigenvalues and 
%\[
% \lambda_5=0.2347 \quad \mbox{and} \quad \lambda_6=-0.2347.
%\]
%

%\subsubsection{Example: two high-curvature points}

\subsection{A concave 2-symmetric domain}

In this subsection, we consider a concave 2-symmetric domain as shown in Fig.~\ref{fig26}, which possesses two high-curvature points that are
denoted by $x_*$ and $x_o$. The largest curvature is
\begin{equation}\label{eq:egg2}
 \kappa_{\max}=\kappa_{x_*}=500.
\end{equation}
\begin{figure}[t]
\includegraphics[width=2.5cm] {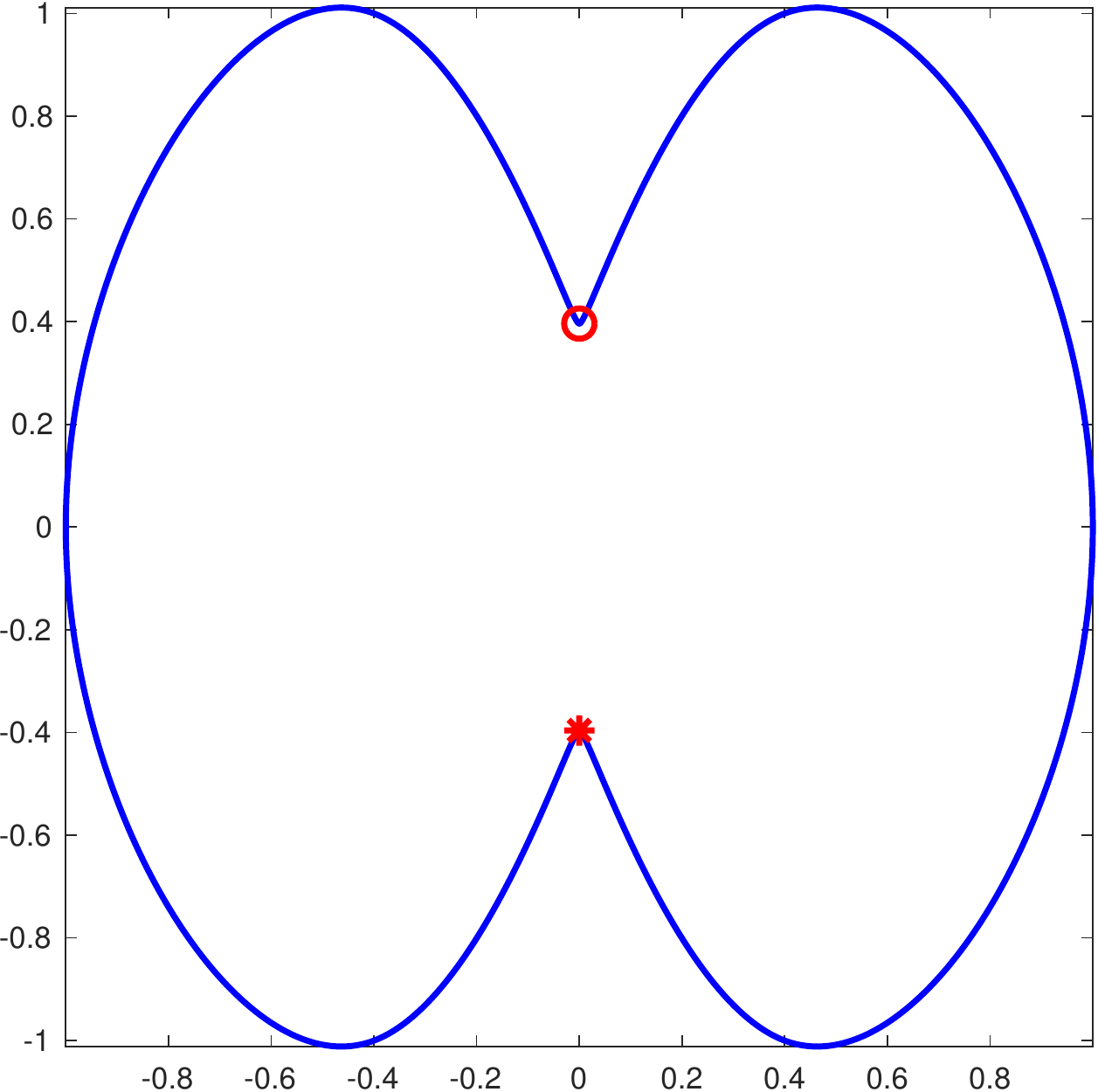}
\caption{\label{fig26} A concave 2-symmetric domain.}
\end{figure}
The first five largest NP eigenvalues (in terms of the absolute value) are
\begin{equation}\label{eq:egg3}
\begin{split}
& \lambda_0=0.5, \quad  \lambda_1=0.3676,\quad \lambda_2=-0.3676,\quad\lambda_3=0.3303,\\
& \quad \lambda_4=-0.3303, \quad \lambda_5=0.2347 \quad \lambda_6=-0.2347.
\end{split}
\end{equation}

%For the eigenvalue $\lambda_0=0.5$, the corresponding eigenfunction, single layer potential and the single layer potential around the high-curvature point are different from that of other positive eigenvalues, and we plot that in Fig.~\ref{figao2}.
%
%\begin{figure}
%\centering
%\subfigure[]{
%\includegraphics[width=0.2\textwidth]{aoei2_0.pdf}}
%\subfigure[]{
%\includegraphics[width=0.2\textwidth]{aopo2_0.pdf}}
%\subfigure[]{
%\includegraphics[width=0.2\textwidth]{aopos2_0.pdf}}\\
%\caption{\label{figao2} $a$, $b$ and $c$ plot the eigenfunction with respect to the arc length for the eigenvalues, the single layer potential and the single layer potential around the high-curvature point for the eigenvalues $\lambda_0=0.5$.}
%\end{figure}

Fig.~\ref{fig27} plots the eigenfunctions, their conormal derivatives and the associated single-layer potentials, respectively, associated with the eigenvalues
 $\lambda_1=0.3676$ and $\lambda_3=0.3303$. The numerical results clearly support our earlier assertion about the NP eigenfunctions associated with simple
 positive eigenvalues. 
 
\begin{figure}[t]
\centering
\subfigure[]{
\includegraphics[width=0.2\textwidth]{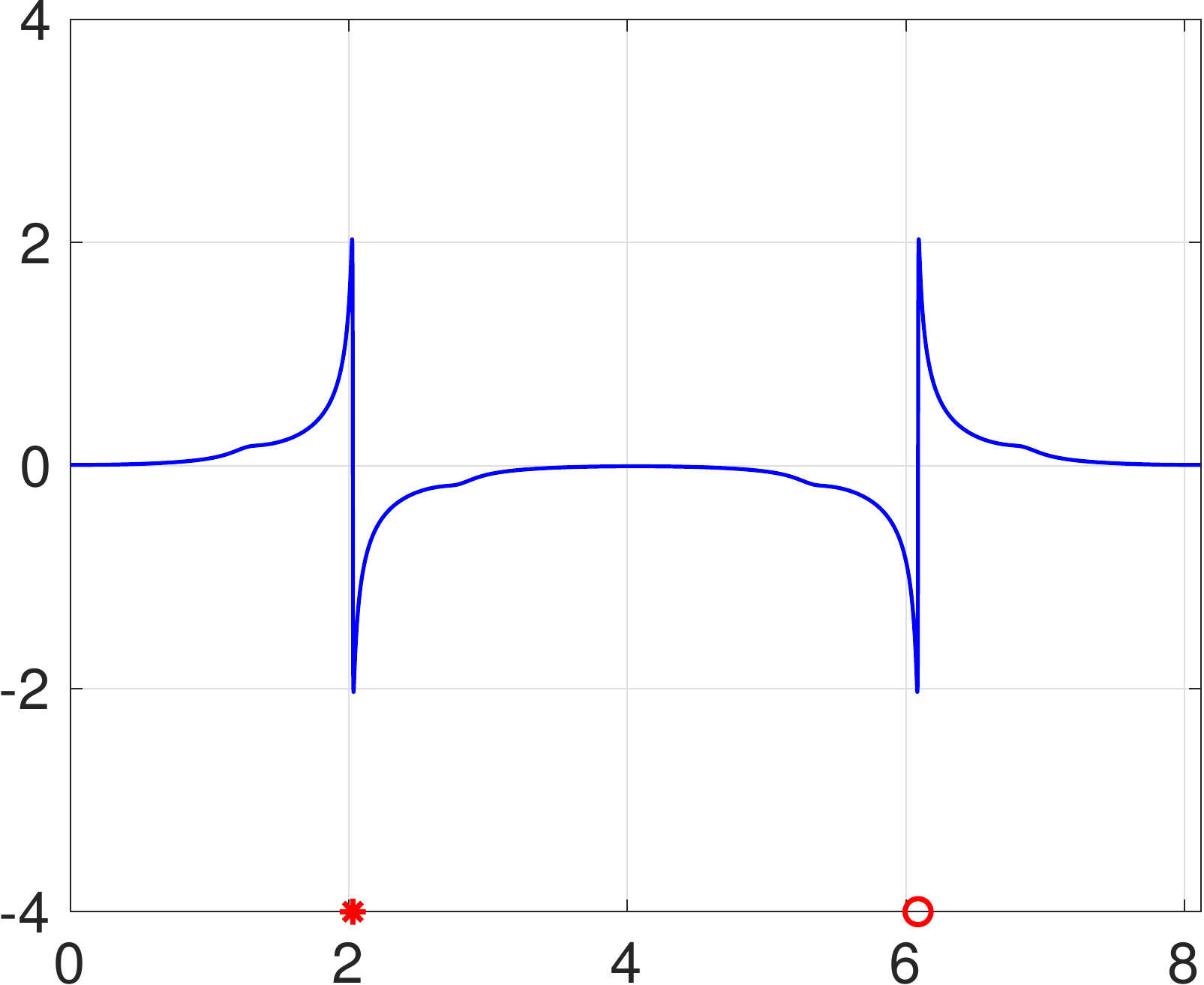}}
\subfigure[]{
\includegraphics[width=0.2\textwidth]{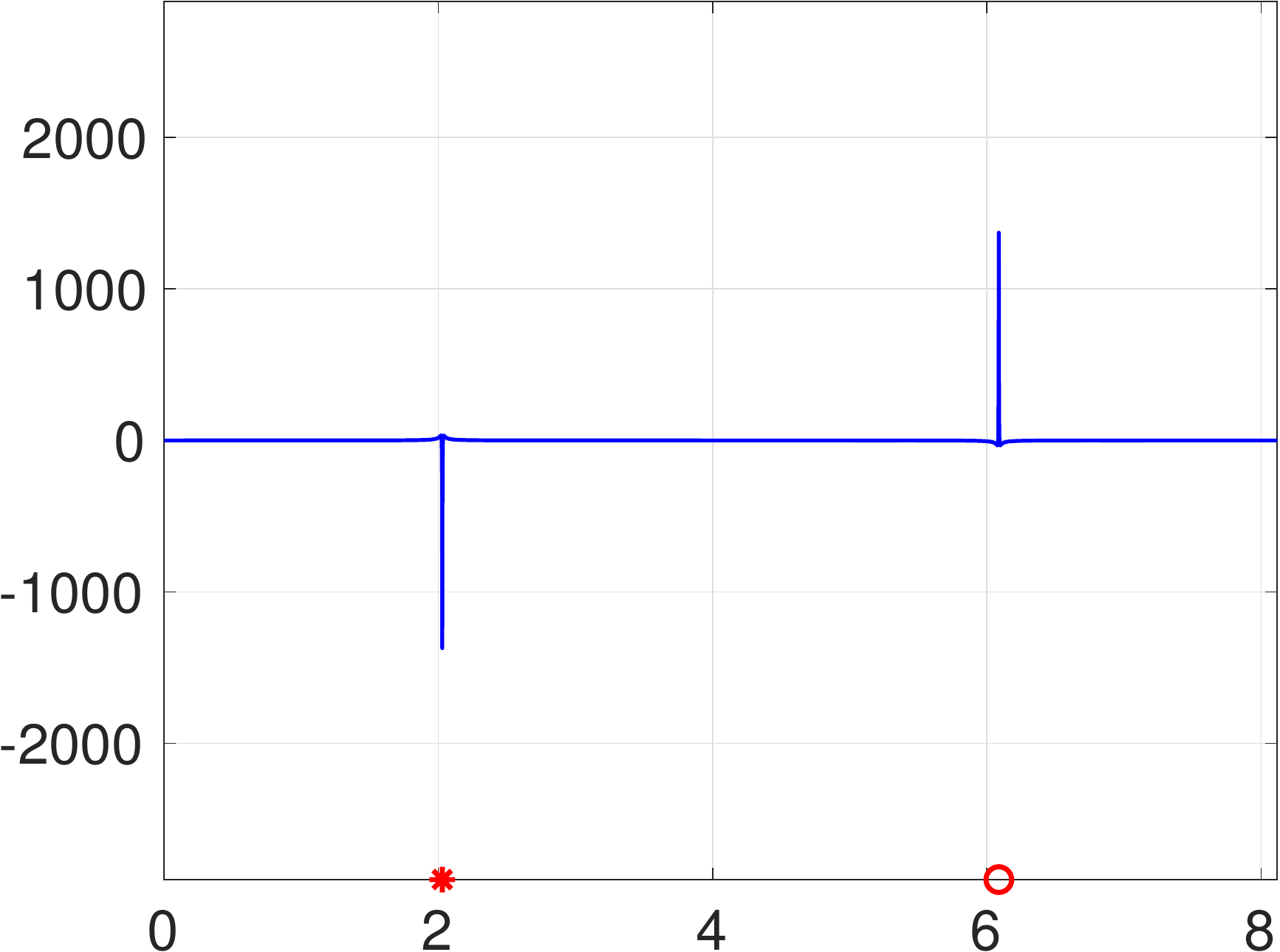}}
\subfigure[]{
\includegraphics[width=0.2\textwidth]{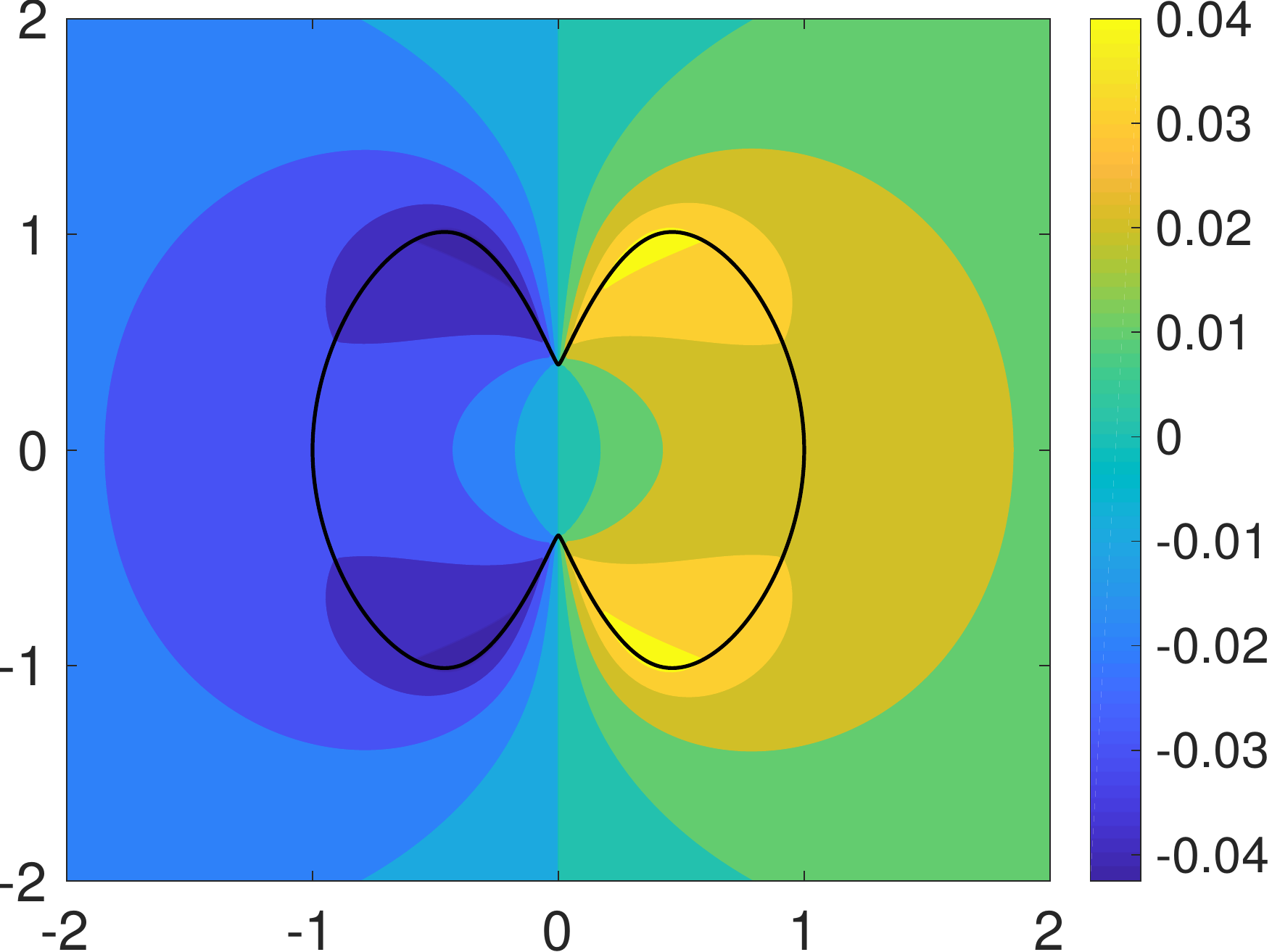}}
\subfigure[]{
\includegraphics[width=0.2\textwidth]{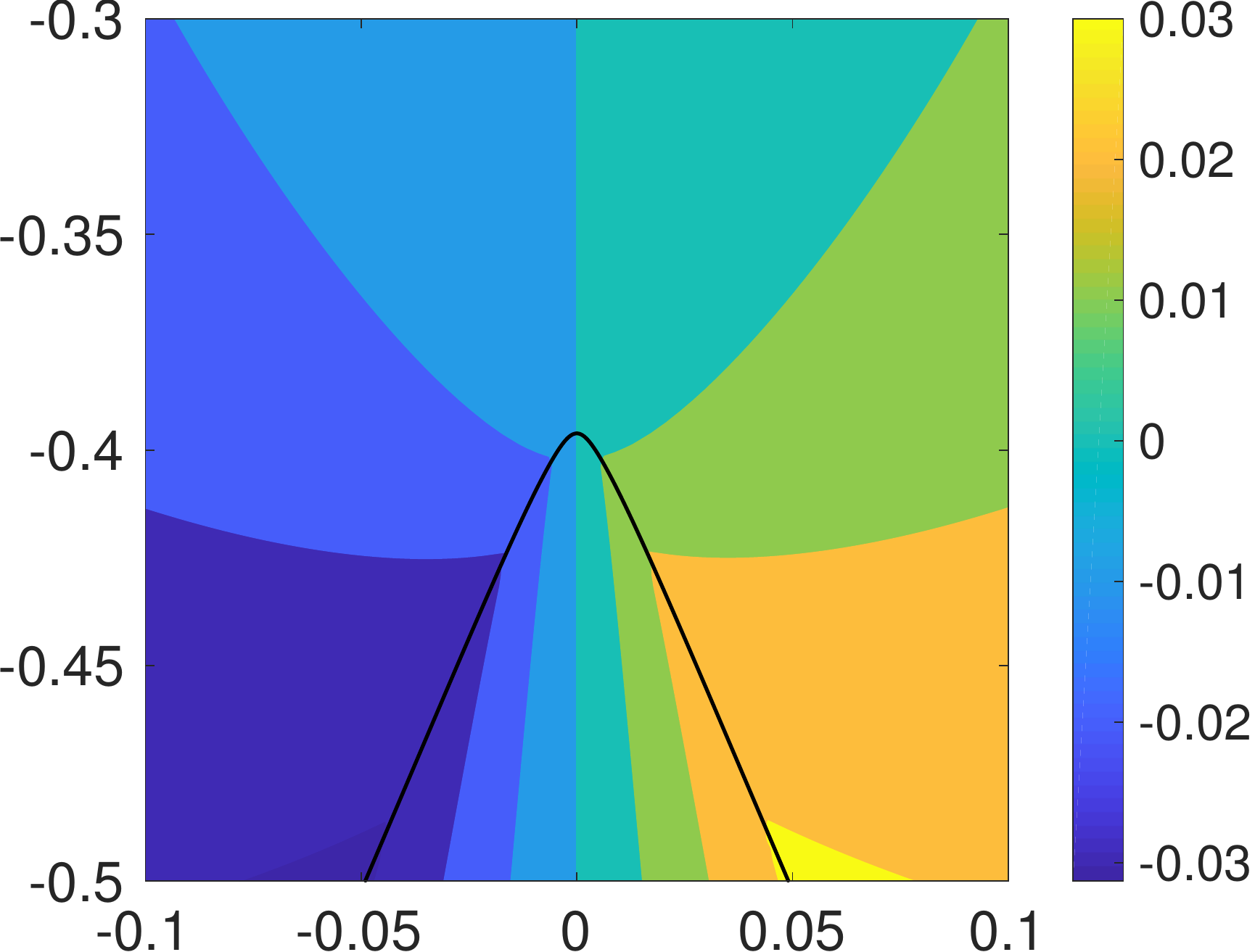}}\\
\subfigure[]{
\includegraphics[width=0.2\textwidth]{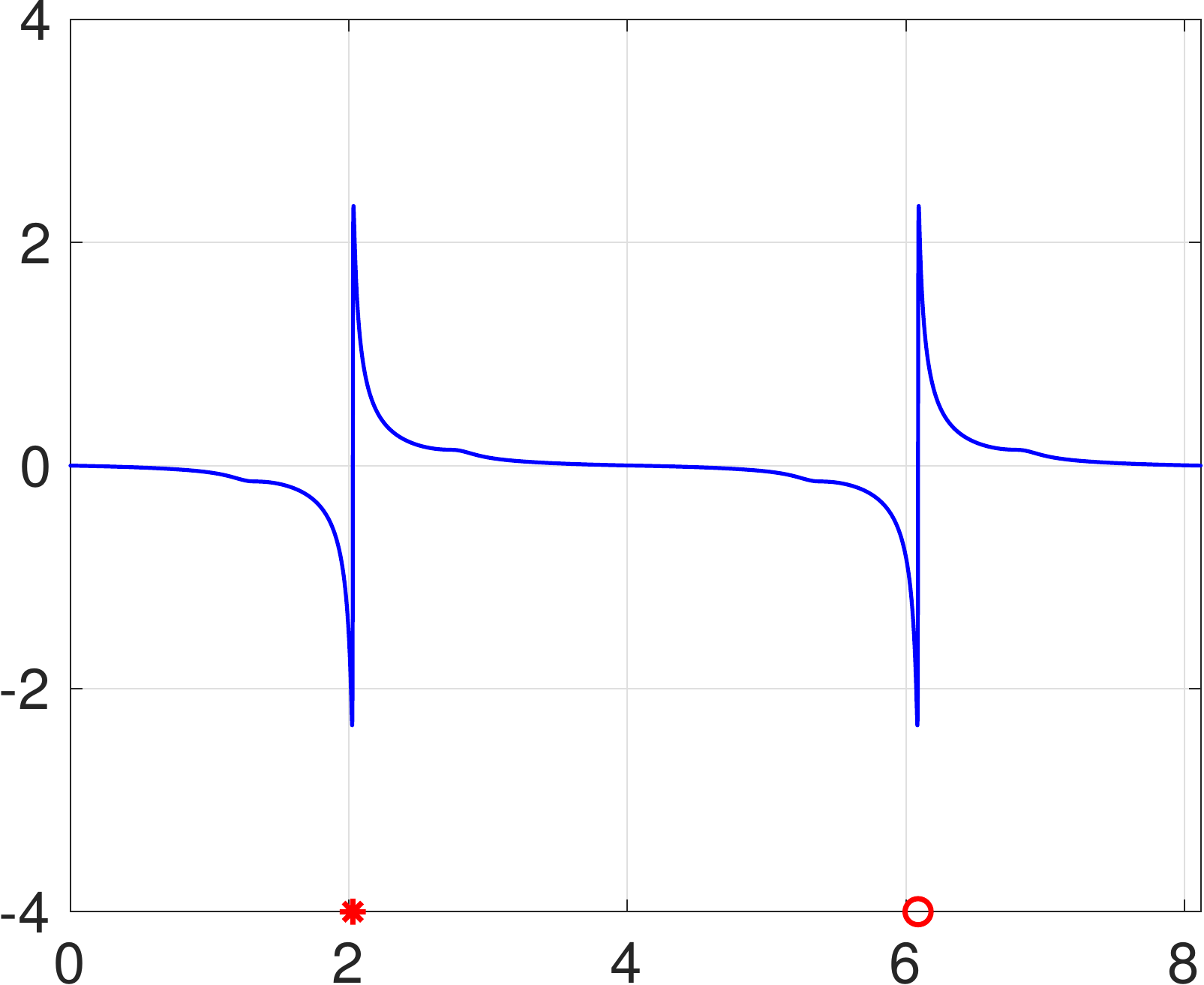}}
\subfigure[]{
\includegraphics[width=0.2\textwidth]{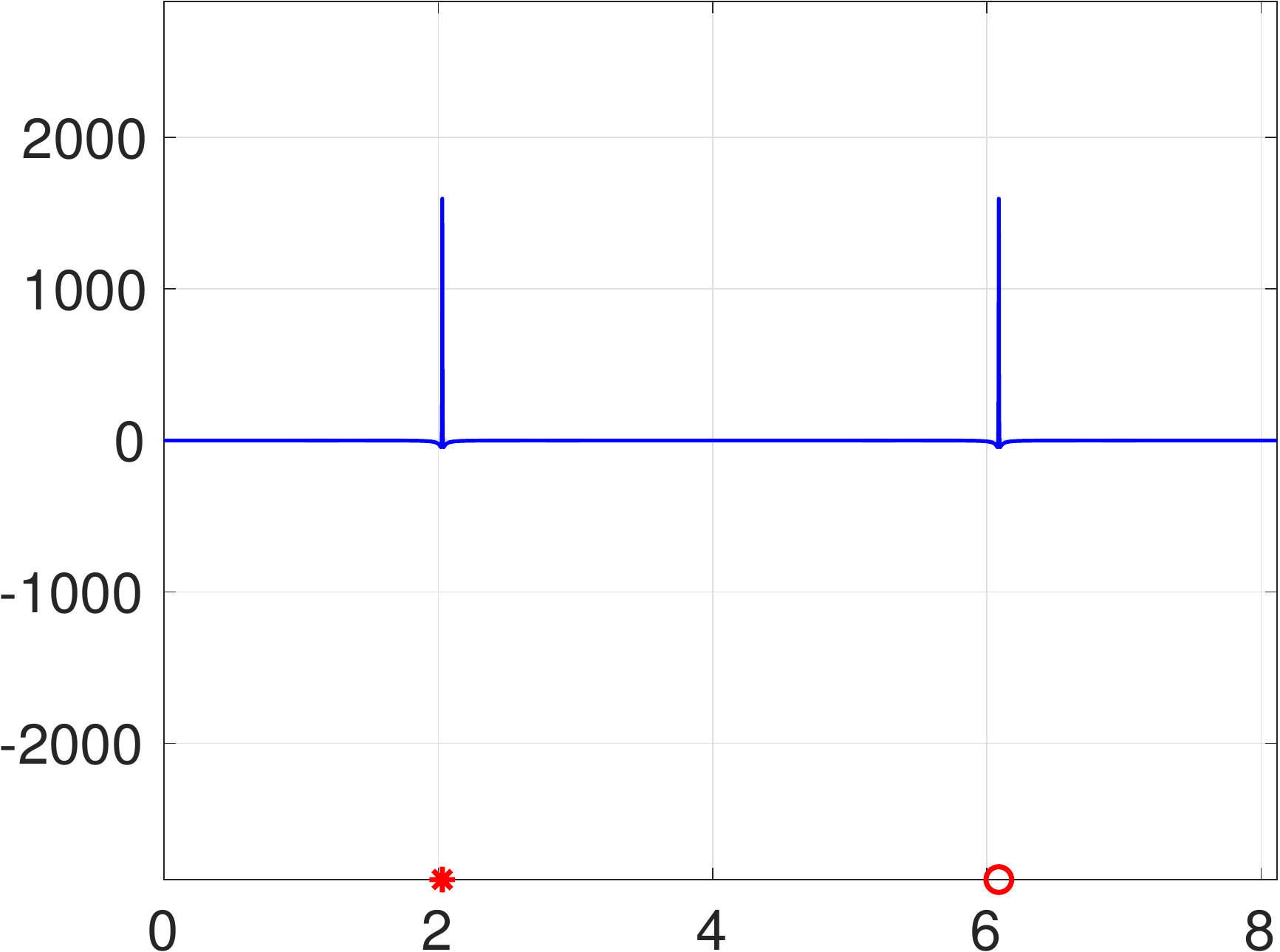}}
\subfigure[]{
\includegraphics[width=0.2\textwidth]{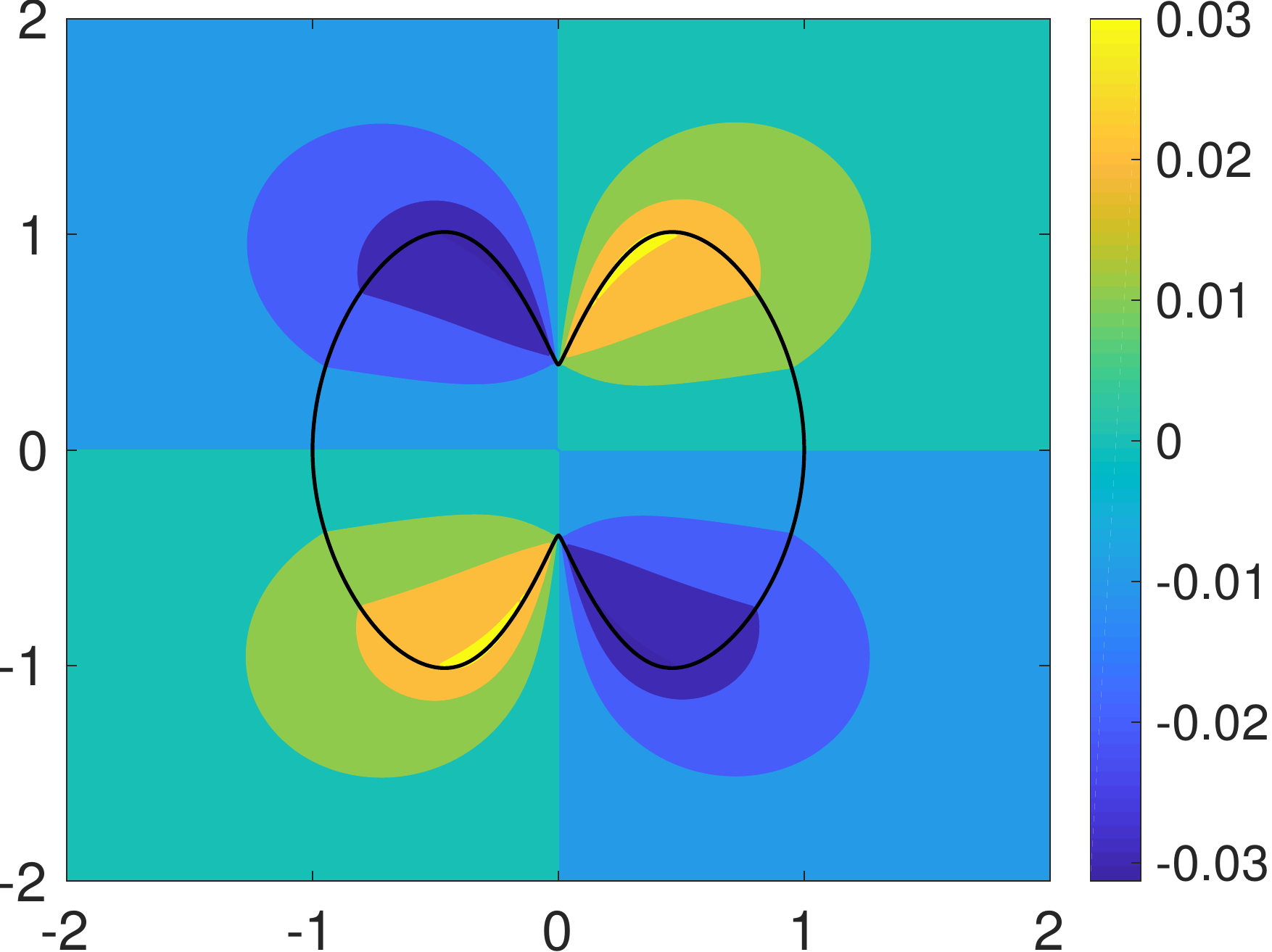}}
\subfigure[]{
\includegraphics[width=0.2\textwidth]{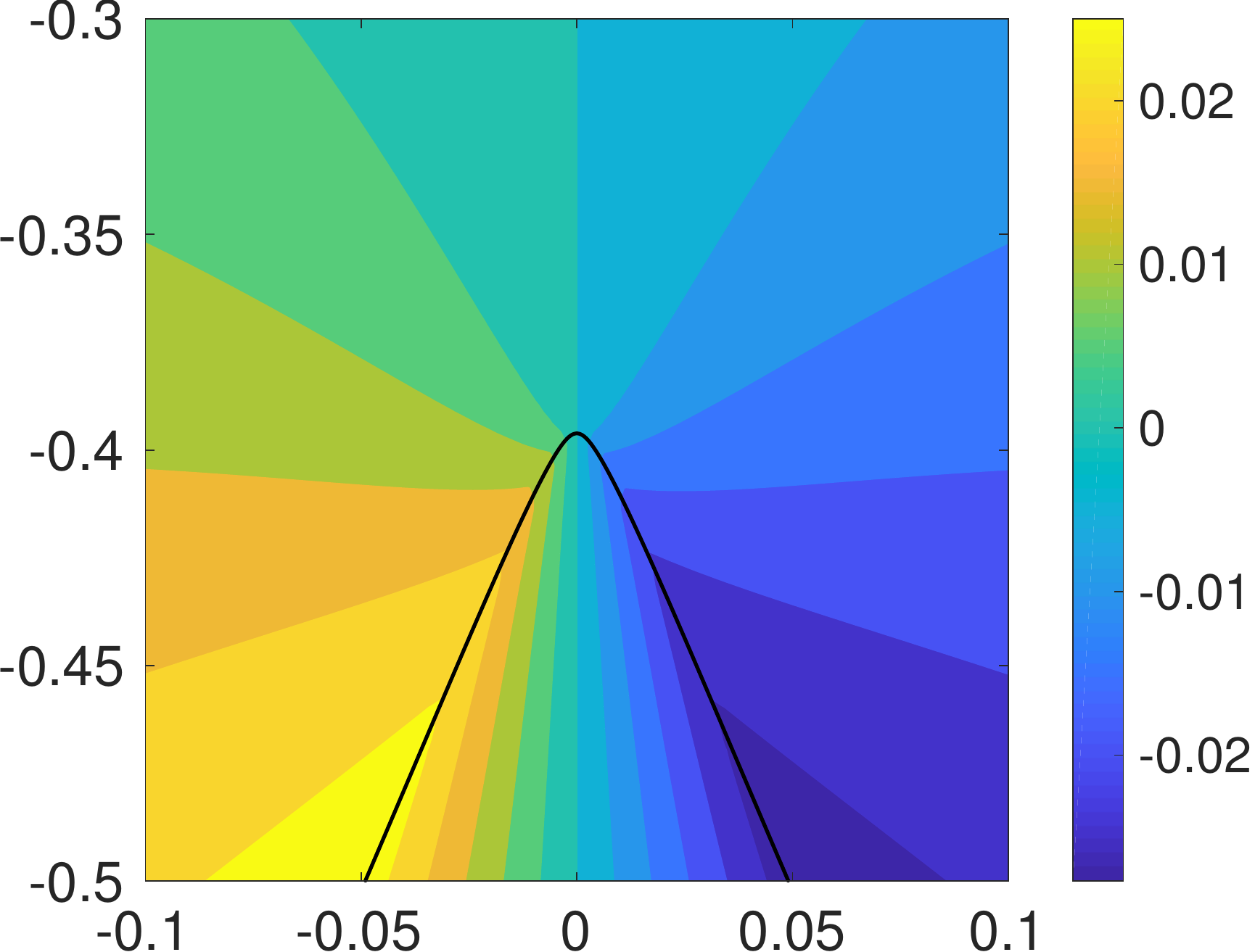}}\\
\caption{\label{fig27} (a), (b), (c), (d). The eigenfunction, its conormal derivative, and the corresponding
single-layer potential associated with $\lambda_1=0.3676$; (e), (f), (g), (h). The corresponding items
associated with $\lambda_3=0.3303$.}
\end{figure}

Fig.~\ref{fig28} plots the eigenfunctions, their conormal derivatives and the associated single-layer potentials, respectively, associated with the eigenvalues
$\lambda_2=-0.3676$ and $\lambda_4=-0.3303$. The numerical results clearly support our earlier assertion about the NP eigenfunctions associated with simple
 negative eigenvalues. 

%Fig.~\ref{fig28} plots the eigenfunctions with respect to the arc length, the corresponding single layer potentials and the single 
%layer potentials around the high-curvature point $x_*$ for the negative eigenvalues 
%$\lambda_2=-0.3676$ and $\lambda_4=-0.3303$. In Fig.~\ref{fig28}, $a,d$ show that the eigenfunctions 
%blow up at the high-curvature points $x_*$ and $x_o$ for the negative eigenvalue. $b,c,e,f$ show that the 
%associated single layer potentials blow up at the high-curvature points $x_*$ and $x_o$ for the negative eigenvalue.

\begin{figure}
\centering
\subfigure[]{
\includegraphics[width=0.2\textwidth]{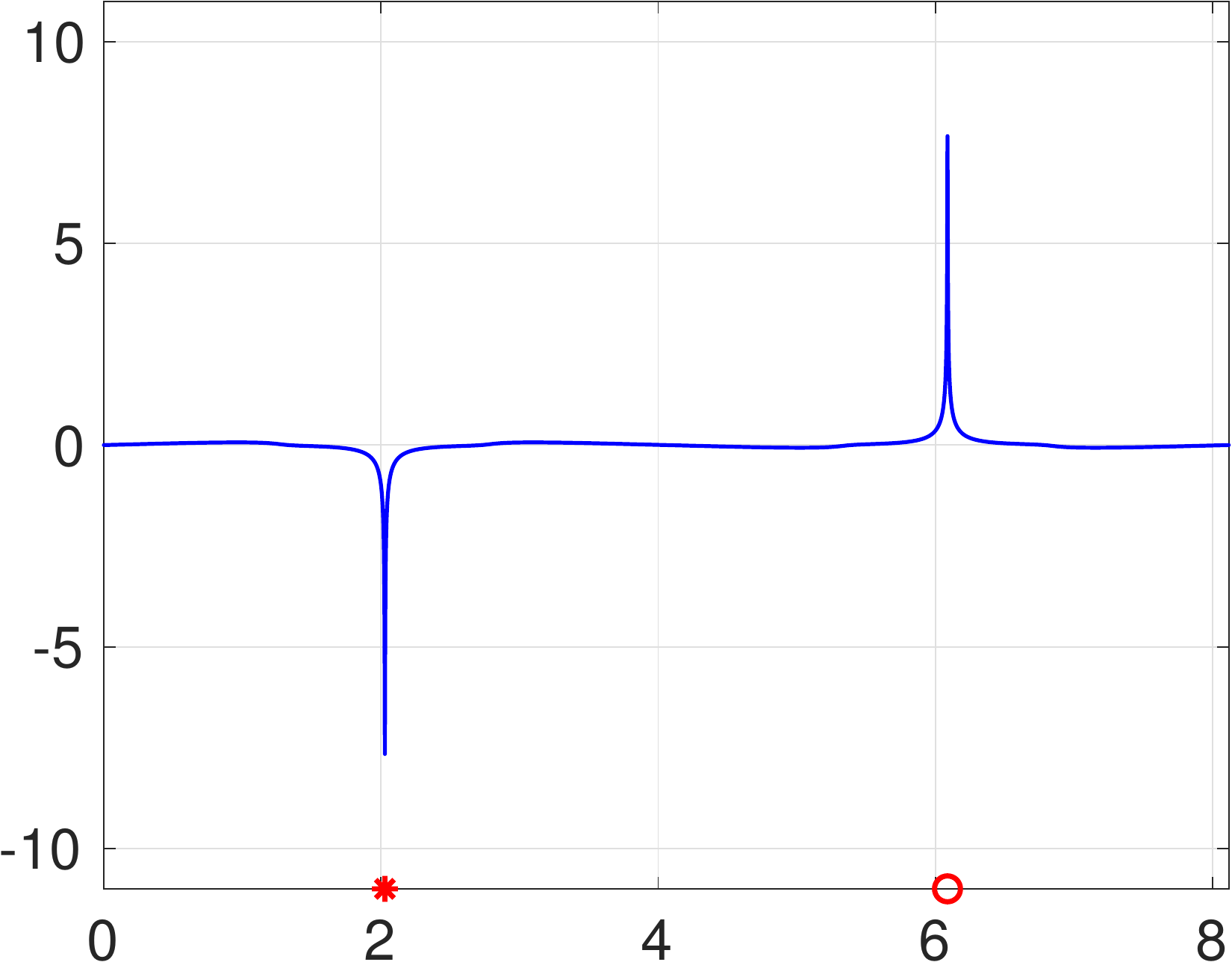}}
\subfigure[]{
\includegraphics[width=0.2\textwidth]{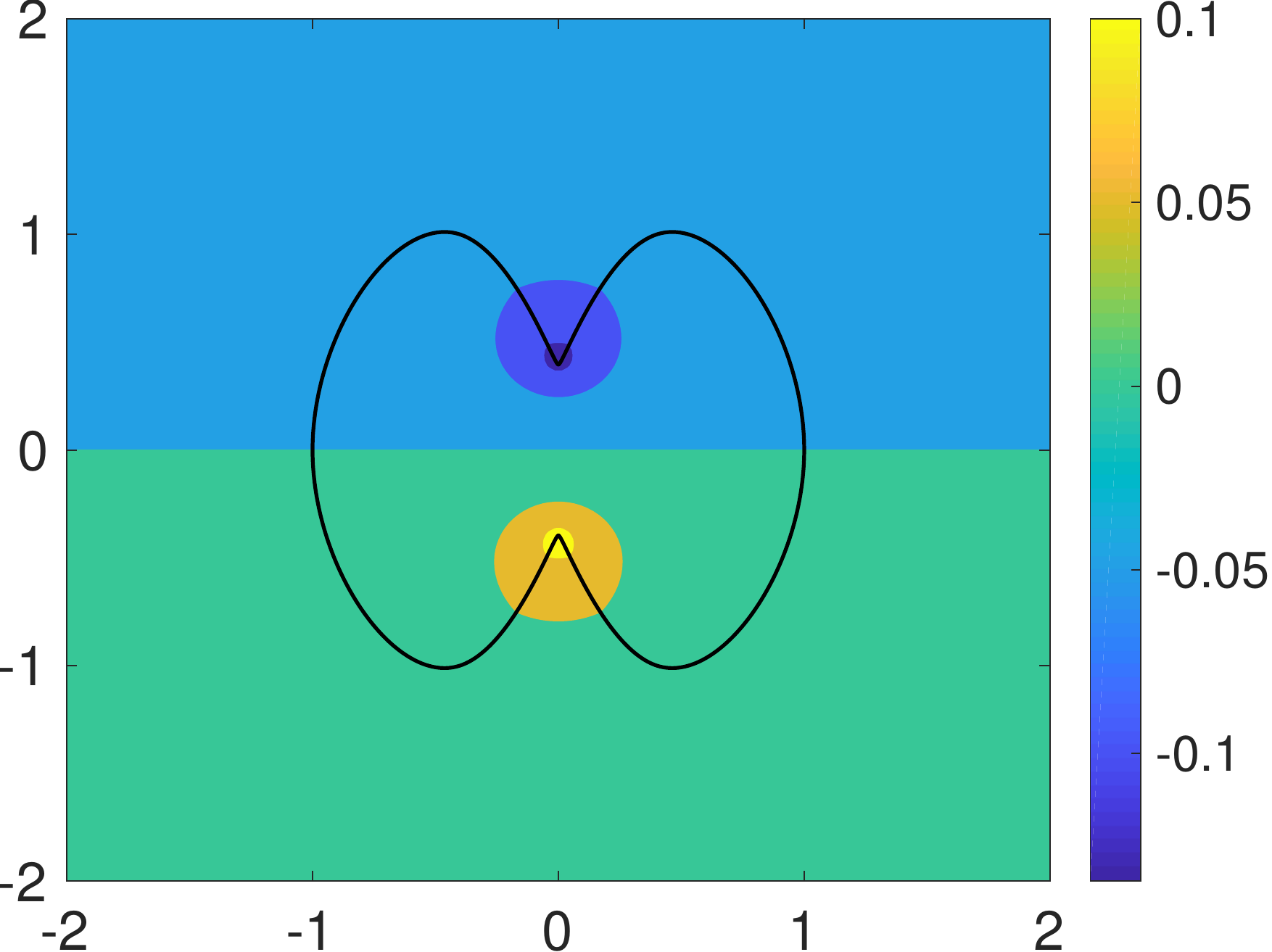}}
\subfigure[]{
\includegraphics[width=0.2\textwidth]{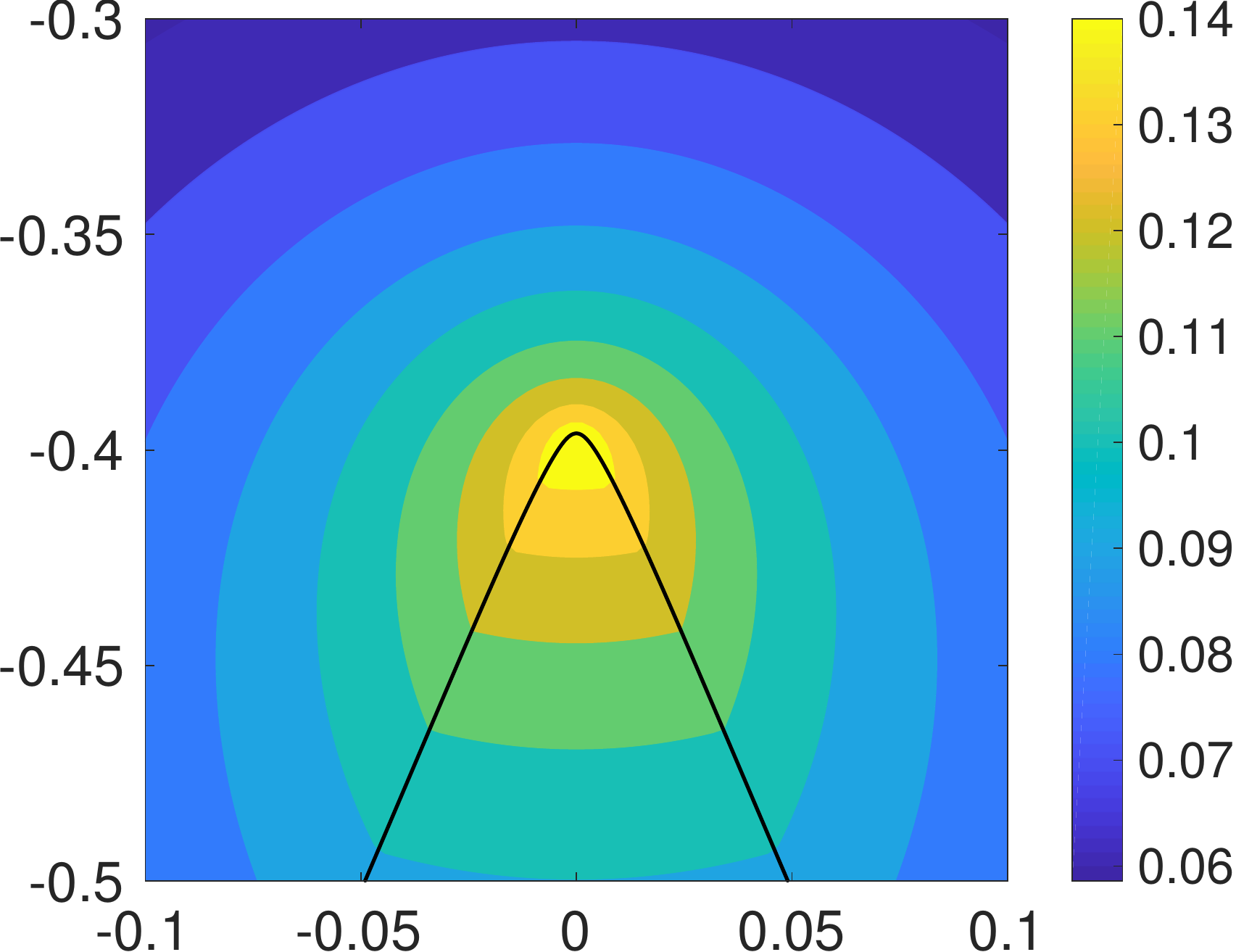}}\\
\subfigure[]{
\includegraphics[width=0.2\textwidth]{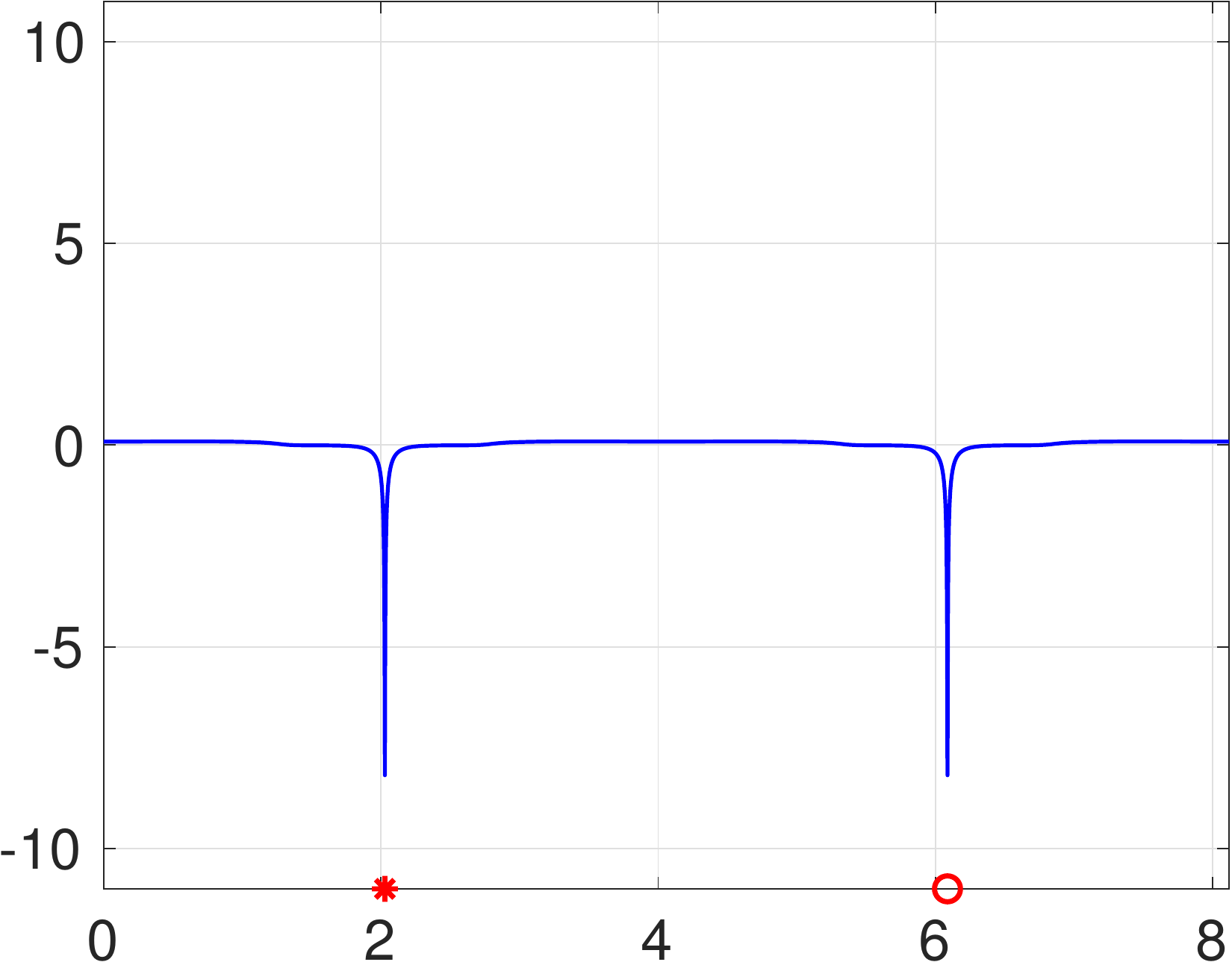}}
\subfigure[]{
\includegraphics[width=0.2\textwidth]{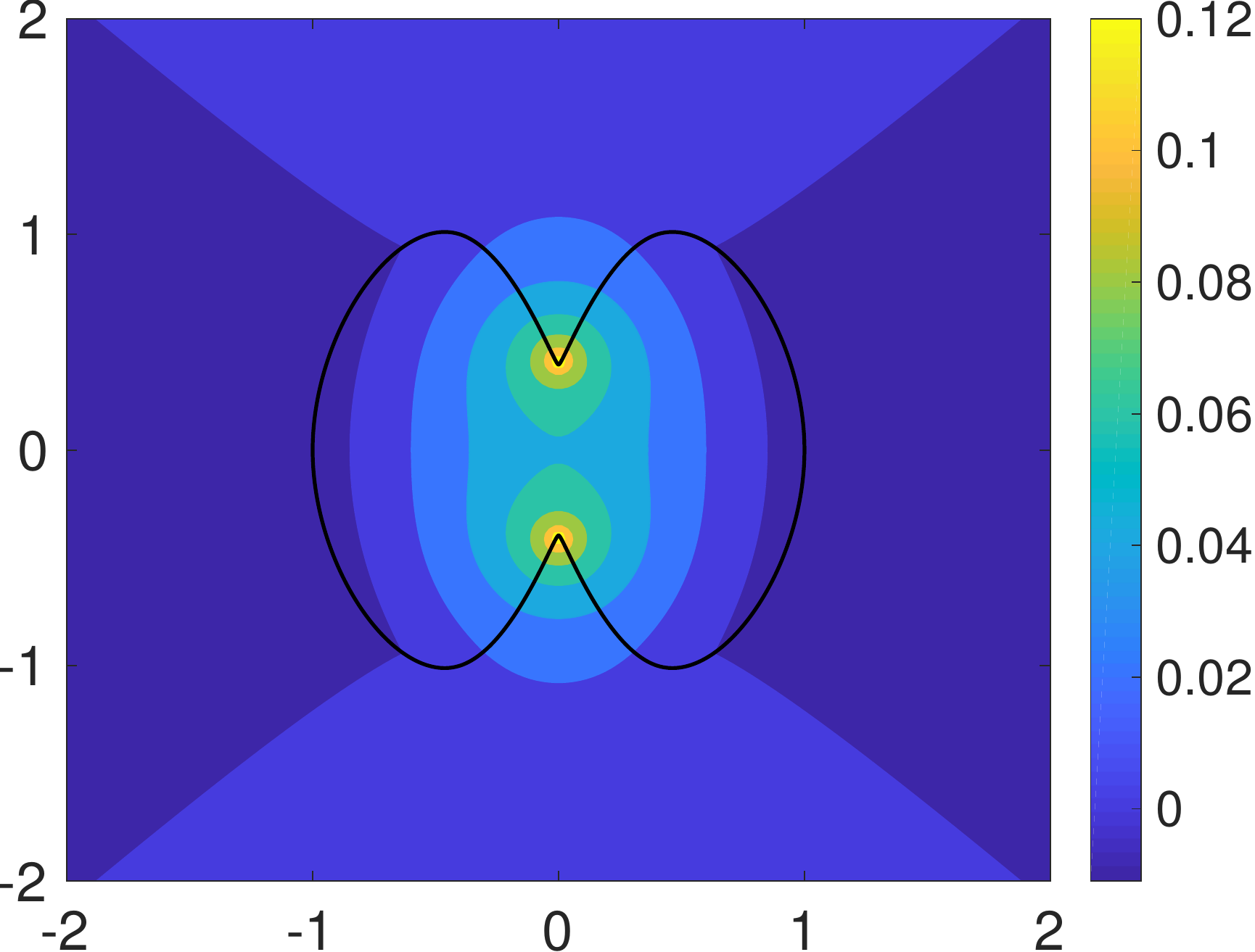}}
\subfigure[]{
\includegraphics[width=0.2\textwidth]{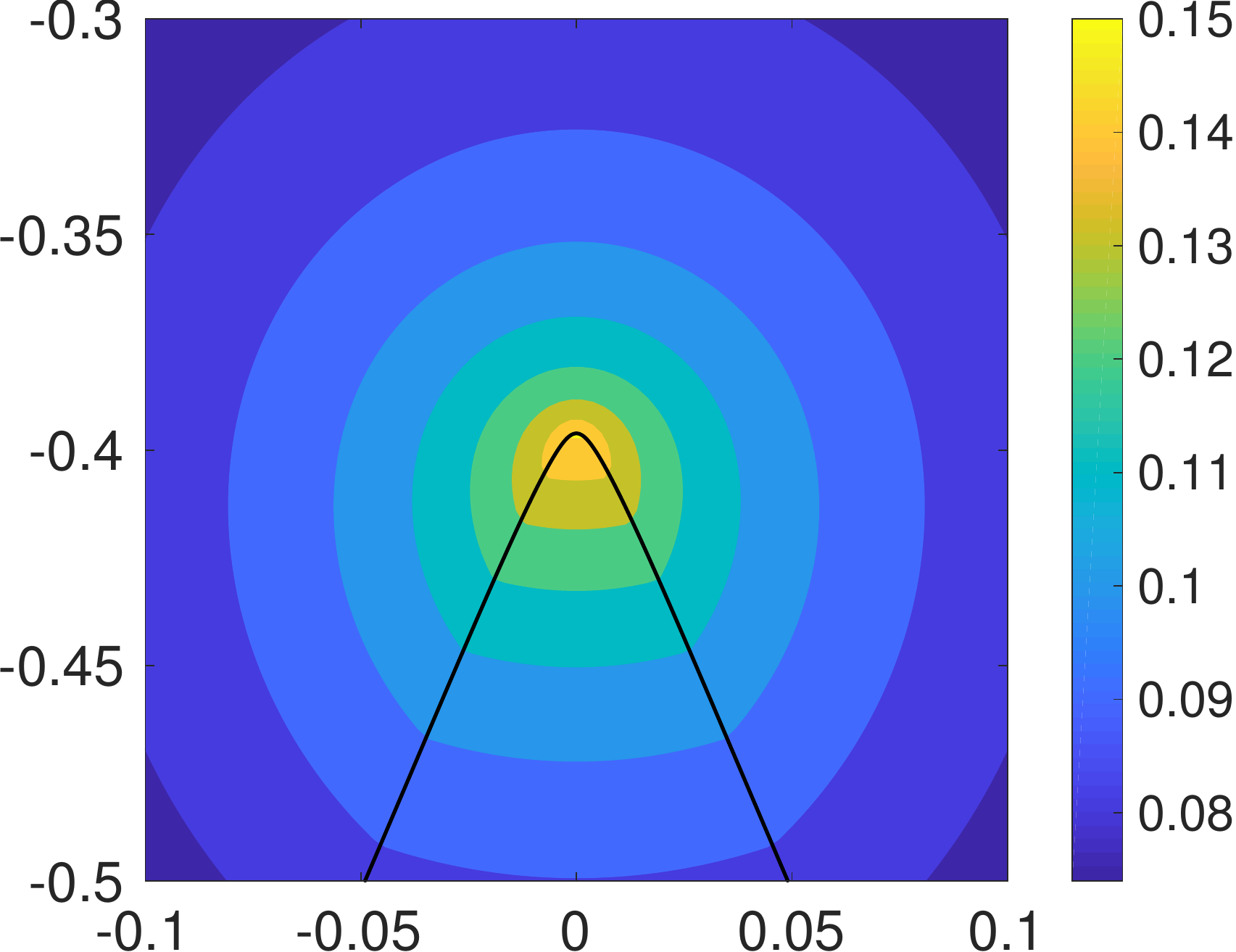}}\\
\caption{\label{fig28} (a), (b), (c). The eigenfunction, its conormal derivative, and the corresponding
single-layer potential associated with $\lambda_2=-0.3676$; (d), (e), (f). The corresponding items
associated with $\lambda_4=-0.3303$.}
\end{figure}

Fig.~\ref{fig29} plots the eigenfunctions with respect to arc length for the eigenvalues $\lambda_1=0.3676$ and $\lambda_2=-0.3676$ with different maximum curvature $500$, $1000$ and $1500$.

\begin{figure}
\centering
\subfigure[]{
\includegraphics[width=0.2\textwidth]{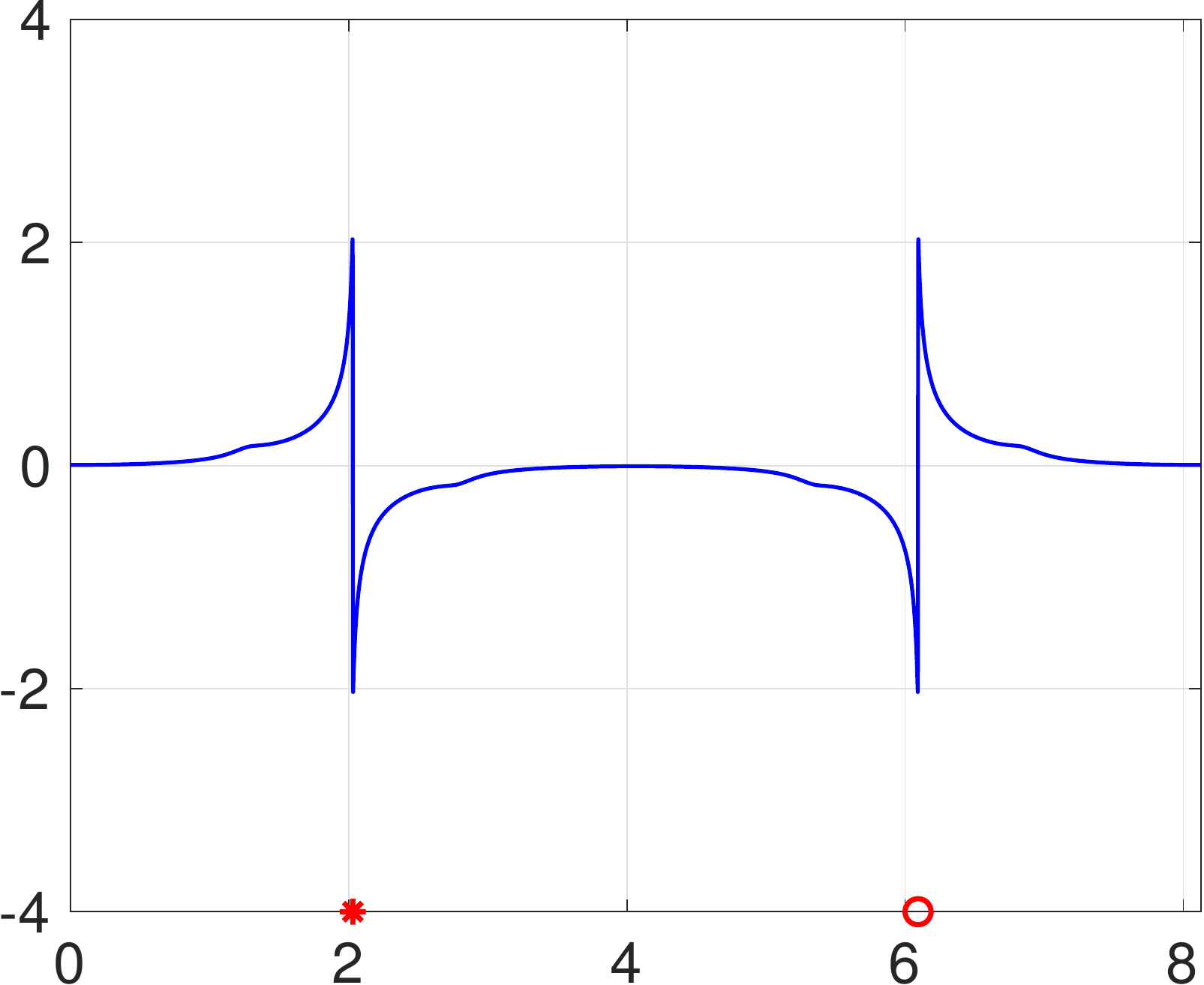}}
\subfigure[]{
\includegraphics[width=0.2\textwidth]{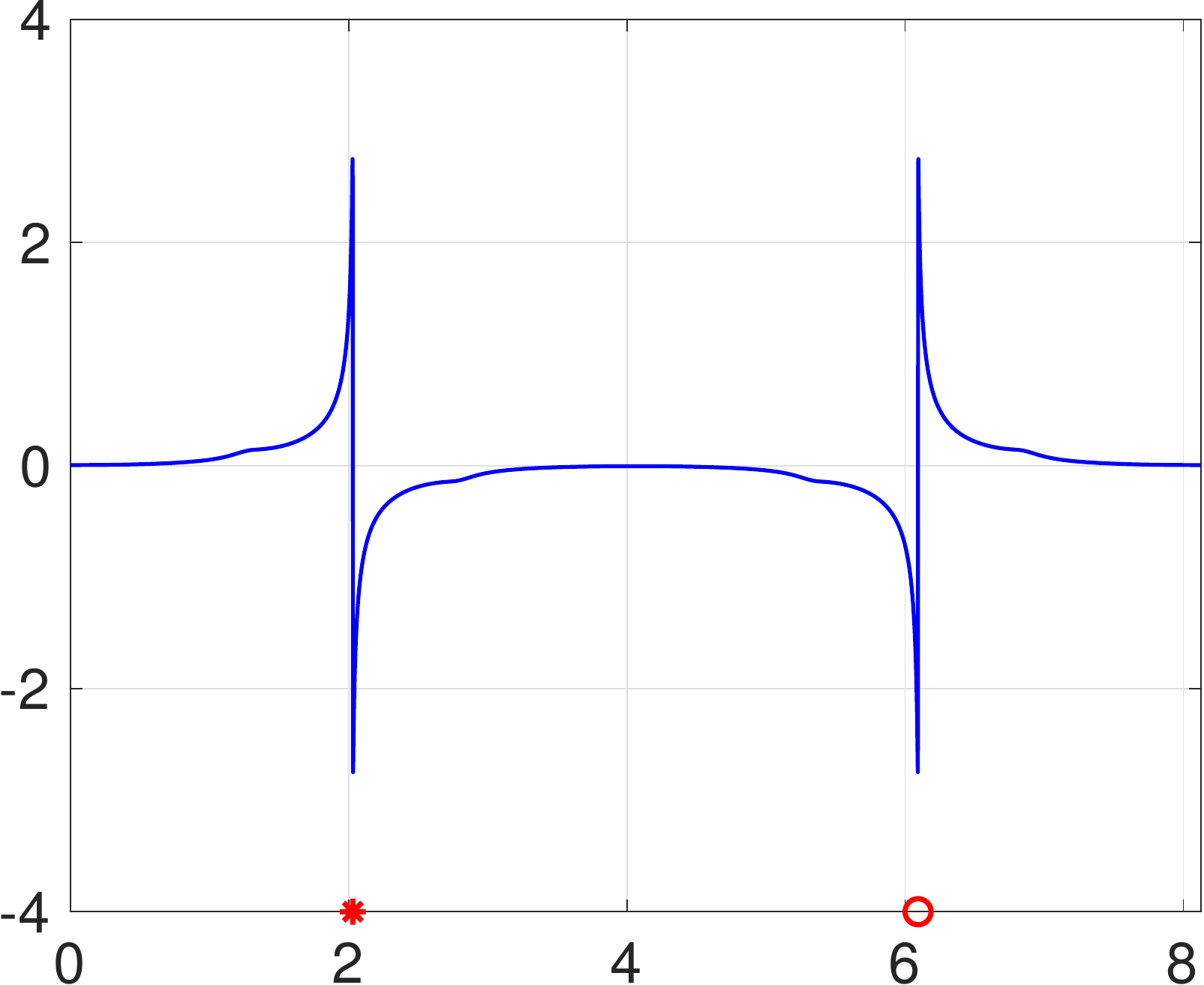}}
\subfigure[]{
\includegraphics[width=0.2\textwidth]{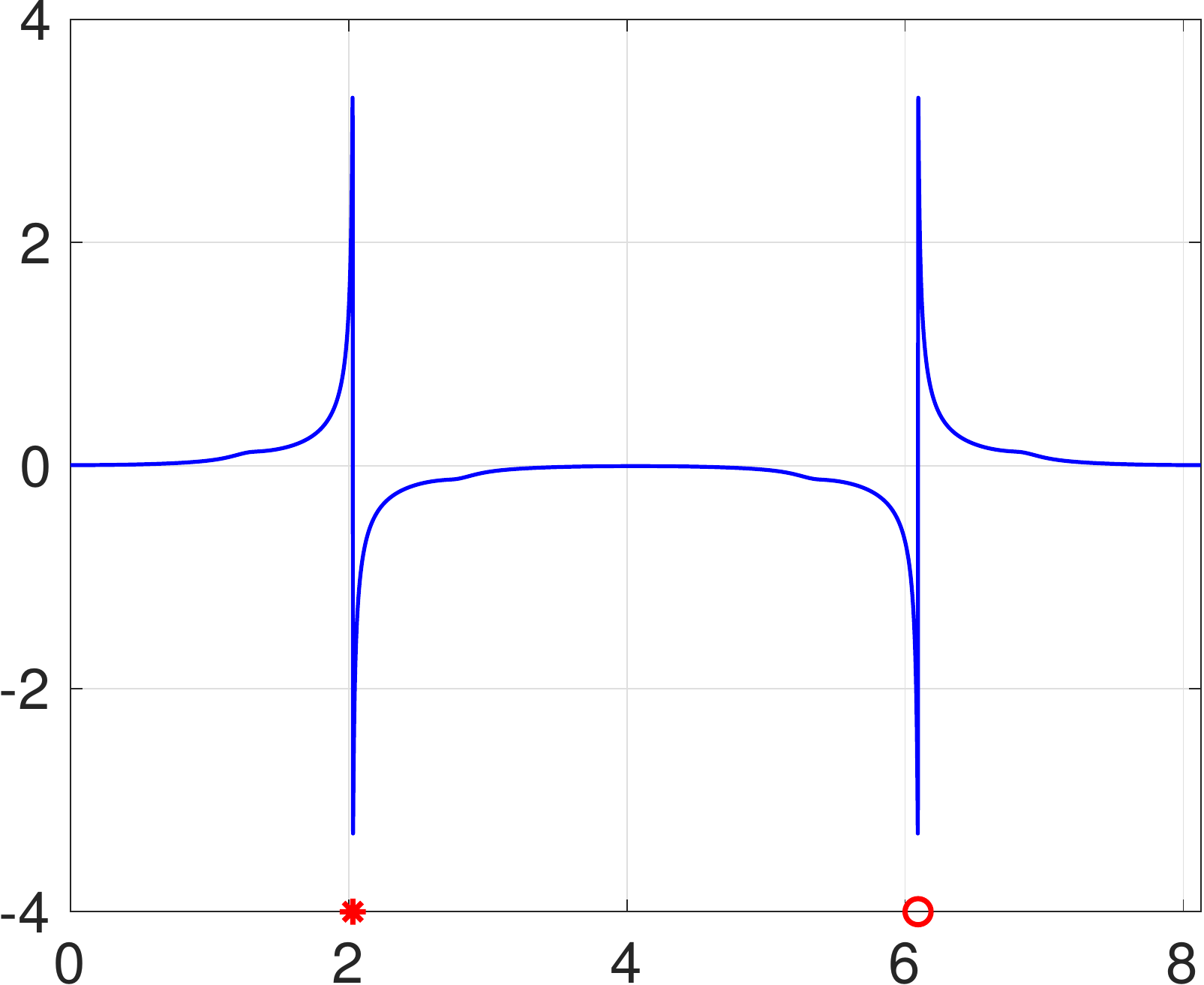}}\\
\subfigure[]{
\includegraphics[width=0.2\textwidth]{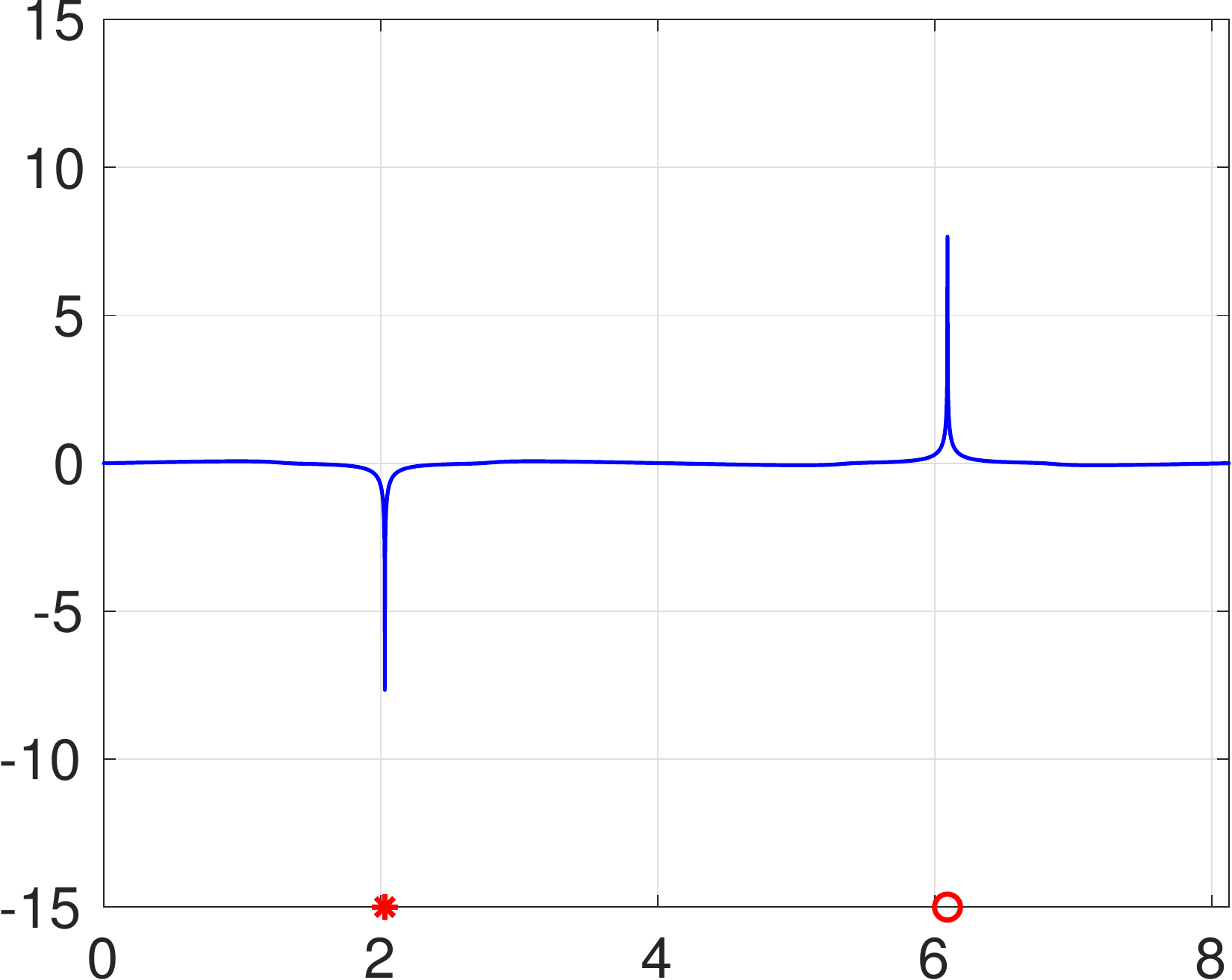}}
\subfigure[]{
\includegraphics[width=0.2\textwidth]{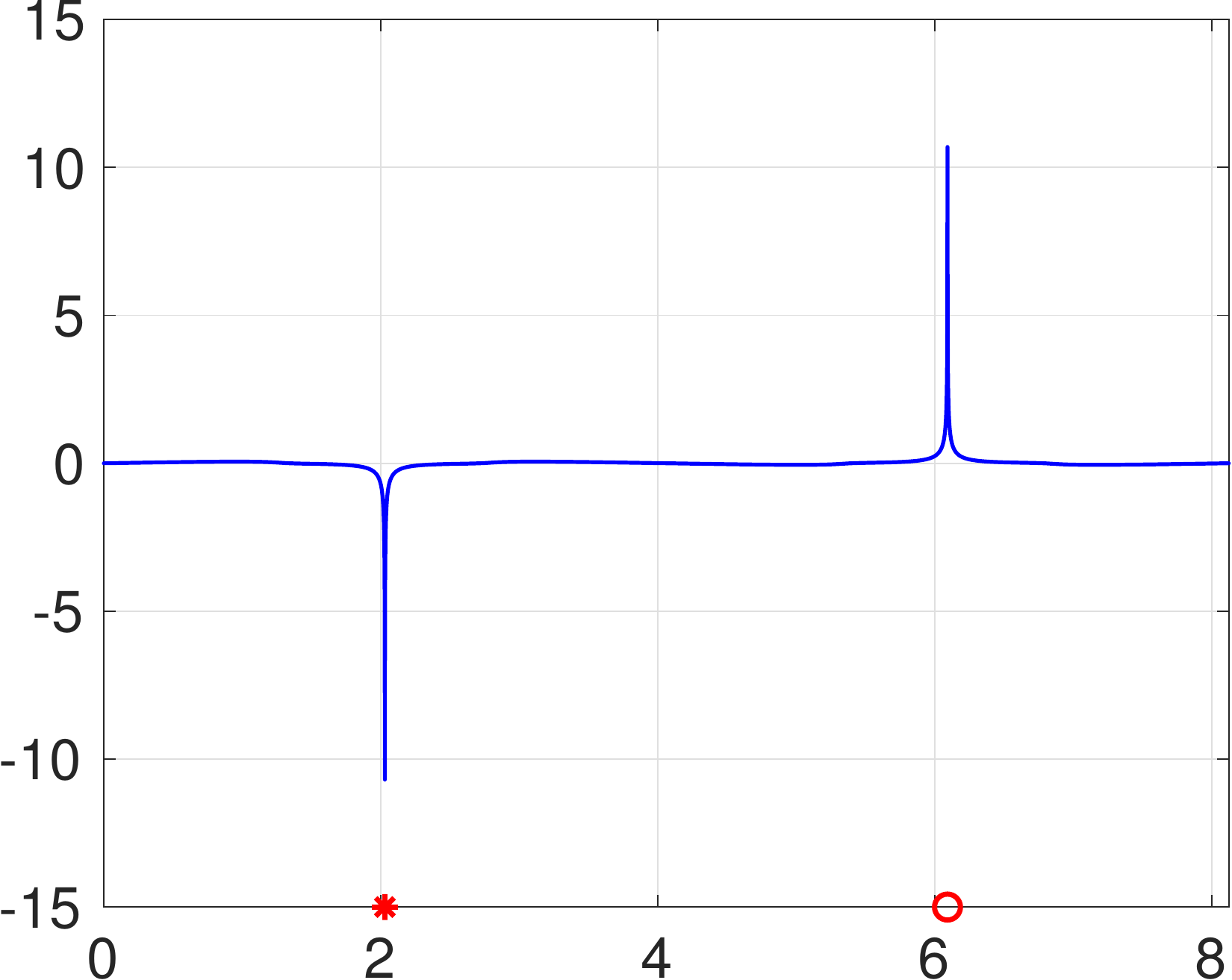}}
\subfigure[]{
\includegraphics[width=0.2\textwidth]{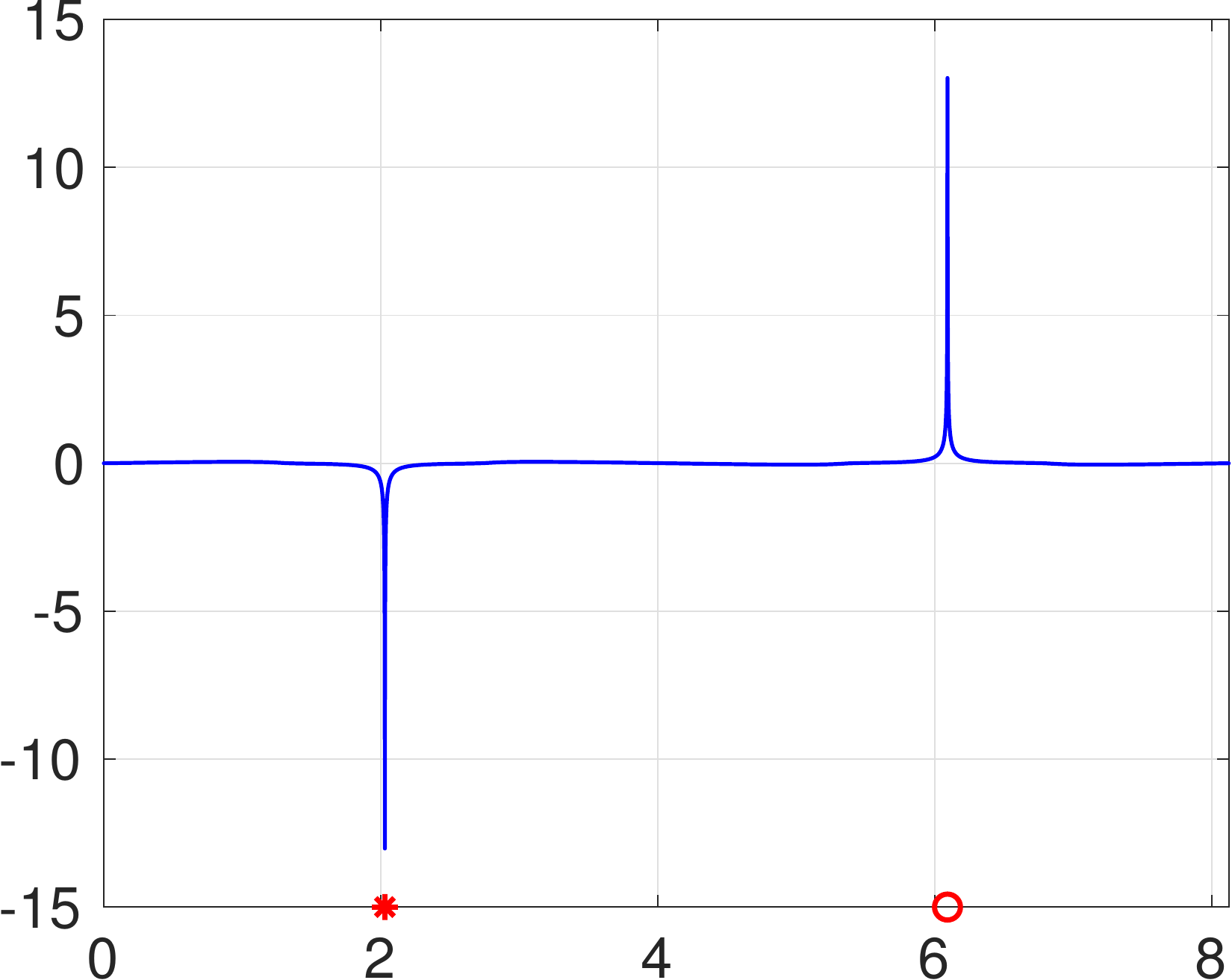}}\\
\caption{\label{fig29} (a), (b), (c).  Plotting the eigenfunctions for the positive eigenvalues $\lambda_1=0.3676$ with different maximum curvature $500$, $1000$ and $1500$. 
(d), (e), (f). The corresponding items for the negative eigenvalue $\lambda_2=-0.3676$.}
\end{figure}

We next investigate the blow-up rate of the NP eigenfunction or its conormal derivative with respect to the curvature.
Therefore we plot the logarithm of the derivative of the absolute value of the eigenfunctions at the high-curvature point for 
 the positive eigenvalues $\lambda_1$, $\lambda_3$ and $\lambda_5$, and the logarithm of the absolute value of the eigenfunctions
  at the high-curvature point for the negative eigenvalues $\lambda_2$, $\lambda_4$ and $\lambda_6$ at the high-curvature 
  point $x_*$ with respect to different curvature  in Fig.~\ref{fig30}. It turns out that blow-up rate also follows the 
rule in \eqref{eq:growthrate}. By regression, we numerically determine the corresponding parameters for those different eigenvalues 
in \eqref{eq:egg3}, and they are listed in Table~\ref{tab6}. 
%
%Fig.~\ref{fig29} shows that both the absolute value of the conormal derivative of the eigenfunction for 
%the positive eigenvalue and the absolute value of the eigenfunction for the negative eigenvalue at high-curvature points
% $x_*$ and $x_o$ increase as the the curvatures of the points $\kappa_{x_*}$ and $\kappa_{x_o}$ increase. 
% Therefore we plot the logarithm of the derivative of the absolute value of the eigenfunctions at the high-curvature point for 
% the positive eigenvalues $\lambda_1$, $\lambda_3$ and $\lambda_5$, and the logarithm of the absolute value of the eigenfunctions
%  at the high-curvature point for the negative eigenvalues $\lambda_2$, $\lambda_4$ and $\lambda_6$ at the high-curvature 
%  point $x_*$ with respect to different curvature  in Fig.~\ref{fig29}. Denote by $\psi_{\max}$ the  maximum of the absolute 
%  value of the eigenfunction for the negative eigenvalue, or the maximum of absolute value of the conormal derivative of the 
%  eigenfunction for the positive eigenvalue and from Fig.~\ref{fig29} one can conclude that by regression
%\[
% \psi_{\max} \sim a \kappa_{\max}^p.
%\]
%The coefficients of the regression are shown in Table.~\ref{tab6}.

\begin{figure}
\includegraphics[width=0.2\textwidth] {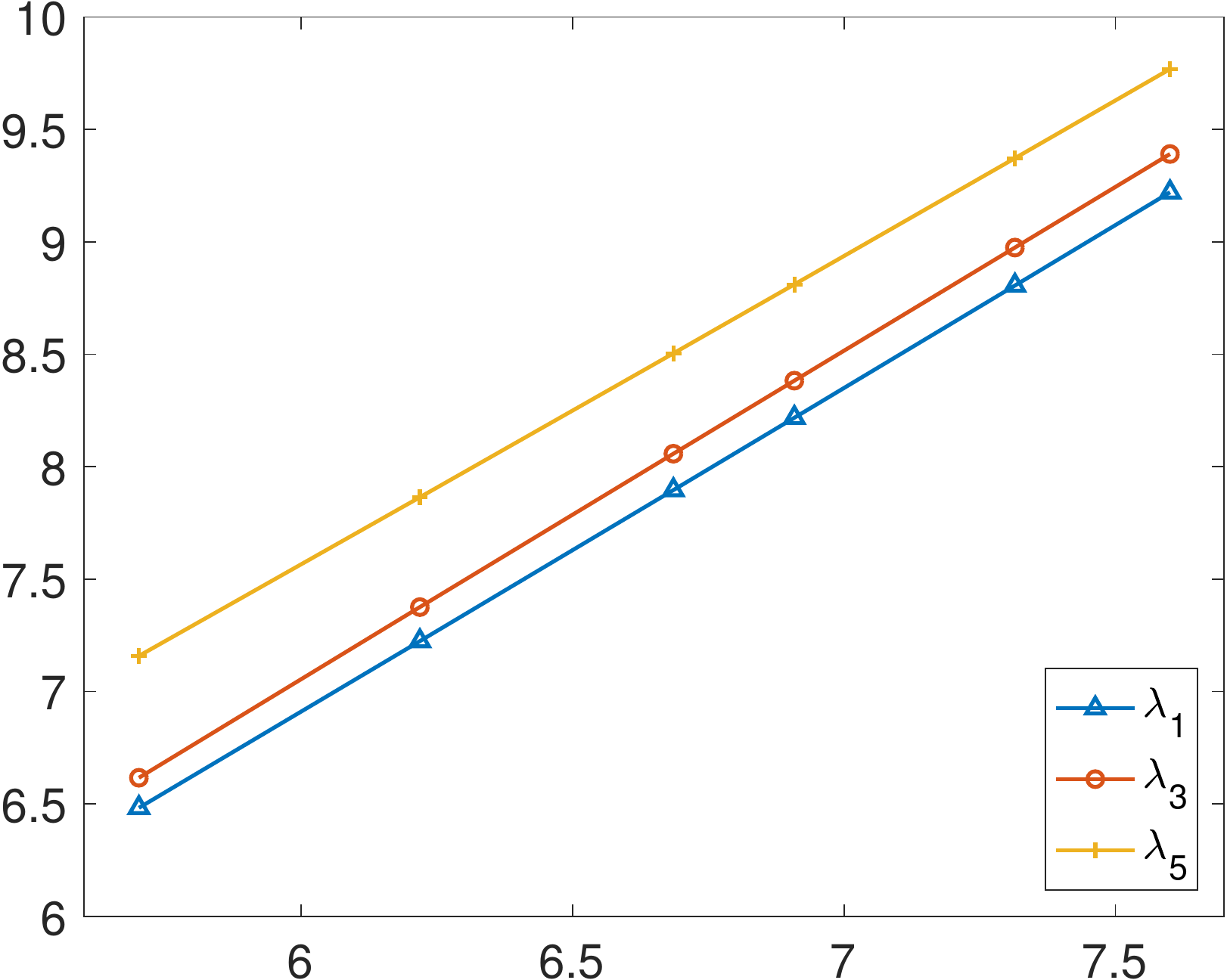}
\hspace{0.8cm}
\includegraphics[width=0.2\textwidth] {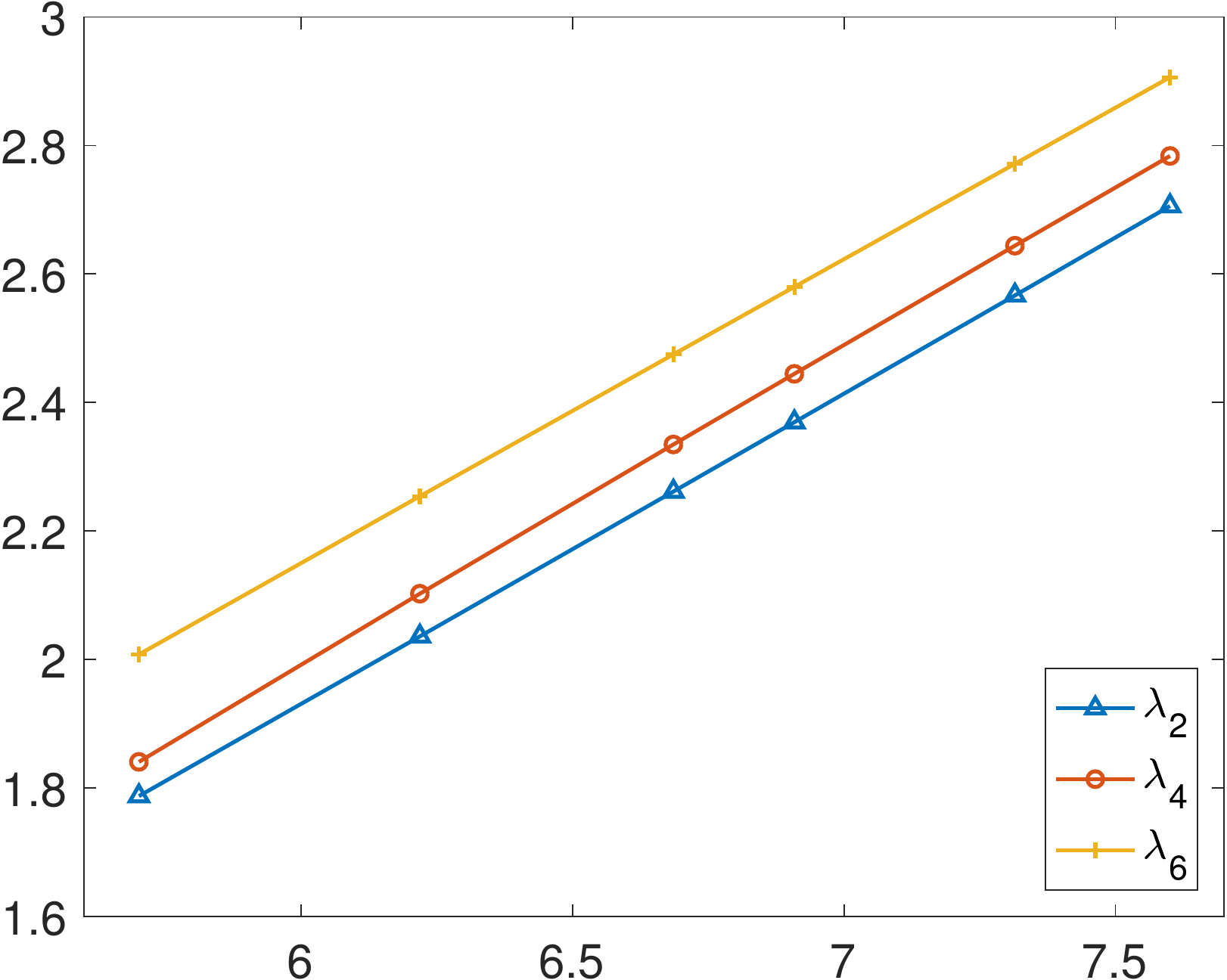}
\caption{\label{fig30} The logarithm of the absolute value of the derivative of the eigenfunction at the high-curvature point is the largest for the positive eigenvalues $\lambda_1$, $\lambda_3$ and $\lambda_5$ and the logarithm of the absolute value of the eigenfunction at the high-curvature point for the negative eigenvalues $\lambda_2$, $\lambda_4$ and $\lambda_6$ with respect to different curvature.}
\end{figure}

\begin{table}[t]
  \centering
  \subtable[]{
    \centering
    \begin{tabular}{cccc}
      \toprule
      & $\lambda_1$ & $\lambda_3$ & $\lambda_5$ \\
      \midrule
      $p$ & 1.4415 & 1.4604 & 1.3742 \\[5pt]
      $\ln(\alpha)$ & -1.7389 & -1.7075 & -0.6801 \\
      \bottomrule
    \end{tabular}}
  %  \caption{}
  \hspace{1.5cm}
\subtable[]{
    \centering
    \begin{tabular}{cccc}
      \toprule
      & $\lambda_2$ & $\lambda_4$ & $\lambda_6$ \\
      \midrule
      $p$ & 0.4836 & 0.4963 & 0.4729 \\[5pt]
      $\ln(\alpha)$ & -0.9712 & -0.9860 & -0.6875 \\
      \bottomrule
    \end{tabular}
    }
 %\end{subtable}
  \caption{The parameters of the form \eqref{eq:growthrate} from the regression associated with the eigenvalues in
  \eqref{eq:egg3}: (a) $\lambda_j, j=1, 3, 5$; (b) $\lambda_j, j=2,4,6$.}
  \label{tab6}
\end{table}

%Here we give two remarks. The first one is that the eigenvalues $\lambda_5$ and $\lambda_6$ are simple. The second one is that the curvatures of the points $\kappa_{x_*}$ and $\kappa_{x_o}$ are changing at the same time and they are always the same, namely
%\[
% \kappa_{x_*}=\kappa_{x_o}.
%\]

\subsection{A non-symmetric domain}

In this subsection, we consider a non-symmetric domain as shown in Fig.~\ref{fig31}, which possesses three boundary points
with relatively large curvatures that are marked as $x_o, x_*$ and $x_\triangle$ in the figure. The corresponding curvatures 
at those three points are respectively given as
\begin{equation}\label{eq:cu}
\kappa_{x_{o}}=500,\quad  \kappa_{x_{*}}=\kappa_{x_{\triangle}}=41.
\end{equation}
It is noted that the domain in Fig.~\ref{fig31} is different from the one in Fig.~\ref{fig11}. Here, we modify the curvatures at the two points $x_*$ and $x_\triangle$ such that the 
domain is no longer symmetric. Obviously, $x_o$ is the high-curvature point. 

\begin{figure}
\includegraphics[width=2.5cm] {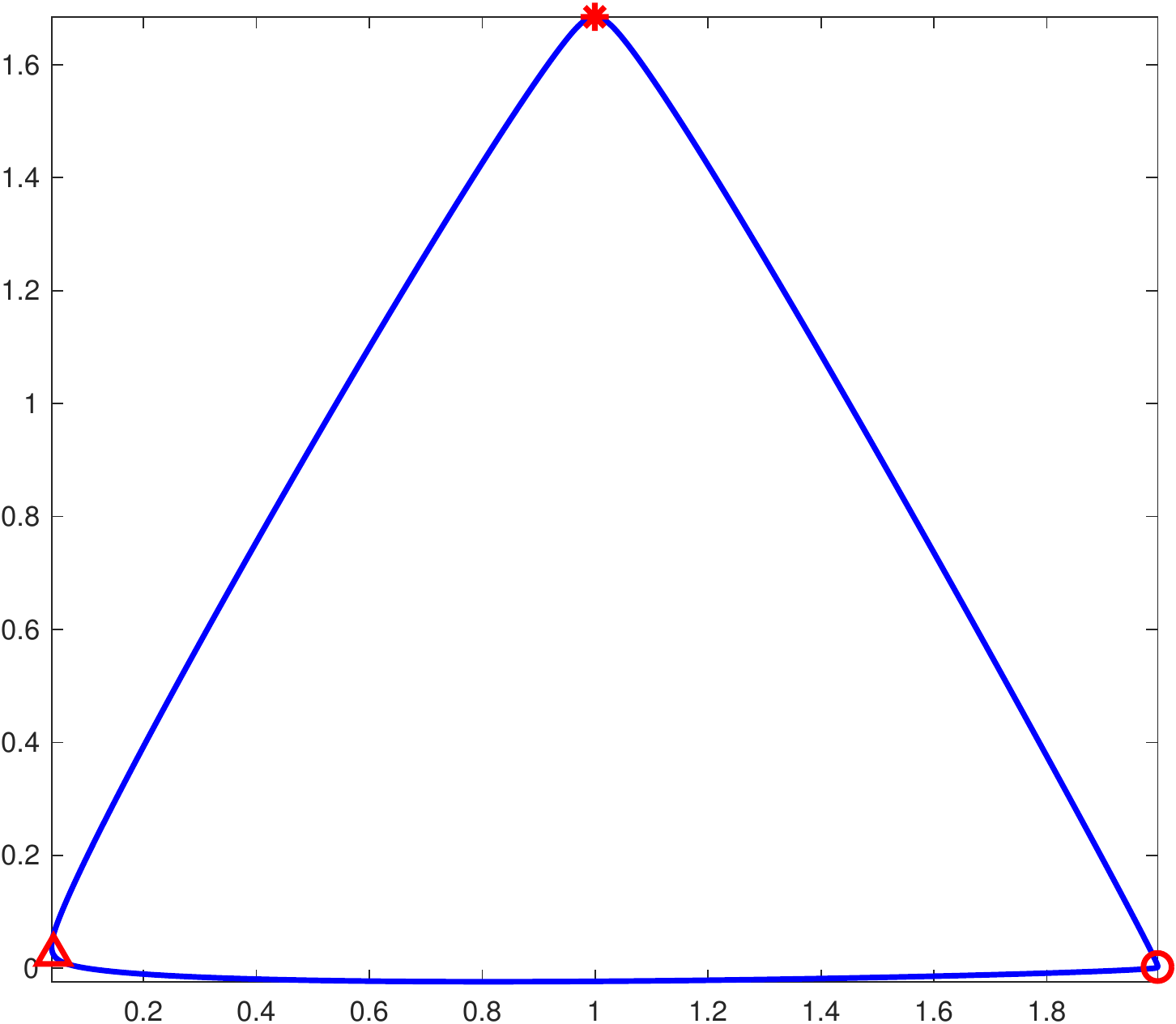}
\caption{\label{fig31} A non-symmetric domain.}
\end{figure}

First, the first five largest NP eigenvalues (in terms of the absolute value) associated with $\partial D$ in Fig.~ \ref{fig31}
are numerically found to be
\begin{equation}\label{eq:egg6}
\lambda_0=0.5, \quad  \lambda_1=0.2710, \quad \lambda_2=-0.2710,\quad  \lambda_3=0.2320,\quad \lambda_4=-0.2320.
\end{equation}

Fig.~\ref{fig32} plots the eigenfunction and the corresponding single-layer potential around the three points
$x_{o}$, $x_{*}$ and $x_{\triangle}$ associated to the eigenvalue $\lambda_3=0.2320$. It can be readily seen that the blow-up 
behaviour does not occur at the high-curvature point $x_o$, and instead it occurs at the two points $x_*$ and $x_\triangle$ which possesses
relatively large curvatures. 

\begin{figure}
\centering
\subfigure[]{
\includegraphics[width=0.2\textwidth]{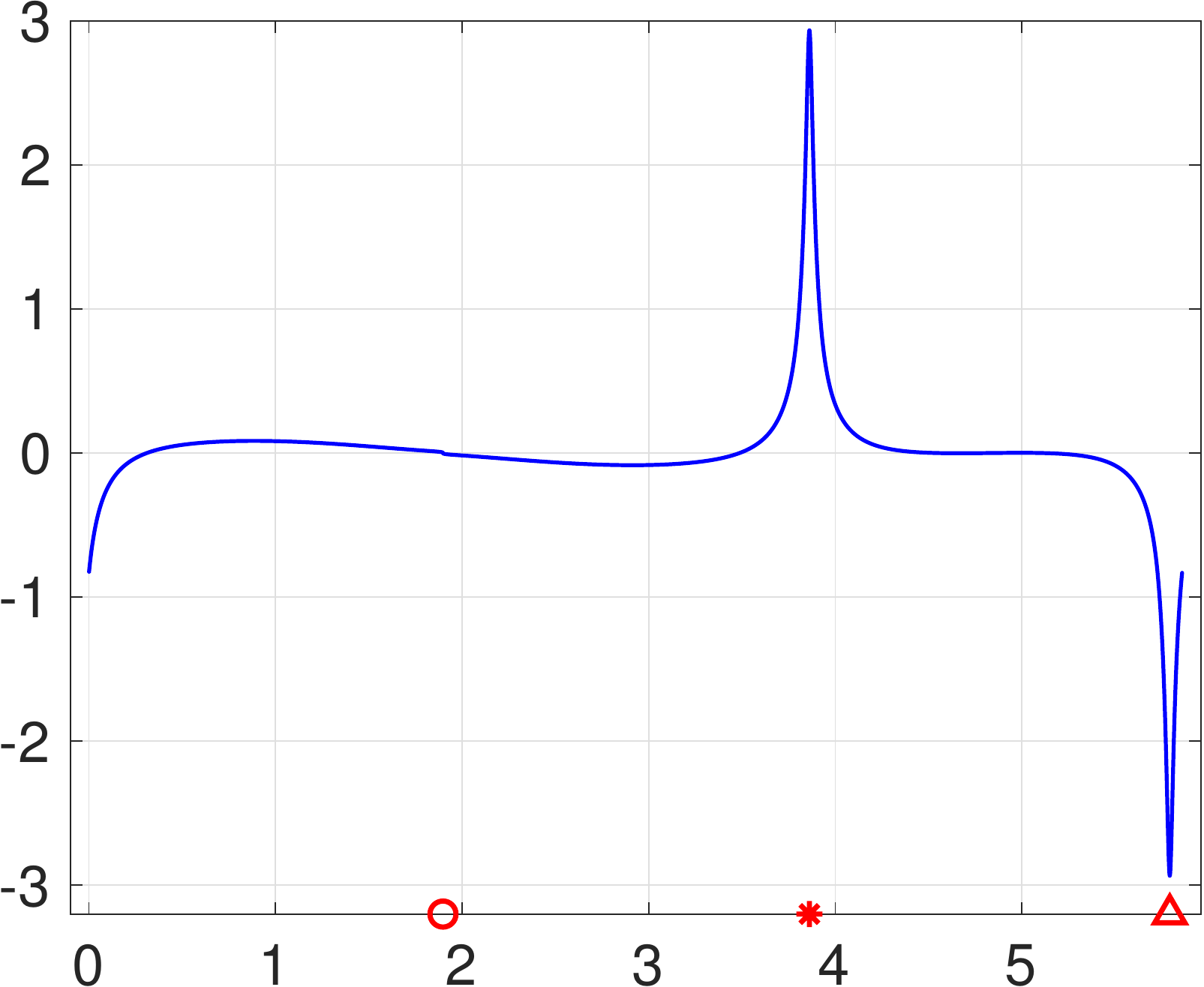}}
\subfigure[]{
\includegraphics[width=0.2\textwidth]{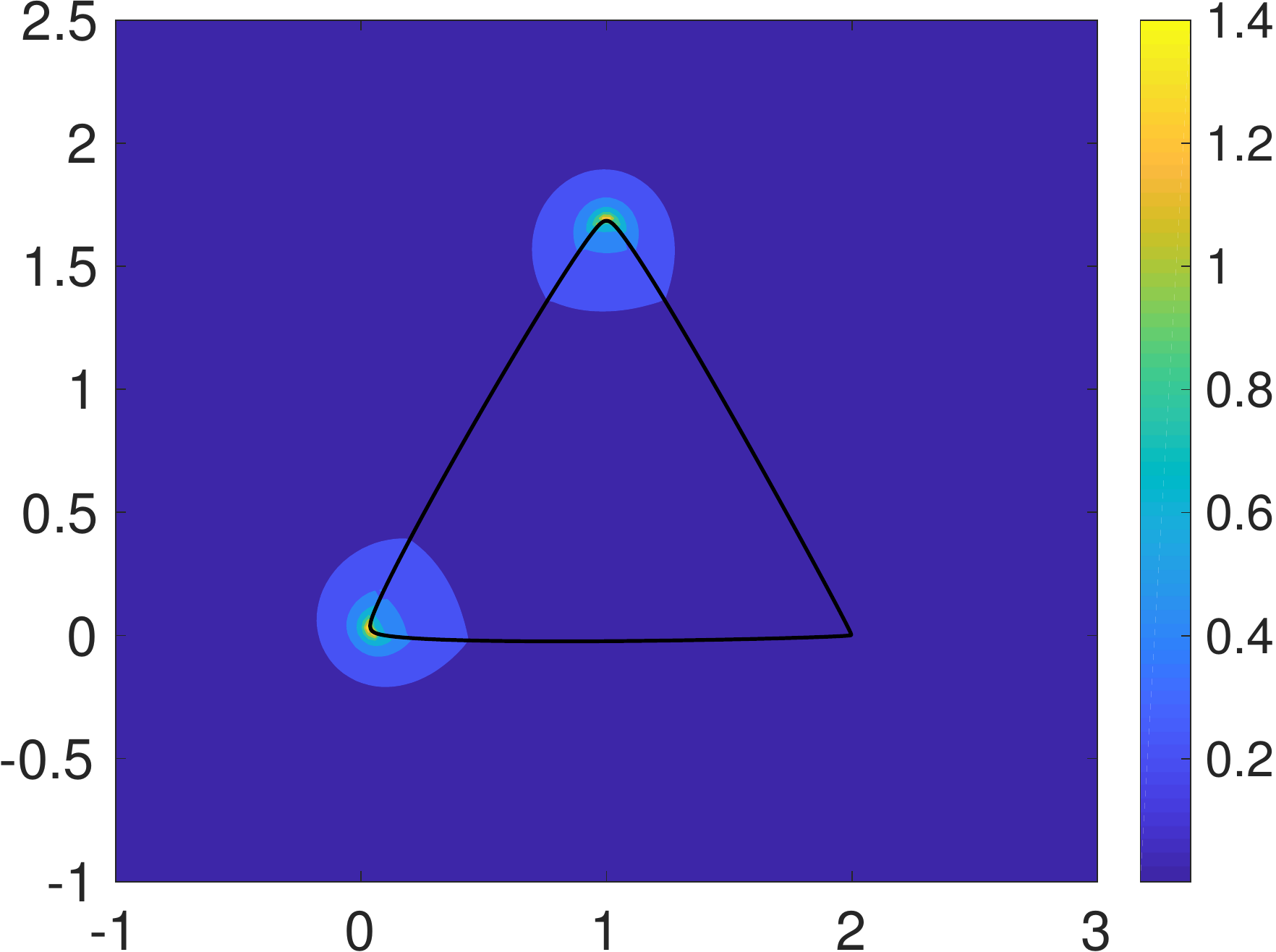}}\\
%\subfigure[]{
%\includegraphics[width=0.2\textwidth]{pos3_1.pdf}}\\
\subfigure[]{
\includegraphics[width=0.2\textwidth]{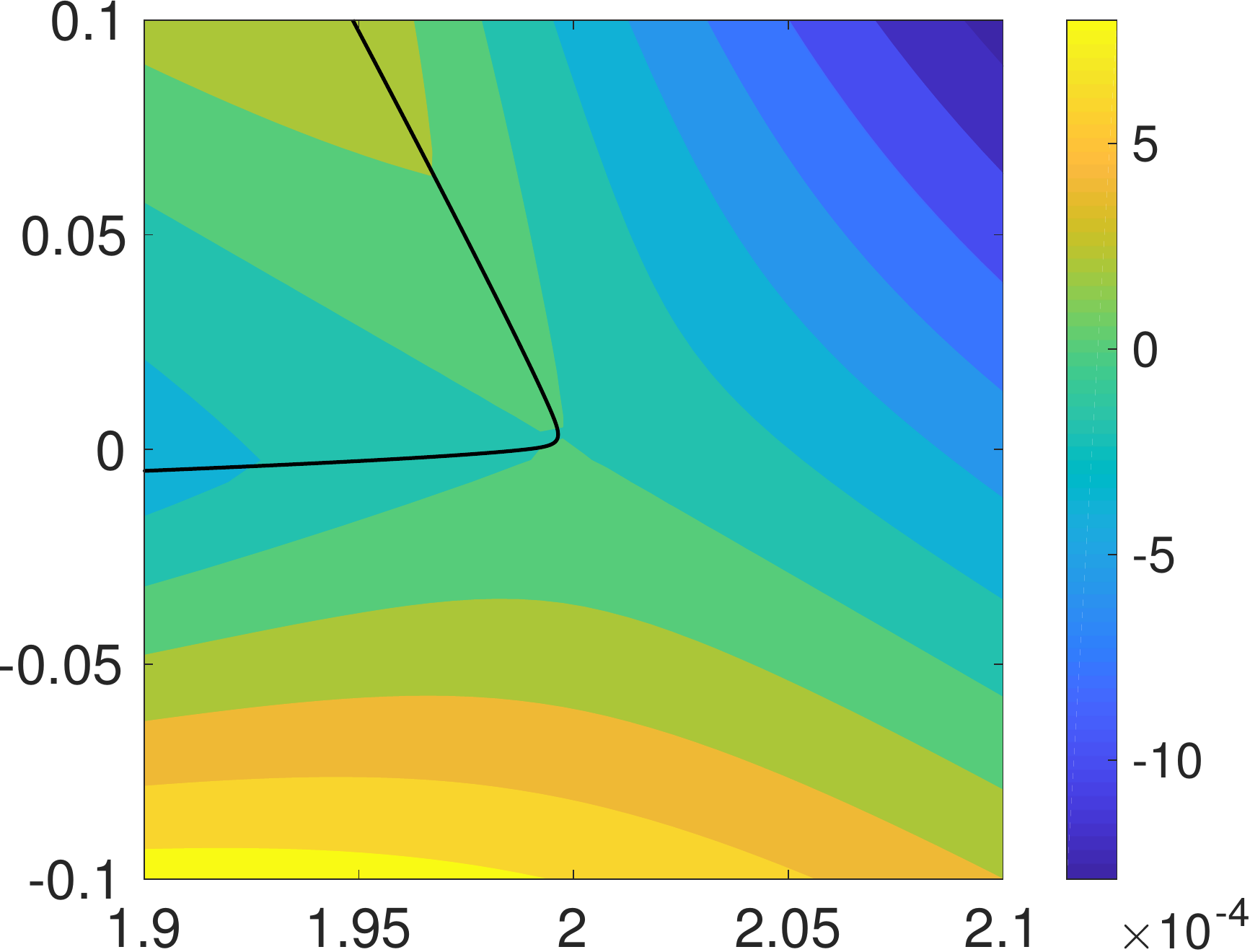}}
\subfigure[]{
\includegraphics[width=0.2\textwidth]{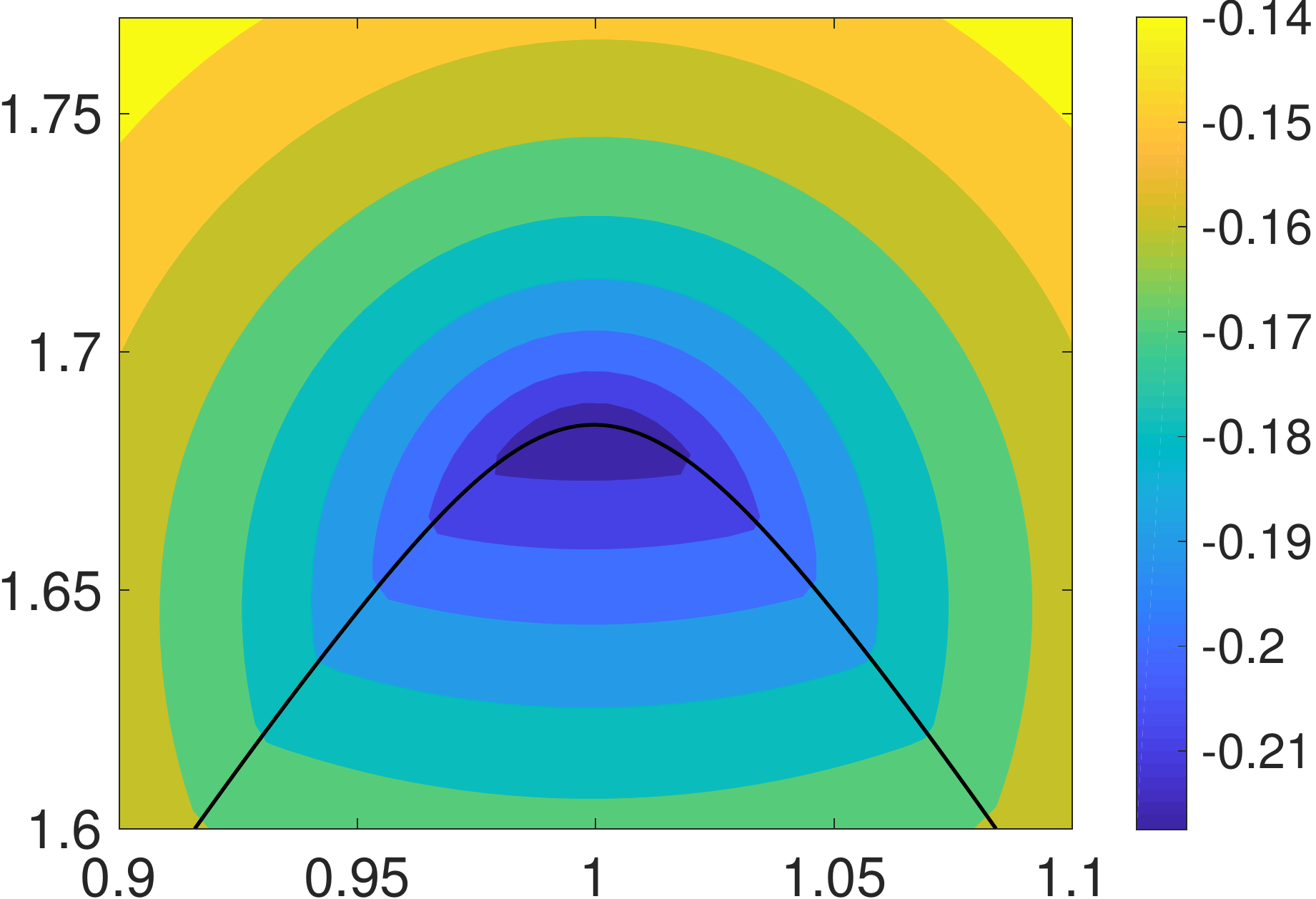}}
\subfigure[]{
\includegraphics[width=0.2\textwidth]{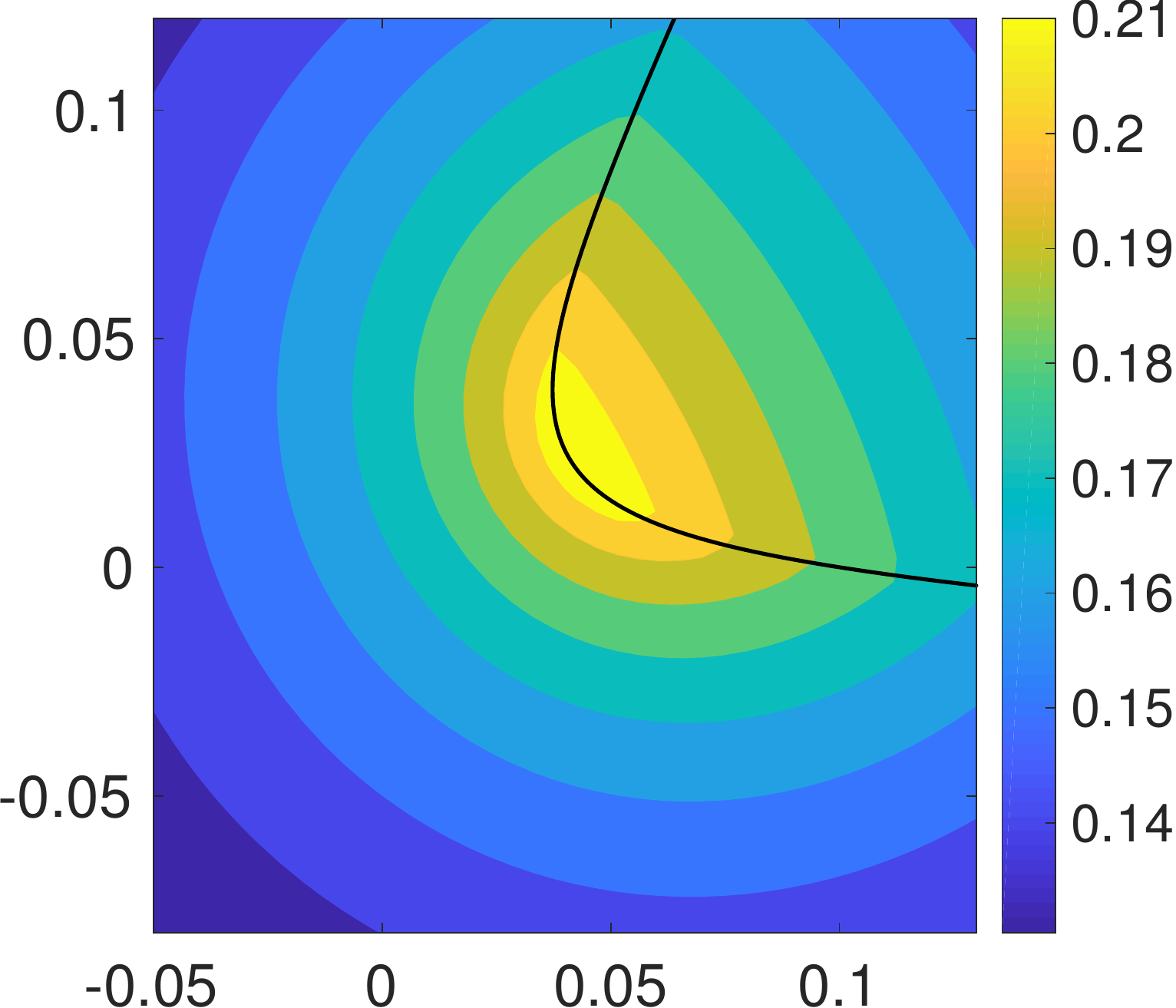}}\\
\caption{\label{fig32}
(a). Eigenfunction with respect to the arc length associated to the eigenvalue $\lambda_3=0.2320$; (b).
The corresponding single-layer potential; (c), (d), (e). The single-layer potential around the three
points $x_{o}$, $x_{*}$ and $x_{\triangle}$, respectively. }
\end{figure}

Fig.~\ref{fig33} plots the eigenfunction, and its conormal derivative as well as the corresponding single-layer potential 
associated to the eigenvalue $\lambda_4=-0.2320$. It can be readily seen that the blow-up behaviour of the conormal derivative 
does not occur at the high-curvature point $x_o$, and instead it occurs at the two points $x_*$ and $x_\triangle$ again. 

Clearly, the previous two examples show that the blow-up behaviour does not follow the one observed for symmetric domains. The two points $x_*$ and $x_\triangle$
are symmetric with respect to $x_o$, and they two may compete with the point $x_o$ to form the blow-up behaviours as observed above. 
However, there are no definite rules for this. 

\begin{figure}[t]
\centering
\subfigure[]{
\includegraphics[width=0.2\textwidth]{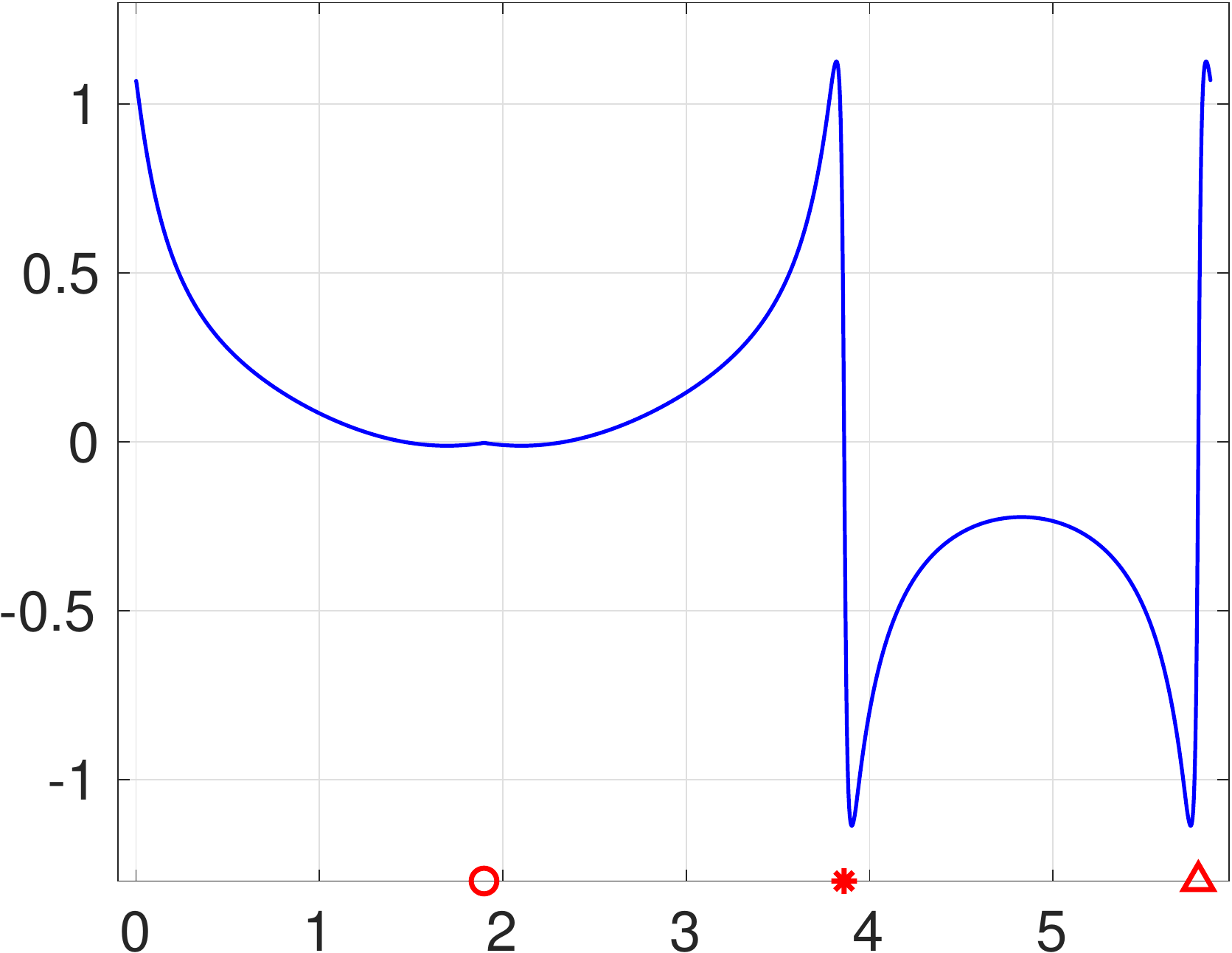}}
\subfigure[]{
\includegraphics[width=0.2\textwidth]{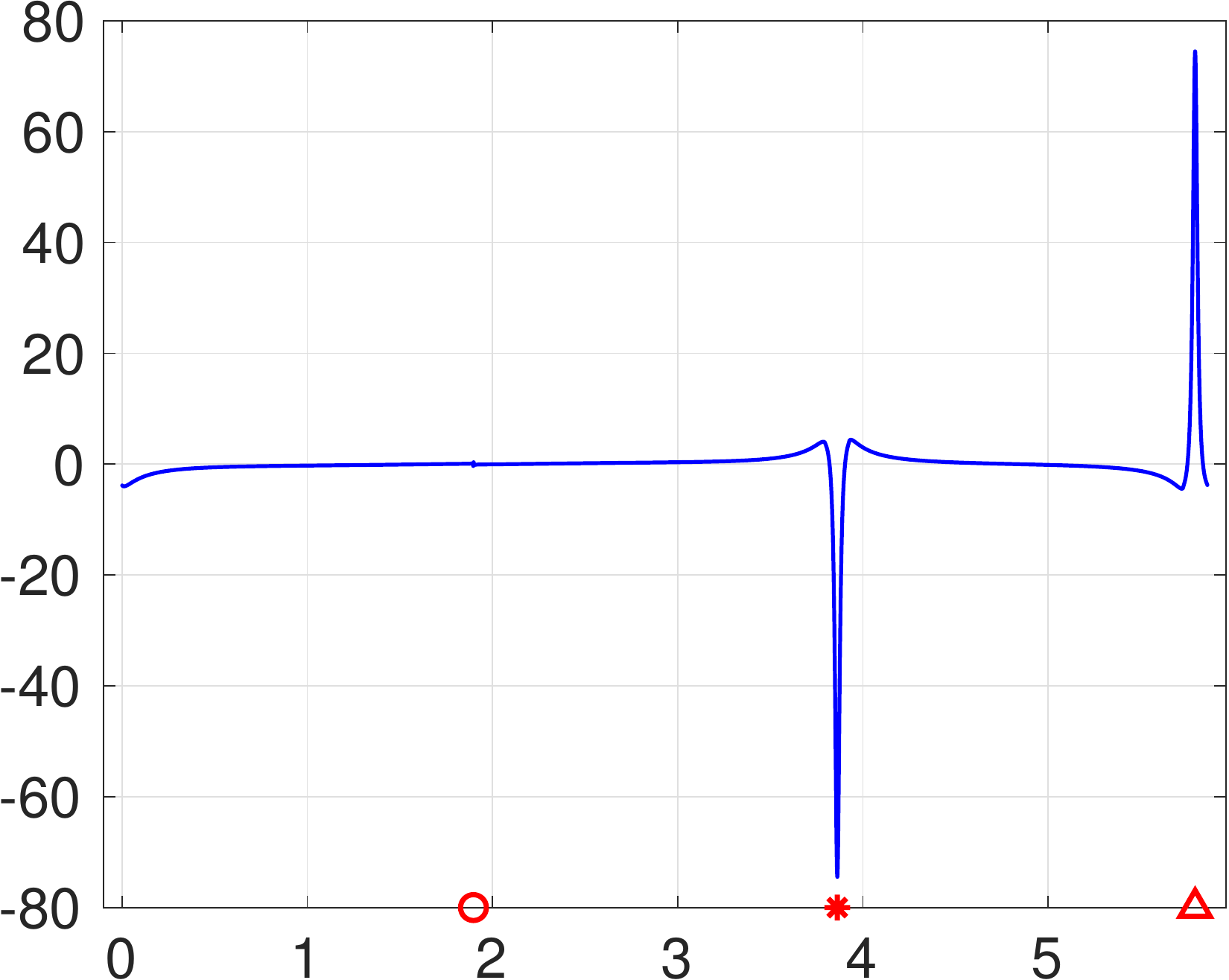}}
\subfigure[]{
\includegraphics[width=0.2\textwidth]{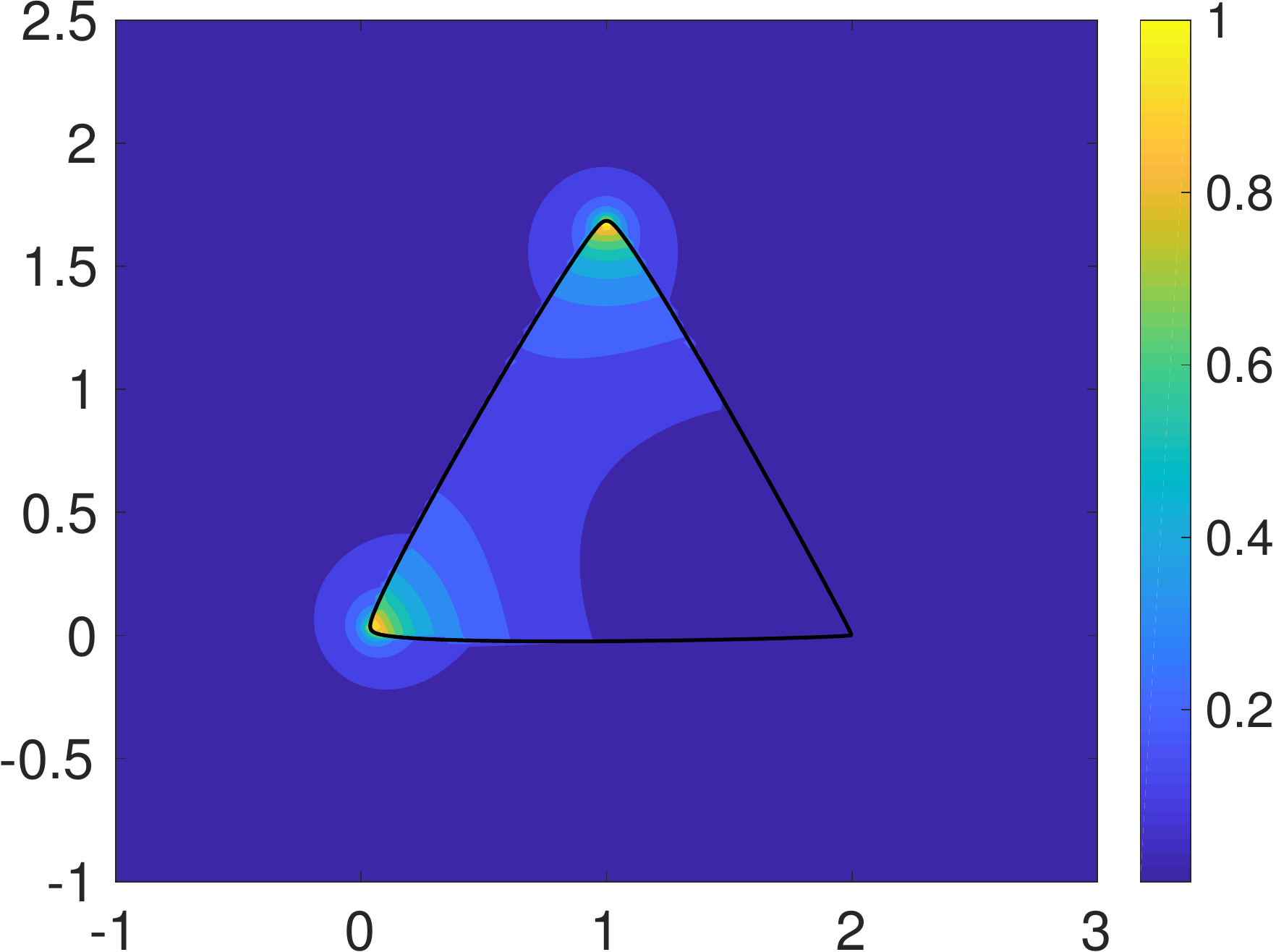}}\\
\subfigure[]{
\includegraphics[width=0.2\textwidth]{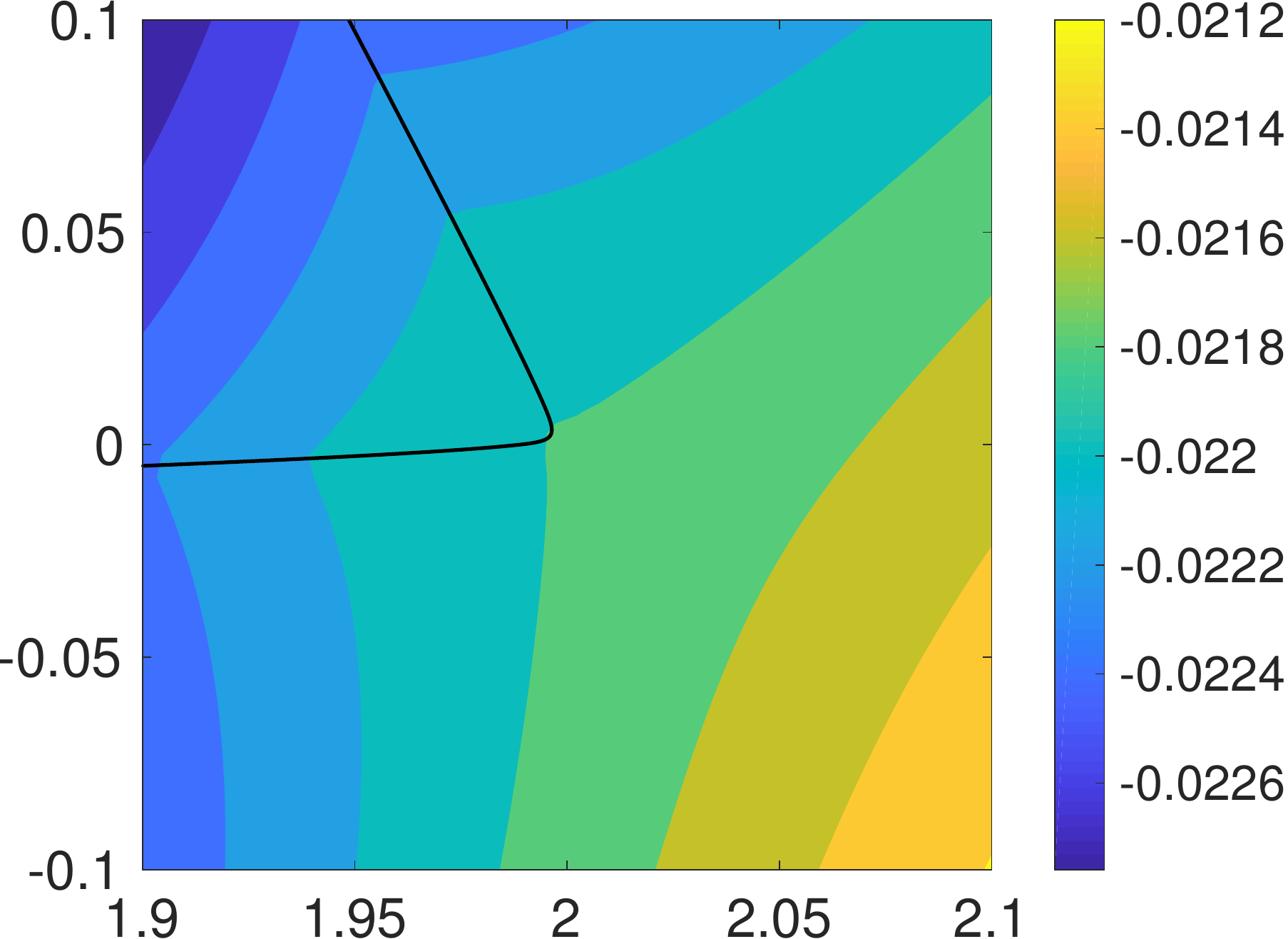}}
\subfigure[]{
\includegraphics[width=0.2\textwidth]{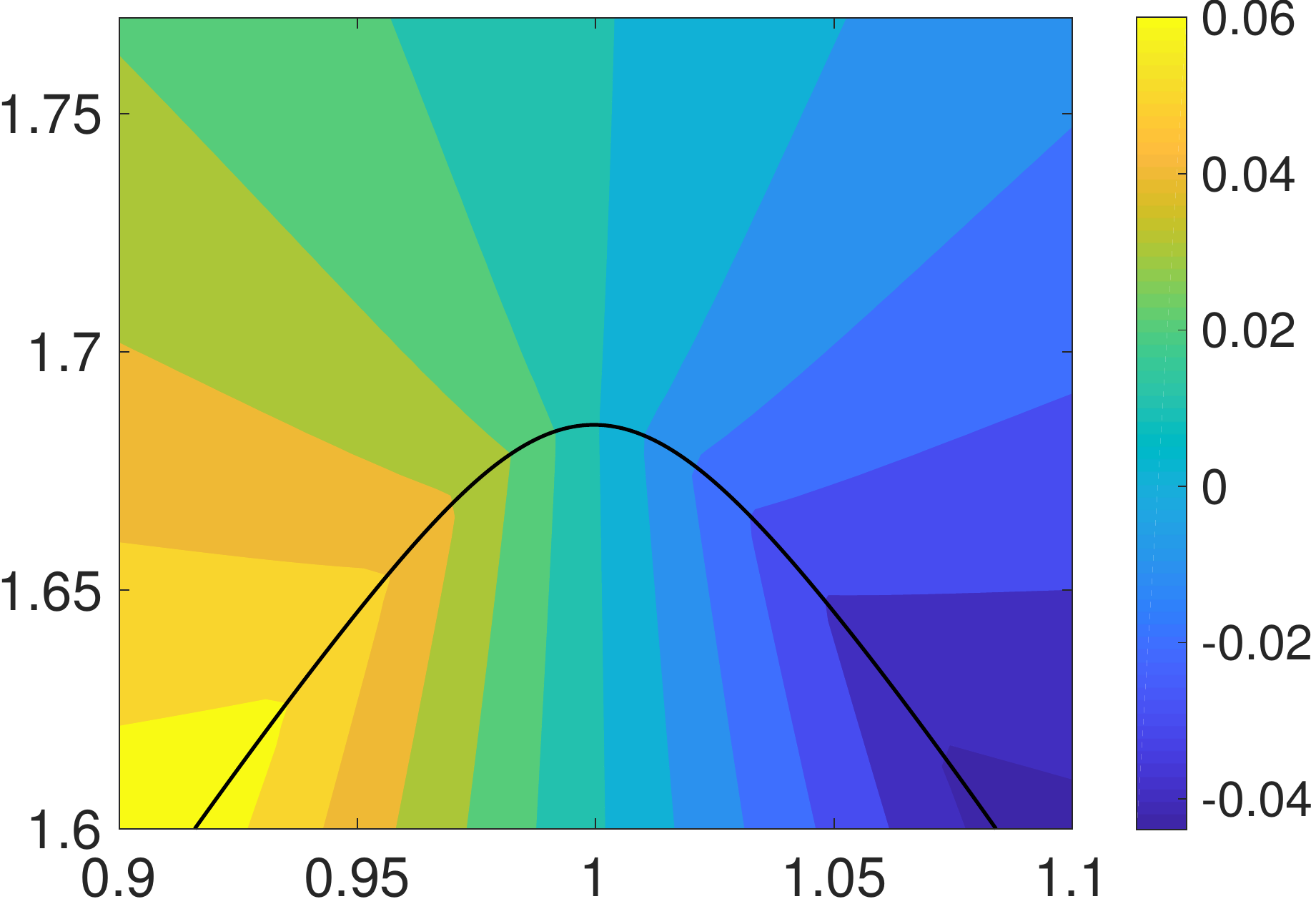}}
\subfigure[]{
\includegraphics[width=0.2\textwidth]{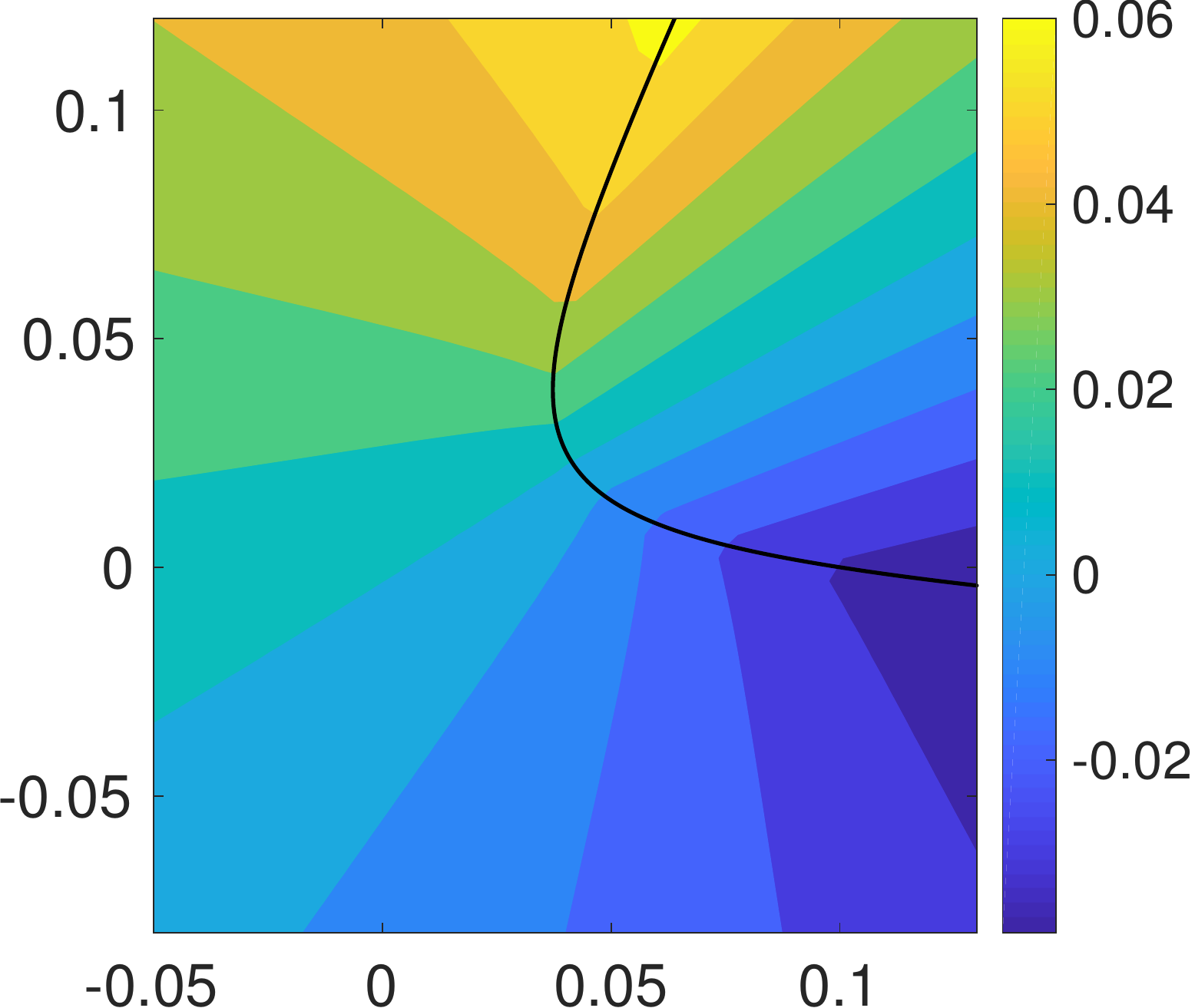}}\\
\caption{\label{fig33} (a), (b). Eigenfunction and its conormal derivative with respect to the arc length associated to the eigenvalue $\lambda_4=-0.2320$; (c), (d), (e), (f).
The corresponding single-layer potential as well as their behaviours around the three points $x_{o}$, $x_{*}$ and $x_{\triangle}$, respectively.}
\end{figure}

\subsection{Another non-symmetric domain}

In this subsection, we consider a non-symmetric domain as shown in Fig.~\ref{fig34}, which possesses three boundary points
with relatively large curvatures that are marked as $x_o, x_*$ and $x_\triangle$ in the figure. The corresponding curvatures 
at those three points are respectively given as
\begin{equation}\label{eq:cu1}
\kappa_{x_{o}}=500, \quad  \kappa_{x_{*}}=804 \quad \mbox{and} \quad \kappa_{x_{\triangle}}=400.
\end{equation}

\begin{figure}
\includegraphics[width=3cm] {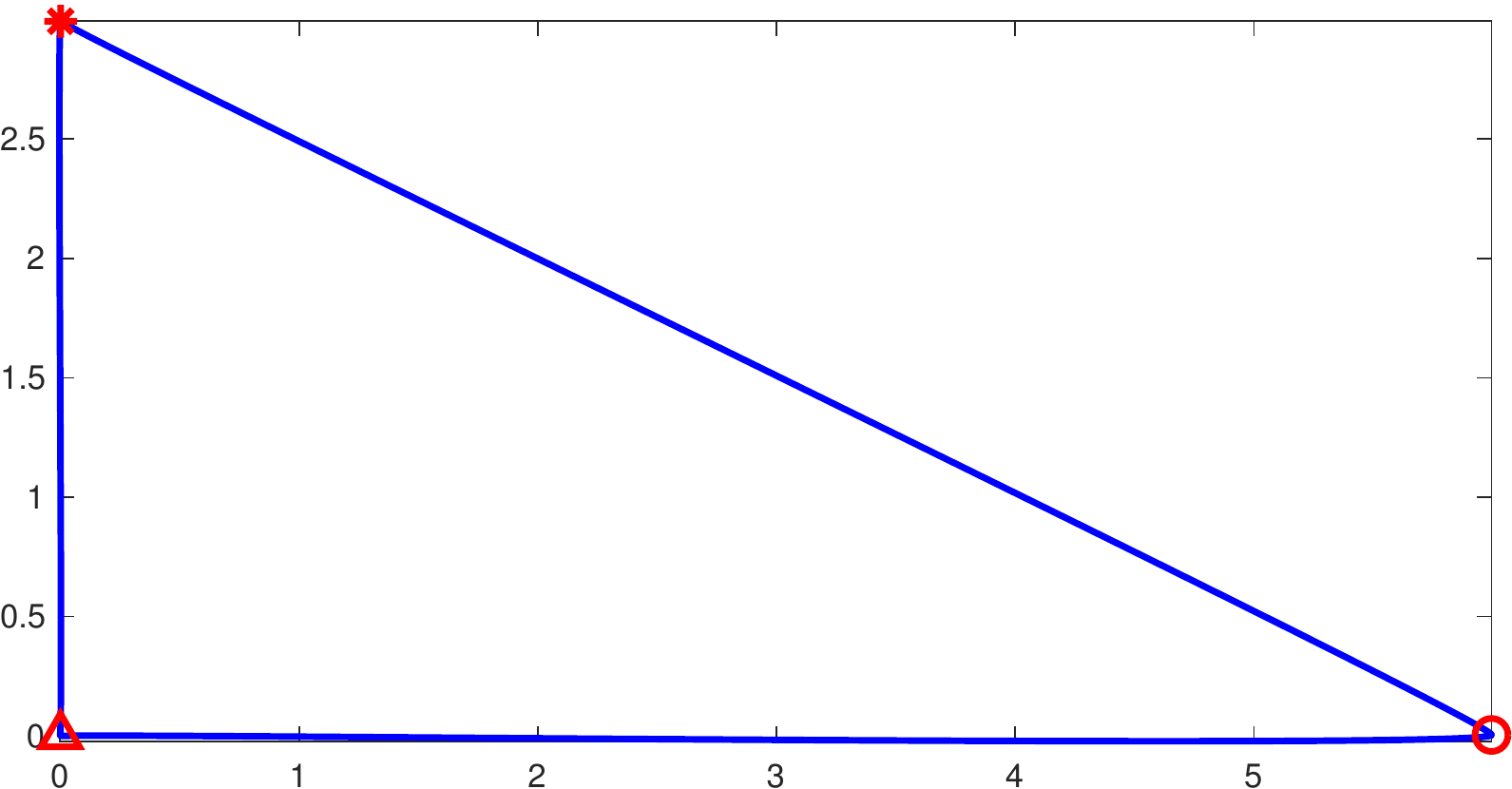}
\caption{\label{fig34} The boundary with three high-curvature points.}
\end{figure}

First, the first three largest NP eigenvalues (in terms of the absolute value) associated with $\partial D$ in Fig.~ \ref{fig34}
are numerically found to be
\begin{equation}\label{eq:egg7}
\lambda_0=0.5, \quad  \lambda_1=0.3792, \quad \lambda_2=-0.3792.
\end{equation}

Fig.~\ref{fig35} plots the eigenfunction and the corresponding single-layer potential around the three points
$x_{o}$, $x_{*}$ and $x_{\triangle}$ associated to the eigenvalue $\lambda_1=0.3792$. It can be readily seen that even if the curvature at the points $x_{o}$ and $x_{*}$ has the relationship 
\[
 \kappa_{x_o}<\kappa_{x_*},
\]
given in \eqref{eq:cu1}, in the Fig.~\ref{fig35}, the figure $a$ shows that the absolute value of the eigenfunction $\varphi$ at the point $x_{o}$ is larger than that at the point $x_{*}$, namely
\[
 \varphi(x_o)>\varphi(x_*).
\]
As for the single layer potential $S_{\partial D}[\varphi]$, $b$, $c$ and $d$ show that 
\[
 S_{\partial D}[\varphi](x_o)>S_{\partial D}[\varphi](x_*).
\]
Therefore for the non-symmetric domain $D$, the larger curvature point does not yield the larger value of the eigenvalue and the associated single layer potential for the positive eigenvalue.

\begin{figure}
\centering
\subfigure[]{
\includegraphics[width=0.2\textwidth]{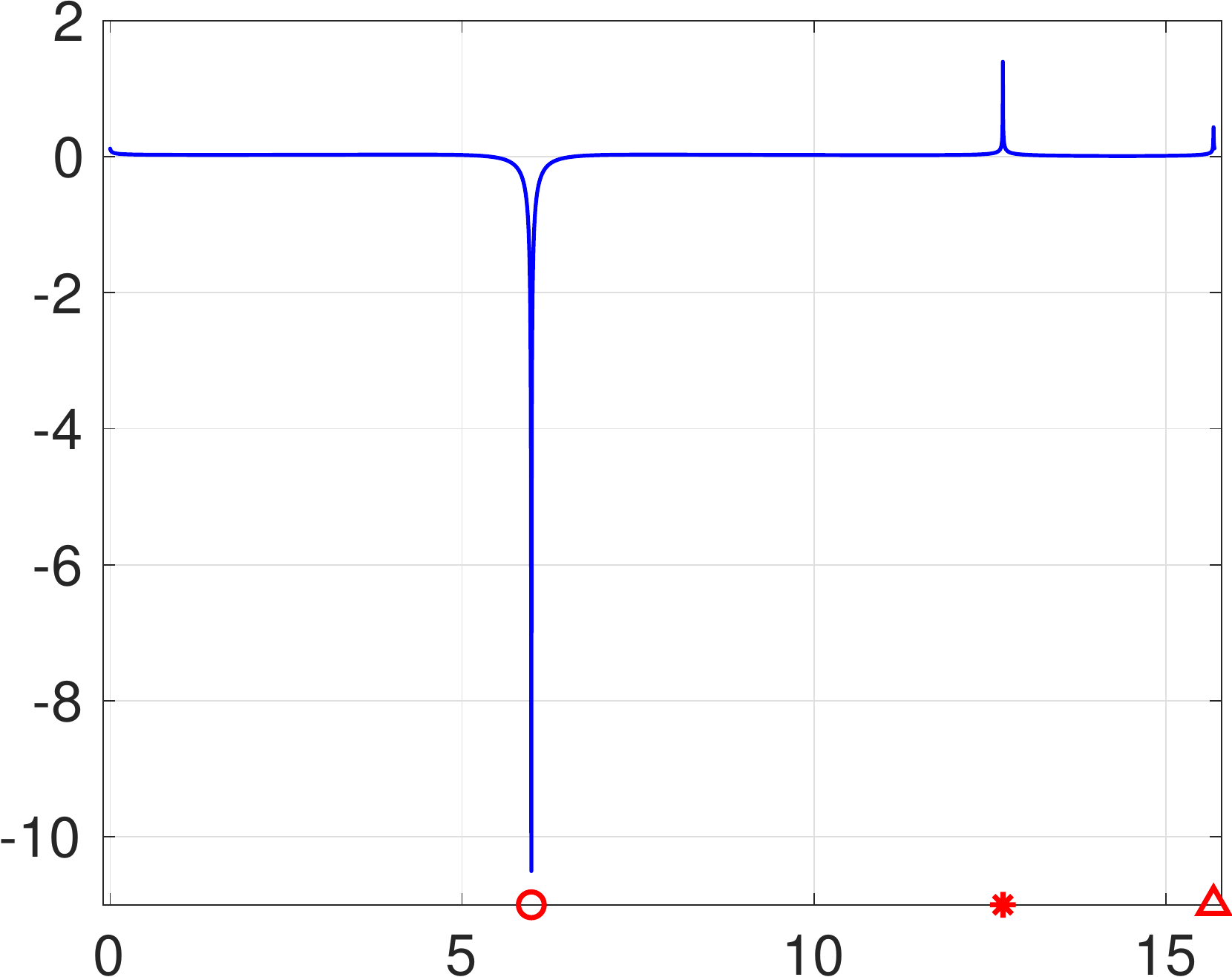}}
\subfigure[]{
\includegraphics[width=0.2\textwidth]{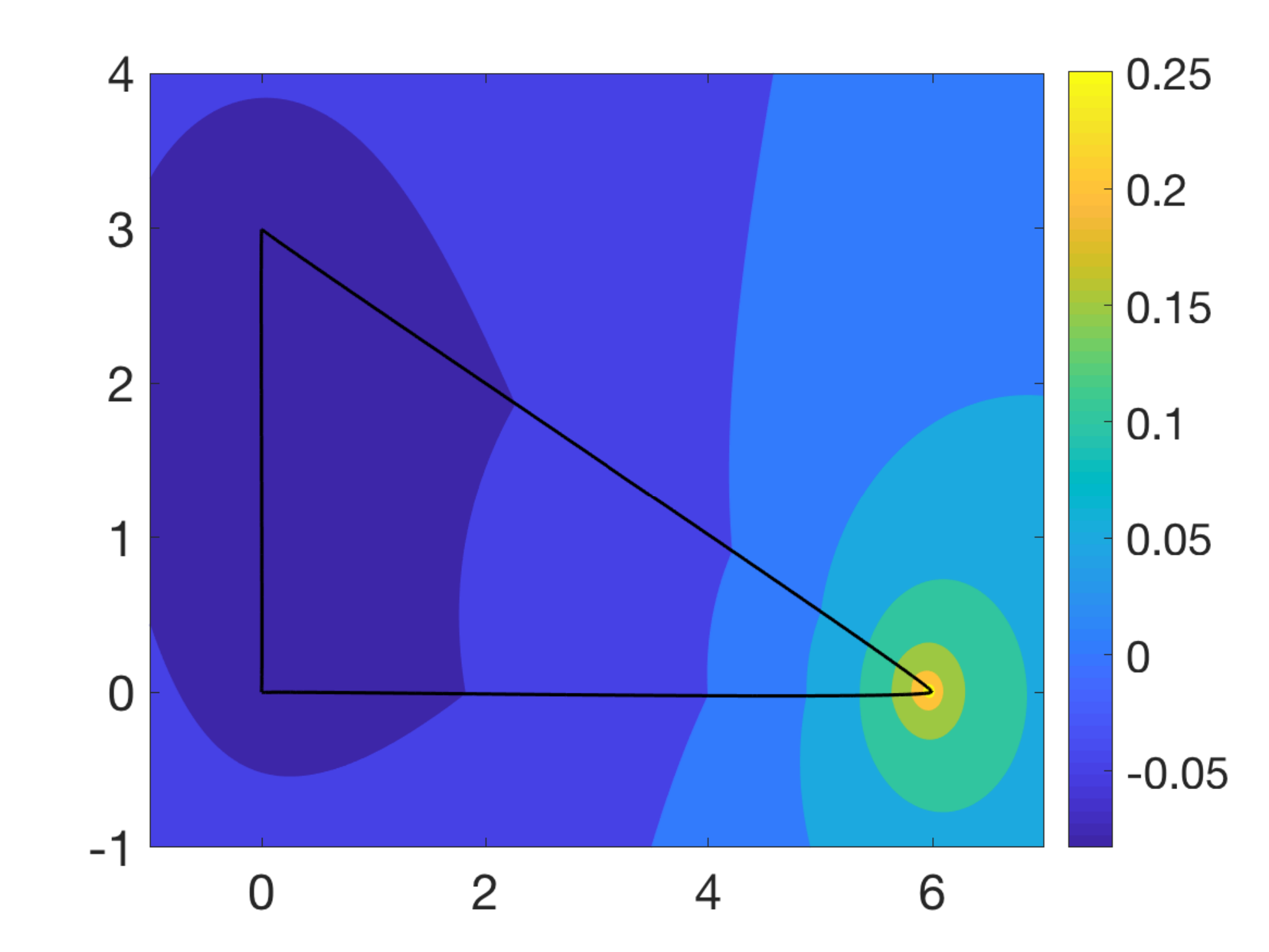}}\\
%\subfigure[]{
%\includegraphics[width=0.2\textwidth]{pos3_1.pdf}}\\
\subfigure[]{
\includegraphics[width=0.2\textwidth]{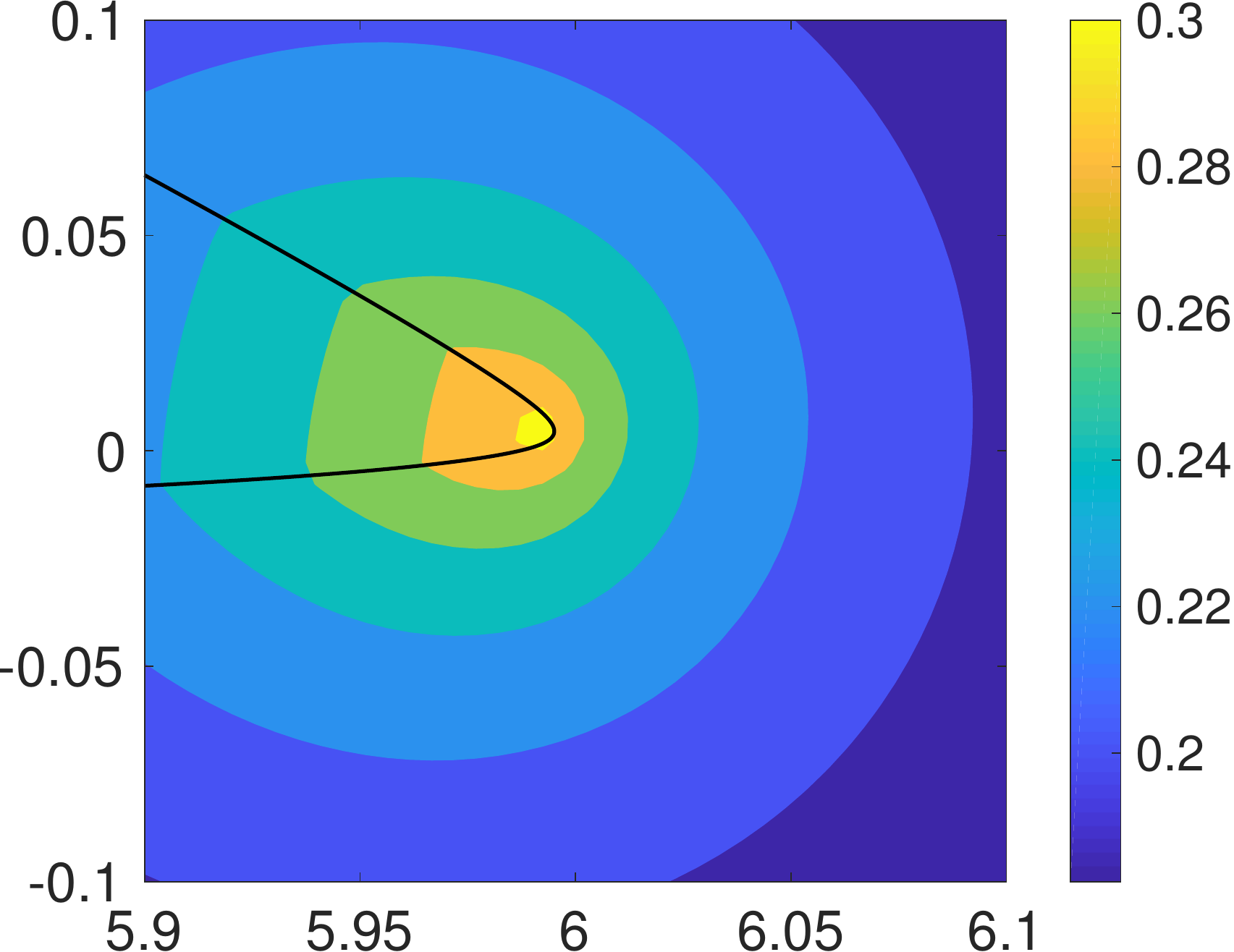}}
\subfigure[]{
\includegraphics[width=0.2\textwidth]{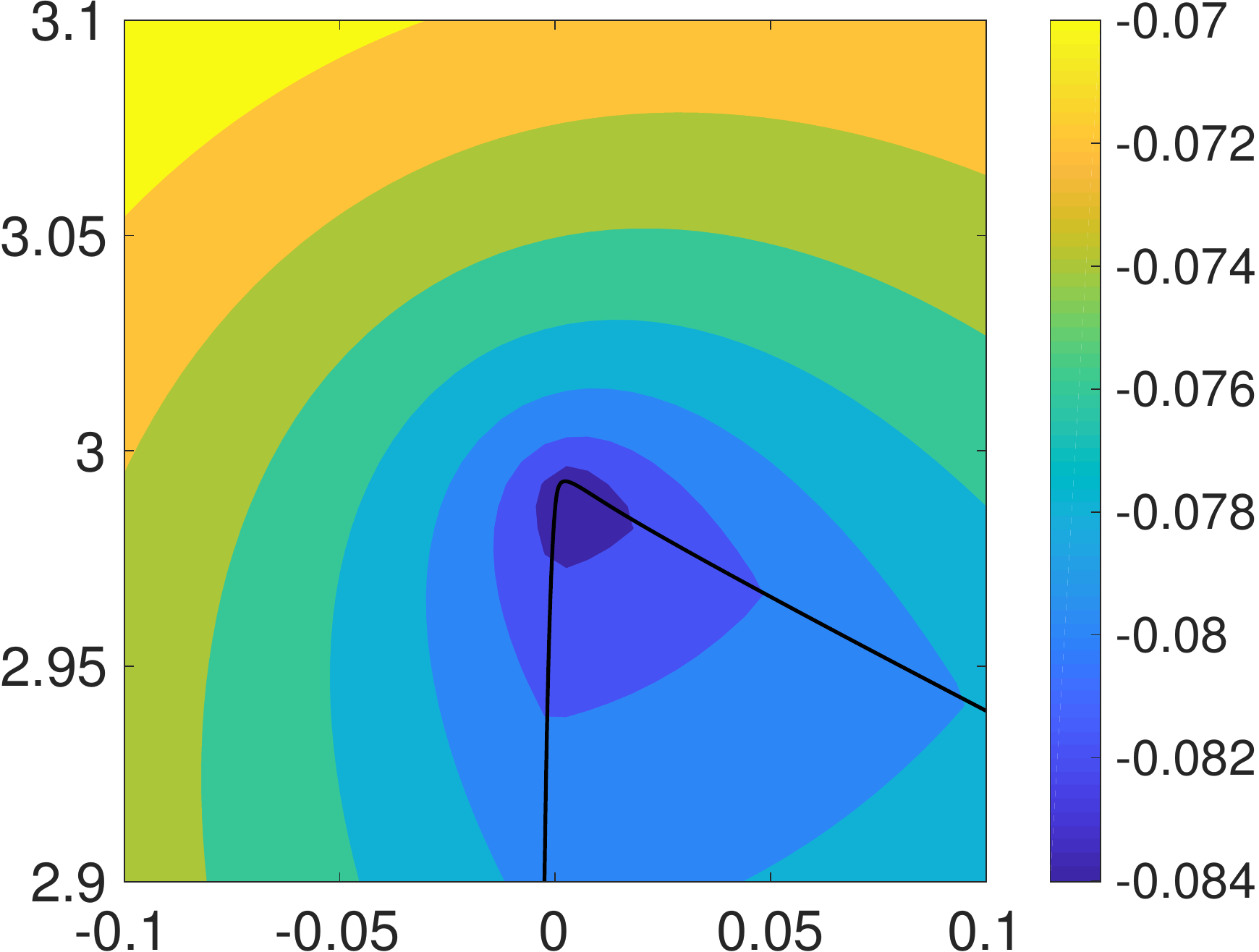}}
\subfigure[]{
\includegraphics[width=0.2\textwidth]{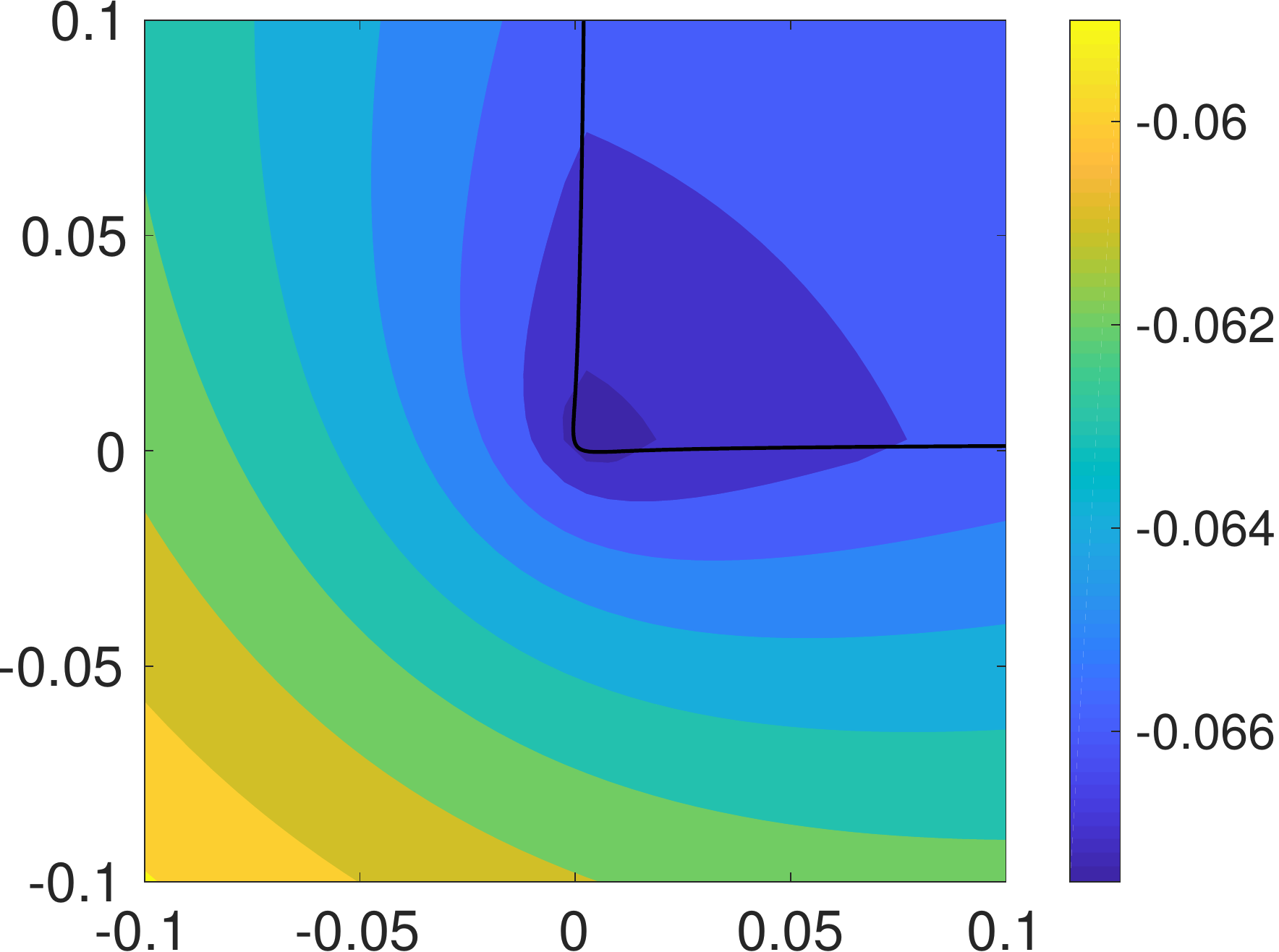}}\\
\caption{\label{fig35} 
(a). Eigenfunction with respect to the arc length associated to the eigenvalue $\lambda_1=0.3792$; (b).
The corresponding single-layer potential; (c), (d), (e). The single-layer potential around the three
points $x_{o}$, $x_{*}$ and $x_{\triangle}$, respectively. }
\end{figure}

Fig.~\ref{fig36} plots the eigenfunction, and its conormal derivative as well as the corresponding single-layer potential 
associated to the eigenvalue $\lambda_2=-0.3792$. It can be readily seen that even if the curvature at the points $x_{o}$ and $x_{*}$ has the relationship 
\[
 \kappa_{x_o}<\kappa_{x_*},
\]
from the figure $a$ and $b$, the values of the derivative of the eigenfunction at the high-curvature points $x_{o}$ and $x_{*}$ satisfy
\[
 d\varphi(x_o)>d\varphi(x_*).
\]
Therefore for the non-symmetric domain $D$, the larger curvature point does not yield the larger value of the conormal derivative of the eigenvalue for the negative eigenvalue.

\begin{figure}
\centering
\subfigure[]{
\includegraphics[width=0.2\textwidth]{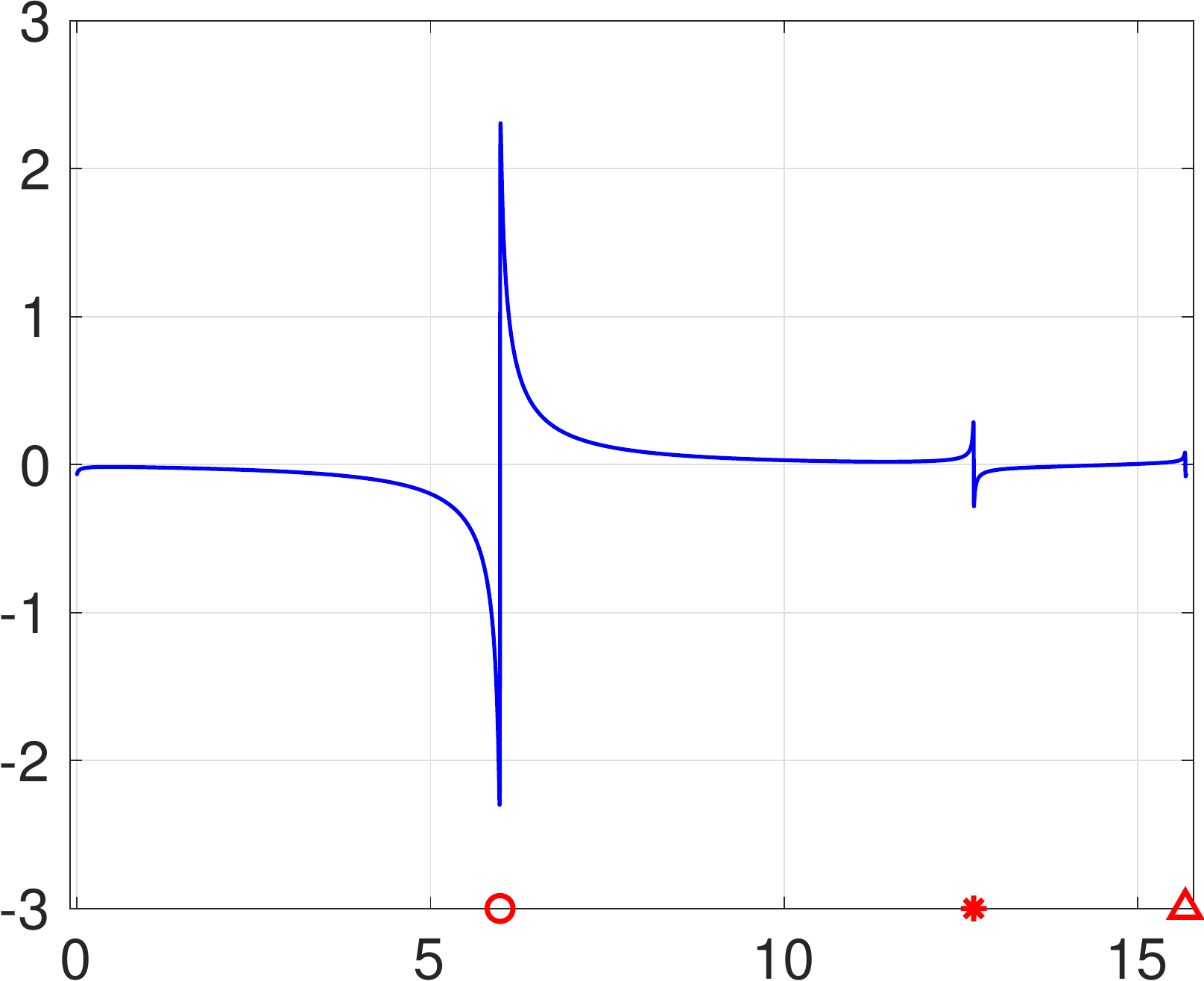}}
\subfigure[]{
\includegraphics[width=0.2\textwidth]{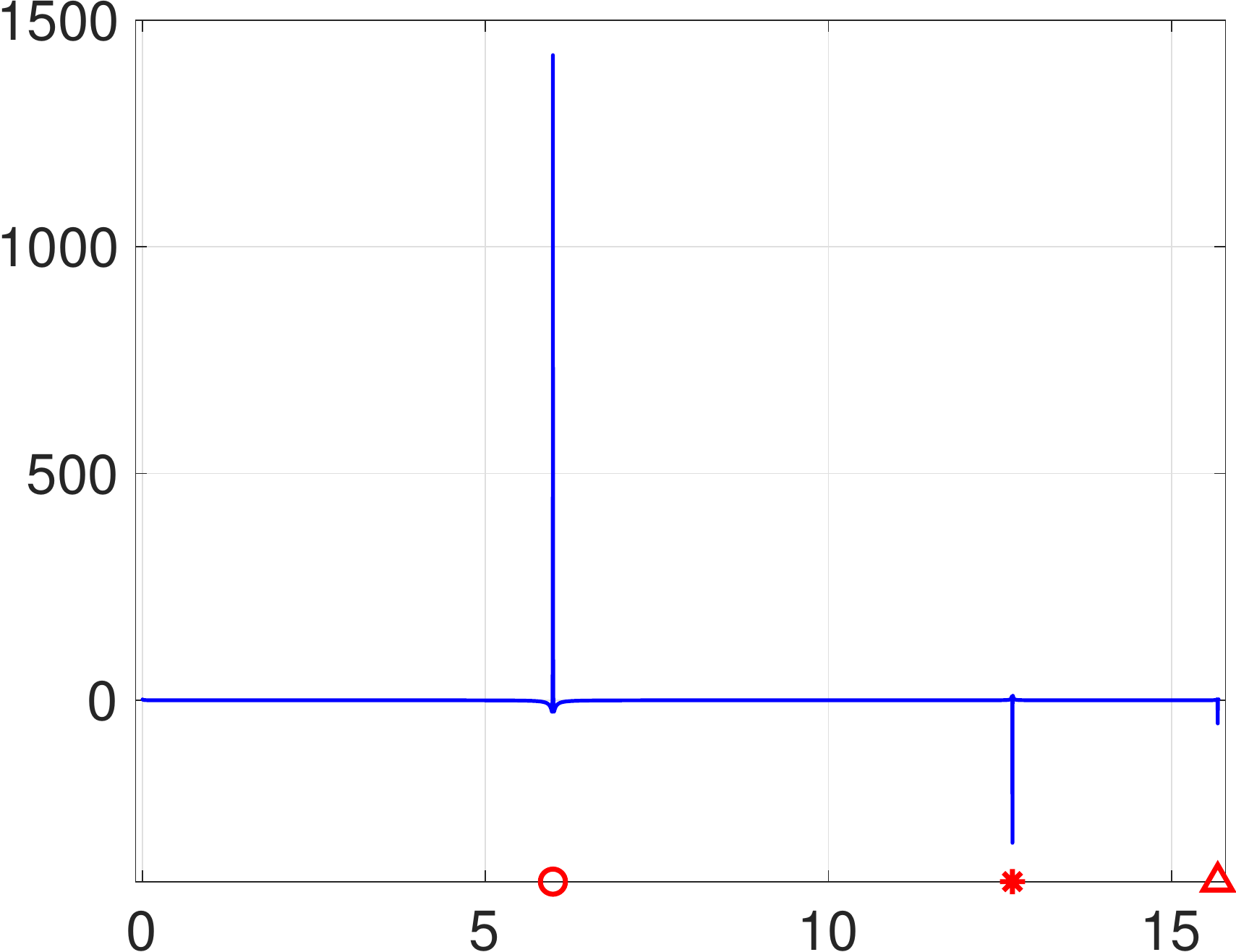}}
\subfigure[]{
\includegraphics[width=0.2\textwidth]{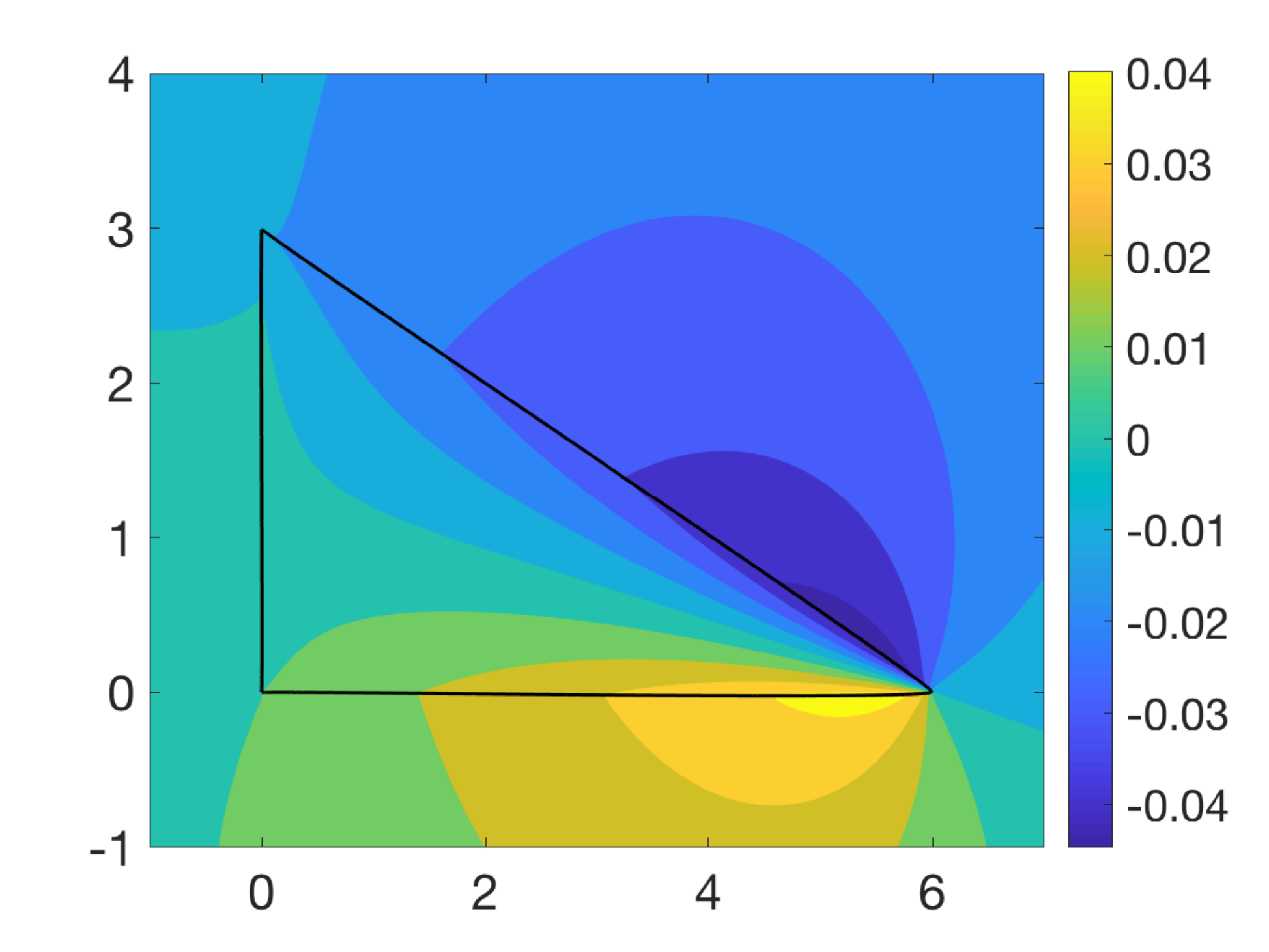}}\\
\subfigure[]{
\includegraphics[width=0.2\textwidth]{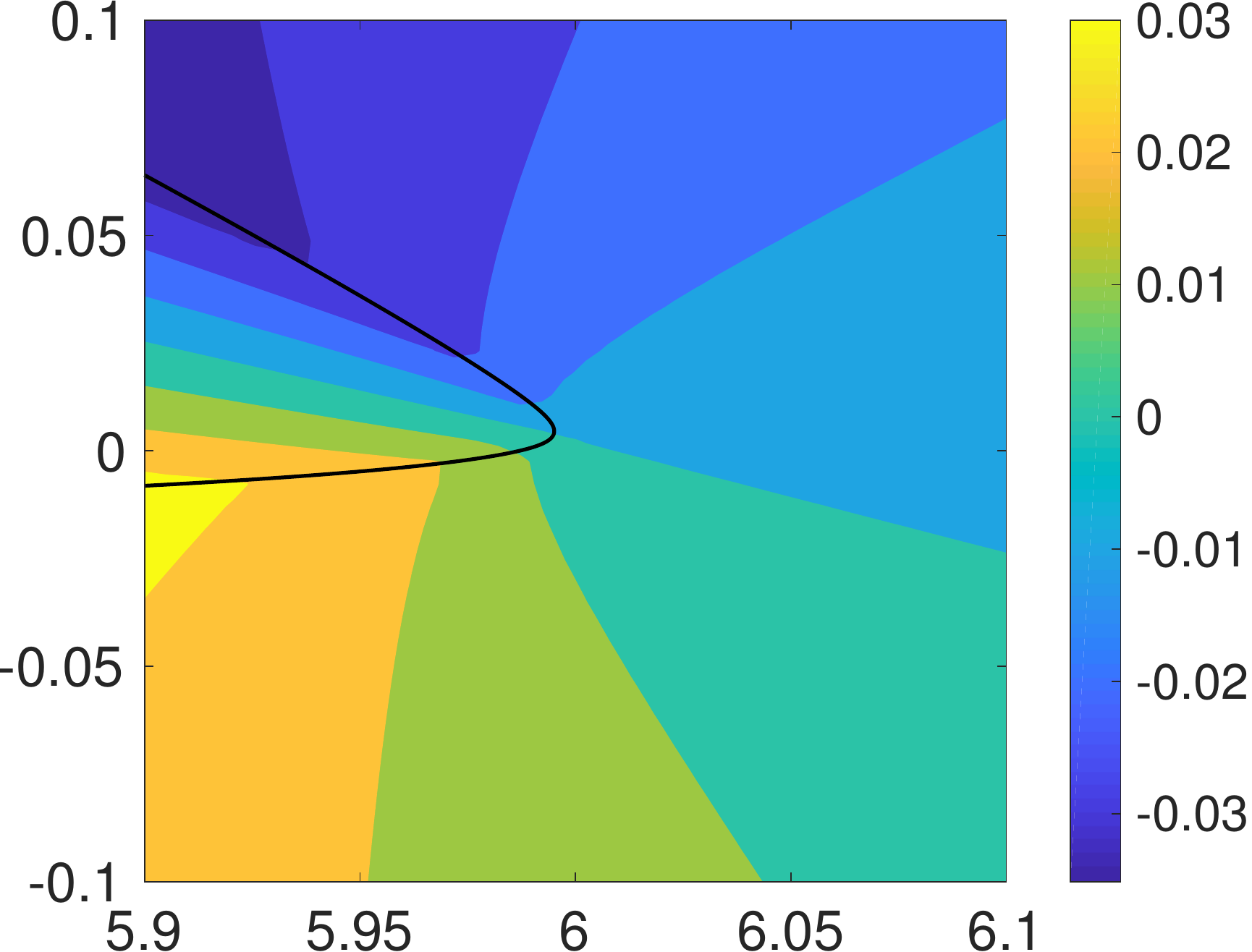}}
\subfigure[]{
\includegraphics[width=0.2\textwidth]{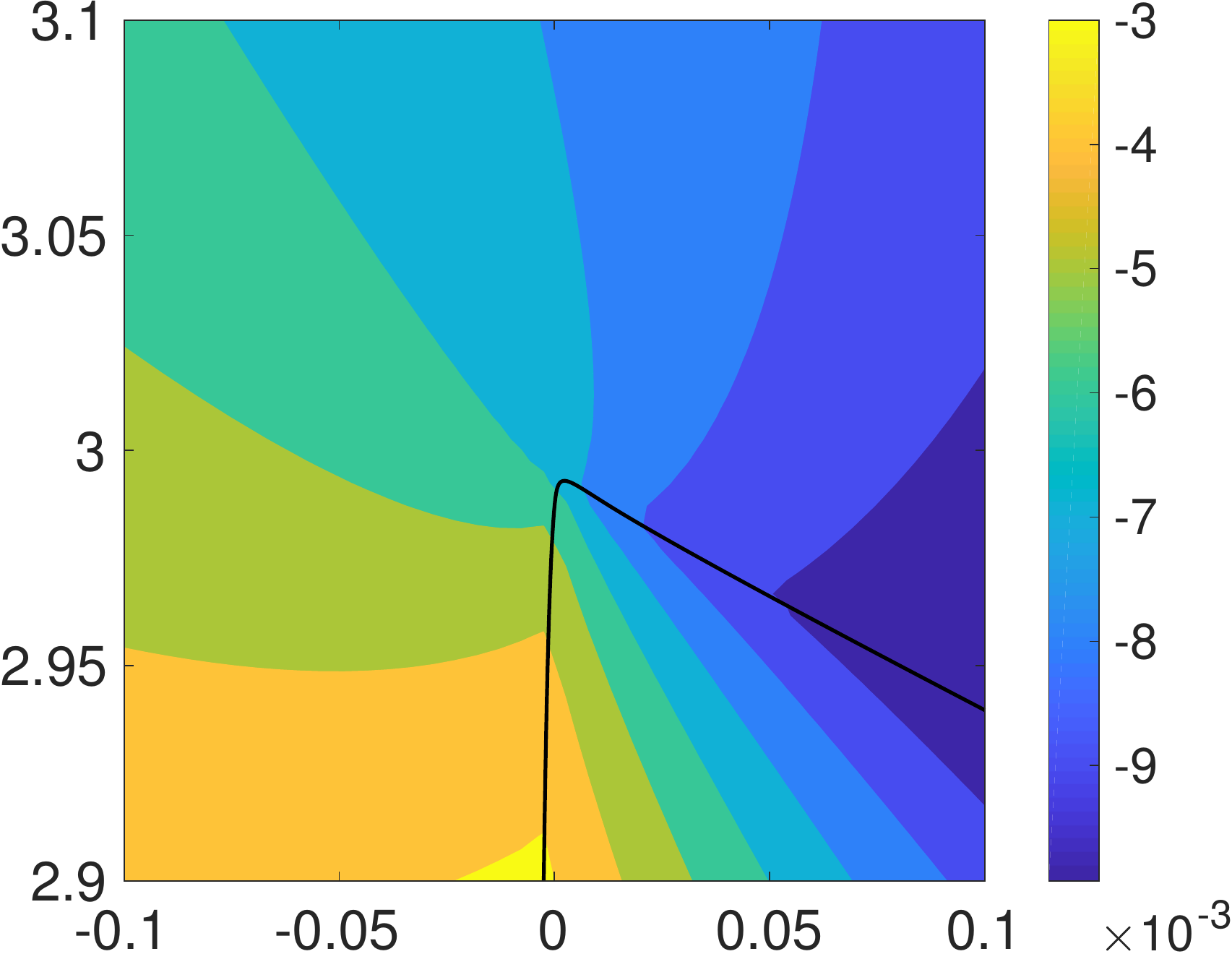}}
\subfigure[]{
\includegraphics[width=0.2\textwidth]{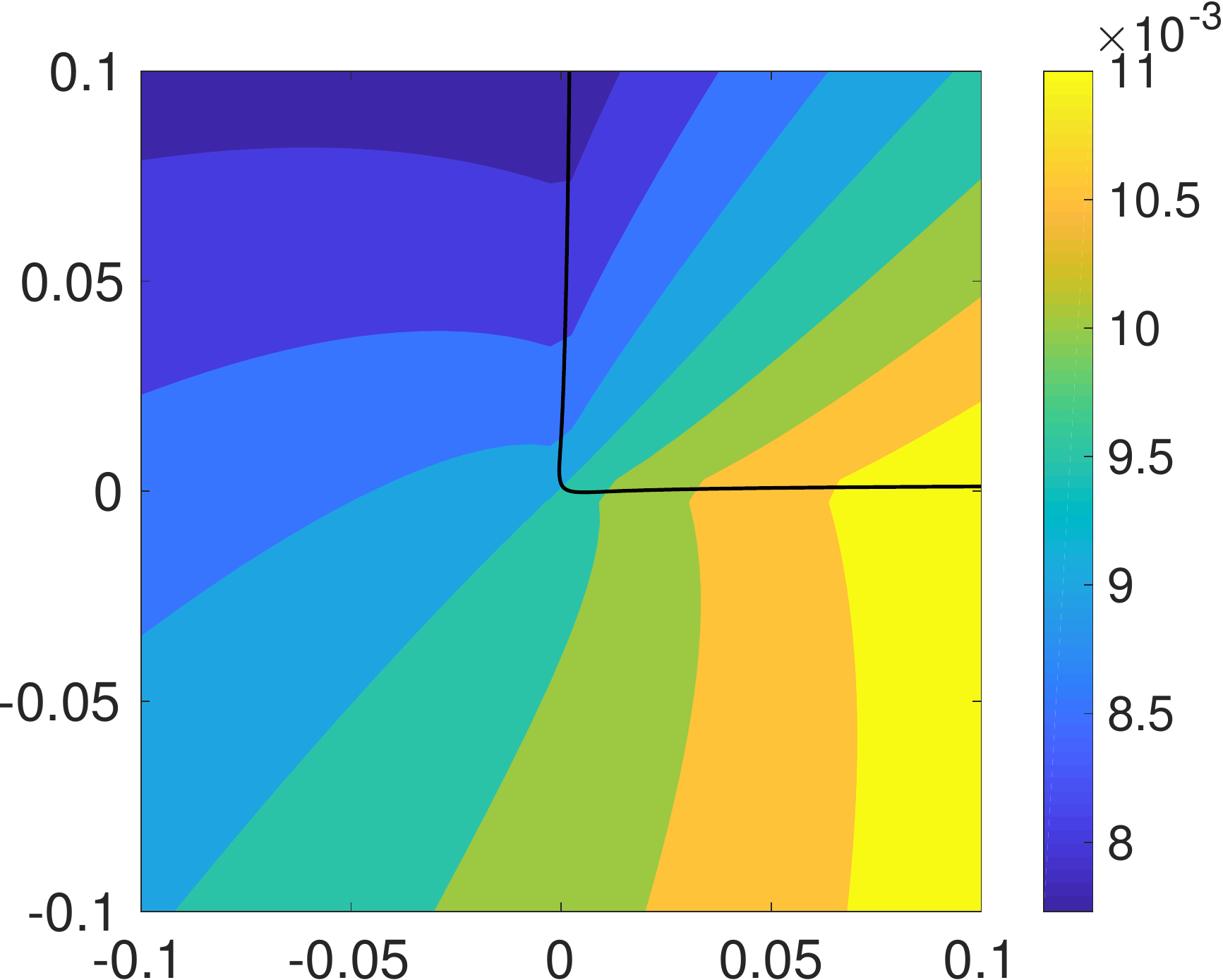}}\\
\caption{\label{fig36}
(a), (b). Eigenfunction and its conormal derivative with respect to the arc length associated to the eigenvalue $\lambda_2=-0.3792$; (c), (d), (e), (f).
The corresponding single-layer potential as well as their behaviours around the three points $x_{o}$, $x_{*}$ and $x_{\triangle}$, respectively.}
\end{figure}

\section{Concluding remarks}

In this paper, we show that the Neumann-Poincar\'e eigenfunctions possess certain delicate and intriguing geometric structures. The results are of independent interest
and significant importance in the spectral theory for the Neumann-Poincar\'e operator. It also opens up an exciting new field for further developments. Furthermore, the results 
can be used to explain the localization and geometrization phenomanon in the plasmon resonances, which is another novel and intriguing discovery made in the present article.
The localization and geometrization phenomenon might be used to produce super-resolution effect in wave imaging. To illustrate this, we present a last numerical example. 
\begin{figure}
\centering

\subfigure[]{
\includegraphics[width=0.2\textwidth]{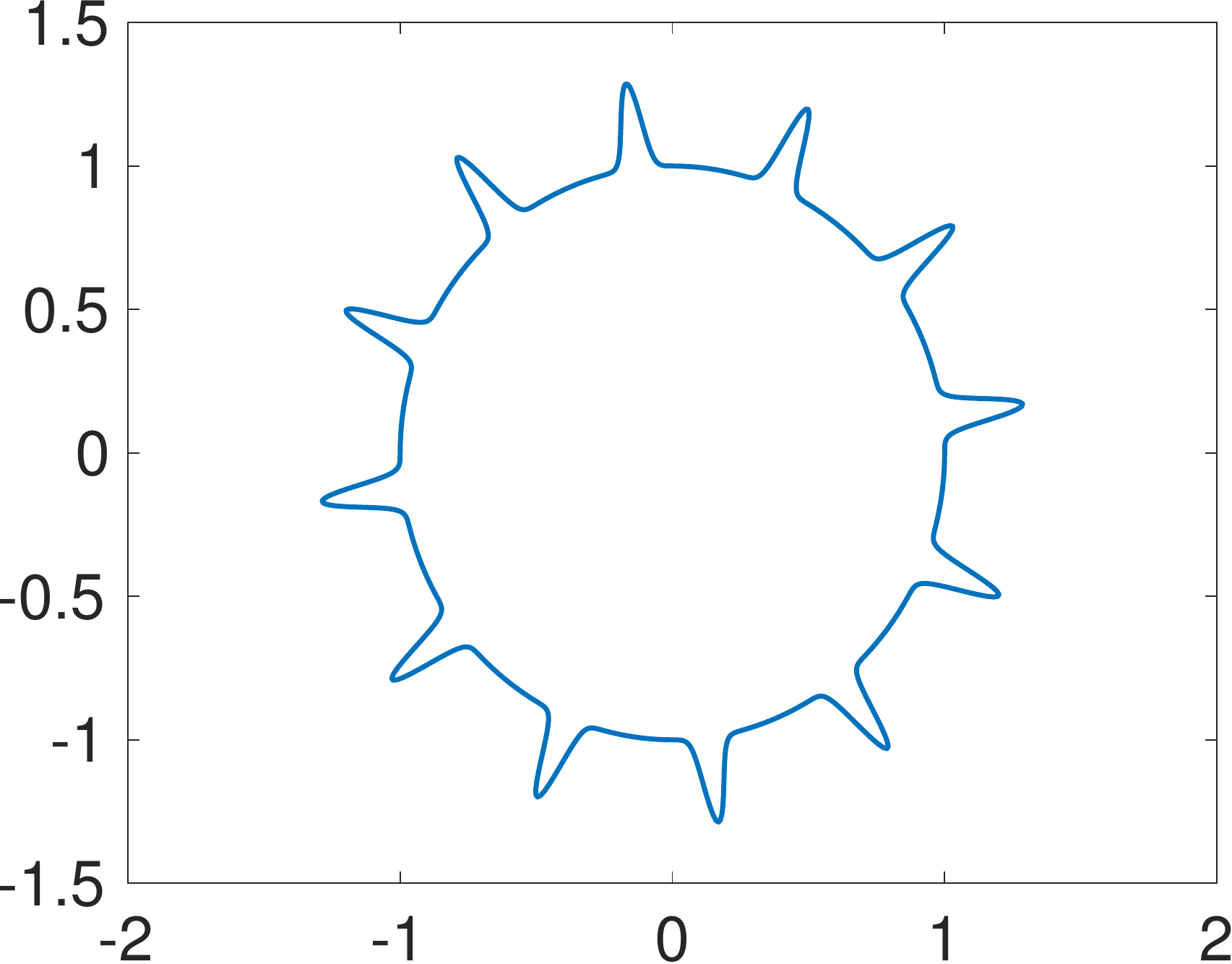}}
\subfigure[]{
\includegraphics[width=0.2\textwidth]{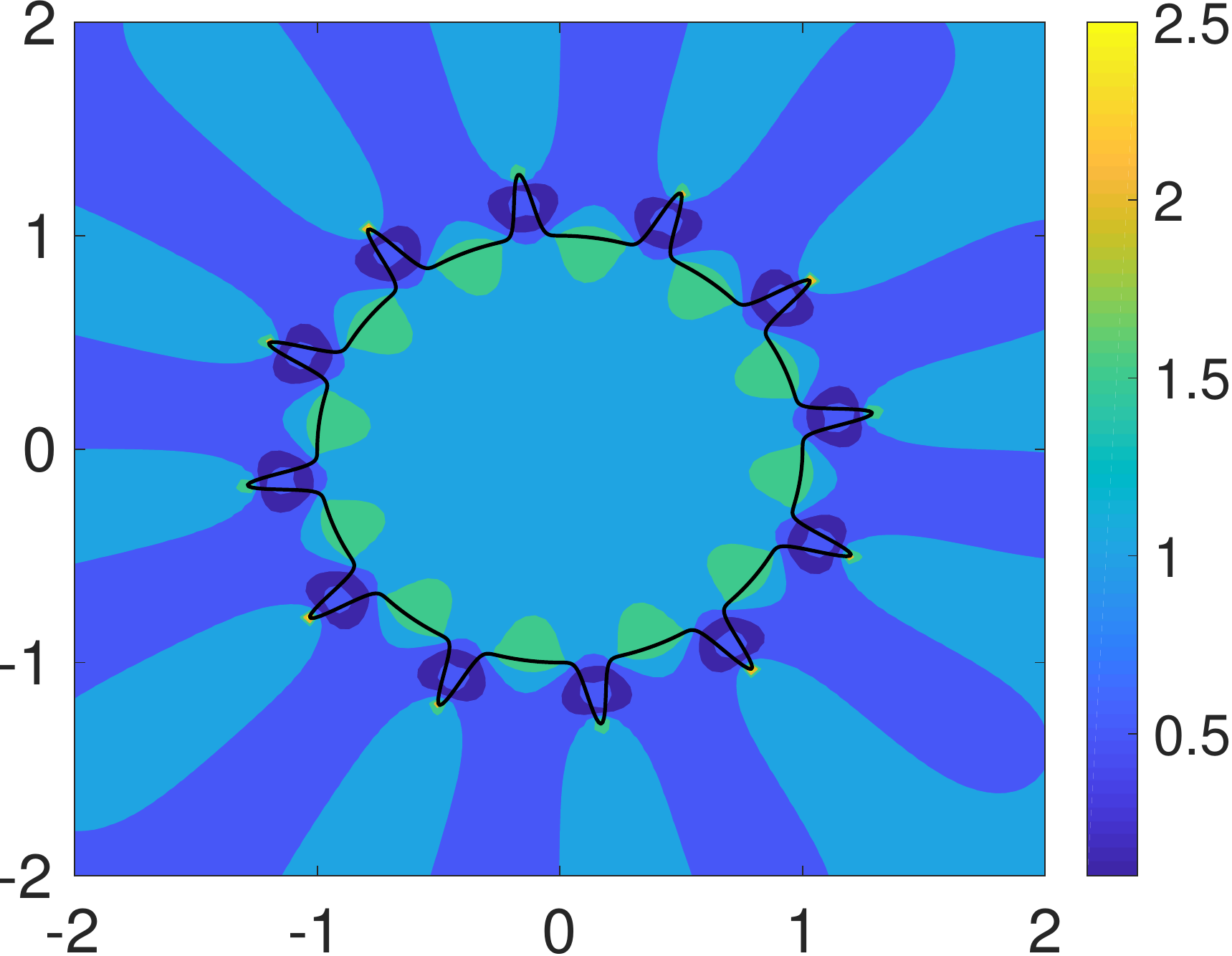}}
\subfigure[]{
\includegraphics[width=0.2\textwidth]{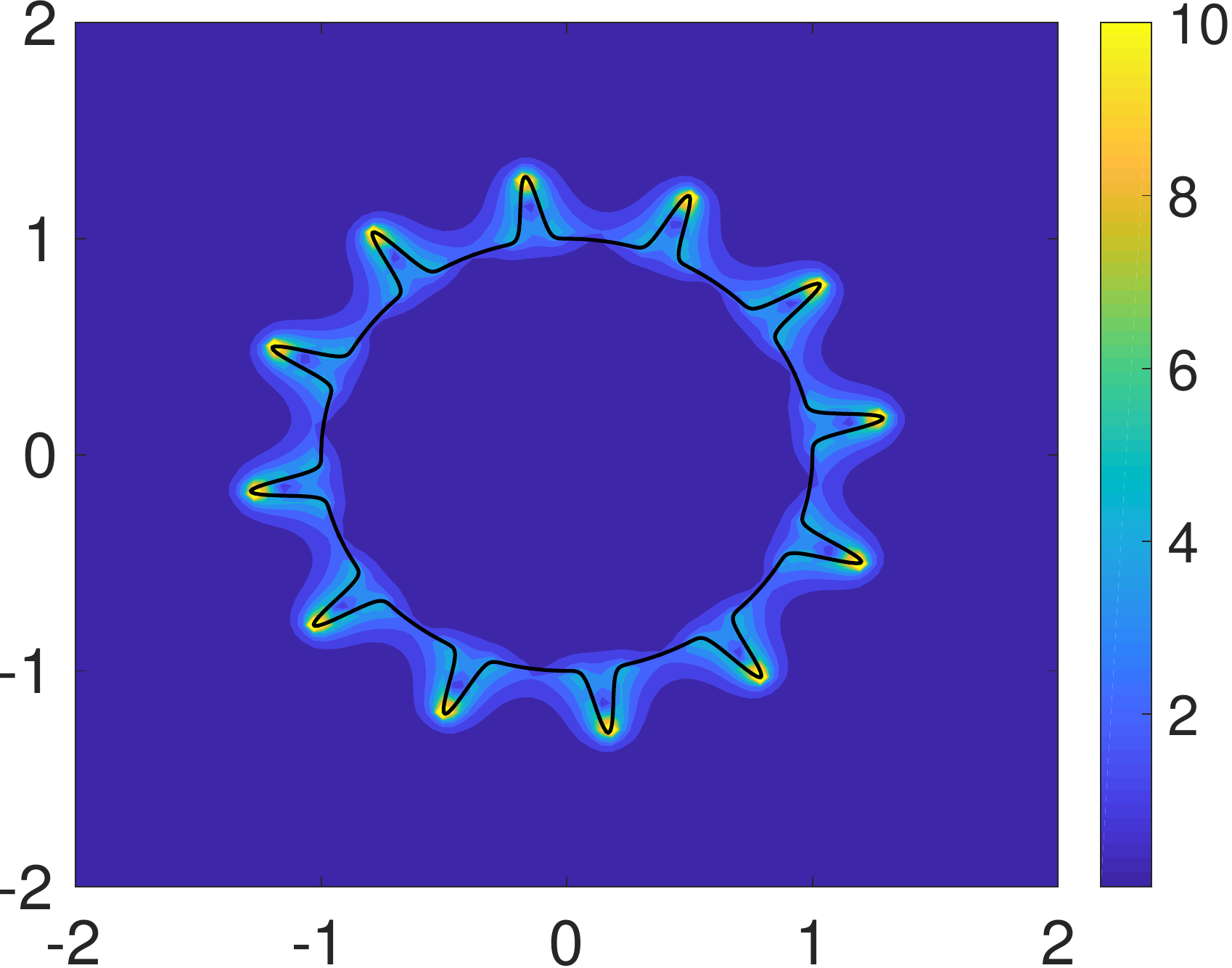}}\\
\caption{\label{fbstar} Left. A star-shaped plasmonic inclusion; Middle. Modulus of the resonant field; Right. Modulus of the gradient of the field. }
\end{figure}
Consider a plasmonic inclusion $D$ as plotted in Fig.~\ref{fbstar}, whose parametrization is given by 
\begin{equation}\label{eq:fr}
 r(\theta)= 1+0.0001e^{8\sin(12\theta)}.
\end{equation}  
There are $12$ cusped points and we denote that by $r_{max}\hat{x}_n$, $n=0,1,\cdots,11$, where
\[
 r_{\max}=1+0.0001e^8,\quad \hat{x}_n=\hat{x}(\pi/24 +n \pi/6),  \quad n=0,1,\cdots,11.
\]
It is easily seen that the distance $\tilde{d}$ between two cusped points is
\begin{equation}\label{eq:distance}
\tilde{d}= |r(\max)\hat{x}_n -r(\max)\hat{x}_{n+1}|=0.6719. 
\end{equation}
Now, let us choose the plasmon parameters inside the domain $D$ as follows,
\[
 \epsilon_c=-2.48907 \quad \mbox{and} \quad \delta=0.00001,
\]
which follows the rule in \eqref{eq:lambda1}, since one of the NP eigenvalues of the operator $K_{\partial D}^*$ can be numerically determined to be 
\[
 \lambda=0.21339.
\]
We also take an incident plane wave of the form \eqref{eq:nn2} with the wave number $k$ replaced by $k=0.01$. By our earlier discussion on the localization
and geometrization in plasmon resonances, it can be observed that strong resonant behaviours occur around the cusped points which possess very large curvatures. 
Those resonant behaviours can apparently be used to locate those cusp parts on the plasmonic inclusion. However, we note that the underlying 
wavelength is given by  
\[
 \lambda=2\pi/k=628.32,
\]
which is much bigger than $\tilde{d}$ in \eqref{eq:distance}. Hence, one might expect to have a certain super-resolution imaging effect. In this last example, we also note
that $k\cdot \mbox{diam}(D)\approx 0.02$, which means that the localization and geometrization phenomenon occurs at an even finer scale in the quasi-static regime.
Clearly, this is mainly due to the peculiar geometric structures of the NP eigenfunctions near high-curvature points. We shall investigate these intriguing problems in our 
future study.

%
%It is clearly that 
%\[
% \tilde{d}\ll w/2,
%\]
%which means that one can not see these cusped points clearly under such incident wave generally. However if choose the plasmon parameters inside the domain $D$ as follows,
%\[
% \epsilon_c=-2.48907 \quad \mbox{and} \quad \delta=0.00001,
%\]
%then the Fig.~\ref{fbstar} plots moduli of the total wave field which implies that one identifies these cusped points, since the total wave field at the cusped points is larger than other places.
%
%
%The reason why the plasmon parameter $\epsilon_c=-2.48907$ follows from \eqref{eq:lambda1} and one of the eigenvalue of the NP operator $K_{\partial D}^*$ is 
%\[
% \lambda=0.21339.
%\]

\section*{Acknowledgement}

The work of H Liu was supported by the FRG fund from Hong Kong Baptist University and the Hong Kong RGC grants (projects 12302017 and 12302018).


\begin{thebibliography}{99}


%\bibitem{Acm13}
%H.~Ammari, G.~Ciraolo, H.~Kang, H.~Lee, and G.W. Milton, \emph{Anomalous
%  localized resonance using a folded geometry in three dimensions,}, Proc. R.
%  Soc. A, \textbf{469} (2013), 20130048.

\bibitem{Ack13}
H.~Ammari, G.~Ciraolo, H.~Kang, H.~Lee, and G.W. Milton, \emph{Spectral theory of a Neumann-Poincar\'{e}-type operator and
  analysis of cloaking due to anomalous localized resonance}, Arch. Ration.
  Mech. Anal., \textbf{208} (2013), 667--692.

\bibitem{Ack14}
H.~Ammari, G.~Ciraolo, H.~Kang, H.~Lee, and G.W. Milton, \emph{Spectral theory of a Neumann-Poincar\'{e}-type operator and analysis of cloaking due to anomalous localized resonance {I}{I}},
  Contemporary Math., \textbf{615} (2014), 1--14.

\bibitem{ADM} {H.~Ammari, Y. Deng and P. Millien}, \emph{Surface plasmon resonance of nanoparticles and applications in imaging},
Arch. Ration. Mech. Anal., \textbf{220} (2016), 109--153.

\bibitem{AMRZ} {H.~Ammari, P. Millien, M. Ruiz and H. Zhang}, \emph{Mathematical analysis of plasmonic nanoparticles: the scalar case}, Archive for Rational Mechanics and Analysis, {\bf 224} (2017), 597--658. 

\bibitem{ARYZ} {H.~Ammari, M. Ruiz, S. Yu and H. Zhang}, \emph{Mathematical analysis of plasmonic resonances for nanoparticles: the full Maxwell equations}, preprint, arXiv:1511.06817

%\bibitem{AK} K. Ando and H. Kang, {\it Analysis of plasmon resonance on smooth domains using spectral properties of the Neumann-Poincar\'e operator}, J. Math. Anal. Appl., \textbf{435} (2016), 162--178.


\bibitem{AKKY} {K. Ando, Y. Ji, H. Kang, K. Kim and S. Yu}, \emph{Spectral properties of the Neumann-Poincar\'e operator and cloaking by anomalous localized resonance for the elastostatic system}, preprint, European J. Appl. Math., in press, 2017.

\bibitem{AKKY2}  {K. Ando, Y. Ji, H. Kang, K. Kim and S. Yu}, {\it Cloaking by anomalous localized resonance for linear elasticity on a coated structure}, SIAM J. Math. Anal., in press, 2017.


\bibitem{AKL} {K. Ando, H. Kang and H. Liu}, \emph{Plasmon resonance with finite frequencies: a validation of the quasi-static approximation for diametrically small inclusions}, SIAM J. Appl. Math., \textbf{76} (2016), 731--749.

\bibitem{AKM1} {K. Ando, H. Kang and Y. Miyanishi}, {\it Elastic Neumann--Poincar\'e operators on three dimensional smooth domains: Polynomial compactness and spectral structure}, Int. Math. Res. Notices, in press, 2017.

\bibitem{AKM2} {K. Ando, H. Kang and Y. Miyanishi}, {\it Spectral structure of elastic Neumann--Poincar\'e operators}, preprint

\bibitem{BL} E. Bl{\aa}sten and H. Liu, {\it  Scattering by curvatures, radiationless sources, transmission eigenfunctions and inverse scattering problems}, arXiv: 1808.01425

\bibitem{Bos10}
G.~Bouchitt\'{e} and B.~Schweizer, \emph{Cloaking of small objects by anomalous localized resonance}, Quart. J. Mech. Appl. Math., \textbf{63} (2010), 438--463.

\bibitem{Brl07} O.P. Bruno and S.~Lintner, \emph{Superlens-cloaking of small dielectric bodies in the quasistatic regime}, J. Appl. Phys., \textbf{102} (2007), 124502.

 \bibitem{CKKL} D. Chung, H. Kang, K. Kim and H. Lee, {\it Cloaking due to anomalous localized resonance in plasmonic structures of confocal ellipses}, SIAM J. Appl. Math., \textbf{74} (2014), no. 5, 1691--1707.

 %\bibitem{CK} {D.~Colton and R.~Kress}, {\it Inverse
%Acoustic and Electromagnetic Scattering Theory}, 2nd Edition,
%Springer-Verlag, Berlin, 1998.

 \bibitem{DLL} Y. Deng, H. Li and H. Liu, {\it On spectral properties of Neumann-Poincare operator and plasmonic cloaking in 3D elastostatics}, J. Spectral Theory, in press.
 
 %\bibitem{ELLWarXiv} E. Bl{\aa}sten, H. Li, H. Liu and Y. Wang, {\it Localization and geometrization in plasmon resonances and geometric structures of Neumann-Poincaré eigenfunctions},
 %arXiv:1809.08533
 
 \bibitem{HKL} J. Helsing, H. Kang and M. Lim, {\it Classification of spectra of the Neumann--Poincar\'{e} operator on planar domains with corners by resonance}, Ann. I. H. Poincare-AN, {\bf 34} (2017), 991--1011. 
 
 \bibitem{HP} J. Helsing and K. M. Perfekt, {\it The spectra of harmonic layer potential operators on domains with rotationally symmetric conical points}, J. Math. Pures Appl., in press, DOI: 10.1016/j.matpur.2017.10.012
 
 \bibitem{JK18} Y. Ji and H. Kang, {\it A concavity condition for existence of a negative Neumann-Poincar\'e eigenvalue in three dimensions}, arXiv:1808.10621
 
\bibitem{KK18} H. Kang and D. Kawagoe, {\it Surface Riesz transforms and spectral property of elastic Neumann--Poinca\'e operators on less smooth domains in three dimensions}, arXiv:1806.02026

\bibitem{KLY} H. Kang, M. Lim and S. Yu, {\it Spectral resolution of the Neumann-Poincar\'{e} operator on intersecting disks and analysis of plasmon resonance}, Arch. Rati. Mech. Anal., {\bf 226} (2017), 83--115.

\bibitem{KLO} H. Kettunen, M. Lassas and P. Ola, {\it On absence and existence of the anomalous localized resonace without the quasi-static approximation}, SIAM J. Appl. Math., {\bf 78} (2018), 609--628. 

\bibitem{KM} {D. M. Kochmann and G. W. Milton}, \emph{Rigorous bounds on the effective moduli of
composites and inhomogeneous bodies with negative-stiffness phases}, J. Mech. Phys.
Solids, {\bf 71} (2014), 46--63.

\bibitem{Klsap} {R.V. Kohn, J.Lu, B.~Schweizer and M.I. Weinstein}, \emph{A variational
  perspective on cloaking by anomalous localized resonance}, Comm. Math. Phys., {\bf 328} (2014), 1--27.

 % \bibitem{Kup} V. D. Kupradze, {\it Three-dimensional Problems of the Mathematical Theory of Elasticity and Thermoelasticity}, Amsterdam, North-Holland, 1979.

\bibitem{LLBW} {R.S. Lakes, T. Lee, A. Bersie, and Y. Wang}, \emph{Extreme damping in composite materials
with negative-stiffness inclusions}, Nature, {\bf 410} (2001), 565--567.

\bibitem{LiLiu2d} {H. Li and H. Liu}, \emph{On anomalous localized resonance for the elastostatic system}, SIAM J. Math. Anal., {\bf 48} (2016), 3322--3344.

\bibitem{LiLiu3d} {H. Li and H. Liu}, \emph{On three-dimensional plasmon resonance in elastostatics},
Annali di Matematica Pura ed Applicata, {\bf 196} (2017), 1113--1135.

\bibitem{s25} {H. Li and H. Liu}, On anomalous localized resonance and plasmonic cloaking beyond the quasistatic limit, Proceedings A, at press, 2018. 

\bibitem{LLL} {H. Li, J. Li and H. Liu}, \emph{On quasi-static cloaking due to anomalous localized resonance in $\mathbb{R}^3$}, SIAM J. Appl. Math., {\bf 75}  (2015), no. 3, 1245--1260.

\bibitem{LLL2} {H. Li, J. Li and H. Liu}, {\it On novel elastic structures inducing plariton resonances with finite frequencies and cloaking due to anomalous localized resonance}, {J. Math. Pures Appl.}, DOI:10.1016/j.matpur.2018.06.014

\bibitem{LLLW} H. Li, S. Li, H. Liu and X. Wang, {\it Analysis of electromagnetic scattering from plasmonic inclusions at optical frequencies and applications}, arXiv:1804.09517

%\bibitem{LASS} {A. E. H. Love}, \emph{A Treatise on the Mathematical Theory of Elasticity}, 4th Edition, Cambridge University Press, 2013.

\bibitem{GWM1} {R.C. McPhedran, N.-A.P. Nicorovici, L.C. Botten and G.W. Milton}, \emph{Cloaking by plasmonic
resonance among systems of particles: cooperation or combat?} C.R. Phys., {\bf 10} (2009), 391--399.

\bibitem{GWM2} {D. A. B. Miller}, \emph{On perfect cloaking}, Opt. Express, {\bf 14} (2006), 12457--12466.

\bibitem{GWM3} {G.W. Milton and N.-A.P. Nicorovici}, \emph{On the cloaking effects associated with anomalous
localized resonance}, Proc. R. Soc. A, {\bf 462} (2006), 3027--3059.

\bibitem{GWM4} {G.W. Milton, N.-A.P. Nicorovici, R.C. McPhedran, K. Cherednichenko and Z. Jacob},
\emph{Solutions in folded geometries, and associated cloaking due to anomalous resonance}, New. J. Phys., {\bf 10} (2008), 115021.

\bibitem{Ned} {J. C. N\'ed\'elec}, \emph{Acoustic and Electromagnetic Equations}, Applied Mathematical Sciences 144, 2001, Springer-Verlag, New York. 

%\bibitem{GWM5} {G.W. Milton, N.-A.P. Nicorovici, R.C. McPhedran, and V.A. Podolskiy}, \emph{ Proof of
%superlensing in the quasistatic regime, and limitations of superlenses in this regime due
%to anomalous localized resonance}, Proc. R. Soc. A, {\bf 461} (2005), 3999--4034.

%\bibitem{Ngu1} H. Nguyen, \emph{Cloaking an arbitrary object via anomalous localized resonance: the cloak is independent of the object}, SIAM J. Math. Anal., {\bf 49} (4), 3208--3232.  

\bibitem{Ngu2} H. Nguyen, \emph{ Cloaking via anomalous localized resonance for doubly complementary media in the finite frequency regime }, arXiv:1511.08053. 

\bibitem{GWM6} {N.-A.P. Nicorovici, R.C. McPhedran, S. Enoch and G. Tayeb}, \emph{Finite wavelength cloaking
by plasmonic resonance}, New. J. Phys., {\bf 10} (2008), 115020.

\bibitem{GWM7} {N.-A.P. Nicorovici, R.C. McPhedran and G.W. Milton}, \emph{Optical and dielectric properties
of partially resonant composites}, Phys. Rev. B, {\bf 49} (1994), 8479--8482.

\bibitem{GWM8} {N.-A.P. Nicorovici, G.W. Milton, R.C. McPhedran and L.C. Botten}, \emph{Quasistatic cloaking
of two-dimensional polarizable discrete systems by anomalous resonance}, Optics
Express, {\bf 15} (2007), 6314--6323.

\bibitem{GWM9} G.W. Milton and N.-A.P. Nicorovici, \emph{On the cloaking effects associated
 with anomalous localized resonance}, Proc. R. Soc. A, \textbf{462} (2006),
  3027--3059.

\bibitem{Pen1} J. B. Pendry, {\it Negative refraction makes a perfect lens}, Phys. Rev. Lett., {\bf 85} (2000), 3966.

\bibitem{Pen2} D. R. Smith, J. B. Pendry and M. C. K. Wiltshire, {\it Metamaterials and negative refractive index}, Science, {\bf 305} (2004), 788--792.

 \bibitem{Ves} V. G. Veselago, {\it The electrodynamics of substances with simultaneously negative values of $\epsilon$ and $\mu$}, Sov. Phys. Usp., {\bf 10} (1968), 509.



\end{thebibliography}
\end{document}